\documentclass[a4paper,11pt,reqno]{amsart}

\usepackage[utf8]{inputenc}
\usepackage[left=2.5cm,right=2.5cm,top=3cm, bottom=3cm]{geometry}
\usepackage{amsmath, amssymb, mathrsfs, amsthm}
\usepackage{tikz, tikz-cd}
\usepackage{mathtools}
\usepackage{graphicx}
\usepackage{quiver}
\usepackage[hidelinks]{hyperref}
\usepackage{stmaryrd} 
\usepackage[style=alphabetic,sorting=nyt]{biblatex}
\usepackage{url}

\newcommand{\Q}{\mathbb Q}

\newcommand{\R}{\mathbb R}
\newcommand{\Z}{\mathbb Z}

\def\Spec{\operatorname{Spec}}

\newtheoremstyle{mytheorem}
{}
{}
{\itshape}
{}
{\scshape}
{:}
{.5em}
{}%

\newtheoremstyle{mydefi}
{}
{}
{}
{}
{\scshape}
{:}
{.5em}
{}%

\theoremstyle{mytheorem}
\newtheorem{teo}{Theorem}
\numberwithin{teo}{section}
\newtheorem{prop}[teo]{Proposition}
\newtheorem{lem}[teo]{Lemma}
\newtheorem{cor}[teo]{Corollary}

\theoremstyle{mydefi}
\newtheorem{ex/}[teo]{Example}
\newenvironment{ex}
  {%
   \pushQED{\qed}\begin{ex/}}
  {\popQED\end{ex/}}

\newtheorem{rmk/}[teo]{Remark}
\newenvironment{rmk}
  {%
   \pushQED{\qed}\begin{rmk/}}
  {\popQED\end{rmk/}}

\newtheorem{defi/}[teo]{Definition}
\newenvironment{defi}
  {%
   \pushQED{\qed}\begin{defi/}}
  {\popQED\end{defi/}}

\newtheorem{ass/}[teo]{Assumption}
\newenvironment{ass}
  {%
   \pushQED{\qed}\begin{ass/}}
  {\popQED\end{ass/}}

\makeatletter
\renewenvironment{proof}[1][\proofname]{\par
  \vspace{-\topsep+6pt}
  \pushQED{\qed}%
  \normalfont
  \topsep0pt \partopsep0pt 
  \trivlist
  \item[\hskip\labelsep
        \scshape
    #1\@addpunct{.}]\ignorespaces
}{%
  \popQED\endtrivlist\@endpefalse
  \addvspace{6pt plus 6pt} 
}
\makeatother

\title[Intrinsic Non-Reductive GIT]{Intrinsic Non-Reductive Geometric Invariant Theory and Moduli Spaces of Principal Bundles with Non-Reductive Structure Groups}
\author{George Cooper and Ludvig Modin}
\address{G. Cooper: \textbf{Centro di Ricerca Matematica Ennio De Giorgi, Scuola Normale Superiore}, Collegio Puteano, Piazza dei Cavalieri, 3, 56126 Pisa (PI), REPUBBLICA ITALIANA \emph{(former)}} 
\email{george.cooper@sns.it}
\address{L. Modin: \textbf{Institut f\"ur Algebraische Geometrie, Fakult\"at f\"ur Mathematik und Physik, Gottfried Wilhelm Leibniz Universit\"at Hannover}, Welfengarten 1, 30167 Hannover, BUNDESREPUBLIK DEUTSCHLAND}
\email{modin@math.uni-hannover.de}

\begin{document}
\emergencystretch 3em 

\begin{abstract}
    We prove the existence of moduli spaces for algebraic stacks with a weak action of $\Theta = [\mathbb{A}^1/\mathbb{G}_m]$ (including stacks of filtered objects and unstable $\Theta$-strata) for which the center admits a good moduli space and the retraction to the center has flat, finitely presented relative inertia, providing a stack-intrinsic generalization of the existence of non-reductive Geometric Invariant Theory quotients. As an application, we construct quasi-projective coarse moduli spaces of principal bundles over a smooth projective curve with non-reductive structure groups, including moduli spaces of Ramanathan unstable principal bundles of fixed Harder--Narasimhan type.
\end{abstract}

\maketitle

\setcounter{tocdepth}{1}
\tableofcontents

\section{Introduction}

Building on the results of \cite{Modin}, this article serves to provide a stack-intrinsic generalization of the existence of non-reductive Geometric Invariant Theory (GIT) quotients, providing a link between the non-reductive GIT of Bérczi--Kirwan et al. \cite{BDHK-Projective-Completions, BDHK-git-graded-unipotent, BK-moment-maps} and the Beyond GIT program of Alper--Halpern-Leistner et al. \cite{AHL-Beyond-GIT-Survey}. We then illustrate how our \emph{Intrinsic Non-Reductive GIT} can be carried out in practice, by using it to construct new examples of quasi-projective coarse moduli spaces of principal bundles over a smooth projective curve, where the principal bundles have a non-reductive structure group.

\subsection{Background and Previous Work}

\subsubsection{Non-Reductive GIT}

At a fundamental level, non-reductive GIT is concerned with the existence of good quotients for actions of non-reductive algebraic groups $H = U \rtimes L$ on schemes $X$, extending the classical theory of Mumford et al. \cite{MFK-GIT} for actions of reductive algebraic groups. A key assumption when constructing non-reductive GIT quotients is that the center $Z(L)$ of the Levi subgroup $L \subset H$ admits a distinguished cocharacter $\lambda$ which both grades the unipotent radical $U \subset H$, in the sense that the conjugation action of $\lambda(\mathbb{G}_m)$ contracts $U$ to the identity, and simultaneously contracts $X$ to the $\lambda(\mathbb{G}_m)$-fixed point locus $Z = X^{\lambda,0}$. 

The construction of such quotients can be realized via a two-step procedure, as carried out explicitly in \cite{HHJRelativeNRGIT} and implicitly in the proof of Theorem \ref{thm:nrgit-intro} of the present article. The first step involves constructing a schematic geometric quotient $X \to X/U$ for the action of $U$ on $X$, through which the contraction morphism $X \to Z$, $x \mapsto \lim_{t \to 0} \lambda(t) \cdot x$ factors. The formation of the quotient $X/U$ requires the $U$-stabilizers of $X$ to be sufficiently well-behaved; though different authors in the literature impose slightly different assumptions on the $U$-stabilizers of points of $X$, all of these assumptions imply that the stabilizer group scheme $\mathrm{Stab}_{U,X} \to X$ is flat (see Remarks \ref{rmk:generalizing-constancy-unipotent-stabilizers} and \ref{rmk:generalizing-qiao-cdrs}). The second step is to then apply relative reductive GIT with respect to the induced morphism $X/U \to Z$, which is surjective and equivariant with respect to naturally inherited actions of $L$. As with reductive GIT, in practice an $L$-linearized invertible sheaf is used to determine a notion of semistability for points of $X/U$ which is characterizable in terms of cocharacters of $L$, and for which the resulting quotient by $L$ is a good quotient of the open locus of semistable points in $X/U$. The result is an $H$-invariant open subscheme $X' \subset X$ admitting a schematic categorical $H$-quotient $X' \to X' /\!\!/ H$, with the points of $X'$ being characterized in terms of cocharacters of $H$.

Non-reductive GIT has seen numerous recent applications, including the construction of moduli spaces of GIT unstable objects \cite{Hamilton-Thesis, Jackson-HNL2, Qiao-HNL2, Qiao-Thesis, Hoskins-Jackson-QIS}, moduli spaces of hypersurfaces in toric orbifolds \cite{Bunnett-Orbifolds} and flags of hypersurfaces in projective space \cite{Cooper-Flags}, compactifications of Hilbert schemes of complete intersections in projective space \cite{Wiaterek-Thesis}, moduli spaces of representations of quivers over truncated polynomial rings \cite{HHJ-Affine-NRGIT, HJV-Quivers} and, most strikingly, to prove the Green--Griffiths--Lang and Kobayashi hyperbolicity conjectures for generic hypersurfaces of polynomial degree \cite{BK-Hyperbolicity}.

\subsubsection{The Beyond GIT Program} \label{sec:beyond-git-program}

The traditional approach to studying moduli problems and constructing moduli spaces, especially via Mumford's reductive GIT, requires that the moduli problem is sufficiently rigidified (so that the moduli problem can be described as a global quotient stack), which in practice requires introducing extrinsic data to the moduli problem. The goal of the Beyond GIT program is to instead develop the tools necessary to study moduli problems intrinsically, using the theory of algebraic stacks. 

The Beyond GIT program postulates that a moduli problem is the same as an algebraic stack $\mathcal{X}$ (or a higher analog, such as a derived algebraic stack). The study of the moduli problem is then the study of the geometry of the algebraic stack $\mathcal{X}$. In particular, the problem of the existence of a moduli space for the moduli problem becomes the problem of the existence of a stack-theoretic moduli space $\pi : \mathcal{X} \to Y$; depending on the context, $\pi$ may be sought to be a (stack-theoretic) coarse moduli space, a good moduli space, in the sense of Alper \cite{Alper-GMS}, an adequate moduli space, also in the sense of Alper \cite{Alper-AMS}, or more generally a topological moduli space, in the sense of upcoming work of Rydh \cite{Rydh-TMS-Talk-2, Rydh-TMS-Talk-1}.

One of the recent successes of the Beyond GIT program has been the valuative criterion of Alper--Halpern-Leistner--Heinloth \cite{AHLH-Existence-GMS} for when an algebraic stack admits a good moduli space, in terms of S-completeness and $\Theta$-completeness. Another success has been the development by Halpern-Leistner \cite{Halpern-Leistner} of the theory of $\Theta$-stratifications and of the notion of semistability with respect to numerical invariants on algebraic stacks, generalizing both the Hesselink--Kempf--Kirwan--Ness instability stratification \cite{Hesselink-instability, Kempf-instability, Kirwan-Thesis, Ness-instability} and the Hilbert--Mumford criterion from reductive GIT. These tools provide effective criteria for proving the existence of and studying good moduli spaces, which in practice are more straightforward to implement compared to GIT. The Beyond GIT program has successfully been used to construct good moduli spaces which lie beyond the reach of GIT, most notably the projective moduli space of K-polystable $\mathbb{Q}$-Fano varieties by Xu et al \cite{xu-k-stability}; it has also been used by Halpern-Leistner to prove the $D$-equivalence conjecture for smooth Calabi--Yau varieties birationally equivalent to smooth moduli spaces of semistable coherent sheaves on a K3 surface \cite{Halpern-Leistner-Derived}.

\subsubsection{Closing the Gap Between Non-Reductive GIT and Beyond GIT}

The existence of non-reductive GIT quotients has yet to be fully addressed from the perspective of the Beyond GIT program. Whereas the GIT quotient of an action of a linearly reductive group produces a good moduli space, and the GIT quotient of a geometrically reductive group produces an adequate moduli space, the quotients produced by non-reductive GIT in general are neither good nor adequate moduli spaces, since the automorphism groups of the closed points of a quotient stack admitting a non-reductive GIT quotient in general have non-trivial unipotent radicals. As such, the toolkit from the Beyond GIT program for constructing and studying good or adequate moduli spaces cannot directly be applied to non-reductive GIT quotients.

The first bridge between non-reductive GIT and the Beyond GIT program was provided in previous work by the second named author \cite{Modin}, where an existence result for relative moduli spaces for the associated graded morphism $\mathrm{gr} : \mathrm{Filt}(\mathcal{Y}) \to \mathrm{Grad}(\mathcal{Y})$ was established; here $\mathrm{Filt}(\mathcal{Y})$ (resp. $\mathrm{Grad}(\mathcal{Y})$) denotes the stack of filtered (resp. graded) objects of the algebraic stack $\mathcal{Y}$, in the sense of Halpern-Leistner \cite{Halpern-Leistner}. This same work also showed how this existence result can be used to extend the \emph{$\widehat{U}$-Theorem} of Bérczi--Doran--Hawes--Kirwan \cite{BDHK-git-graded-unipotent} (i.e. the existence of non-reductive GIT quotients when the group $H$ is of the form $\widehat{U} = U \rtimes \mathbb{G}_m$) to the case where the base space is a Noetherian scheme of arbitrary characteristic.

By developing the results of \cite{Modin}, this article seeks to further close the gap between the theory of non-reductive GIT and the Beyond GIT program, by establishing a stack-intrinsic existence result for (absolute) moduli spaces of algebraic stacks $\mathcal{X}$ with a weak action of the monoid stack $\Theta = [\mathbb{A}^1/\mathbb{G}_m]$ for which there is a suitable notion of semistability, in our case given by $\Theta$-semistability with respect to suitable invertible sheaves. The class of stacks with weak $\Theta$-actions include stacks of filtered objects $\mathcal{X} = \mathrm{Filt}(\mathcal{Y})$ as well as unstable $\Theta$-strata appearing in a $\Theta$-stratification of an algebraic stack. Given an action of $H = U \rtimes L$ on a scheme $X$ for which the distinguished cocharacter $\lambda : \mathbb{G}_m \to Z(L)$ grades the unipotent radical $U$ and contracts $X$ onto the $\mathbb{G}_m$-fixed locus $Z = X^{\lambda,0}$, there is an induced weak $\Theta$-action on the quotient stack $[X/H]$. By applying our main result to the quotient stack $[X/H]$ with this weak $\Theta$-action, we are able to recover the existence of non-reductive GIT quotients (as formulated in \cite{HHJRelativeNRGIT}); consequently our main result should be viewed as a stack-intrinsic generalization of non-reductive GIT.

\subsection{Main Results}

Let $\mathcal{X}$ be an algebraic stack which is locally of finite presentation over a quasi-separated base scheme $B$; assume that the diagonal of $\mathcal{X} \to B$ is affine. Suppose $\mathcal{X}$ is endowed with a weak action of $\Theta_B = [\mathbb{A}_B^1/\mathbb{G}_{m,B}]$. This weak action determines an algebraic stack $\mathcal{Z}$ with a weak action of $B\mathbb{G}_{m,B}$, known as the \emph{center} of the weak $\Theta_B$-action, and a pair of morphisms $\mathrm{gr} : \mathcal{X} \to \mathcal{Z}$, $\sigma : \mathcal{Z} \to \mathcal{X}$, which are equivariant with respect to the natural maps $\Theta_B \to B\mathbb{G}_{m,B}$ and $B\mathbb{G}_{m,B} \to \Theta_B$ respectively. Heuristically, $\mathcal{Z}$ should be thought of as the fixed point locus for a contracting action of $\mathbb{G}_{m,B}$ on $\mathcal{X}$, and $\mathrm{gr} : \mathcal{X} \to \mathcal{Z}$ as the retraction of $\mathcal{X}$ onto this fixed locus. Indeed, for the weak action on $[X/H]$ determined by $\lambda$, the center is given by the stack quotient $[Z/L]$ of $Z = X^{\lambda, 0}$, $\mathrm{gr} : [X/H] \to [Z/L]$ is the morphism corresponding to the retraction $X \to Z$ given by $x \mapsto \lim_{t \to 0} \lambda(t) \cdot x$, and $\sigma : [Z/L] \to [X/H]$ corresponds to the inclusion $Z \hookrightarrow X$.

\subsubsection{Statement of the Main Result}

The main result of this article, concerning the existence of moduli spaces for the stack $\mathcal{X}$, is the following theorem.

\begin{teo}[Intrinsic Non-Reductive GIT] \label{thm:main-theorem}
    Let $B$ be a quasi-separated scheme. Let $\mathcal{X}$ be an algebraic stack which is locally of finite presentation over $B$, with $\mathcal{X} \to B$ having an affine diagonal. Assume that $\mathcal{X}$ admits a weak action of the stack $\Theta_B$ with center $\mathcal{Z}$; denote by $\mathrm{gr} : \mathcal{X} \to \mathcal{Z}$ the retraction of $\mathcal{X}$ onto its center. Assume further that the relative inertia $\mathcal{I}_{\mathrm{gr}} \to \mathcal{X}$ is flat and of finite presentation. 

    Then the following statements hold:

    \begin{enumerate}
        \item The rigidification $\rho : \mathcal{X} \to \mathcal{X}^{\mathrm{rig}} := \mathcal{X} \!\!\! \fatslash  \mathcal{I}_{\mathrm{gr}}$ is well-defined, and $\mathcal{X}^{\mathrm{rig}}$ inherits a weak action of $\Theta_B$ from that on $\mathcal{X}$, with the same center $\mathcal{Z}$, and with respect to which $\rho$ is $\Theta_B$-equivariant. Moreover, the induced morphism $\overline{\mathrm{gr}} : \mathcal{X}^{\mathrm{rig}} \to \mathcal{Z}$ is affine and of finite presentation.

        \item If $\mathcal{Z}$ admits a good moduli space $\pi : \mathcal{Z} \to Z$, then the composite $\pi \circ \overline{\mathrm{gr}} : \mathcal{X}^{\mathrm{rig}} \to Z$ is a good moduli space. Assuming further that $\mathcal{X}$ is locally Noetherian, if $\mathcal{L}$ is an invertible sheaf on $\mathcal{X}^{\mathrm{rig}}$, then there are open substacks $(\mathcal{X}^{\mathrm{rig}})^s(\mathcal{L}) \subset (\mathcal{X}^{\mathrm{rig}})^{ss}(\mathcal{L})$ of $\mathcal{X}^{\mathrm{rig}}$ such that $(\mathcal{X}^{\mathrm{rig}})^{ss}(\mathcal{L})$ admits a good moduli space of the form
        $$ \pi_{\mathcal{L}} : (\mathcal{X}^{\mathrm{rig}})^{ss}(\mathcal{L}) \to X(\mathcal{L}) := \underline{\mathrm{Proj}}_Z \bigoplus_{m \geq 0} (\pi \circ \overline{\mathrm{gr}})_{\ast}(\mathcal{L}^m) = \underline{\mathrm{Proj}}_Z \bigoplus_{m \geq 0} (\pi \circ \mathrm{gr})_{\ast}(\rho^{\ast}\mathcal{L}^m), $$
        and such that the restriction of $\pi_{\mathcal{L}}$ to $(\mathcal{X}^{\mathrm{rig}})^{s}(\mathcal{L})$ is a tame moduli space onto its image. Moreover, for any algebraically closed field $k$, the following holds:
            \begin{enumerate}
                \item[(i)] For any $k$-point $\overline{x} \in \mathcal{X}^{\mathrm{rig}}(k)$, $ \overline{x} \in (\mathcal{X}^{\mathrm{rig}})^{(s)s}(\mathcal{L})(k)$ if and only if $\overline{x}$ is $\Theta_k$-(semi)stable with respect to $\mathcal{L}$.

                \item[(ii)] If $\overline{x} \in (\mathcal{X}^{\mathrm{rig}})^s(\mathcal{L})(k)$ then $\dim \mathrm{Aut}_{\mathcal{X}^{\mathrm{rig}}}(\overline{x}) = 0$.

                \item[(iii)] $\rho^{-1}((\mathcal{X}^{\mathrm{rig}})^{(s)s}(\mathcal{L}))(k)$ is given by the set of $k$-points $x \in \mathcal{X}(k)$ which are $\Theta_k$-(semi)stable with respect to $\rho^{\ast} \mathcal{L}$.
            \end{enumerate}
            In particular, there is a well-defined morphism of algebraic stacks $\mathcal{X}^{ss}(\rho^{\ast}\mathcal{L}) := \rho^{-1}((\mathcal{X}^{\mathrm{rig}})^{ss}(\mathcal{L})) \to X(\mathcal{L})$, which factors as a gerbe $\rho : \mathcal{X}^{ss}(\rho^{\ast} \mathcal{L}) \to (\mathcal{X}^{\mathrm{rig}})^{ss}(\mathcal{L})$ followed by a good moduli space $(\mathcal{X}^{\mathrm{rig}})^{ss}(\mathcal{L}) \to X(\mathcal{L})$; the restriction of this morphism to $\mathcal{X}^{s}(\rho^{\ast}\mathcal{L}) := \rho^{-1}((\mathcal{X}^{\mathrm{rig}})^{s}(\mathcal{L}))$ is a coarse moduli space onto its image. The $k$-points of $\mathcal{X}^{(s)s}(\rho^{\ast}\mathcal{L})$ are given by the $k$-points of $\mathcal{X}$ which are $\Theta_k$-(semi)stable with respect to $\rho^{\ast} \mathcal{L}$.
    \end{enumerate}
\end{teo}

\begin{rmk}
    In general, the morphism $\rho \circ \pi_{\mathcal{L}} : \mathcal{X}^{ss}(\rho^{\ast}\mathcal{L}) \to X(\mathcal{L})$ is neither a good moduli space nor an adequate moduli space, since the automorphism groups of closed points of $\mathcal{X}^{ss}(\rho^{\ast}\mathcal{L})$ can have non-trivial unipotent radicals. 
\end{rmk}

\begin{rmk}
    If $B$ is an excellent scheme (for instance, a scheme of finite type over a field or over $\Z$), if $\mathcal{X}$ is of finite type over $B$ and if the good moduli space $Z$ is a scheme, then the $\mathcal{O}_Z$-algebras $\bigoplus_{m \geq 0} (\pi \circ \overline{\mathrm{gr}})_{\ast}(\mathcal{L}^m) \cong \bigoplus_{m \geq 0} (\pi \circ \mathrm{gr})_{\ast}(\rho^{\ast}\mathcal{L}^m)$ are of finite type and the morphism $X(\mathcal{L}) \to Z$ is (Zariski) locally projective (see Lemma \ref{lem:when-moduli-space-is-projective}).
\end{rmk}

\subsubsection{Existence of Non-Reductive GIT Quotients}

Take $B = \Spec k$ to be the spectrum of a field. Suppose $X$ is a separated $k$-scheme of finite type which is acted on by $H = U \rtimes L$, and suppose $\lambda : \mathbb{G}_{m,k} \to Z(L)$ is a fixed cocharacter which simultaneously grades $U \subset H$ and contracts $X$ onto $Z = X^{\lambda,0}$. By taking $\mathcal{X} = [X/H]$ with the weak $\Theta_k$-action induced by $\lambda$, we recover as a special case of Theorem \ref{thm:main-theorem} the existence of non-reductive GIT quotients for the action of the internally graded group $H$ on $X$, as formulated in \cite[Theorem 1.2]{HHJRelativeNRGIT} in the case where $k$ is algebraically closed and of characteristic $0$ (which, by \cite[Theorem 6.10]{HHJRelativeNRGIT}, generalizes the existence results of \cite{BDHK-Projective-Completions, BDHK-git-graded-unipotent}); compare as well with \cite[Corollary 5.3 and Theorem 5.5]{Modin}.

\begin{teo}[Existence of Non-Reductive GIT Quotients for Internally Graded Groups] \label{thm:nrgit-intro}
Let $k$ be a field. Let $X$ be a separated $k$-scheme locally of finite type, acted on by a smooth linear algebraic group $H$ over $k$ with a split maximal torus. Assume that $H$ admits a cocharacter $\lambda : \mathbb{G}_{m,k} \to H$ for which $H = P_H(\lambda)$ coincides with its own attractor locus for the conjugation action of $\lambda(\mathbb{G}_{m,k})$ on $H$. Let $L = L_H(\lambda)$ be the centralizer of $\lambda(\mathbb{G}_{m,k})$ in $H$ and let $U = U_H(\lambda)$ be the kernel of the natural homomorphism $H \to L$ induced by $\lambda$, so that $H = U \rtimes L$ and $\lambda$ factors through the center $Z(L) \subset H$. Assume further that the action of $\mathbb{G}_{m,k}$ on $X$ via $\lambda$ extends to an action of the monoid $\mathbb{A}_k^1$, so that $X = X^{\lambda,+}$, and that the $U$-stabilizer group scheme $\mathrm{Stab}_{U,X} \to X$ is flat. Let $Z := X^{\lambda, 0}$ be the $\lambda(\mathbb{G}_{m,k})$-fixed locus of $X$.

Then the following statements hold:

\begin{enumerate}
    \item There exists a geometric good $U$-quotient $X \to X/U$, and $X/U$ is affine over the closed subscheme $Z$ of $X$. If the action of $U$ on $X$ is free then $X \to X/U$ is a principal $U$-bundle.

    \item Assume further that under the induced action of $L$ on $Z$, the quotient stack $[Z/L]$ admits a good moduli space $\pi : [Z/L] \to Z/\!\!/L$. Then the quotient stack $[(X/U)/L] = [X/H] \!\! \fatslash \ \mathcal{I}_{[X/H]/[Z/L]}$ admits a good moduli space, given by the composite $[(X/U)/L] \to [Z/L] \to Z /\!\!/ L$. If $\chi$ is a character of $L$, then there are $L$-invariant open subschemes $(X/U)^{\chi-s} \subset (X/U)^{\chi-ss}$ of $X/U$ such that $[(X/U)^{\chi-ss}/L]$ admits a good moduli space of the form
    $$ \pi_{\chi} : [(X/U)^{\chi-ss}/L] \to X /\!\!/_{\chi} H := \underline{\mathrm{Proj}}_{Z /\!\!/L} \bigoplus_{m \geq 0} (\pi \circ \mathrm{gr})_{\ast} (\mathcal{O}_{[X/H]}(\chi^m)), $$
    and such that the restriction of $\pi_{\chi}$ to $[(X/U)^{\chi-s}/L]$ is a tame moduli space onto its image; if $X$ is a $k$-scheme of finite type and if $Z /\!\!/ L$ is a scheme then the natural morphism $X /\!\!/_{\chi} H \to Z /\!\!/ L$ is projective. Moreover, for any algebraically closed field $K \supset k$ the following statements hold:
    \begin{enumerate}
        \item[(i)] Let $\overline{x} \in (X/U)(K)$ be a $K$-point of $X/U$. Let $L_{\overline{x}} \subset L_K$ be the stabilizer for the unique polystable orbit contained in the $L_K$-orbit closure of $\overline{x}$ in $(X/U)_K$. Then $\overline{x} \in (X/U)^{\chi-(s)s}(K)$ if and only if for all cocharacters $\eta : \mathbb{G}_{m,K} \to L_{\overline{x}}$ for which the limit $\lim_{t \to 0} \eta(t) \cdot x$ exists in $(X/U)_K$, $\langle \chi, \eta \rangle > (\geq) \ 0$.

        \item[(ii)] Let $X^{\chi-(s)s}$ denote the preimage of $(X/U)^{\chi-(s)s}$ in $X$. Then, for $x \in X(K)$ lying over $\overline{x} \in (X/U)(K)$, we have that $x \in X^{\chi-(s)s}(K)$ if and only if for all cocharacters $\eta : \mathbb{G}_{m,K} \to U_K \rtimes L_{\overline{x}}$ for which the limit $\lim_{t \to 0} \eta(t) \cdot x$ exists in $X_K$, $\langle \chi, \eta \rangle > (\geq) \ 0$.
    \end{enumerate}
    In particular, the composite of the geometric good $U$-quotient $X^{\chi-ss} \to (X/U)^{\chi-ss}$ with the induced good $L$-quotient $(X/U)^{\chi-ss} \to X /\!\!/_{\chi} H$ is a good $H$-quotient $X^{\chi-ss} \to X /\!\!/_{\chi} H$, and the restriction of this morphism to $X^{\chi-s}$ is a geometric good $H$-quotient onto its image.
\end{enumerate}

\end{teo}

\begin{rmk}
    If the characteristic of $k$ is $0$ and if $X$ is reduced, then the flatness of $\mathrm{Stab}_{U,X} \to X$ is equivalent to the condition that the map $x \mapsto \dim \mathrm{Stab}_U(x)$ is locally constant on $|X|$ (see Remark \ref{rmk:generalizing-constancy-unipotent-stabilizers}).
\end{rmk}

\begin{rmk}
    Denote by $\psi$ the morphism $X \to Z /\!\!/ L$. Then the algebraic space $X /\!\!/_{\chi} H$ coincides with the relative projective spectrum (over $Z /\!\!/ L$) of the algebra of $\chi$-twisted $H$-semi-invariants for the $\mathcal{O}_{Z/\!\!/L}$-algebra $\psi_{\ast} \mathcal{O}_X$, where $\chi$ is viewed as a character of $H$ via composition with the quotient $H \to L$ (see Remark \ref{rmk:H-quotient-is-spectrum-of-semiinvariants}).
\end{rmk}

By modifying the proof of Theorem \ref{thm:nrgit-intro}, we also recover the existence of non-reductive GIT quotients in the case where the group $H = U \rtimes L$ instead admits an \emph{external} grading (arising from a cocharacter $\lambda : \mathbb{G}_{m,k} \to \mathrm{Aut}(H)$) for which $X$ is an attractor locus for this external $\mathbb{G}_{m,k}$; compare this with \cite[Theorem 1.3]{HHJRelativeNRGIT}.

\begin{teo}[Existence of Non-Reductive GIT Quotients for Externally Graded Groups] \label{thm:nrgit-external-intro}
    Let $k$ be a field. Let $X$ be a separated $k$-scheme locally of finite type, acted on by a smooth linear algebraic group $H = U \rtimes L$ over $k$ with a split maximal torus. Assume that there exists a cocharacter $\lambda : \mathbb{G}_{m,k} \to \mathrm{Aut}(H)$ such that $\lambda(\mathbb{G}_{m,k})$ commutes with $L$, and such that for the group $\widehat{H} := H \rtimes_{\lambda} \mathbb{G}_{m,k}$, we have $P_{\widehat{H}}(\widehat{\lambda}) = \widehat{H}$, $U_{\widehat{H}}(\widehat{\lambda}) = U$ and $L_{\widehat{H}}(\widehat{\lambda}) = L \rtimes_{\lambda} \mathbb{G}_{m,k} \cong L \times \mathbb{G}_{m,k}$, where $\widehat{\lambda} : \mathbb{G}_{m,k} \to \widehat{H}$ corresponds to the external $\mathbb{G}_{m,k}$. Assume further that the action of $H$ on $X$ extends to an action of $\widehat{H}$, that the action of the external $\mathbb{G}_{m,k} \subset \widehat{H}$ extends to an action of the monoid $\mathbb{A}_k^1$, and that the $U$-stabilizer group scheme $\mathrm{Stab}_{U,X} \to X$ is flat. Let $Z := X^{\widehat{\lambda},0}$ be the fixed locus for the action of the external $\mathbb{G}_{m,k} \subset \widehat{H}$.

    Then the following statements hold:

    \begin{enumerate}
        \item There exists a geometric good $U$-quotient $X \to X/U$, and $X/U$ is affine over the closed subscheme $Z$ of $X$. If the action of $U$ on $X$ is free then $X \to X/U$ is a principal $U$-bundle.

        \item Assume further that under the induced action of $L$ on $Z$, the quotient stack $[Z/L]$ admits a good moduli space $\pi : [Z/L] \to Z/\!\!/L$. Then the quotient stack $[(X/U)/L] = [X/H] \!\! \fatslash \ \mathcal{I}_{[X/H]/[Z/L]}$ admits a good moduli space $\psi : [(X/U)/L] \to X /\!\!/ H$, and $X /\!\!/ H$ is affine over $Z /\!\!/ L$. If $\chi$ is a character of $L$, then there are $L$-invariant open subschemes $(X/U)^{\chi-s} \subset (X/U)^{\chi-ss}$ of $X/U$ such that $[(X/U)^{\chi-ss}/L]$ admits a good moduli space of the form
        $$ \pi_{\chi} : [(X/U)^{\chi-ss}/L] \to X /\!\!/_{\chi} H := \underline{\mathrm{Proj}}_{X /\!\!/H} \bigoplus_{m \geq 0} (\psi \circ \rho)_{\ast} (\mathcal{O}_{[X/H]}(\chi^m)), $$
        and such that the restriction of $\pi_{\chi}$ to $[(X/U)^{\chi-s}/L]$ is a tame moduli space onto its image; if $X$ is a $k$-scheme of finite type and if $Z /\!\!/ L$ is a scheme then the natural morphism $X /\!\!/_{\chi} H \to X /\!\!/ H$ is projective. Moreover, for any algebraically closed field $K \supset k$ the following statements hold:
        \begin{enumerate}
            \item[(i)] Let $\overline{x} \in (X/U)(K)$ be a $K$-point of $X/U$. Let $L_{\overline{x}} \subset L_K$ be the stabilizer for the unique polystable orbit contained in the $L_K$-orbit closure of $\overline{x}$ in $(X/U)_K$. Then $\overline{x} \in (X/U)^{\chi-(s)s}(K)$ if and only if for all cocharacters $\eta : \mathbb{G}_{m,K} \to L_{\overline{x}}$ for which the limit $\lim_{t \to 0} \eta(t) \cdot x$ exists in $(X/U)_K$, $\langle \chi, \eta \rangle > (\geq) \ 0$.

            \item[(ii)] Let $X^{\chi-(s)s}$ denote the preimage of $(X/U)^{\chi-(s)s}$ in $X$. Then, for $x \in X(K)$ lying over $\overline{x} \in (X/U)(K)$, we have that $x \in X^{\chi-(s)s}(K)$ if and only if for all cocharacters $\eta : \mathbb{G}_{m,K} \to U_K \rtimes L_{\overline{x}}$ for which the limit $\lim_{t \to 0} \eta(t) \cdot x$ exists in $X_K$, $\langle \chi, \eta \rangle > (\geq) \ 0$.
        \end{enumerate}

        In particular, the composite of the geometric good $U$-quotient $X^{\chi-ss} \to (X/U)^{\chi-ss}$ with the induced good $L$-quotient $(X/U)^{\chi-ss} \to X /\!\!/_{\chi} H$ is a good $H$-quotient $X^{\chi-ss} \to X /\!\!/_{\chi} H$, and the restriction of this morphism to $X^{\chi-s}$ is a geometric good $H$-quotient onto its image.
    \end{enumerate}
\end{teo}

Taking $X$ to be of finite type over $\Spec k$, we highlight that the primary difference in the conclusions of Theorems \ref{thm:nrgit-intro} and \ref{thm:nrgit-external-intro} is that in Theorem \ref{thm:nrgit-intro} (the internally graded case) the good $L$-quotient of $X/U$ is given by $Z /\!\!/ L$, whence $X /\!\!/_{\chi} H \to Z /\!\!/ L$ is projective, whereas in Theorem \ref{thm:nrgit-external-intro} (the externally graded case) the good $L$-quotient of $X/U$ is an algebraic space $X /\!\!/ H$ which is affine over $Z /\!\!/ L$ (but which need not coincide with $Z /\!\!/ L$), whence $X /\!\!/_{\chi} H \to Z /\!\!/ L$ is projective-over-affine.

\subsubsection{Outline of the Proof of the Main Theorem}

Our approach to proving Theorem \ref{thm:main-theorem} mirrors the two-step procedure used to construct non-reductive GIT quotients. 

First, under the assumption that the relative inertia $\mathcal{I}_{\mathrm{gr}} \to \mathcal{X}$ is flat and of finite presentation, the rigidification $\mathcal{X}^{\mathrm{rig}} := \mathcal{X} \!\!\! \fatslash \mathcal{I}_{\mathrm{gr}}$ is well-defined (Lemma \ref{lem:existence-rigidification}) and the induced morphism $\overline{\mathrm{gr}} : \mathcal{X}^{\mathrm{rig}} \to \mathcal{Z}$ is affine (Theorem \ref{thm:affinity-grbar}); this is essentially the core result of \cite{Modin}, and makes use of the affinity results for $\mathbb{G}_m$-attractor loci in algebraic spaces of Drinfeld \cite{Drinfeld-Gm} and Richarz \cite{Richarz-Gm}. The requirement that $\mathcal{I}_{\mathrm{gr}} \to \mathcal{X}$ is flat and of finite presentation should be thought of as generalizing the requirement in non-reductive GIT for the $U$-stabilizers of $X$ to be sufficiently well-behaved; indeed, in the setting of Theorem \ref{thm:nrgit-intro} when $\mathcal{X} = [X/H]$ and $\mathcal{Z} = [Z/L]$, the relative inertia $\mathcal{I}_{[X/H]/[Z/L]} \to [X/H]$ is flat if and only if the stabilizer group scheme $\mathrm{Stab}_{U,X} \to X$ is flat. Since $\mathcal{I}_{\mathrm{gr}}$ is invariant for the weak $\Theta_B$-action, the descent of the weak $\Theta_B$-action to $\mathcal{X}^{\mathrm{rig}}$ follows from the universal property of rigidification (Proposition \ref{prop:induced-action-on-rigidification}).

The second step of our approach departs from \cite{Modin}, where it is shown that the stack $(\mathcal{X} \setminus \sigma(\mathcal{Z})) \!\!\! \fatslash \mathcal{I}_{\mathrm{gr}}$ is an open substack of $\mathcal{X}^{\mathrm{rig}}$ admitting a relative tame moduli space over the rigidification $\mathcal{Z} \!\!\! \fatslash \mathcal{G}_{\mathcal{Z}}^{\mathrm{can}}$ of $\mathcal{Z}$ by the canonical subgroup stack $\mathcal{G}_{\mathcal{Z}}^{\mathrm{can}}$ (cf. Proposition \ref{prop:rigidified-center}), with $(\mathcal{X} \setminus \sigma(\mathcal{Z})) \!\!\! \fatslash \mathcal{I}_{\mathrm{gr}} \to \mathcal{Z} \!\!\! \fatslash \mathcal{G}_{\mathcal{Z}}^{\mathrm{can}}$ smooth locally projective. Instead, we assume that the center $\mathcal{Z}$ admits a good moduli space $\pi : \mathcal{Z} \to Z$, and show that the composite $\pi \circ \overline{\mathrm{gr}} : \mathcal{X}^{\mathrm{rig}} \to Z$ is also a good moduli space (Proposition \ref{prop:gms-of-Xrig}). Given an invertible sheaf $\mathcal{L}$ on $\mathcal{X}^{\mathrm{rig}}$, this allows us to implement a version of ``variation of good moduli space'' (cf. \cite[Theorem 5.6.1]{Halpern-Leistner} and \cite[Theorem 2.6.3]{IbanezNunez-RefinedHN}) to produce an absolute good moduli space $X(\mathcal{L})$ of the open $\Theta$-semistable locus $(\mathcal{X}^{\mathrm{rig}})^{ss}(\mathcal{L})$ in $\mathcal{X}^{\mathrm{rig}}$ determined by $\mathcal{L}$ (Definition \ref{def:L-semistable-locus}), with $X(\mathcal{L})$ given by the relative projective spectrum over $Z$ corresponding to the graded $\mathcal{O}_Z$-algebra $\bigoplus_{m \geq 0} (\pi \circ \overline{\mathrm{gr}})_{\ast}(\mathcal{L}^m)$. More specifically, we define $(\mathcal{X}^{\mathrm{rig}})^{ss}(\mathcal{L})$ as a saturated projective spectrum, in the sense of \cite{Edidin-Rydh}, so that $(\mathcal{X}^{\mathrm{rig}})^{ss}(\mathcal{L})$ inherits from $\mathcal{X}^{\mathrm{rig}}$ the property of having a good moduli space, with the good moduli space given by the relative projective spectrum $X(\mathcal{L})$ (Proposition \ref{prop:gms-from-line-bundle}). We then show that $(\mathcal{X}^{\mathrm{rig}})^{ss}(\mathcal{L})$ coincides with the $\Theta$-semistable locus determined by $\mathcal{L}$ (Proposition \ref{prop:semistable-loci-are-theta-semistable-loci}), by using the étale local structure theorem for good moduli spaces of Alper--Hall--Rydh \cite{AHR-Luna} to reduce the problem to GIT, then applying the relative Hilbert--Mumford criterion of Gulbrandsen--Halle--Hulek \cite{GHH-Relative-HM} (or alternatively \cite[Theorem 1.1]{HHJRelativeNRGIT}).

If $\rho : \mathcal{X} \to \mathcal{X}^{\mathrm{rig}}$ denotes the rigidification morphism, we then show that $\rho^{-1}((\mathcal{X}^{\mathrm{rig}})^{ss}(\mathcal{L}))$ coincides with the $\Theta$-semistable locus $\mathcal{X}^{ss}(\rho^{\ast} \mathcal{L})$ (Proposition \ref{prop:semistable_locus_upstairs}), by showing that any morphism $\Theta_k \to \mathcal{X}^{\mathrm{rig}}$ (where $k$ is an algebraically closed field) admits a lift to a morphism $\Theta_k \to \mathcal{X}$ (Proposition \ref{prop:lifting_filts_along_gerbe}). Consequently, we obtain an open substack $\mathcal{X}^{ss}(\rho^{\ast} \mathcal{L}) \subset \mathcal{X}$ with an explicit Hilbert--Mumford-type description (i.e. $\mathcal{X}^{ss}(\rho^{\ast} \mathcal{L})$ is characterized in terms of morphisms $\Theta_k \to \mathcal{X}$) and which admits $X(\mathcal{L})$ as a moduli space (more precisely, $\mathcal{X}^{ss}(\rho^{\ast} \mathcal{L}) \to X(\mathcal{L})$ factors as a gerbe followed by a good moduli space).

\subsection{Application to Moduli Spaces of Principal Bundles with Non-Reductive Structure Groups}

The second part of this article serves to illustrate how our main theorem can be applied in practice, by constructing new examples of quasi-projective coarse moduli spaces of principal bundles over a smooth projective curve; the distinguishing feature of these moduli spaces is that the principal bundles being parameterized have non-reductive structure groups. Examples of coarse moduli spaces of such principal bundles where the structure group of the bundles is a parabolic subgroup of $\mathrm{GL}_r$ have previously been constructed by Jackson \cite{Jackson-HNL2}, Qiao \cite{Qiao-HNL2, Qiao-Thesis}, Brambila Paz--Rios Sierra \cite{BPRS-HNL2-Proceedings} and Hoskins--Jackson \cite{Hoskins-Jackson-QIS}; in all of these examples, the resulting moduli spaces are moduli spaces of slope unstable vector bundles of fixed Harder--Narasimhan type.\footnote{The constructions of Jackson, Qiao \cite{Qiao-HNL2} and Hoskins--Jackson are more general in that they consider the problem of constructing coarse moduli spaces of Gieseker unstable torsion-free coherent sheaves on a polarized projective scheme $(X, \mathcal{O}_X(1))$ of fixed Harder--Narasimhan type. When $X$ is a smooth curve then all torsion-free sheaves are locally free, and Gieseker semistability coincides with slope semistability for vector bundles.} The moduli spaces we obtain allow for non-reductive structure groups beyond parabolic subgroups of $\mathrm{GL}_r$ or $\mathrm{SL}_r$.

In principle, the moduli spaces we construct could be constructed using non-reductive GIT, since they are moduli spaces associated to global quotient stacks with weak actions of $\Theta_k$ (the aforementioned constructions of Jackson, Qiao and Hoskins--Jackson all utilize non-reductive GIT). However, our use of Intrinsic Non-Reductive GIT leads to a conceptually cleaner construction, without the need to introduce extrinsic data to frame the moduli problem. Though we do not carry this out in the present article, we expect our approach to be easily transferable to related moduli problems, such as the moduli of singular principal bundles on a projective scheme or the moduli of principal Higgs bundles over a smooth projective curve.

\subsubsection{Weak $\Theta$-Actions on Moduli Stacks of Principal Bundles}

Let $k$ be an algebraically closed field of characteristic $0$, and let $C$ be a smooth connected projective algebraic curve over $k$ of genus $p_a(C) \geq 2$. If $L$ is a connected linear algebraic group over $k$ a famous result of Ramanathan \cite{Ramanathan-Moduli-1, Ramanathan-Moduli-2} proves the existence of a projective good moduli space $M_L^d$ of the moduli stack $\mathcal{B}un_L^{d,ss}$ of semistable principal $L$-bundles over $C$ of a fixed topological type $d$ (see Theorem \ref{thm:good-moduli-spaces=semistable-bundles}). We bootstrap off of this result, by setting things up in such a way that the moduli stacks of principal bundles with a weak $\Theta_k$-action we consider have centers given by substacks of $\mathcal{B}un_L^{d,ss}$.

Specifically, the structure groups we consider are of the form $H = U \rtimes L$, where $U \subset H$ is the unipotent radical of $H$ and the Levi subgroup $L \subset H$ is as above. We further assume that there is a distinguished cocharacter $\lambda$ of $Z(L)$ which grades $U$ (Assumption \ref{ass:structure-group-bundles}; compare this with the assumptions underlying non-reductive GIT). The cocharacter $\lambda$ then yields a weak $\Theta_k$-action on the moduli stack $\mathcal{B}un_H$ of principal $H$-bundles over $C$, whose center is given by $\mathcal{B}un_L$ and whose associated retraction $\mathcal{B}un_H \to \mathcal{B}un_L$ is given by extension of structure group along the Levi quotient $H \to L$ (Lemma \ref{lem:weak-theta-action-bunh}). We then restrict attention to appropriate locally closed substacks $\mathcal{X}$ of $\mathcal{B}un_H^{d,ss} := \mathrm{gr}^{-1}(\mathcal{B}un_L^{d,ss})$. In order to ensure that $\mathcal{I}_{\mathrm{gr}} \to \mathcal{X}$ is flat, we additionally need to impose conditions on the unipotent automorphisms of the principal $H$-bundles $E$ we parameterize; in particular, we require that the Lie algebra of the unipotent radical of the automorphism group of $E$ has the same dimension as that of the associated graded principal $H$-bundle $\sigma(\mathrm{gr}(E)) = (E \times^H L) \times^L H$. In order to do this in a way which preserves the property of the center having a good moduli space, we further assume that the associated principal $L$-bundle $E \times^H L$ is stable, as opposed to being semistable.\footnote{This could possibly be relaxed by implementing the reduction of stabilizers procedure of Edidin--Rydh \cite{Edidin-Rydh}, however we do not consider this.}

\subsubsection{Stability Conditions Yielding Coarse Moduli Spaces}

In order to obtain stability conditions on principal $H$-bundles which yield coarse moduli spaces\footnote{Without any choice of stability condition, in the resulting moduli space (cf. Corollary \ref{cor:moduli-space-h-bundles-no-character}) all principal $H$-bundles are identified with their extensions with respect to the Levi quotient $H \to L$, i.e. we obtain a moduli space of principal $L$-bundles which is far from being a coarse moduli space of principal $H$-bundles.} and for which the corresponding notion of stability can be explicitly described, in a manner analogous to the non-splitting conditions imposed by Jackson, Qiao, Brambila Paz--Rios Sierra and Hoskins--Jackson, we impose a stronger set of assumptions on $H$, Assumption \ref{ass:structure-group-bundles-strong}; these assumptions take inspiration from \cite[Section 4]{Hoskins-Jackson-QIS}. Intuitively, Assumption \ref{ass:structure-group-bundles-strong} assumes that the adjoint action of the torus $Z(L)^{\circ}/(Z(L)^{\circ} \cap C_L(U)) \cong \mathbb{G}_m^n$ on $\mathfrak{u} = \mathrm{Lie}(U)$ is fully contracting, in the sense that all weights for this action are contained in $(\mathbb{Z}^{\geq 0})^n \setminus \{0\}$. Assumption \ref{ass:structure-group-bundles-strong} further assumes that we have a fixed basis of linearly independent cocharacters of this torus; these cocharacters are used to decompose $H$ as an iterated semidirect product
$$ H = U_1 \rtimes \left( U_2 \rtimes \left( \cdots \rtimes \left( U_n \rtimes L \right) \cdots \right) \right). $$
The stability condition resulting from this basis of cocharacters is then described in terms of non-existence of reductions of structure groups with respect to the above decomposition.

Under the Assumption \ref{ass:structure-group-bundles-strong} on $H$, we use Theorem \ref{thm:main-theorem} to prove the following result.

\begin{teo}[Corollary \ref{cor:adapted-character-tame-ms}] \label{thm:moduli-principal-bundles-intro}
    Let $k$ be an algebraically closed field of characteristic $0$. Let $C$ be a smooth, connected, projective algebraic curve over $k$ of genus $p_a(C) \geq 2$. Let $H$ be a connected linear algebraic group over $k$ which is non-reductive, with Levi decomposition $H = U \rtimes L$. Fix a rational cocharacter $d \in \mathfrak{X}_{\ast}(T)_{\Q}$, where $T$ is a fixed maximal torus of $L$, and fix a non-negative integer $\delta \in \Z^{\geq 0}$. Assume further that Assumption \ref{ass:structure-group-bundles-strong} holds for $H$, so that
    $$ H = U_1 \rtimes \left( U_2 \rtimes \left( \cdots \rtimes \left( U_n \rtimes L \right) \cdots \right) \right) $$
    is an iterated semidirect product with each $U_i$ a non-trivial subgroup of $U$ and with $n = \dim Z(L) - \dim (Z(L) \cap C_L(U)) > 0$. Set $H_n := L$, and define $H_0, \dots, H_{n-1} \subset H$ inductively by $H_i := U_{i+1} \rtimes H_{i+1}$ (so that $H_0 = H$).

    Then there exists a quasi-projective coarse moduli space $M_H^{d, \delta, \mathrm{adpt}}$ whose $k$-points are in bijection with principal $H$-bundles $E \to C$ satisfying the following properties:
    \begin{enumerate}
        \item The type of $E$ as a principal $H$-bundle is equal to $d$.
        \item If $E_L := E \times^H L$ denotes the extension of structure group with respect to the Levi quotient $H \to L$, then $E_L$ is a stable principal $L$-bundle (in the sense of \cite{Ramanathan-Riem-Surface}).
        \item There are equalities $\dim H^0(C, E \times^H \mathfrak{u}) = \dim H^0(C, E_L \times^L \mathfrak{u}) = \delta$, where $\mathfrak{u} = \mathrm{Lie}(U)$, and where $E \times^H \mathfrak{u}$ (resp. $E_L \times^L \mathfrak{u}$) is the bundle associated to the adjoint action of $H$ (resp. $L$) on $\mathfrak{u}$.
        \item For each $m = 1, \dots, n$, set $F_m := ((E \times^{H_0} H_1) \times^{H_1} \cdots) \times^{H_{m-1}} H_m$ and $E_m := F_m \times^{H_m} H$. Then for all $m = 1, \dots, n$, $E_{m-1} \not \cong E_m$ as principal $H$-bundles.
    \end{enumerate}
\end{teo}

\begin{rmk}
    $M_H^{d,\delta,\mathrm{adpt}}$ is the coarse moduli space of the locally closed substack $\mathcal{X}_{\delta}^{\chi-ss}$ of $\mathcal{B}un_H$ appearing the statement of Corollary \ref{cor:moduli-space-h-bundles-general-character}, in the case where $\chi$ is an adapted character of $L$ in the sense of Definition \ref{def:adapted-characters}. Corollary \ref{cor:moduli-space-h-bundles-general-character} gives an explicit description of the groupoid $\mathcal{X}_{\delta}^{\chi-ss}(S)$ for any $k$-scheme $S$. A more general version of Theorem \ref{thm:moduli-principal-bundles-intro} valid for any character of $L$ (not just those which are adapted) and valid under the weaker Assumption \ref{ass:structure-group-bundles} is given by Theorem \ref{thm:moduli-principal-bundles-general-character}.
\end{rmk}

\begin{rmk}
    Suppose $\{e\} = U^{(q+1)} \lhd U^{(q)} \lhd \cdots \lhd U^{(1)} = U$ is a central series such that $\{e\} = U^{(q+1)} \lhd U^{(q)} \lhd \cdots \lhd U^{(1)} = U \lhd H$ is a normal series (for example, the lower central series $U^{(i+1)} = [U^{(i)}, U]$). Then a sufficient condition for $\dim H^0(C, E \times^H \mathfrak{u}) = \dim H^0(C, E_L \times^L \mathfrak{u})$ to hold for $E$ is that for any irreducible $L$-representation appearing as a subrepresentation of $\mathrm{Lie}(U^{(i)}/U^{(i+1)})$ for some $i \geq 2$, the corresponding sheaf of sections $\Gamma_C(E_L \times^L V)$ satisfies
    $$ \mu(\Gamma_C(E_L \times^L V)) > 2p_a(C) - 2, $$
    where $\mu(\mathcal{E}) = \deg(\mathcal{E})/\mathrm{rank}(\mathcal{E})$ denotes the Mumford slope of a vector bundle. This is a condition which depends only on the type $d$ (see Remark \ref{rmk:unipotent-auts-bundles-non-abelian}).
\end{rmk}

In the statement of Theorem \ref{thm:moduli-principal-bundles-intro}, it is possible to take $H$ to be a proper parabolic subgroup $P \subset G$ of a connected reductive linear algebraic group $G$. By choosing the type $d$ to be a Harder--Narasimhan type for unstable principal $G$-bundles whose canonical reduction has structure group $P$ (cf. Example \ref{ex:unstable-hn-strata-in-our-framework}), Theorem \ref{thm:moduli-principal-bundles-intro} yields the existence of quasi-projective coarse moduli spaces of unstable principal $G$-bundles over $C$. This generalizes the existence results for moduli spaces of slope unstable vector bundles of fixed Harder--Narasimhan type arising from \cite{Jackson-HNL2, Qiao-HNL2, Qiao-Thesis, BPRS-HNL2-Proceedings} (see Section \ref{sec:parabolic-subgroups-gl}), and realizes for principal bundles over curves the program of constructing schematic moduli spaces of unstable objects, following on the work of \cite{Jackson-HNL2, Qiao-HNL2, Qiao-Thesis, Cooper-Thesis, Hoskins-Jackson-QIS} for unstable coherent sheaves on projective varieties and \cite{Hamilton-Thesis} for unstable $\mathrm{GL}_r$-Higgs bundles over curves.

\subsection{Further Directions and Future Work}

This work does not address the question of what the correct stack-theoretic notion of a ``non-reductive moduli space'' should be in full generality. In fact, the definition of a \emph{non-reductive good moduli space}\footnote{In turn, these are special cases of the more general upcoming notion of a topological moduli space, as mentioned in Section \ref{sec:beyond-git-program}.} has previously been announced by Rydh \cite{Rydh-TMS-Talk-2, Rydh-TMS-Talk-1}, and will appear in upcoming work; this class of morphisms should provide the correct complete intrinsic generalization, analogously to how the notions of good and adequate moduli spaces collectively provide a complete intrinsic generalization of the notion of a reductive GIT quotient. We expect the moduli spaces $X(\mathcal{L})$ constructed in Theorem \ref{thm:main-theorem} to be examples of such non-reductive good moduli spaces.

In the setting of Theorem \ref{thm:nrgit-intro}, if $\mathrm{Stab}_{U,X} \to X$ is flat over a dense open subscheme of $X$ but not over all of $X$, then in favorable situations the blow-up procedure of Qiao \cite{Qiao-Blowups} can be used to replace $X$ with a scheme $\widetilde{X}$ which is acted on by the same group $H$, which is contracted to its fixed locus by the same cocharacter $\lambda$, and for which $\mathrm{Stab}_{U, \widetilde{X}} \to \widetilde{X}$ is flat. Theorem \ref{thm:nrgit-intro} can therefore be applied for the action of $H$ on the scheme $\widetilde{X}$.\footnote{The paper \cite{BDHK-Projective-Completions} also sought to describe a blow-up procedure for flattening unipotent stabilizers in non-reductive GIT, however the proof of the efficacy of the procedure contained therein contains a gap, so is not guaranteed to work in all circumstances as claimed. This was communicated to us by Frances Kirwan \cite{kirwan-private-communication}.} In general, the upcoming results of Rydh on non-reductive good moduli spaces include the non-reductive analog of the reduction of stabilizers procedure of Edidin--Rydh \cite{Edidin-Rydh}.\footnote{The upcoming reduction of stabilizers for non-reductive good moduli spaces was communicated to us by David Rydh \cite{rydh-private-communication}.} We expect that this procedure can be used in the setting of Theorem \ref{thm:main-theorem} to perform a stack-intrinsic blow-up procedure to similarly flatten the relative inertia $\mathcal{I}_{\mathrm{gr}} \to \mathcal{X}$. As such, we do not consider the problem of flattening the relative inertia in the present article.

In general, variation of good moduli space can be carried out for stability conditions other than those coming from $\Theta$-stability with respect to invertible sheaves, albeit at the cost of having the resulting good moduli spaces no-longer necessarily admitting explicit descriptions as relative projective spectra (cf. \cite[Theorem 5.6.1]{Halpern-Leistner}, \cite[Theorem 2.6.3]{IbanezNunez-RefinedHN}). Since in practice the stability conditions of interest to us are those arising from invertible sheaves on $\mathcal{X}^{\mathrm{rig}}$, we elect to not consider this more general situation.

Though our main result, Theorem \ref{thm:main-theorem}, holds over any quasi-separated base scheme $B$, and in particular over any quasi-separated scheme of characteristic $p > 0$, the notion of a good moduli space in positive characteristic is significantly restrictive. Unfortunately, the arguments of Section \ref{sec:projective-moduli-spaces-inv-sheaves} fail to carry over as stated to the case where $\pi : \mathcal{Z} \to Z$ is assumed to be an adequate moduli space which is not necessarily good, due to an absence of a general local structure theorem for adequate moduli spaces. On the other hand, when \(\mathcal{X}\) is a quotient stack whose center admits an adequate moduli space, no local structure theorem is required, and analogs of Theorems \ref{thm:nrgit-intro} and \ref{thm:nrgit-external-intro} can be established without the need for significant modifications; we will carry this out in \cite{CooperModin-upcoming}.

\subsection{Notation and Conventions}

Throughout this article we adopt the following conventions, notation and terminology:\footnote{In Section \ref{sec:principal-bundles} we adopt additional notation and terminology concerning principal bundles over curves and their moduli.}

\begin{enumerate}
    \item All stacks and algebraic spaces are assumed to be defined relative to the fppf topology on the site of schemes, as in \cite{stacks-project}. If $\mathcal{X}$ is an algebraic stack, an algebraic space or a scheme, we denote the underlying topological space by $|\mathcal{X}|$.
    \item By a \emph{geometric point} of an algebraic stack, an algebraic space or a scheme $\mathcal{X}$, we mean a morphism $\Spec k \to \mathcal{X}$, where $k$ is an algebraically closed field.
    \item If $\mathcal{X}$ is an algebraic stack and if $\mathcal{I}$ is an ideal of $\mathcal{O}_{\mathcal{X}}$, we denote the corresponding closed substack $\underline{\Spec}_{\mathcal{X}}(\mathcal{O}_{\mathcal{X}}/\mathcal{I})$ of $\mathcal{X}$ as $\mathbb{V}(\mathcal{I})$.
    \item Given an algebraic space $X$ over a base scheme $B$, acted on by the group $\mathbb{G}_{m,B}$, we denote the $\mathbb{G}_{m,B}$-attractor locus (resp. $\mathbb{G}_{m,B}$-fixed locus) by $X^+$ (resp. $X^0$); see Definition \ref{def:attractor-loci}. Where the $\mathbb{G}_{m,B}$-action arises from a choice of cocharacter $\lambda$, we may emphasize this by writing $X^{\lambda, +}$ in place of $X^+$ and $X^{\lambda, 0}$ in place of $X^0$.

    \item Given an algebraic group $H$ over a field $k$ and a cocharacter $\lambda: \mathbb{G}_{m,k} \to H$, we denote by $P_H(\lambda)$ (resp. $L_H(\lambda)$) the algebraic group with underlying scheme $H^{\lambda,+}$ (resp. $H^{\lambda, 0}$, where $\mathbb{G}_{m,k}$ acts on $H$ via conjugation: $(t,h) \mapsto \lambda(t)h\lambda(t)^{-1}$ (if $H$ is reductive and has a split maximal torus, the groups $P_H(\lambda)$ are exactly the parabolic subgroups of $H$). We denote by $U_H(\lambda)$ the kernel of the natural homomorphism $P_H(\lambda) \to L_H(\lambda)$.

    \item If $f : \mathcal{X} \to \mathcal{Y}$ is a morphism of algebraic stacks (over a base $B$), we denote the relative inertia of $f$ as both $\mathcal{I}_f$ and $\mathcal{I}_{\mathcal{X}/\mathcal{Y}}$; by definition, this is the kernel of the induced map $f_{\ast} : \mathcal{I}_{\mathcal{X}} \to \mathcal{I}_{\mathcal{Y}} \times_{\mathcal{Y}} \mathcal{X}$ on inertia stacks, where $\mathcal{I}_{\mathcal{X}} = \mathcal{X} \times_{\Delta, \mathcal{X} \times_B \mathcal{X}, \Delta} \mathcal{X}$.
\end{enumerate}

\subsection{Acknowledgments}

Work on this project started following on from the \emph{New Equivariant Methods in Algebraic and Differential Geometry} workshop, which took place in 2024 at the Isaac Newton Institute, University of Cambridge. In addition to the organizers of this workshop, we would also like to especially thank Eloise Hamilton, Victoria Hoskins and Jessica (née Joshua) Jackson for providing a draft copy of the paper \cite{HHJRelativeNRGIT} to us before their results were announced and for helpful conversations concerning the present article. Additionally, we would like to thank Roberto Fringuelli, Jochen Heinloth, Andrés Ibáñez Núñez, Frances Kirwan, Michele Pernice, David Rydh, Luca Schaffler and Angelo Vistoli for helpful conversations and feedback.

\section{Preliminaries}

In this section we provide the necessary preliminaries we require in the sequel to prove Theorems \ref{thm:main-theorem} and \ref{thm:nrgit-intro} (there is a separate preliminary subsection in Section \ref{sec:principal-bundles}, concerning principal bundles and their moduli). Throughout this section, let $B$ be a quasi-separated scheme;\footnote{For the results concerning filtered and graded objects of algebraic stacks, $B$ can be taken to be an algebraic stack, however for the results concerning attractor loci for actions of $\mathbb{A}^1$, we require the base $B$ to be a scheme.} unless stated otherwise, we work relative to $B$.

\subsection{Attractor Loci for Actions of $\mathbb{G}_m$} 

Let $X$ be an algebraic space which is quasi-separated and locally of finite presentation over $B$, which is endowed with an action of $\mathbb{G}_{m,B}$ for which the structure morphism $X \to B$ is invariant.

\begin{defi} \label{def:attractor-loci}
    We denote by
    $$ X^+ := \underline{\mathrm{Map}}_B^{\mathbb{G}_{m,B}}(\mathbb{A}_B^1, X) $$
    the functor assigning to a $B$-scheme $S$ the set of $\mathbb{G}_{m,B}$-equivariant morphisms $\mathbb{A}_B^1 \times_B S \to X$, where $\mathbb{G}_{m,B}$ acts on $\mathbb{A}_B^1$ by coordinate-wise multiplication (and acts on $S$ trivially). Similarly, we denote
    $$ X^0 := \underline{\mathrm{Map}}_B^{\mathbb{G}_{m,B}}(B, X), $$
    where $\mathbb{G}_{m,B}$ acts on $B$ trivially.
\end{defi}

There is a natural transformation $X^+ \to X^0$ induced by the zero section $0: B \to \mathbb{A}_B^1$, and a natural transformation $X^0 \to X^+$ induced by the structure morphism $\mathbb{A}_B^1 \to B$.

\begin{teo} \label{theorem:attractor_loci_existence}
The functor $X^0$ is representable by a closed subspace of $X$, and $X^+$ is representable by an algebraic space locally of finite presentation over $B$, for which the natural morphism $X^+ \to X^0$ is affine and of finite presentation. Moreover, if $X \to B$ is separated (resp. quasi-separated, quasi-compact, of finite presentation), the same is true of $X^+ \to B$. In particular, if $X$ is a scheme, then the same is true of $X^0$ and $X^+$.

In addition, the natural morphism $X^+ \to X$ is an isomorphism if and only if the action of $\mathbb{G}_{m,B}$ on $X$ extends to an action of the monoid $\mathbb{A}_B^1$; if such an extension exists, it is necessarily unique.
\end{teo}

\begin{proof}
    By \cite[Corollary 10.2]{AHR-Etale}, the action of $\mathbb{G}_{m,B}$ on $X$ is \'etale-locally linearizable, in the sense of \cite{Richarz-Gm}. The first paragraph is then \cite[Theorem A]{Richarz-Gm}. The final statement is \cite[Proposition 1.4.15]{Drinfeld-Gm} in the case where $B$ is the spectrum of a field, however Drinfeld's arguments carry over to the case where $B$ is a quasi-separated scheme (see \cite[Theorem I. 1.0.1 (5)]{Modin-Thesis}).
\end{proof}

\begin{defi} \label{defn:attractor_fixed_loci}
    The algebraic space $X^+$ is known as the \emph{$\mathbb{G}_{m,B}$-attractor locus} of $X$, and the algebraic space $X^0$ is known as the \emph{$\mathbb{G}_{m,B}$-fixed locus} of $X$. We say that the $\mathbb{G}_{m,B}$-action on $X$ is \emph{contracting} if it extends to an action of $\mathbb{A}_B^1$, so that $X = X^+$.
\end{defi}

\begin{rmk}
    The natural action of $\mathbb{G}_{m,B}$ on $\mathbb{A}_B^1$ yields an action of $\mathbb{G}_{m,B}$ on the attractor $X^+$ with respect to which $X^+ \to X$ is $\mathbb{G}_{m,B}$-equivariant; the action of $\mathbb{G}_{m,B}$ on $X^+$ also extends to an action of the monoid $\mathbb{A}_B^1$. The closed immersion $X^0 \to X$ and the projection $X^+ \to X^0$ are also $\mathbb{G}_{m,B}$-equivariant, where $X^0$ is endowed with the trivial action of $\mathbb{G}_{m,B}$.
\end{rmk}

\begin{rmk}
    If $B = \Spec k$ where $k$ is a field, then Theorem \ref{theorem:attractor_loci_existence} is a consequence of the results of \cite{Drinfeld-Gm}, which in turn provides an extension of the results of Białynicki-Birula \cite{BB-Theorem} of tori on smooth projective varieties to algebraic spaces over a field.
\end{rmk}

\subsection{Stacks of Filtered and Graded Objects}

Here we recall the basic definitions and properties of stacks of filtered and graded objects in an algebraic stack, as introduced by Halpern-Leistner \cite{Halpern-Leistner}, and how they relate to weak actions of the stack $\Theta$ on an algebraic stack. We also recall the description of $\mathrm{Filt}([X/G])$ and $\mathrm{Grad}([X/G])$ in terms of quotients of $\mathbb{G}_m$-attractor loci, where $G$ is a smooth linear algebraic group acting on a scheme $X$ of finite type over a field.

\subsubsection{The Stack $\Theta$}

Following Halpern-Leistner, we introduce the stack $\Theta_B$, which plays an important role in the present article.

\begin{defi}[\cite{Halpern-Leistner}]
    Let $\mathbb{G}_{m,B}$ act on the affine line $\mathbb{A}_B^1$ via scalar multiplication. We define $\Theta_B$ to be the corresponding quotient stack:
    $$ \Theta_B := [\mathbb{A}_B^1 / \mathbb{G}_{m, B}]. $$
\end{defi}

By \cite[Proposition 10.3.7]{Olsson}, if $S$ is a scheme over $B$ then giving a morphism $S \to \Theta_B$ is equivalent to specifying an invertible $\mathcal{O}_S$-module $\mathcal{L}$ and a homomorphism $s : \mathcal{L} \to \mathcal{O}_S$. 

Identifying $B\mathbb{G}_{m,B}$ with the classifying stack of invertible sheaves over $B$-schemes, there are natural morphisms relating $\Theta_B$, $B\mathbb{G}_{m,B}$ and $B$:
\begin{itemize}
    \item $0 : B \mathbb{G}_{m,B} \to \Theta_B$, given by $\mathcal{L} \mapsto (\mathcal{L}, 0)$.

    \item $q : \Theta_B \to B\mathbb{G}_{m,B}$, given by $(\mathcal{L},s) \mapsto \mathcal{L}$.

    \item $1 : B \to \Theta_B$ given by $S \mapsto (\mathcal{O}_S, \mathrm{id}_{\mathcal{O}_S})$ (where $S$ is a $B$-scheme).
\end{itemize}

There is also a multiplication map $\mu : \Theta_B \times_B \Theta_B \to \Theta_B$, given by
$$ ((\mathcal{L}_1, s_1), (\mathcal{L}_2, s_2)) \mapsto (\mathcal{L}_1 \otimes \mathcal{L}_2, s_1 \otimes s_2). $$
Under this multiplication map, $\Theta_B$ is a monoid stack over $B$, for which the maps $0 : B\mathbb{G}_{m,B} \to \Theta_B$ and $q : \Theta_B \to B\mathbb{G}_{m,B}$ are homomorphisms of monoid stacks (endowing $B\mathbb{G}_{m,B}$ with its usual group stack structure).

\subsubsection{The Stacks $\mathrm{Filt}(\mathcal{Y})$ and $\mathrm{Grad}(\mathcal{Y})$}

Let $\mathcal{Y}$ be an algebraic stack over $B$. We denote by $\mathrm{Filt}(\mathcal{Y})$ (resp. $\mathrm{Grad}(\mathcal{Y})$) the presheaf of groupoids
$$ S \mapsto \mathrm{Map}(\Theta_B \times_B S, \mathcal{Y}) \ \ (\text{resp.}\ \mathrm{Map}(B\mathbb{G}_{m,B} \times_B S, \mathcal{Y})), $$
where $S$ is a $B$-scheme, and where $\mathrm{Map}(\mathcal{Y}' \times_B S, \mathcal{Y})$ denotes the groupoid of $1$-morphisms of $B$-stacks $\mathcal{Y}' \times_B S \to \mathcal{Y}$.

\begin{teo} \label{thm:algebraicity-filt-grad}
    Assume that the morphism $\mathcal{Y} \to B$ is quasi-separated and locally of finite presentation, and that all points of $\mathcal{Y}$ have affine automorphism groups in their fiber over $B$. Then $\mathrm{Filt}(\mathcal{Y})$ and $\mathrm{Grad}(\mathcal{Y})$
    are both representable by algebraic stacks, locally of finite presentation and quasi-separated over $B$. Moreover, if $\mathcal{Y} \to B$ has an affine (resp. quasi-affine, separated) diagonal, then the same is true for both $\mathrm{Filt}(\mathcal{Y}) \to B$ and $\mathrm{Grad}(\mathcal{Y}) \to B$.
\end{teo}

\begin{proof}
    This is \cite[Proposition 1.1.2]{Halpern-Leistner}.
\end{proof}

\begin{defi}
    The stack $\mathrm{Filt}(\mathcal{Y})$ (resp. $\mathrm{Grad}(\mathcal{Y})$) is known as the \emph{stack of $\Z$-filtered objects} (resp. \emph{stack of $\Z$-graded objects}) of $\mathcal{Y}$.
\end{defi}

\begin{rmk} \label{rmk:filt-n}
    For any integer $n \geq 1$, there are algebraic stacks $\mathrm{Filt}^n(\mathcal{Y})$ and $\mathrm{Grad}^n(\mathcal{Y})$ of $\Z^n$-filtered and $\Z^n$-graded objects of $\mathcal{Y}$ respectively. We only consider these stacks when $n = 1$, since we have no need to work with $\mathrm{Filt}^n(\mathcal{Y})$ or $\mathrm{Grad}^n(\mathcal{Y})$ when $n > 1$.
\end{rmk}

The natural morphisms relating $\Theta_B$, $B\mathbb{G}_{m,B}$ and $B$ induce natural morphisms relating $\mathrm{Filt}(\mathcal{Y})$ and $\mathrm{Grad}(\mathcal{Y})$:
\begin{itemize}
    \item The morphism $0 : B\mathbb{G}_{m,B} \to \Theta_B$ corresponds to the \emph{associated graded} morphism, denoted $\mathrm{gr} : \mathrm{Filt}(\mathcal{Y}) \to \mathrm{Grad}(\mathcal{Y})$.

    \item The morphism $q : \Theta_B \to B\mathbb{G}_{m,B}$ corresponds to the \emph{split filtration} morphism, denoted $\sigma : \mathrm{Grad}(\mathcal{Y}) \to \mathrm{Filt}(\mathcal{Y})$.

    \item The morphism $1 : B \to \Theta_B$ corresponds to the \emph{evaluation at $1$} morphism, denoted $\mathrm{ev}_1 : \mathrm{Filt}(\mathcal{Y}) \to B$.
\end{itemize}

We denote by $\mathrm{ev}_0 : \mathrm{Filt}(\mathcal{Y}) \to B$ the composition $\mathrm{ev}_1 \circ \sigma \circ \mathrm{gr}$ and denote by $u : \mathrm{Grad}(\mathcal{Y}) \to \mathcal{Y}$ the composition $\mathrm{ev}_1 \circ \sigma$.

\begin{ex}
    Take $B = \Spec k$ (where $k$ is a field), and take $\mathcal{Y}$ to be the moduli stack of vector bundles on a smooth projective curve $C$ over $k$ of rank $r$ and degree $d$. Then, as shown in \cite[Proposition 1.0.1]{Halpern-Leistner} and \cite[Lemma 1.10]{Heinloth-HM}, specifying a morphism $f : \Theta_k \to \mathcal{Y}$ is equivalent to specifying a vector bundle $\mathcal{E}$ over $C$ of rank $r$ and degree $d$, together with a decreasing $\Z$-indexed filtration
    $$ \mathcal{E}^{\bullet} : \cdots \subset \mathcal{E}^{i+1} \subset \mathcal{E}^i \subset \mathcal{E}^{i-1} \subset \cdots $$
    of subbundles of $\mathcal{E}$, with $\mathcal{E}^i = 0$ for $i \gg 0$ and $\mathcal{E}^i = \mathcal{E}$ for $i \ll 0$. The associated graded morphism $\mathrm{gr}$ sends this to the $\Z$-indexed sequence of vector bundles $(\mathcal{E}^i/\mathcal{E}^{i+1})_{i \in \Z}$, and the evaluation morphism $\mathrm{ev}_1$ sends this to $\mathcal{E}$ (i.e. forgets the filtration). The morphism $\sigma$ assigns to $(\mathcal{E}^i/\mathcal{E}^{i+1})_{i \in \Z}$ the vector bundle $\mathcal{F} := \bigoplus_{i \in \mathbb{Z}} \mathcal{E}^i/\mathcal{E}^{i+1}$, with the induced split filtration given by $\mathcal{F}_i := \bigoplus_{j \geq i} \mathcal{E}^j/\mathcal{E}^{j+1}$.
\end{ex}

\subsubsection{An Alternative Description of $\mathrm{Grad}(\mathcal{Y})$}

The algebraic stack $\mathrm{Grad}(\mathcal{Y})$ admits an alternative description in terms of the stack of homomorphisms from $\mathbb{G}_{m, \mathcal{Y}}$ to the inertia stack $\mathcal{I}_{\mathcal{Y}} = \mathcal{I}_{\mathcal{Y}/B}$ of $\mathcal{Y}$.

Let $S$ be a scheme over $B$, and let $\xi \in \mathcal{Y}(S)$. We denote by $\underline{\mathrm{Aut}}_{\mathcal{Y}}(\xi)$ the sheaf of groups on $S$ assigning to a morphism $S' \to S$ the group $\mathrm{Aut}_{\mathcal{Y}}(\xi|_{S'})$.

\begin{prop} \label{prop:alt-desc-grad}
    Assume that the morphism $\mathcal{Y} \to B$ is quasi-separated and locally of finite presentation, and that all points of $\mathcal{Y}$ have affine automorphism groups in their fiber over $B$. Let $S$ be a scheme over $B$, and let $\xi \in \mathcal{Y}(S)$. Then there exists a Cartesian diagram
    \[\begin{tikzcd}[ampersand replacement=\&,cramped]
	{\underline{\mathrm{Hom}}_{S-\mathrm{gp}}(\mathbb{G}_{m,S}, \underline{\mathrm{Aut}}_{\mathcal{Y}}(\xi))} \&\& {\mathrm{Grad}(\mathcal{Y})} \\
	\\
	S \&\& {\mathcal{Y}}
	\arrow[from=1-1, to=1-3]
	\arrow[from=1-1, to=3-1]
	\arrow["\lrcorner"{anchor=center, pos=0.125}, draw=none, from=1-1, to=3-3]
	\arrow["u", from=1-3, to=3-3]
	\arrow["\xi", from=3-1, to=3-3]
\end{tikzcd}\]
In particular, there exists a canonical equivalence
$$ \mathrm{Grad}(\mathcal{Y}) \cong \underline{\mathrm{Hom}}_{\mathcal{Y}-\mathrm{gp}}(\mathbb{G}_{m,\mathcal{Y}}, \mathcal{I}_{\mathcal{Y}}) $$
of stacks over $\mathcal{Y}$. Moreover, after using this isomorphism to view the objects of $\mathrm{Grad}(\mathcal{Y})(S)$ as pairs of the form $(\xi : S \to \mathcal{Y}, \lambda : \mathbb{G}_{m,S} \to \underline{\mathrm{Aut}}_{\mathcal{Y}}(\xi))$, an isomorphism $(\xi_1, \lambda_1) \cong (\xi_2, \lambda_2)$ in $\mathrm{Grad}(\mathcal{Y})(S)$ is the same data as an isomorphism $\alpha : \xi_1 \stackrel{\simeq}{\to} \xi_2$ in $\mathcal{Y}(S)$ for which
$$ \alpha = \lambda_2(t) \circ \alpha \circ \lambda_1(t)^{-1} $$
for all $t \in \mathbb{G}_{m,S}$. 

\end{prop}

\begin{proof}
    This is \cite[Lemma 1.1.5]{Halpern-Leistner}.
\end{proof}

\subsubsection{Weak $\Theta$-Actions and Stacks of Filtered Objects}

We recall the notion of a weak action of a monoid stack $\mathcal{M}$ on a stack $\mathcal{X}$.

\begin{defi}
    Let $\mathcal{X}$ be a stack and let $\mathcal{M}$ be a monoid stack. A \emph{weak action} of $\mathcal{M}$ on $\mathcal{X}$ consists of an action map $a: \mathcal{M} \times \mathcal{X} \to \mathcal{X}$ together with $2$-isomorphisms $a|_{\{1\} \times \mathcal{X}} \cong \mathrm{id}_{\mathcal{X}}$ and $a \circ (\mathrm{id}_{\mathcal{M}} \times a) \cong a \circ (\mu \times \mathrm{id}_{\mathcal{X}}) : \mathcal{M} \times \mathcal{M} \times \mathcal{X} \to \mathcal{X}$, where $\mu : \mathcal{M} \times \mathcal{M} \to \mathcal{M}$ is the monoid multiplication.
\end{defi}

Denoting by $\mu : \Theta_B \times_B \Theta_B \to \Theta_B$ multiplication map on $\Theta_B$, there is a natural morphism $\Theta_B \times_B \mathrm{Filt}(\mathcal{Y}) \to \mathrm{Filt}(\mathcal{Y})$, corresponding to the composition
\[\begin{tikzcd}[ampersand replacement=\&,cramped]
	{\Theta_B \times_B \Theta_B \times_B \mathrm{Filt}(\mathcal{Y})} \&\& {\Theta_B \times_B \mathrm{Filt}(\mathcal{Y})} \&\& {\mathcal{Y}}
	\arrow["{\mu \times \mathrm{id}_{\mathrm{Filt}(\mathcal{Y}).}}", from=1-1, to=1-3]
	\arrow["{\mathrm{ev}}", from=1-3, to=1-5]
\end{tikzcd}\]
Similarly, there is a morphism $B\mathbb{G}_{m,B} \times_B \mathrm{Grad}(\mathcal{Y}) \to \mathrm{Grad}(\mathcal{Y})$.

\begin{prop} \label{prop:filt-weak-actions}
Assume that the morphism $\mathcal{Y} \to B$ is quasi-separated and locally of finite presentation, and that all points of $\mathcal{Y}$ have affine automorphism groups in their fiber over $B$. The morphisms $\Theta_B \times_B \mathrm{Filt}(\mathcal{Y}) \to \mathrm{Filt}(\mathcal{Y})$ and $B\mathbb{G}_{m,B} \times_B \mathrm{Grad}(\mathcal{Y}) \to \mathrm{Grad}(\mathcal{Y})$ both define weak actions, with respect to which the morphisms $\mathrm{gr}: \mathrm{Filt}(\mathcal{Y}) \to \mathrm{Grad}(\mathcal{Y})$ and $\sigma : \mathrm{Grad}(\mathcal{Y}) \to \mathrm{Filt}(\mathcal{Y})$ are equivariant (up to $2$-isomorphism) with respect to $q : \Theta_B \to B\mathbb{G}_{m,B}$ and $0 : B\mathbb{G}_{m,B} \to \Theta_B$ respectively. Moreover, the weak $\Theta_B$-action (resp. weak $B\mathbb{G}_{m,B}$-action) restricts to any open and closed substack $\mathcal{X} \subset \mathrm{Filt}(\mathcal{Y})$ (resp. $\mathcal{X} \subset \mathrm{Grad}(\mathcal{Y})$).

Conversely, let $\mathcal{X}$ be an algebraic stack which is locally of finite presentation over $B$ with a weak action of $\Theta_B$ (resp. $B\mathbb{G}_{m,B}$). Assume further that the diagonal of $\mathcal{X} \to B$ is affine. Then the map $\phi : \mathcal{X} \to \mathrm{Filt}(\mathcal{X})$ (resp. $\phi : \mathcal{X} \to \mathrm{Grad}(\mathcal{X})$) corresponding to the action map $\Theta_B \times_B \mathcal{X} \to \mathcal{X}$ (resp. $B\mathbb{G}_{m,B} \times_B \mathcal{X} \to \mathcal{X}$) is a $\Theta_B$-equivariant (resp. $B\mathbb{G}_{m,B}$-equivariant) open and closed immersion for which $\mathrm{ev}_1 : \phi(\mathcal{X}) \to \mathcal{X}$ (resp. $u : \phi(\mathcal{X}) \to \mathcal{X}$) is an isomorphism.
\end{prop}

\begin{proof}
    For the assertion that $\Theta_B \times_B \mathrm{Filt}(\mathcal{Y}) \to \mathrm{Filt}(\mathcal{Y})$ defines a weak action which restricts to any open and closed substack, see \cite[Proposition 1.2.4]{Halpern-Leistner-Derived}. For the assertion that weak $\Theta_B$-actions on algebraic stacks with affine diagonals induce open and closed immersions into the stack of filtered objects, see \cite[Proposition 1.4.1, Corollary 1.4.2.1]{Halpern-Leistner-Derived}. The analogous statements involving weak actions of $B\mathbb{G}_{m,B}$ and open and closed immersions into stacks of graded objects follow \emph{mutatis mutandis}. The equivariance of the morphisms $\mathrm{gr}: \mathrm{Filt}(\mathcal{Y}) \to \mathrm{Grad}(\mathcal{Y})$ and $\sigma : \mathrm{Grad}(\mathcal{Y}) \to \mathrm{Filt}(\mathcal{Y})$ is easily checked.
\end{proof}

\begin{defi}[\cite{Halpern-Leistner-Derived}] \label{defn:center-weak-action}
    Let $\mathcal{X}$ be an algebraic stack which is locally of finite presentation over $B$ and for which $\mathcal{X} \to B$ has an affine diagonal. Let $a : \Theta_B \times_B \mathcal{X} \to \mathcal{X}$ be a weak action of $\Theta_B$ on $\mathcal{X}$. Let $\phi : \mathcal{X} \to \mathrm{Filt}(\mathcal{X})$ be the morphism classifying this action. The \emph{center} of the weak action of $\Theta_B$ on $\mathcal{X}$ is the stack $\mathcal{Z} := \mathcal{X} \times_{\phi, \mathrm{Filt}(\mathcal{X}), \sigma} \mathrm{Grad}(\mathcal{X}) = \sigma^{-1}(\mathcal{X})$:
    \[\begin{tikzcd}[ampersand replacement=\&]
	{\mathcal{Z}} \&\& {\mathrm{Grad}(\mathcal{X})} \\
	\\
	{\mathcal{X}} \&\& {\mathrm{Filt}(\mathcal{X})}
	\arrow[from=1-1, to=1-3]
	\arrow[from=1-1, to=3-1]
	\arrow["\lrcorner"{anchor=center, pos=0.125}, draw=none, from=1-1, to=3-3]
	\arrow["\sigma", from=1-3, to=3-3]
	\arrow["\phi"', from=3-1, to=3-3]
    \end{tikzcd}\]
\end{defi}

In the setting of Definition \ref{defn:center-weak-action}, the weak action of $B\mathbb{G}_{m,B}$ on $\mathrm{Grad}(\mathcal{X})$ restricts to a weak action on the open and closed substack $\mathcal{Z} \subset \mathrm{Grad}(\mathcal{X})$, and there are induced equivariant morphisms $\mathrm{gr} : \mathcal{X} \to \mathcal{Z}$ and $\sigma : \mathcal{Z} \to \mathcal{X}$, for which there are isomorphisms $\sigma \circ \mathrm{gr} \cong a(0,-) : \mathcal{X} \to \mathcal{X}$ and $\mathrm{gr} \circ \sigma \cong \mathrm{id}_{\mathcal{Z}}$, and for which the diagram
\[\begin{tikzcd}
	{\mathcal{X}} && {\mathrm{Filt}(\mathcal{X})} \\
	\\
	{\mathcal{Z}} && {\mathrm{Grad}(\mathcal{X})}
	\arrow["\phi", from=1-1, to=1-3]
	\arrow["{\mathrm{gr}}"', from=1-1, to=3-1]
	\arrow["\lrcorner"{anchor=center, pos=0.125}, draw=none, from=1-1, to=3-3]
	\arrow["{\mathrm{gr}}", from=1-3, to=3-3]
	\arrow[from=3-1, to=3-3]
\end{tikzcd}\]
is Cartesian. 

\begin{defi}
    The induced morphism $\mathrm{gr} : \mathcal{X} \to \mathcal{Z}$ is known as the \emph{retraction} associated to the weak $\Theta_B$-action on $\mathcal{X}$. The morphism $\sigma : \mathcal{Z} \to \mathcal{X}$ is known as the \emph{zero section}.
\end{defi}

The morphism $\mathcal{Z} \to B$ is locally of finite presentation and has an affine diagonal (as this is true for $\mathrm{Grad}(\mathcal{X}) \to B$). In particular, $\mathrm{gr} : \mathcal{X} \to \mathcal{Z}$ and $\sigma : \mathcal{Z} \to \mathcal{X}$ are both morphisms locally of finite presentation. We also have that the equivariant open and closed immersion $\mathcal{Z} \to \mathrm{Grad}(\mathcal{X})$ factors as an equivariant open and closed immersion $\mathcal{Z} \to \mathrm{Grad}(\mathcal{Z})$ followed by the morphism $\mathrm{Grad}(\mathcal{Z}) \to \mathrm{Grad}(\mathcal{X})$ induced by $\sigma : \mathcal{Z} \to \mathcal{X}$. 

We remark that the stacks $\mathcal{X}$ and $\mathcal{Z}$, together with the morphisms $\mathrm{gr}$ and $\sigma$, define a \emph{$\Theta_B$-action retract}, in the sense of \cite[Sections 5.1.10 and 5.1.11]{BHLINK-Intrinsic-DT}. In particular, this implies that the retraction $\mathrm{gr}$ is a \emph{$\Theta_B$-deformation retract}, in the sense of \cite[Section 1.3.1]{Halpern-Leistner}, i.e. the diagram
    \[\begin{tikzcd}[ampersand replacement=\&]
	{\Theta_B \times_B \mathcal{X}} \&\& {\mathcal{X}} \\
	\\
	{\mathcal{X}} \&\& {\mathcal{Z}}
	\arrow["a", from=1-1, to=1-3]
	\arrow["{\mathrm{pr}_{\mathcal{X}}}"', from=1-1, to=3-1]
	\arrow["{{\mathrm{gr}}}", from=1-3, to=3-3]
	\arrow["{\mathrm{gr}}"', from=3-1, to=3-3]
    \end{tikzcd}\]
is (2-)commutative. Consequently, if $\mathcal{U} \subset \mathcal{X}$ is any open substack with $\sigma^{-1}(\mathcal{U}) = \mathcal{Z}$, then $\mathcal{U} = \mathcal{X}$, and moreover that $\mathrm{gr}$ induces a bijection on connected components (with inverse given by $\sigma$).

\begin{ex}
    If $\mathcal{S}$ is a $\Theta_B$-stratum in an algebraic stack $\mathcal{X}$ over $B$, in the sense of \cite{Halpern-Leistner}, then $\mathcal{S}$ is isomorphic over $B$ to an open and closed substack of the stack $\mathrm{Filt}(\mathcal{X})$, so admits a weak $\Theta_B$-action. This action can be chosen such that the center of $\mathcal{S}$ associated to the weak $\Theta_B$-action coincides with the center of $\mathcal{S}$ as a $\Theta_B$-stratum.
\end{ex}

\begin{rmk}
    An analog of Proposition \ref{prop:filt-weak-actions} holds for weak actions of $\Theta_B^n$ for any $n \geq 1$, by replacing $\mathrm{Filt}$ with $\mathrm{Filt}^n$ and $\mathrm{Grad}$ with $\mathrm{Grad}^n$. Given a weak action of $\Theta_B^n$ on an algebraic stack $\mathcal{X}$ over $B$ (quasi-separated, locally of finite presentation and with an affine diagonal), there is an induced weak action of $\Theta_B$, given by taking the diagonal morphism $\Theta_B \to \Theta_B^n$; these actions will have the same center $\mathcal{Z}$. For the purposes of this article, it is only this induced diagonal action we are interested in considering (cf. Remark \ref{rmk:filt-n}).
\end{rmk}

\subsubsection{The Canonical Rigidification of the Center}

Let $\mathcal{X}$ be an algebraic stack which is quasi-separated and locally of finite presentation over $B$, with an affine diagonal, and with a weak action $a : \Theta_B \times_B \mathcal{X} \to \mathcal{X}$ of $\Theta_B$. Let $\mathcal{Z}$ be the center of this action.

There exists a canonical homomorphism $\lambda^{\mathrm{can}} : \mathbb{G}_{m,\mathcal{Z}} \to \mathcal{I}_{\mathcal{Z}}$ of group stacks over $\mathcal{Z}$, characterized as follows. Suppose $\xi : S \to \mathcal{Z} \subset \mathrm{Grad}(\mathcal{X})$ is a morphism from a scheme. By Proposition \ref{prop:alt-desc-grad}, the $S$-point $\xi$ corresponds to a pair $(u(\xi), \lambda_{\xi} : \mathbb{G}_{m,S} \to \underline{\mathrm{Aut}}_{\mathcal{X}}(u(\xi)))$. An automorphism of $\xi$ over $S$ is then an automorphism of $u(\xi) \in \mathcal{X}(S)$ invariant under conjugating by $\lambda_{\xi}$. Since $\lambda_{\xi}$ commutes with itself, it defines a cocharacter $\mathbb{G}_{m,S} \to \underline{\mathrm{Aut}}_{\mathcal{Z}}(\xi)$; over $S$, $\lambda^{\mathrm{can}}$ is this cocharacter.

Let $\mathcal{G}_{\mathcal{Z}}^{\mathrm{can}} \subset \mathcal{I}_{\mathcal{Z}}$ denote the image of $\lambda^{\mathrm{can}}$.

\begin{prop} \label{prop:rigidified-center}
    The image $\mathcal{G}_{\mathcal{Z}}^{\mathrm{can}}$ is a central closed subgroup stack of $\mathcal{I}_{\mathcal{Z}}$ which is flat and of finite presentation over $\mathcal{Z}$. Moreover, the rigidification $\mathcal{Z} \to \mathcal{Z} \!\!\! \fatslash \mathcal{G}_{\mathcal{Z}}^{\mathrm{can}}$ is well-defined, and over any connected component $\mathcal{Z}'$ of $\mathcal{Z}$ is either a $\mathbb{G}_m$-gerbe or an isomorphism, the latter happening if and only if $\lambda^{\mathrm{can}}|_{\mathcal{Z}'}$ factors through the identity section $\mathcal{Z}' \to \mathcal{I}_{\mathcal{Z}'}$.
\end{prop}

\begin{proof}
    This is \cite[Proposition 4.9]{Modin}.
\end{proof}

\subsubsection{The Internal $\mathbb{G}_m$-Action Arising from a Weak $\Theta$-Action} \label{sec:internal-weak-Gm-action}

Let $\mathcal{X}$ and $\mathcal{Z}$ be as above. By pulling back along $B \to B\mathbb{G}_{m,B}$, the weak $\Theta_B$-action on $\mathcal{X}$ lifts to a weak $\mathbb{A}_B^1$-action on $\mathcal{X}$, with respect to which the retraction $\mathrm{gr} : \mathcal{X} \to \mathcal{Z}$ is invariant. By restricting to the open subscheme $\mathbb{G}_{m,B} \subset \mathbb{A}_B^1$, we obtain a weak action\footnote{In the notation of \cite{Modin}, this is denoted by $\mathrm{rs}_{\ast}$. In fact, as shown in \cite[Lemma 4.2]{Modin}, the morphism $a'$ is a strict $\mathbb{G}_{m,B}$-action.} $a' : \mathbb{G}_{m,B} \times_B \mathcal{X} \to \mathcal{X}$.

On the other hand, suppose $\xi : S \to \mathcal{Z} \subset \mathrm{Grad}(\mathcal{X})$ is a morphism from a scheme. Then, on $\mathcal{X}_S := \mathcal{X} \times_\mathcal{Z} S$, there is a weak action\footnote{In the notation of \cite{Modin}, this is denoted by 
$\mu^{\mathrm{can}}$.} $a_S^{\mathrm{can}} : \mathbb{G}_{m,S} \times_S \mathcal{X}_S \to \mathcal{X}_S$, defined as follows. Given an $S$-scheme $S'$, an $S'$-point of $\mathcal{X}_S$ is equivalent to the data of an object $\zeta \in \mathcal{X}(S')$ and an invariant isomorphism $\alpha : u(\mathrm{gr}(\zeta)) \stackrel{\simeq}{\to} u(\xi|_{S'})$ in $\mathcal{X}(S')$, where invariance is with respect to the canonical cocharacters $\lambda_{\xi|_{S'}} : \mathbb{G}_{m,S'} \to \underline{\mathrm{Aut}}_{\mathcal{X}}(u(\xi|_{S'}))$ and $\lambda_{\mathrm{gr}(\zeta)} : \mathbb{G}_{m,S'} \to \underline{\mathrm{Aut}}_{\mathcal{X}}(u(\mathrm{gr}(\zeta)))$, as in Proposition \ref{prop:alt-desc-grad}. Over $S'$, the morphism $a_S^{\mathrm{can}}$ is defined by
$$ (t, (\zeta, \alpha)) \mapsto (\zeta, \alpha \circ \lambda_{\mathrm{gr}(\zeta)}(t^{-1})). $$

By pulling back $a'$, we obtain a second weak action $a_S' : \mathbb{G}_{m,S} \times_S \mathcal{X}_S \to \mathcal{X}_S$. In fact, these two weak actions turn out to be equivalent.

\begin{lem} \label{lem:internal-Gm-action}
    There is an isomorphism $a_S' \cong a_S^{\mathrm{can}}$ of morphisms $\mathbb{G}_{m,S} \times_S \mathcal{X}_S \to \mathcal{X}_S$. Moreover, $a_S^{\mathrm{can}}$ factors through $\mathcal{G}_S^{\mathrm{can}} \times_S \mathcal{X}_S$. 
    
    In particular, if $\mathcal{G}_{\mathcal{Z}}^{\mathrm{can}}$ is the trivial $\mathcal{Z}$-group stack then $a'$ is equivalent to the trivial $\mathbb{G}_{m,B}$-action on $\mathcal{X}$, and hence the weak action on $\mathcal{X}$ is equivalent to the trivial $\Theta_B$-action.
\end{lem}

\begin{proof}
    The assertions in the first paragraph are \cite[Lemmas 4.4 and 4.10]{Modin}. It then easily follows that the triviality of $\mathcal{G}_{\mathcal{Z}}^{\mathrm{can}}$ implies that $a'$ is equivalent to the trivial action. 
    
    For the final assertion, the triviality of $a'$ implies that the weak $\mathbb{A}_B^1$-action on $\mathcal{X}$ is equivalent to the trivial action, which implies that the original weak $\Theta_B$-action is equivalent to the trivial action.
\end{proof}

\subsubsection{Filtered and Graded Objects of Quotient Stacks} \label{section:filt-grad-quotient-stacks}

Let $k$ be a field. We restrict attention to $\mathcal{Y} = [X/G]$, where $X$ is a quasi-separated scheme over $B = \Spec k$ and where $G$ is a smooth linear algebraic group over $k$ with a split maximal torus $T \subset G$. Let $W = N_G(T)/C_G(T)$ be the corresponding Weyl group.

Given a cocharacter $\lambda : \mathbb{G}_{m,k} \to T$, denote by $X^{\lambda,+}$ (resp. $X^{\lambda,0}$) the attractor (resp. fixed) locus for the induced action of $\mathbb{G}_{m,k}$ on $X$ via $\lambda$. The homomorphism $\lambda$ also yields a conjugation action of $\mathbb{G}_{m,k}$ on $G$, given by $t \cdot g = \lambda(t)g\lambda(t)^{-1}$; let $P_G(\lambda)$ (resp. $L_G(\lambda)$) denote the attractor (resp. fixed) locus for this action. $P_G(\lambda)$ and $L_G(\lambda)$ are subgroups of $G$, and the natural map $P_G(\lambda) \to L_G(\lambda)$ is a split surjective group homomorphism; denote by $U_G(\lambda)$ the kernel of this homomorphism, so that $P_G(\lambda) = U_G(\lambda) \rtimes L_G(\lambda)$. The cocharacter $\lambda : \mathbb{G}_{m,k} \to G$ factors through the center $Z(L_G(\lambda)) \subset G$ of $L_G(\lambda)$. If $G$ is reductive, then $P_G(\lambda)$ is a parabolic subgroup of $G$, and $U_G(\lambda)$ is the unipotent radical of $P_G(\lambda)$.

Since the formation of attractor and fixed loci commutes with products, applying $(-)^{\lambda,+}$ (resp. $(-)^{\lambda,0}$) to the action map $G \times X \to X$ yields a group action $P_G(\lambda) \times X^{\lambda,+} \to X^{\lambda,+}$ (resp. $L_G(\lambda) \times X^{\lambda,0} \to X^{\lambda,0}$). The natural map $X^{\lambda,+} \to X^{\lambda,0}$ is equivariant with respect to $P_G(\lambda) \to L_G(\lambda)$, so there is an induced map on quotient stacks $[X^{\lambda,+}/P_G(\lambda)] \to [X^{\lambda,0}/L_G(\lambda)]$. Similarly, there are natural maps $[X^{\lambda,0}/L_G(\lambda)] \to [X^{\lambda,+}/P_G(\lambda)]$ and $[X^{\lambda,+}/P_G(\lambda)] \to [X/G]$.

There is also a natural map $[X^{\lambda,+}/P_G(\lambda)] \to \mathrm{Filt}([X/G])$, defined as follows. The trivial principal $G$-bundle over $\mathbb{A}_k^1 \times X^{\lambda,+}$ admits a well-defined $\mathbb{G}_{m,k} \times P_G(\lambda)$-equivariant structure, given by the formula
$$ (t,p) \cdot (z,x,g) = (tz, p \cdot x, \lambda(tz)p\lambda(z)^{-1}g), $$
yielding an action of $\mathbb{G}_{m,k} \times P_G(\lambda)$ on $\mathbb{A}_k^1 \times X^{\lambda,+} \times G$ which commutes with right multiplication by $G$, and for which the map
$$ \mathbb{A}_k^1 \times X^{\lambda,+} \times G \to X, \quad (z,x,g) \mapsto g^{-1}\lambda(z) \cdot x $$
is $\mathbb{G}_{m,k} \times P_G(\lambda)$-equivariant. This defines a morphism $\Theta_k \times [X^{\lambda,+}/P_G(\lambda)] \to [X/G]$, and hence a morphism $[X^{\lambda,+}/P_G(\lambda)] \to \mathrm{Filt}([X/G])$.

We also have a natural map $[X^{\lambda,0}/L_G(\lambda)] \to \mathrm{Grad}([X/G])$, corresponding to the morphism $B\mathbb{G}_{m,k} \times [X^{\lambda,0}/L_G(\lambda)] = [X^{\lambda,0}/(\mathbb{G}_{m,k} \times L_G(\lambda))] \to [X/G]$ defined by the inclusion $X^{\lambda,0} \hookrightarrow X$, which is equivariant with respect to the homomorphism $\mathbb{G}_{m,k} \times L_G(\lambda) \to G$ given by $(t,l) \mapsto \lambda(t)l$.

\begin{teo} \label{thm:filt-grad-quotient-stack}
    Let $k$ be a field, and let $\mathcal{Y} = [X/G]$ be the quotient of a quasi-separated $k$-scheme $X$ by a smooth linear algebraic $k$-group $G$, with a split maximal torus $T$ and corresponding Weyl group $W$. 
    
    Then, for any cocharacter $\lambda : \mathbb{G}_m \to G$, the natural maps $[X^{\lambda,+}/P_G(\lambda)] \to \mathrm{Filt}([X/G])$ and $[X^{\lambda,0}/L_G(\lambda)] \to \mathrm{Grad}([X/G])$ are open and closed immersions. The restriction of the morphism $\mathrm{gr} : \mathrm{Filt}([X/G]) \to \mathrm{Grad}([X/G])$ to $[X^{\lambda,+}/P_G(\lambda)]$ is given by the natural morphism $[X^{\lambda,+}/P_G(\lambda)] \to [X^{\lambda,0}/L_G(\lambda)]$, the restriction of $\sigma : \mathrm{Grad}([X/G]) \to \mathrm{Filt}([X/G])$ to $[X^{\lambda,0}/L_G(\lambda)]$ is the natural map $[X^{\lambda,0}/L_G(\lambda)] \to [X^{\lambda,+}/P_G(\lambda)]$, and the restriction of the evaluation morphism $\mathrm{ev}_1 : \mathrm{Filt}([X/G]) \to [X/G]$ to $[X^{\lambda,+}/P_G(\lambda)]$ is the natural morphism $[X^{\lambda,+}/P_G(\lambda)] \to [X/G]$. 
    
    Moreover, after fixing representatives for each orbit of the action of $W$ on $\mathrm{Hom}(\mathbb{G}_{m,k},T)$, the natural maps $[X^{\lambda,+}/P_G(\lambda)] \to \mathrm{Filt}([X/G])$ and $[X^{\lambda,0}/L_G(\lambda)] \to \mathrm{Grad}([X/G])$ yield isomorphisms
    $$ \mathrm{Filt}([X/G]) \cong \bigsqcup_{\lambda \in \mathrm{Hom}(\mathbb{G}_{m,k},T)/W} [X^{\lambda,+}/P_G(\lambda)] $$ 
    and
    $$ \mathrm{Grad}([X/G]) \cong \bigsqcup_{\lambda \in \mathrm{Hom}(\mathbb{G}_{m,k},T)/W} [X^{\lambda,0}/L_G(\lambda)]. $$
\end{teo}

\begin{proof}
    This is \cite[Theorem 1.4.8]{Halpern-Leistner}.
\end{proof}

\subsection{Good Moduli Spaces}

We recall Alper's notion of a good moduli space of an algebraic stack, which is the stack-intrinsic analog of the notion of a GIT good quotient for an action of a linearly reductive group.

\begin{defi}[\cite{Alper-GMS}]
    A morphism $\phi : \mathcal{Y} \to Y$ from an algebraic stack $\mathcal{Y}$ to an algebraic space $Y$ is a \emph{good moduli space} if:
    \begin{enumerate}
        \item $\phi$ is quasi-compact and quasi-separated;
        \item the natural map $\mathcal{O}_Y \to \phi_{\ast} \mathcal{O}_{\mathcal{Y}}$ is an isomorphism; and
        \item $\phi$ is cohomologically affine, i.e. the functor $\phi_{\ast} : \mathrm{QCoh}(\mathcal{Y}) \to \mathrm{QCoh}(Y)$ is exact.
    \end{enumerate}
    We say $\phi$ is a \emph{tame moduli space} if in addition, for every algebraically closed field $k$, the induced morphism $[\mathcal{Y}(k)] \to Y(k)$ from the set of isomorphism classes of $k$-points of $\mathcal{Y}$ to the set of $k$-points of $Y$ is a bijection.
\end{defi}

\begin{rmk}
    More generally, a morphism $\phi : \mathcal{X} \to \mathcal{Y}$ between algebraic stacks is said to be a \emph{good moduli space morphism} if $\phi$ satisfies conditions $(1)-(3)$ above; in this sense, a good moduli space is a good moduli space morphism whose target is an algebraic space.
\end{rmk}

\begin{ex} \label{ex:alper-affine-projective-git}
    Let $k$ be a field, and let $G$ be a linear algebraic $k$-group which is linearly reductive, acting on a quasi-compact $k$-scheme $X$. If $X = \Spec A$ is affine, then by \cite[Theorem 13.6]{Alper-GMS} the natural map
    $$ [X/G] \to X /\!\!/G := \Spec A^G $$
    is a good moduli space. In general, let $\mathcal{L}$ be a $G$-linearized invertible sheaf on $X$, and let $X^{ss}(\mathcal{L})$ be the associated semistable locus, in the sense of \cite{MFK-GIT}. Then by \cite[Theorem 13.6]{Alper-GMS} the natural map
    $$ [X^{ss}(\mathcal{L})/G] \to \mathrm{Proj} \bigoplus_{m \geq 0} H^0(X, \mathcal{L}^m)^G $$
    has open image $X /\!\!/_{\mathcal{L}} G$ (if $X$ is projective and if $\mathcal{L}$ is ample then there is an equality $X /\!\!/_{\mathcal{L}} G = \mathrm{Proj} \bigoplus_{m \geq 0} H^0(X, \mathcal{L}^m)^G$), and the map $ \phi : [X^{ss}(\mathcal{L})/G] \to X /\!\!/_{\mathcal{L}} G$ is a good moduli space. Moreover, if $X^{Ms}(\mathcal{L}) \subset X^{ss}(\mathcal{L})$ denotes the Mumford stable locus,\footnote{i.e. the locus of semistable points $x \in X^{ss}(\mathcal{L})$ whose $G$-orbits in $X^{ss}(\mathcal{L})$ are closed, but without the requirement that $\dim \mathrm{Stab}_G(x) = 0$.} then the restriction of $\phi$ to $[X^{Ms}(\mathcal{L})/G]$ is a tame moduli space onto its image.
\end{ex}

\subsubsection{Basic Properties of Good Moduli Spaces}

Let $\mathcal{Y}$ be an algebraic stack with a good moduli space $\phi : \mathcal{Y} \to Y$. Assume further that the stack $\mathcal{Y}$ is quasi-separated. Then the morphism $\phi$ enjoys the following properties (cf. \cite{Alper-GMS}, unless indicated otherwise):

\begin{enumerate}
    \item $\phi$ is surjective, universally closed and universally submersive. In particular, by the same argument used to prove \cite[Tag 09MQ]{stacks-project}, the target $Y$ is also quasi-separated.
    \item $\phi$ is universal for morphisms to algebraic spaces \cite[Theorem 3.12]{AHR-Etale}. 
    \item For any two closed substacks $\mathcal{Z}_1, \mathcal{Z}_2$ of $\mathcal{Y}$, we have $\mathrm{im}(\mathcal{Z}_1) \cap \mathrm{im}(\mathcal{Z}_2) = \mathrm{im}(\mathcal{Z}_1 \cap \mathcal{Z}_2)$ as closed subspaces of $Y$, where the intersections and images are scheme-theoretic.
    \item Let $k$ be an algebraically closed field, and let $[\mathcal{Y}(k)]_{\mathrm{S}}$ denote the set of $\sim_{\mathrm{S}}$-equivalence classes of $k$-points of $\mathcal{Y}$, where $y_1 \sim_{\mathrm{S}} y_2$ if and only if their closures in $\mathcal{Y} \times_{\Z} \Spec k$ intersect. Then the natural map $[\mathcal{Y}(k)]_{\mathrm{S}} \to Y(k)$ is a bijection.
    \item Good moduli spaces are stable under base changes $Y' \to Y$, where $Y'$ is a quasi-separated algebraic space. Conversely, if $\phi' : \mathcal{Y} \to Y$ is a morphism from an algebraic stack to a quasi-separated algebraic space, and if $Y' \to Y$ is an fpqc morphism of quasi-separated algebraic spaces such that the base change $\phi'' : \mathcal{Y} \times_Y Y' \to Y'$ of $\phi'$ is a good moduli space, then $\phi'$ is a good moduli space.
    \item If $\mathcal{A}$ is a quasi-coherent sheaf of $\mathcal{O}_{\mathcal{Y}}$-algebras, then $\underline{\Spec}_{\mathcal{Y}} \mathcal{A} \to \underline{\Spec}_Y \phi_{\ast} \mathcal{A}$ is a good moduli space. In particular, an algebraic stack $\mathcal{Y}'$ which is affine over $\mathcal{Y}$ (including any closed substack of $\mathcal{Y}$) admits a good moduli space $Y'$ which is affine over $Y$.
    \item If $\mathcal{U} \subset \mathcal{Y}$ is an open substack of $\mathcal{Y}$, then the restriction of $\phi$ to $\mathcal{U}$ yields a good moduli space $\mathcal{U} \to \phi(\mathcal{U})$ if and only if $\mathcal{U}$ is \emph{$\phi$-saturated}, i.e. that $\phi^{-1}(\phi(\mathcal{U})) = \mathcal{U}$.
    \item For any morphism $g : Y' \to Y$ of algebraic spaces and for any quasi-coherent sheaf $\mathcal{F}$ of $\mathcal{O}_{\mathcal{Y}}$-modules, the base change map $g^{\ast}\phi_{\ast} \mathcal{F} \to \phi_{\ast}' (g')^{\ast} \mathcal{F}$ is an isomorphism.
    \item If in addition $\mathcal{Y}$ is locally Noetherian then the following all hold:
    \begin{enumerate}
        \item[(i)] Any closed point of $\mathcal{Y}$ has a linearly reductive automorphism group.
        \item[(ii)] For any closed point $y \in |Y|$, there exists a unique closed point $\widetilde{y} \in |\phi^{-1}(y)|$.
        \item [(iii)] The pullback functor $\phi^{\ast}$ induces an equivalence of categories between the category of locally free $\mathcal{O}_Y$-modules of finite rank and the full subcategory of locally free $\mathcal{O}_{\mathcal{Y}}$-modules of finite rank with trivial stabilizer action at closed points; the inverse functor is given by $\phi_{\ast}$.
    \end{enumerate}
    
    \item If $\mathcal{Y}$ is of finite type over an excellent scheme $B$, then so is $Y$.
\end{enumerate}

\subsubsection{Relation to GIT Good Quotients}

We recall the notion of a good quotient, in the sense of Seshadri \cite{Seshadri-Quotient}.\footnote{The definition of a good quotient we provide here is slightly more general, in that we allow the base to be an algebraic space, as opposed to a scheme in \cite{Seshadri-Quotient}.}

\begin{defi}
    Let $G$ be a linear algebraic group over $k$, acting on a quasi-separated $k$-scheme $X$. Let $f : X \to Y$ be a morphism to an algebraic space $Y$. We say that $f$ is a \emph{good $G$-quotient} if the following properties hold:
    \begin{enumerate}
        \item the morphism $f$ is surjective, affine and $G$-invariant;
        \item the natural map $\mathcal{O}_Y \to f_{\ast} \mathcal{O}_X$ induces an isomorphism $\mathcal{O}_Y \cong (f_{\ast} \mathcal{O}_X)^G$;
        \item the image of a $G$-invariant closed subspace of $|X|$ is closed in $|Y|$; and
        \item $f$ separates disjoint $G$-invariant closed subspaces of $|X|$.
    \end{enumerate}
    If in addition the geometric fibers of $f$ are in bijection with the $G$-orbits in $X$, we say that $f$ is a \emph{geometric good $G$-quotient}. 
\end{defi}

Let $G$ be a linear algebraic $k$-group, acting on a quasi-separated $k$-scheme $X$. Extending Example \ref{ex:alper-affine-projective-git}, we have the following relationship between the notion of a good moduli space of $[X/G]$ and a good quotient of $X$ by $G$:

\begin{itemize}
    \item If $[X/G]$ admits a good moduli space $[X/G] \to X /\!\!/G$, then $X \to X /\!\!/G$ is a good $G$-quotient.

    \item If $G$ is linearly reductive and if there exists a good $G$-quotient $X \to X /\!\!/ G$, then the natural map $[X/G] \to X/\!\!/G$ is a good moduli space.

    \item If $G$ is geometrically reductive and if there exists a good $G$-quotient $X \to X /\!\!/ G$, then in general $[X/G] \to X /\!\!/ G$ is not a good moduli space, only an adequate moduli space (in the sense of \cite{Alper-AMS}). Assuming in addition that $X$ is of finite type over $k$, it follows from \cite[Theorem 4.21]{AHR-Etale} that $[X/G] \to X /\!\!/ G$ is a good moduli space if and and only if all closed $G$-orbits in $X$ have a linearly reductive stabilizer.
\end{itemize}

In characteristic $0$, the notions of good and adequate moduli spaces coincide, as well as the notions of linearly reductive and geometrically reductive algebraic groups. In this case, if $G$ is reductive then $[X/G]$ admits a schematic good moduli space if and only if there exists a good $G$-quotient $X \to X/\!\!/G$.

\subsubsection{The Étale Local Structure of a Good Moduli Space}

Over an algebraically closed field, the étale local structure of a good moduli space is given by affine reductive GIT, in the following sense.

\begin{teo} \label{thm:gms-etale-local-structure}
    Let $\mathcal{Y}$ be a Noetherian algebraic stack over $B = \Spec k$, where $k$ is an algebraically closed field. Assume that $\mathcal{Y}$ admits a good moduli space $\phi : \mathcal{Y} \to Y$ with an affine diagonal. Let $y \in \mathcal{Y}(k)$ be a closed point. Then there exists an affine scheme $\Spec A$ with an action of the (linearly reductive) $k$-group $G_y := \mathrm{Aut}_{\mathcal{Y}}(y)$ and a Cartesian diagram
    \[\begin{tikzcd}[ampersand replacement=\&]
	{[\Spec A/G_y]} \&\& {\mathcal{Y}} \\
	\\
	{\Spec A /\!\!/ G_y} \&\& Y
	\arrow[from=1-1, to=1-3]
	\arrow[from=1-1, to=3-1]
	\arrow["\lrcorner"{anchor=center, pos=0.125}, draw=none, from=1-1, to=3-3]
	\arrow[from=1-3, to=3-3]
	\arrow[from=3-1, to=3-3]
\end{tikzcd}\]
    with $\Spec A /\!\!/ G_y = \Spec A^{G_y} \to Y$ an étale neighborhood of $\phi(y) \in Y(k)$.
\end{teo}

\begin{proof}
    This is \cite[Theorem 4.12]{AHR-Luna}.
\end{proof}

\subsection{Stability Conditions Arising from Invertible Sheaves on Algebraic Stacks}

In the sequel, we consider $\Theta$-stability conditions on an algebraic stack arising from invertible sheaves, generalizing the notion of Hilbert--Mumford stability with respect to a linearized invertible sheaf on a scheme.

Let $\mathcal{Y}$ be an algebraic stack which is locally of finite presentation over $B$. Let $\mathcal{L}$ be an invertible sheaf on $\mathcal{Y}$. If $k$ is a field, and if $f : \Theta_k \to \mathcal{Y}$ is a morphism, then the pullback of $\mathcal{L}$ along $f$ is a $\mathbb{G}_{m,k} = \Spec k[t,t^{-1}]$-equivariant line bundle on $\mathbb{A}_{k}^1 = \Spec k[x]$; in particular, $f^{\ast}\mathcal{L}|_{\mathbb{A}_k^1}$ admits a section $s \in k[x]$, unique up to scalar multiplication, with $H^0(\mathbb{A}_k^1, f^{\ast}\mathcal{L}|_{\mathbb{A}_k^1}) = k[x] \cdot s$. Under the coaction $k[x] \to k[x,t,t^{-1}]$, there exists $d \in \Z$ such that $s$ is sent to $t^ds$; we denote this integer $d$ by $\mathrm{wt}_0(f^{\ast} \mathcal{L})$.

\begin{defi}[\cite{Halpern-Leistner, Heinloth-HM}] \label{def:theta-stability}
    Let $k$ be an algebraically closed field, and let $y \in \mathcal{Y}(k)$. The point $y$ is said to be \emph{$\Theta_k$-(semi)stable} with respect to $\mathcal{L}$ if for all filtrations $f : \Theta_k \to \mathcal{Y}$ for which $\mathrm{ev}_1(f) \cong y$ and for which $f$ does not factor through $\Theta_k \to \Spec k$, we have
    $$ \mathrm{wt}_0(f^{\ast} \mathcal{L}) < (\leq) \ 0. $$
\end{defi}

\subsubsection{Relation to Hilbert--Mumford Stability}

In the case where $\mathcal{Y} = [X/G]$ is the quotient of a separated $k$-scheme $X$ by a smooth linear algebraic group $G$ over $k$, and where $\mathcal{L}$ is the invertible sheaf associated to a $G$-linearized invertible sheaf on $X$ (also denoted $\mathcal{L}$), the locus of $K$-points of $[X/G]$ which are $\Theta_K$-(semi)stable, where $K \supset k$ is an algebraically closed field, is given by the set of $K$-points of $X$ which are Hilbert--Mumford (semi)stable, in the following way.

\begin{ex} \label{ex:theta-stability-for-quotient-stacks}
    Let $\mathcal{Y} = [X/G]$, $\mathcal{L}$ and $K$ be as above. Let $\overline{p} \in [X/G](K)$, with representative $p \in X(K)$. Then, by Theorem \ref{thm:filt-grad-quotient-stack} and \cite[Remark 1.4.9]{Halpern-Leistner}, specifying a morphism $f : \Theta_K \to [X/G]$ with $\mathrm{ev}_1(f) \cong \overline{p}$ and for which $f$ does not factor through $\Theta_K \to \Spec K$ is, up to conjugating $\lambda$ by $P_{G_K}(\lambda)$, equivalent to specifying a non-trivial cocharacter $\lambda : \mathbb{G}_{m,K} \to G_K$ for which $\lim_{t \to 0} \lambda(t) \cdot p$ exists in $X_K$; this latter condition means that the composite morphism
    \[\begin{tikzcd}[ampersand replacement=\&]
	{\mathbb{G}_{m,K}} \&\& {G_K = G_K \times_K \Spec K} \&\& {G_K \times_K X_K} \&\& {X_K}
	\arrow["\lambda", from=1-1, to=1-3]
	\arrow["{\mathrm{id}_{G_K} \times p}", from=1-3, to=1-5]
	\arrow["a", from=1-5, to=1-7]
    \end{tikzcd}\]
    extends to $\mathbb{A}_K^1$ (since $X$ is separated, such an extension is necessarily unique). In this case, one has
    $$ \mathrm{wt}_0(f^{\ast} \mathcal{L}) = -\mu^{\mathcal{L}}(p, \lambda), $$
    where $\mu^{\mathcal{L}}(p, \lambda)$ denotes the Hilbert--Mumford weight of $p$ with respect to the linearization $\mathcal{L}$ and the cocharacter $\lambda$, as defined in \cite{MFK-GIT}. As such, the $\Theta_K$-(semi)stability of the point $\overline{p} \in [X/G](K)$ is equivalent to the Hilbert--Mumford (semi)stability of $p \in X(K)$ with respect to cocharacters defined over $K$, i.e. $\overline{p}$ is $\Theta_K$-(semi)stable if and only if for all cocharacters $\lambda : \mathbb{G}_{m,K} \to G_K$ for which $\lim_{t \to 0} \lambda(t) \cdot p$ exists in $X_K$, one has that $\mu^{\mathcal{L}}(p, \lambda) > (\geq) \ 0$ (this latter criterion is independent of the choice of representative $p \in X(K)$ for $\overline{p}$).
\end{ex}

\section{Proof of the Main Result}

In this section, we prove the main result of this article, Theorem \ref{thm:main-theorem}. There is significant overlap between the proof of Theorem \ref{thm:main-theorem} and \cite[Theorem A]{Modin}, however we keep this section as self-contained as possible. 

Throughout this section, let $B$ be a quasi-separated scheme, and let $\mathcal{X}$ be an algebraic stack which is locally of finite presentation and has an affine diagonal over $B$, admitting a weak $\Theta_B$-action with center $\mathcal{Z}$ for which the associated retraction $\mathrm{gr} : \mathcal{X} \to \mathcal{Z}$ has flat, finitely presented relative inertia $\mathcal{I}_{\mathrm{gr}} \to \mathcal{X}$.

\subsection{Existence of the Rigidification}

We first show that the rigidification of $\mathcal{X}$ by $\mathcal{I}_{\mathrm{gr}}$ is well-defined.

\begin{lem} \label{lem:existence-rigidification}
    The rigidification $\rho : \mathcal{X} \to \mathcal{X}^{\mathrm{rig}} := \mathcal{X} \!\!\! \fatslash \mathcal{I}_{\mathrm{gr}}$ is well-defined and is an fppf gerbe.
\end{lem}

\begin{proof}
    There is a Cartesian diagram
    \[\begin{tikzcd}[ampersand replacement=\&]
	{\mathcal{I}_{\mathrm{gr}}} \&\& {\mathcal{I}_{\mathcal{X}}} \\
	\\
	{\mathcal{X}} \&\& {\mathcal{X} \times_{\mathcal{Z}} \mathcal{I}_{\mathcal{Z}}}
	\arrow[from=1-1, to=1-3]
	\arrow[from=1-1, to=3-1]
	\arrow["\lrcorner"{anchor=center, pos=0.125}, draw=none, from=1-1, to=3-3]
	\arrow[from=1-3, to=3-3]
	\arrow[from=3-1, to=3-3]
    \end{tikzcd}\]
    where $\mathcal{X} \to \mathcal{X} \times_{\mathcal{Z}} \mathcal{I}_{\mathcal{Z}}$ is given by $x \mapsto (x, (\mathrm{gr}(x), \mathrm{id}_{\mathrm{gr}(x)}))$ and where $\mathcal{I}_{\mathcal{X}} \to \mathcal{X} \times_{\mathcal{Z}} \mathcal{I}_{\mathcal{Z}}$ is given by $(x, \alpha) \mapsto (x, (\mathrm{gr}(x), \mathrm{gr}_{\ast}(\alpha))$. Since $\mathcal{X} \to B$ has an affine diagonal, so do $\mathrm{Grad}(\mathcal{X}) \to B$ and $\mathcal{Z} \to B$. By \cite[Tag 0CL0]{stacks-project}, this implies that $\mathcal{I}_{\mathcal{Z}} \to \mathcal{Z}$ is separated. Since a section of a separated morphism is a closed immersion, the morphism $\mathcal{X} \to \mathcal{X} \times_{\mathcal{Z}} \mathcal{I}_{\mathcal{Z}}$ is a closed immersion, hence so is $\mathcal{I}_{\mathrm{gr}} \to \mathcal{I}_{\mathcal{X}}$. Since in addition $\mathcal{I}_{\mathrm{gr}} \to \mathcal{X}$ is flat and of finite presentation, we may apply \cite[Theorem A.1]{AOV-tame-stacks}, which yields the result.
\end{proof}

\subsection{Descending the Weak $\Theta_B$-Action}

Next, we show that the weak $\Theta_B$-action on $\mathcal{X}$ descends to the rigidification $\mathcal{X}^{\mathrm{rig}}$.

\begin{prop} \label{prop:induced-action-on-rigidification}
    The weak $\Theta_B$-action on $\mathcal{X}$ induces a canonical weak action of $\Theta_B$ on the rigidification $\mathcal{X}^{\mathrm{rig}}$, for which the gerbe $\rho : \mathcal{X} \to \mathcal{X}^{\mathrm{rig}}$ is $\Theta_B$-equivariant. With respect to the induced action, the morphism $\overline{\mathrm{gr}} : \mathcal{X}^{\mathrm{rig}} \to \mathcal{Z}$ induced\footnote{i.e. $\overline{\mathrm{gr}}$ is the unique morphism such that $\mathrm{gr} = \overline{\mathrm{gr}} \circ \rho$.} by $\mathrm{gr}$ is equivariant with respect to $q : \Theta_B \to B\mathbb{G}_{m,B}$, and the morphism $\overline{\sigma} := \rho \circ \sigma : \mathcal{Z} \to \mathcal{X}^{\mathrm{rig}}$ is invariant with respect to $0 : B\mathbb{G}_{m,B} \to B$.
\end{prop}

\begin{proof}
    Let $a : \Theta_B \times_B \mathcal{X} \to \mathcal{X}$ denote the action morphism. In order to show that $a$ descends to a weak action on $\mathcal{X}^{\mathrm{rig}}$, we apply the universal property of rigidification, as stated in \cite[Lemma B.2]{Modin}, with respect to the closed subgroup stack $\Theta_B \times_B \mathcal{I}_{\mathrm{gr}} \subset \mathcal{I}_{\Theta_B \times_B \mathcal{X}}$ and the composition of $a$ with the rigidification $\rho : \mathcal{X} \to \mathcal{X}^{\mathrm{rig}}$. This implies that, if the induced morphism $\mathcal{I}_{\Theta_B \times_B \mathcal{X}} \to \mathcal{I}_{\mathcal{X}}$ restricts to a morphism $\Theta_B \times_B \mathcal{I}_{\mathrm{gr}} \to \mathcal{I}_{\mathrm{gr}}$, then there exists a unique dashed arrow making the below diagram (2-)commute:
    \[\begin{tikzcd}[ampersand replacement=\&]
	{\Theta_B \times_B \mathcal{X}} \&\& {\mathcal{X}} \\
	\\
	{\Theta_B \times_B \mathcal{X}^{\mathrm{rig}}} \&\& {\mathcal{X}^{\mathrm{rig}}}
	\arrow["a", from=1-1, to=1-3]
	\arrow["{\mathrm{id}_{\Theta_B} \times \rho}"', from=1-1, to=3-1]
	\arrow["\lrcorner"{anchor=center, pos=0.125}, draw=none, from=1-1, to=3-3]
	\arrow["\rho", from=1-3, to=3-3]
	\arrow["{\exists!}", dashed, from=3-1, to=3-3]
    \end{tikzcd}\]

    In order to show that the induced morphism $\mathcal{I}_{\Theta_B \times_B \mathcal{X}} \to \mathcal{I}_{\mathcal{X}}$ restricts to a morphism $\Theta_B \times_B \mathcal{I}_{\mathrm{gr}} \to \mathcal{I}_{\mathrm{gr}}$, we consider the commutative diagram
    \[\begin{tikzcd}[ampersand replacement=\&]
	{\Theta_B \times_B \mathcal{X}} \&\& {\mathcal{X}} \\
	\\
	{\mathcal{X}} \&\& {\mathcal{Z}}
	\arrow["a", from=1-1, to=1-3]
	\arrow["{\mathrm{pr}_{\mathcal{X}}}"', from=1-1, to=3-1]
	\arrow["{{\mathrm{gr}}}", from=1-3, to=3-3]
	\arrow["{\mathrm{gr}}"', from=3-1, to=3-3]
    \end{tikzcd}\]
    exhibiting $\mathrm{gr} : \mathcal{X} \to \mathcal{Z}$ as a $\Theta_B$-deformation retract.

    By the functoriality of the inertia stack construction, we obtain a commutative diagram of inertia stacks
    \[\begin{tikzcd}[ampersand replacement=\&]
	{\mathcal{I}_{\Theta_B \times_B \mathcal{X}}} \&\& {\mathcal{I}_{\mathcal{X}}} \\
	\\
	{\mathcal{I}_{\mathcal{X}}} \&\& {\mathcal{I}_{\mathcal{Z}}}
	\arrow["{{a_{\ast}}}", from=1-1, to=1-3]
	\arrow["{(\mathrm{pr}_{\mathcal{X}})_{\ast}}"', from=1-1, to=3-1]
	\arrow["{{\mathrm{gr}_{\ast}}}", from=1-3, to=3-3]
	\arrow["{\mathrm{gr}_{\ast}}"', from=3-1, to=3-3]
    \end{tikzcd}\]
    Since $\Theta_B \times_B \mathcal{I}_{\mathrm{gr}} \subset \mathcal{I}_{\Theta_B \times_B \mathcal{X}}$ is contained in the kernel of $(\mathrm{gr})_{\ast} \circ (\mathrm{pr}_{\mathcal{X}})_{\ast} \cong (\mathrm{gr} \circ \mathrm{pr}_{\mathcal{X}})_{\ast}$ and $\mathcal{I}_{\mathrm{gr}} \subset \mathcal{I}_{\mathcal{X}}$ is the kernel of $\mathrm{gr}_{\ast}$, it follows that $a_{\ast}$ must send $\Theta_B \times_B \mathcal{I}_{\mathrm{gr}}$ to $\mathcal{I}_{\mathrm{gr}}$ as claimed. As such, by the universal property of rigidification we obtain an induced morphism $a^{\mathrm{rig}} : \Theta_B \times_B \mathcal{X}^{\mathrm{rig}} \to \mathcal{X}^{\mathrm{rig}}$. To see that $a^{\mathrm{rig}}$ is a weak action, we apply \cite[Corollary B.3]{Modin} to the commutative diagrams
    \[\begin{tikzcd}[ampersand replacement=\&]
	{\Theta_B \times_B \Theta_B \times_B \mathcal{X}} \&\& {\Theta_B \times_B \mathcal{X}} \\
	\\
	{\Theta_B \times_B \mathcal{X}} \&\& {\mathcal{X}}
	\arrow["{\mathrm{id}_{\Theta_B} \times a}", from=1-1, to=1-3]
	\arrow["{\mu \times \mathrm{id}_{\mathcal{X}}}"', from=1-1, to=3-1]
	\arrow[""{name=0, anchor=center, inner sep=0}, "a", from=1-3, to=3-3]
	\arrow[""{name=1, anchor=center, inner sep=0}, "a"', from=3-1, to=3-3]
    \end{tikzcd}\]
    and
    \[\begin{tikzcd}[ampersand replacement=\&]
	{\mathcal{X}} \&\& {\Theta_B \times_B \mathcal{X}} \\
	\\
	\&\& {\mathcal{X}}
	\arrow["{1 \times \mathrm{id}_{\mathcal{X}}}", from=1-1, to=1-3]
	\arrow[""{name=0, anchor=center, inner sep=0}, "{\mathrm{id}_{\mathcal{X}}}"', from=1-1, to=3-3]
	\arrow["a", from=1-3, to=3-3]
    \end{tikzcd}\]
    where $\mu$ denotes the multiplication $\Theta_B \times_B \Theta_B \to \Theta_B$.

    Finally, the equivariance of $\overline{\mathrm{gr}}$ and $\overline{\sigma}$ is a consequence of the $\Theta_B$-equivariance of $\rho$ together with the respective equivariance properties of $\mathrm{gr}$ and $\sigma$.
\end{proof}

\subsection{Affinity of the Rigidification over the Center}

Using the fact that the weak $\Theta_B$-action on $\mathcal{X}$ descends to $\mathcal{X}^{\mathrm{rig}}$, we are in a position to prove that the induced morphism $\overline{\mathrm{gr}} : \mathcal{X}^{\mathrm{rig}} \to \mathcal{Z}$ is affine (compare with \cite[Theorem 4.13]{Modin}).

\begin{teo} \label{thm:affinity-grbar}
    The morphism $\overline{\mathrm{gr}} : \mathcal{X}^{\mathrm{rig}} \to \mathcal{Z}$ induced by $\mathrm{gr} : \mathcal{X} \to \mathcal{Z}$ is affine and of finite presentation. Moreover, the morphism $\overline{\sigma}$ is a closed immersion, which after pulling back along a morphism from a scheme $S \to \mathcal{Z}$ identifies $S$ with the fixed locus in $\mathcal{X}_S^{\mathrm{rig}}$ for the induced action of $\mathbb{G}_{m,S}$.
\end{teo}

\begin{proof}
    We first observe that since $\overline{\mathrm{gr}} \circ \overline{\sigma} \cong \mathrm{id}_{\mathcal{Z}}$, there exists a Cartesian diagram
    \[\begin{tikzcd}[ampersand replacement=\&]
	{\mathcal{Z}} \&\& {\mathcal{X}^{\mathrm{rig}}} \\
	\\
	{\mathcal{X}^{\mathrm{rig}}} \&\& {\mathcal{X}^{\mathrm{rig}} \times_{\mathcal{Z}} \mathcal{X}^{\mathrm{rig}}}
	\arrow["{\overline{\sigma}}", from=1-1, to=1-3]
	\arrow["{\overline{\sigma}}"', from=1-1, to=3-1]
	\arrow["\lrcorner"{anchor=center, pos=0.125}, draw=none, from=1-1, to=3-3]
	\arrow["{\Delta_{\overline{\mathrm{gr}}}}", from=1-3, to=3-3]
	\arrow["{(\overline{\sigma} \circ \overline{\mathrm{gr}}, \mathrm{id}_{\mathcal{X}^{\mathrm{rig}}})}"', from=3-1, to=3-3]
    \end{tikzcd}\]
    In particular, the affinity of $\overline{\mathrm{gr}}$ implies that $\overline{\sigma}$ is a closed immersion.

    By pulling back along $B \to B\mathbb{G}_{m,B}$, the weak $\Theta_B$-action on $\mathcal{X}^{\mathrm{rig}}$ lifts to a weak $\mathbb{A}_B^1$-action on $\mathcal{X}^{\mathrm{rig}}$, with respect to which $\overline{\mathrm{gr}}$ is invariant. Since the property of a morphism being affine of finite presentation is fpqc local on the base, by passing to a smooth presentation we may assume that $\mathcal{Z} = S$ is a scheme. As such, $\mathcal{X}^{\mathrm{rig}}$ is an algebraic space which is quasi-separated and locally of finite presentation over $S$, with a contracting action of $\mathbb{G}_{m,S}$. By Theorem \ref{theorem:attractor_loci_existence} it follows that $\mathcal{X}^{\mathrm{rig}}$ is affine and of finite presentation over the $\mathbb{G}_{m,S}$-fixed locus $\mathcal{X}^{\mathrm{rig},0}$, which in turn is a closed subspace of $\mathcal{X}^{\mathrm{rig}}$.
    
    On the other hand, since $\overline{\sigma}$ is equivariant with respect to $0 : S \to \mathbb{A}_S^1$, the morphism $\overline{\sigma} : S \to \mathcal{X}^{\mathrm{rig}}$ factors through the closed immersion $\mathcal{X}^{\mathrm{rig},0} \subset \mathcal{X}^{\mathrm{rig}}$. Together with the equalities $\overline{\mathrm{gr}} \circ \overline{\sigma} = \mathrm{id}_S$ and $\overline{\sigma} \circ \overline{\mathrm{gr}} = a^{\mathrm{rig}}(0,-)$, this implies that $\overline{\sigma} : S \to \mathcal{X}^{\mathrm{rig},0}$ is an isomorphism with inverse given by $\overline{\mathrm{gr}}$. Since $\overline{\mathrm{gr}}$ factors through the retraction $\mathcal{X}^{\mathrm{rig}} \to \mathcal{X}^{\mathrm{rig},0}$, it follows that $\overline{\mathrm{gr}} : \mathcal{X}^{\mathrm{rig}} \to S$ is affine and of finite presentation. This completes the proof of the theorem.
\end{proof}

\subsubsection{Compatibility of the Centers of the Weak $\Theta_B$-Actions}

Using that $\overline{\sigma}$ is a closed immersion, we may identify the center of the induced weak $\Theta_B$-action on $\mathcal{X}^{\mathrm{rig}}$ with $\mathcal{Z}$.

\begin{lem} \label{lem:centers-agree}
    The center $\mathcal{Z}$ of the weak action of $\Theta_B$ on $\mathcal{X}$ is isomorphic to the center of the induced weak action on $\mathcal{X}^{\mathrm{rig}}$.
\end{lem}

\begin{proof}
    We first compare $\mathcal{Z}$ with $\mathrm{Grad}(\mathcal{X}^{\mathrm{rig}})$, by considering the commutative diagram
    \[\begin{tikzcd}[ampersand replacement=\&]
	{\mathcal{Z}} \&\& {\mathrm{Grad}(\mathcal{X})} \\
	\\
	{\mathrm{Grad}(\mathcal{Z})} \&\& {\mathrm{Grad}(\mathcal{X}^{\mathrm{rig}})}
	\arrow[from=1-1, to=1-3]
	\arrow[from=1-1, to=3-1]
	\arrow["{\mathrm{Grad}(\rho)}", from=1-3, to=3-3]
	\arrow["{\mathrm{Grad}(\sigma)}", from=3-1, to=1-3]
	\arrow["{\mathrm{Grad}(\overline{\sigma})}"', from=3-1, to=3-3]
    \end{tikzcd}\]
    where $\mathcal{Z} \to \mathrm{Grad}(\mathcal{Z})$ is the open and closed immersion classifying the weak $B\mathbb{G}_{m,B}$-action on $\mathcal{Z}$. Since $\rho : \mathcal{X} \to \mathcal{X}^{\mathrm{rig}}$ is an fppf gerbe, it is smooth, by \cite[Proposition 1.3.1]{Halpern-Leistner} the induced morphism $\mathrm{Grad}(\rho)$ is also smooth, hence (by \cite[Tag 06R7]{stacks-project}) has open image. In particular, the composition $\mathcal{Z} \to \mathrm{Grad}(\mathcal{X}^{\mathrm{rig}})$ has open image. On the other hand, since $\overline{\sigma}$ is a closed immersion, by \cite[Corollary 1.1.8]{Halpern-Leistner} the same is true of $\mathrm{Grad}(\overline{\sigma})$, and so $\mathcal{Z} \to \mathrm{Grad}(\mathcal{X}^{\mathrm{rig}})$ is a closed immersion. It follows that the morphism $\mathcal{Z} \to \mathrm{Grad}(\mathcal{X}^{\mathrm{rig}})$ is an open and closed immersion.

    Let $\mathcal{W} := \mathcal{X}^{\mathrm{rig}} \times_{\mathrm{Filt}(\mathcal{X}^{\mathrm{rig}})} \mathrm{Grad}(\mathcal{X}^{\mathrm{rig}})$ be the center of the induced weak action on $\mathcal{X}^{\mathrm{rig}}$. We then have a commutative diagram
    \[\begin{tikzcd}[ampersand replacement=\&]
	{\mathcal{Z}} \&\&\& \\
	\& {\mathcal{W}} \&\& {\mathrm{Grad}(\mathcal{X}^{\mathrm{rig}})} \\
	\\
	\& {\mathcal{X}^{\mathrm{rig}}} \&\& {\mathrm{Filt}(\mathcal{X}^{\mathrm{rig}})}
	\arrow["{\exists!}", dashed, from=1-1, to=2-2]
	\arrow[curve={height=-18pt}, from=1-1, to=2-4]
	\arrow["{\overline{\sigma}}"', curve={height=18pt}, from=1-1, to=4-2]
	\arrow[from=2-2, to=2-4]
	\arrow[from=2-2, to=4-2]
	\arrow["\lrcorner"{anchor=center, pos=0.125}, draw=none, from=2-2, to=4-4]
	\arrow["{\sigma^{\mathrm{rig}}}", from=2-4, to=4-4]
	\arrow["{a^{\mathrm{rig}}}"', from=4-2, to=4-4]
    \end{tikzcd}\]
    whose horizontal arrows are open and closed immersions. We claim that the induced morphism $\mathcal{Z} \to \mathcal{W}$ is an isomorphism. Since both $\mathcal{W} \to \mathrm{Grad}(\mathcal{X}^{\mathrm{rig}})$ and $\mathcal{Z} \to \mathrm{Grad}(\mathcal{X}^{\mathrm{rig}})$ are both open and closed immersions, the same is true of $\mathcal{Z} \to \mathcal{W}$. In particular, $\mathcal{Z} \to \mathcal{W}$ is an isomorphism onto its image. On the other hand, since $\mathrm{gr} : \mathcal{X} \to \mathcal{Z}$ and $\sigma : \mathcal{Z} \to \mathcal{X}$ yield a $\Theta$-deformation retract of $\mathcal{X}$ onto $\mathcal{Z}$, $\sigma$ induces a bijection between the sets of connected components of $|\mathcal{Z}|$ and $|\mathcal{X}|$; it follows that $\overline{\sigma} = \rho \circ \sigma$ induces a bijection between the sets of connected components of $|\mathcal{Z}|$ and $|\mathcal{X}^{\mathrm{rig}}|$, with inverse given by $|\overline{\mathrm{gr}}|$. Similarly, the restricted morphism $\sigma^{\mathrm{rig}} : \mathcal{W} \to \mathcal{X}^{\mathrm{rig}}$ induces a bijection on the level of sets of connected components. This implies that $\mathcal{Z} \to \mathcal{W}$ must be surjective, and hence an isomorphism.
\end{proof}

\subsection{Existence of a Good Moduli Space of the Rigidified Stack}

From now on until the end of this section, assume in addition that $\mathcal{Z}$ admits a good moduli space $\pi : \mathcal{Z} \to Z$ (note that since $\mathcal{Z}$ is quasi-separated, $Z$ is necessarily quasi-separated as well). Our next goal is to show that the composition of $\overline{\mathrm{gr}} : \mathcal{X}^{\mathrm{rig}} \to \mathcal{Z}$ with $\pi : \mathcal{Z} \to Z$ is a good moduli space of $\mathcal{X}^{\mathrm{rig}}$. 

To do this, we first consider the image $\mathcal{G}_{\mathcal{Z}}^{\mathrm{can}} \subset \mathcal{I}_{\mathcal{Z}}$ of the canonical homomorphism $\mathbb{G}_{m, \mathcal{Z}} \to \mathcal{I}_{\mathcal{Z}}$; recall that by Proposition \ref{prop:rigidified-center} there is a well-defined rigidification $\tau : \mathcal{Z} \to \mathcal{Z}^{\mathrm{rig}} := \mathcal{Z} \!\!\! \fatslash  \mathcal{G}_{\mathcal{Z}}^{\mathrm{can}}$, which on each connected component of $\mathcal{Z}$ is either an isomorphism or a $\mathbb{G}_m$-gerbe. Since $Z$ is an algebraic space, the good moduli space $\pi$ must necessarily factor through this rigidification, yielding a morphism $\overline{\pi} : \mathcal{Z}^{\mathrm{rig}} \to Z$.

\begin{lem} \label{lem:gms-Zrig}
    The morphism $\overline{\pi} : \mathcal{Z}^{\mathrm{rig}} \to Z$ is a good moduli space.
\end{lem}

\begin{proof}
    By \cite[Remark 12.2]{Alper-GMS}, the morphism $B\mathbb{G}_{m,S} \to S$ is a good moduli space for any scheme $S$. Since the property of being a good moduli space morphism is fpqc local on the base (the proof of \cite[Proposition 4.7(ii)]{Alper-GMS} carries over word-for-word to the case where the base is an algebraic stack), it follows that a $\mathbb{G}_m$-gerbe is a good moduli space morphism. In particular, $\tau$ is a good moduli space morphism. 
    
    Since the natural morphisms $\mathcal{O}_{\mathcal{Z}^{\mathrm{rig}}} \to \tau_{\ast} \mathcal{O}_{\mathcal{Z}}$ and $\mathcal{O}_Z \to \pi_{\ast}\mathcal{O}_{\mathcal{Z}}$ are isomorphisms, the same is true of $\mathcal{O}_Z \to \overline{\pi}_{\ast}\mathcal{O}_{\mathcal{Z}^{\mathrm{rig}}}$. On the other hand, by \cite[Proposition 4.5]{Alper-GMS} the adjunction morphism $\mathcal{F} \to \tau_{\ast}\tau^{\ast} \mathcal{F}$ is an isomorphism for all $\mathcal{F} \in \mathrm{QCoh}(\mathcal{Z}^{\mathrm{rig}})$, so $\overline{\pi}_{\ast} \cong \pi_{\ast} \circ \tau^{\ast}$. As $\tau$ is flat then $\tau^{\ast}$ is exact, and since $\pi$ is cohomologically affine then $\pi_{\ast}$ is exact. It follows that $\overline{\pi}$ is also cohomologically affine, whence $\overline{\pi}$ is a good moduli space.
\end{proof}

\begin{prop} \label{prop:gms-of-Xrig}
    The composition of $\overline{\mathrm{gr}} : \mathcal{X}^{\mathrm{rig}} \to \mathcal{Z}$ with the good moduli space $\pi : \mathcal{Z} \to Z$ is a good moduli space $\mathcal{X}^{\mathrm{rig}} \to Z$.
\end{prop}

\begin{proof}
    Since the composition of two good moduli space morphisms is a good moduli space morphism, it suffices to show that the morphism $\mathcal{X}^{\mathrm{rig}} \to \mathcal{Z}^{\mathrm{rig}}$ is a good moduli space morphism. Since this problem is fpqc local on $\mathcal{Z}^{\mathrm{rig}}$, we may assume without loss of generality that $\mathcal{Z}$ and $\mathcal{Z}^{\mathrm{rig}}$ are connected. As such, $\tau$ is either an isomorphism or is a $\mathbb{G}_m$-gerbe.

    Suppose first that $\tau$ is trivial, so that by Proposition \ref{prop:rigidified-center} the canonical homomorphism $\lambda_{\mathrm{can}} : \mathbb{G}_{m, \mathcal{Z}} \to \mathcal{I}_{\mathcal{Z}}$ factors through the identity section $\mathcal{Z} \to \mathcal{I}_{\mathcal{Z}}$. By Lemma \ref{lem:internal-Gm-action}, this implies that the weak $\Theta_B$-actions on both $\mathcal{X}$ and $\mathcal{X}^{\mathrm{rig}}$ are equivalent to the trivial action, which in turn implies that the morphism $\overline{\mathrm{gr}} : \mathcal{X}^{\mathrm{rig}} \to \mathcal{Z}$ is an isomorphism. It follows \emph{a fortiori} that $\mathcal{X}^{\mathrm{rig}} \to \mathcal{Z}^{\mathrm{rig}} = \mathcal{Z}$ is a good moduli space morphism.

    Suppose instead that $\tau$ is a $\mathbb{G}_m$-gerbe. Neutralizing this gerbe, there exists a scheme $S$, an fppf cover $S \to \mathcal{Z}^{\mathrm{rig}}$ and a Cartesian diagram
    \[\begin{tikzcd}[ampersand replacement=\&]
	{\mathcal{X}_S^{\mathrm{rig}}} \&\& {\mathcal{X}^{\mathrm{rig}}} \\
	\\
	{B\mathbb{G}_{m,S}} \&\& {\mathcal{Z}} \\
	\\
	S \&\& {\mathcal{Z}^{\mathrm{rig}}}
	\arrow[from=1-1, to=1-3]
	\arrow[from=1-1, to=3-1]
	\arrow["\lrcorner"{anchor=center, pos=0.125}, draw=none, from=1-1, to=3-3]
	\arrow["{\overline{\mathrm{gr}}}", from=1-3, to=3-3]
	\arrow[from=3-1, to=3-3]
	\arrow[from=3-1, to=5-1]
	\arrow["\lrcorner"{anchor=center, pos=0.125}, draw=none, from=3-1, to=5-3]
	\arrow[from=3-3, to=5-3]
	\arrow[from=5-1, to=5-3]
    \end{tikzcd}\]

    By Theorem \ref{thm:affinity-grbar} the morphism $\overline{\mathrm{gr}}$ is affine, and so we may write $\mathcal{X}_S^{\mathrm{rig}} = [\Spec_S \mathcal{A}/\mathbb{G}_{m,S}]$, for some graded $\mathcal{O}_S$-algebra $\mathcal{A} = \bigoplus_{m \in \mathbb{Z}} \mathcal{A}_m$. By \cite[Proposition 1.3.9]{Halpern-Leistner} the kernel of $\lambda^{\mathrm{can}}|_S$ is isomorphic to the multiplicative group $D_S(\Z/q\Z)$ for some positive integer $q$. Therefore, by Lemma \ref{lem:internal-Gm-action} the $\mathbb{G}_{m,S}$-action on $\Spec_S \mathcal{A} = \mathcal{X}_S^{\mathrm{rig}} \times_{B\mathbb{G}_{m,S}} S$ used to form the quotient is equivalent to the $\mathbb{G}_{m,S}$-action induced from the weak $\Theta_B$-action on $\mathcal{X}^{\mathrm{rig}}$, up to the finite $S$-group $D_S(\Z/q\Z)$. Since this latter action is contracting and has fixed locus $\overline{\sigma}(S)$, it follows that
    $$ \mathcal{A}_m = \begin{cases} 0 & \text{if } m < 0, \\ \mathcal{O}_S & \text{if } m = 0. \end{cases} $$
    Consequently, by \cite[Theorem 13.2]{Alper-GMS} the composition $\mathcal{X}_S^{\mathrm{rig}} = [\Spec_S \mathcal{A}/\mathbb{G}_{m,S}] \to B\mathbb{G}_{m,S} \to S$ is a good moduli space, and so $\mathcal{X}^{\mathrm{rig}} \to \mathcal{Z}^{\mathrm{rig}}$ is a good moduli space morphism. This completes the proof of the proposition.
\end{proof}

\begin{rmk}
    Without any flatness hypotheses on the relative inertia $\mathcal{I}_{\mathrm{gr}}$, it is shown in \cite{IbanezNunez-GMSGradedPoints} that in the case where $B$ is assumed to be locally Noetherian, the morphism $\pi \circ \mathrm{gr} : \mathcal{X} \to Z$ is a categorical moduli space (in the sense that is it universal for morphisms from $\mathcal{X}$ to algebraic spaces over $B$), and that the adjunction morphism $\mathcal{O}_Z \to (\pi \circ \mathrm{gr})_{\ast} \mathcal{O}_{\mathcal{X}}$ is an isomorphism.
\end{rmk}

\subsection{Projective Moduli Spaces Arising from Invertible Sheaves on the Rigidified Stack} \label{sec:projective-moduli-spaces-inv-sheaves}

Continue to assume that $\mathcal{Z}$ admits a good moduli space $\pi : \mathcal{Z} \to Z$. From now until the end of the section, we additionally assume that the stack $\mathcal{X}$ is locally Noetherian. Given an invertible sheaf $\mathcal{L}$ on $\mathcal{X}^{\mathrm{rig}}$, we establish the existence of open substacks $(\mathcal{X}^{\mathrm{rig}})^s(\mathcal{L}) \subset (\mathcal{X}^{\mathrm{rig}})^{ss}(\mathcal{L})$ of $\mathcal{X}^{\mathrm{rig}}$ such that $(\mathcal{X}^{\mathrm{rig}})^{ss}(\mathcal{L})$ admits a good moduli space of the form
$$ \pi_{\mathcal{L}} : (\mathcal{X}^{\mathrm{rig}})^{ss}(\mathcal{L}) \to X(\mathcal{L}) = \underline{\mathrm{Proj}}_Z \bigoplus_{m \geq 0} (\pi \circ \overline{\mathrm{gr}})_{\ast}(\mathcal{L}^m) $$
and such that the restriction of $\pi_{\mathcal{L}}$ to $(\mathcal{X}^{\mathrm{rig}})^{s}(\mathcal{L})$ is a tame moduli space onto its image. We will also establish that for any algebraically closed field $k$, the $k$-points of $(\mathcal{X}^{\mathrm{rig}})^{(s)s}(\mathcal{L})$ are the $k$-points of $\mathcal{X}^{\mathrm{rig}}$ which are $\Theta_k$-(semi)stable with respect to $\mathcal{L}$, and that the $k$-points of $(\mathcal{X}^{\mathrm{rig}})^{s}(\mathcal{L})$ have $0$-dimensional automorphism groups. 

Our strategy is to essentially implement variation of good moduli space \cite[Theorem 5.6.1]{Halpern-Leistner} (see also \cite[Theorem 2.6.3]{IbanezNunez-RefinedHN}) for the $\Theta$-stability condition on $\mathcal{X}^{\mathrm{rig}}$ determined by $\mathcal{L}$, in the sense of Definition \ref{def:theta-stability}. However, since we wish to also take into account the $\Theta$-stable locus in addition to the $\Theta$-semistable locus, whereas \cite[Theorem 5.6.1]{Halpern-Leistner} and \cite[Theorem 2.6.3]{IbanezNunez-RefinedHN} only considers the $\Theta$-semistable locus, we instead will establish the version of variation of good moduli space we require ourselves.

For the rest of this section, let $\mathcal{L}$ be a fixed invertible sheaf on $\mathcal{X}^{\mathrm{rig}}$.

\subsubsection{The Semistable Loci as Algebraic Stacks}

We first define the substacks $(\mathcal{X}^{\mathrm{rig}})^{(s)s}(\mathcal{L})$ under consideration in Theorem \ref{thm:main-theorem}(3).

\begin{defi} \label{def:L-semistable-locus}
    Set $\mathcal{A} := \bigoplus_{m \geq 0} \mathcal{L}^m$ and $\mathcal{A}_+ := \bigoplus_{m > 0} \mathcal{L}^m$. Let $\phi := \pi \circ \overline{\mathrm{gr}} : \mathcal{X}^{\mathrm{rig}} \to Z$ be the good moduli space of $\mathcal{X}^{\mathrm{rig}}$. Let $\phi^{-1}\phi_{\ast}\mathcal{A}_+$ denote the image of the natural homomorphism $\phi^{\ast}\phi_{\ast} \mathcal{A}_+ \to \mathcal{A}_+ \to \mathcal{A}$. We then define
    $$ (\mathcal{X}^{\mathrm{rig}})^{ss}(\mathcal{L}) := \mathcal{X}^{\mathrm{rig}} \setminus \mathbb{V}(\phi^{-1} \phi_{\ast} \mathcal{A}_+). $$
\end{defi}

In other words, $(\mathcal{X}^{\mathrm{rig}})^{ss}(\mathcal{L})$ is the relative saturated projective spectrum $\underline{\mathrm{Proj}}_{\mathcal{X}^{\mathrm{rig}}}^{\phi}(\mathcal{A})$, in the sense of Edidin--Rydh \cite{Edidin-Rydh} (compare this with \cite[Proposition 3.17]{Edidin-Rydh}).

By upper semi-continuity for the dimension of group schemes \cite[Exposé VI-B, Proposition 4.1]{SGA3}, the set
$$ \mathcal{C} := \{ x \in |(\mathcal{X}^{\mathrm{rig}})^{ss}(\mathcal{L})| : \dim \mathrm{Aut}_{\mathcal{X}^{\mathrm{rig}}}(x) > 0 \} $$
is closed in $|(\mathcal{X}^{\mathrm{rig}})^{ss}(\mathcal{L})|$. By \cite[Proposition 2.6]{Edidin-Rydh},\footnote{Although \cite[Proposition 2.6]{Edidin-Rydh} assumes that the algebraic stacks involved are Noetherian, their argument carries over to the locally Noetherian setting without changes.} it follows from the existence of the good moduli space $\pi_{\mathcal{L}} : (\mathcal{X}^{\mathrm{rig}})^{ss}(\mathcal{L}) \to X(\mathcal{L})$ (see Proposition \ref{prop:gms-from-line-bundle} below) that the set of points $|(\mathcal{X}^{\mathrm{rig}})^{ss}(\mathcal{L})| \setminus \pi_{\mathcal{L}}^{-1}(\pi_{\mathcal{L}}(\mathcal{C}))$ is open in $|(\mathcal{X}^{\mathrm{rig}})^{ss}(\mathcal{L})|$.

\begin{defi} \label{def:L-stable-locus}
    We define $(\mathcal{X}^{\mathrm{rig}})^{s}(\mathcal{L})$ to be the open substack of $(\mathcal{X}^{\mathrm{rig}})^{ss}(\mathcal{L})$ with
    $$ |(\mathcal{X}^{\mathrm{rig}})^{s}(\mathcal{L})| = |(\mathcal{X}^{\mathrm{rig}})^{ss}(\mathcal{L})| \setminus \pi_{\mathcal{L}}^{-1}(\pi_{\mathcal{L}}(\mathcal{C})). $$
\end{defi}

In other words, $(\mathcal{X}^{\mathrm{rig}})^{s}(\mathcal{L})$ is the properly stable locus of the good moduli space $\pi_{\mathcal{L}} : (\mathcal{X}^{\mathrm{rig}})^{ss}(\mathcal{L}) \to X(\mathcal{L})$, in the sense of \cite{Edidin-Rydh}. By construction, every point of $(\mathcal{X}^{\mathrm{rig}})^{s}(\mathcal{L})$ is a regular point of $(\mathcal{X}^{\mathrm{rig}})^{ss}(\mathcal{L})$, in the sense of \cite[Definition 11.9]{Alper-GMS}, since each geometric point of $(\mathcal{X}^{\mathrm{rig}})^s(\mathcal{L})$ admits a neighborhood in $(\mathcal{X}^{\mathrm{rig}})^{ss}(\mathcal{L})$ where the stabilizer dimension is constantly $0$.

\subsubsection{Existence of a Good Moduli Space of the Semistable Locus}

Immediately from the definition, a point $x \in |\mathcal{X}^{\mathrm{rig}}|$ is a point of $|(\mathcal{X}^{\mathrm{rig}})^{ss}(\mathcal{L})|$ if and only if there exists a positive integer $m > 0$ such that the natural homomorphism $\phi^{\ast} \phi_{\ast} \mathcal{L}^m \to \mathcal{L}^m$ is surjective at $x$. In particular, $(\mathcal{X}^{\mathrm{rig}})^{ss}(\mathcal{L})$ coincides with the domain of definition of the natural partial morphism
$$ \pi_{\mathcal{L}} : \mathcal{X}^{\mathrm{rig}} = \underline{\mathrm{Proj}}_{\mathcal{X}^{\mathrm{rig}}} \bigoplus_{m \geq 0} \mathcal{L}^m \dashrightarrow X(\mathcal{L}) := \underline{\mathrm{Proj}}_Z \bigoplus_{m \geq 0} \phi_{\ast} \mathcal{L}^m. $$

\begin{prop} \label{prop:gms-from-line-bundle}
    The morphism $\pi_{\mathcal{L}} : (\mathcal{X}^{\mathrm{rig}})^{ss}(\mathcal{L}) \to X(\mathcal{L})$ is a good moduli space (in particular, $(\mathcal{X}^{\mathrm{rig}})^{s}(\mathcal{L})$ is well-defined). If $(\mathcal{X}^{\mathrm{rig}})^s(\mathcal{L})$ is non-empty, then the restriction of $\pi_{\mathcal{L}}$ to $(\mathcal{X}^{\mathrm{rig}})^s(\mathcal{L})$ has open image $X^s(\mathcal{L}) \subset X(\mathcal{L})$, and $\pi_{\mathcal{L}} : (\mathcal{X}^{\mathrm{rig}})^{s}(\mathcal{L}) \to X^s(\mathcal{L})$ is a tame moduli space.
\end{prop}

\begin{proof}
    Since $\phi : \mathcal{X}^{\mathrm{rig}} \to Z$ is a good moduli space, the result that $\pi_{\mathcal{L}} : (\mathcal{X}^{\mathrm{rig}})^{ss}(\mathcal{L}) \to X(\mathcal{L})$ is a good moduli space follows from \cite[Proposition 3.4]{Edidin-Rydh}; although the assumption is that the stacks involved are Noetherian and the algebras involved are finitely generated, the proof carries over without any modifications to the locally Noetherian setting.

    If $(\mathcal{X}^{\mathrm{rig}})^s(\mathcal{L})$ is non-empty, by \cite[Proposition 2.6]{Edidin-Rydh} the substack $(\mathcal{X}^{\mathrm{rig}})^{s}(\mathcal{L})$ of $(\mathcal{X}^{\mathrm{rig}})^{ss}(\mathcal{L})$ is saturated. Additionally, by \cite[Lemma 11.10]{Alper-GMS}, each geometric point $x : \Spec k \to (\mathcal{X}^{\mathrm{rig}})^s(\mathcal{L})$ has a closed orbit, in the sense that the natural morphism from the residual gerbe of $x$ to $(\mathcal{X}^{\mathrm{rig}})^{s}(\mathcal{L}) \times_B \Spec k$ is a closed immersion. It then follows from \cite[Remark 6.2, Proposition 7.7]{Alper-GMS} that the image $X^s(\mathcal{L})$ of $(\mathcal{X}^{\mathrm{rig}})^s(\mathcal{L})$ under $\pi_{\mathcal{L}}$ is open in $X(\mathcal{L})$, and that $(\mathcal{X}^{\mathrm{rig}})^s(\mathcal{L}) \to X^s(\mathcal{L})$ is a tame moduli space.
\end{proof}

Concerning the projectivity of the morphism $X(\mathcal{L}) \to Z$, we have the following result.

\begin{lem} \label{lem:when-moduli-space-is-projective}
    If $B$ is an excellent scheme, if $\mathcal{X}$ is of finite type over $B$ and if the good moduli space $Z$ is a scheme, then $\bigoplus_{m \geq 0} \phi_{\ast} \mathcal{L}^m = \bigoplus_{m \geq 0} (\pi \circ \overline{\mathrm{gr}})_{\ast} (\mathcal{L}^m)$ is an $\mathcal{O}_Z$-algebra of finite type, and the natural morphism $X(\mathcal{L}) \to Z$ is locally projective. If $Z$ is quasi-compact then $X(\mathcal{L}) \to Z$ is projective.
\end{lem}

\begin{proof}
    Observe that if $\mathcal{X}$ is of finite type over $B$ then so are $\mathcal{X}^{\mathrm{rig}}$ and $\mathcal{Z}$, hence by \cite[Theorem 4.16(xi)]{Alper-GMS} so is their common good moduli space $Z$. Since $\phi : \mathcal{X}^{\mathrm{rig}} \to Z$ is a good moduli space, the induced morphism
    $$ \underline{\Spec}_{\mathcal{X}^{\mathrm{rig}}} \left( \bigoplus_{m \geq 0} \mathcal{L}^m \right) \to \underline{\Spec}_Z \left( \bigoplus_{m \geq 0} \phi_{\ast} \mathcal{L}^m \right) $$
    is also a good moduli space. But $\underline{\Spec}_{\mathcal{X}^{\mathrm{rig}}} \left( \bigoplus_{m \geq 0} \mathcal{L}^m \right)$ is of finite type over $\mathcal{X}^{\mathrm{rig}}$ (and thus $B$), hence $\underline{\Spec}_Z \left( \bigoplus_{m \geq 0} \phi_{\ast} \mathcal{L}^m \right)$ is of finite type over $B$. This implies that $\bigoplus_{m \geq 0} \phi_{\ast} \mathcal{L}^m$ must be an $\mathcal{O}_Z$-algebra of finite type. Further, the graded piece in degree $0$ is $\phi_{\ast} \mathcal{O}_{\mathcal{X}^{\mathrm{rig}}} = \mathcal{O}_Z$. As such, the remaining statements follow from standard properties of relative projective spectra.
\end{proof}

We have an alternative description of $X(\mathcal{L})$ as a relative projective spectrum, in terms of the pullback $\rho^{\ast} \mathcal{L}$ along the gerbe $\rho : \mathcal{X} \to \mathcal{X}^{\mathrm{rig}}$.

\begin{lem} \label{lem:alt-desc-of-relative-proj}
    There is an equality
    $$ X(\mathcal{L}) := \underline{\mathrm{Proj}}_Z \bigoplus_{m \geq 0} (\pi \circ \overline{\mathrm{gr}})_{\ast} (\mathcal{L}^m) = \underline{\mathrm{Proj}}_Z \bigoplus_{m \geq 0} (\pi \circ \mathrm{gr})_{\ast}(\rho^{\ast} \mathcal{L}^m). $$
\end{lem}

\begin{proof}
    Since $\rho : \mathcal{X} \to \mathcal{X}^{\mathrm{rig}}$ is an fppf gerbe, for any locally free sheaf $\mathcal{F}$ of finite rank on $\mathcal{X}^{\mathrm{rig}}$ we have that the adjunction map $\mathcal{F} \to \rho_{\ast} \rho^{\ast} \mathcal{F}$ is an isomorphism. Indeed, when $\mathcal{F} = \mathcal{O}_{\mathcal{X}^{\mathrm{rig}}}$, the problem of showing that $\mathcal{O}_{\mathcal{X}^{\mathrm{rig}}} \to \rho_{\ast} \rho^{\ast} \mathcal{O}_{\mathcal{X}^{\mathrm{rig}}}$ is an isomorphism is local for the fppf topology, and for a trivial gerbe $\rho': B_SG \to S$ over an algebraic space, the fact that the adjunction map $\mathcal{O}_S \to \rho'_{\ast} (\rho')^{\ast} \mathcal{O}_S$ is an isomorphism easily follows once one identifies $(\rho')^{\ast}$ with the functor assigning to a quasi-coherent sheaf the trivial $G$-structure and $\rho'_{\ast}$ with the functor of $G$-invariants. The isomorphism for a general locally free sheaf $\mathcal{F}$ of finite rank then follows from the projection formula.

    Taking $\mathcal{F} = \mathcal{L}^m$, it follows that for each $m \geq 0$ there is a natural isomorphism
    $$ \phi_{\ast} \mathcal{L}^m = (\pi \circ \overline{\mathrm{gr}})_{\ast}(\mathcal{L}^m) \stackrel{\simeq}{\to} (\pi \circ \mathrm{gr})_{\ast} (\rho^{\ast} \mathcal{L}^m), $$
    from which the lemma follows.
\end{proof}

\subsubsection{Compatibility with Passing to Geometric Fibers}

Let $k$ be an algebraically closed field, and let $z \in Z(k)$ be a $k$-point of $Z$. By the compatibility of good moduli spaces with respect to schematic base changes, the base change $\phi_z : \mathcal{X}_z^{\mathrm{rig}} \to \Spec k$ is a good moduli space. Let $\mathcal{L}_z$ denote the pullback of $\mathcal{L}$ to $\mathcal{X}_z^{\mathrm{rig}}$; we may then define $(\mathcal{X}_z^{\mathrm{rig}})^{(s)s}(\mathcal{L}_z)$ in the same way as $(\mathcal{X}^{\mathrm{rig}})^{(s)s}(\mathcal{L})$, working relative to $\phi_z$. The morphism $\mathcal{X}_z^{\mathrm{rig}} \to \mathcal{X}^{\mathrm{rig}}$ induces inclusions $(\mathcal{X}^{\mathrm{rig}})^{(s)s}(\mathcal{L})_z \subset (\mathcal{X}_z^{\mathrm{rig}})^{(s)s}(\mathcal{L}_z)$.

\begin{lem} \label{lem:stability_is_fiberwise}
    The restriction morphisms $(\mathcal{X}^{\mathrm{rig}})^{(s)s}(\mathcal{L})_z \to (\mathcal{X}_z^{\mathrm{rig}})^{(s)s}(\mathcal{L}_z)$ are both isomorphisms.
\end{lem}

\begin{proof}
    For any quasi-coherent sheaf $\mathcal{F}$ of $\mathcal{O}_{\mathcal{X}^{\mathrm{rig}}}$-modules, by \cite[Proposition 4.7(iii)]{Alper-GMS} the formation of $\phi_{\ast}\mathcal{F}$ commutes with base changes $S \to Z$, where $S$ is any quasi-separated algebraic space. Consequently, the formation of the saturated projective spectrum $(\mathcal{X}^{\mathrm{rig}})^{ss}(\mathcal{L}) = \underline{\mathrm{Proj}}_{\mathcal{X}^{\mathrm{rig}}}^{\phi} \bigoplus_{m \geq 0} \mathcal{L}^m$ commutes with taking the base change along $z : \Spec k \to Z$: the restriction of the natural morphism $\phi^{\ast} \phi_{\ast} \mathcal{A}_+ \to \mathcal{A}$ to $\mathcal{X}_z^{\mathrm{rig}}$ coincides with the natural morphism $\phi_z^{\ast} (\phi_z)_{\ast} (\mathcal{A}_z)_+ \to \mathcal{A}_z$, where $\mathcal{A}_z := \bigoplus_{m \geq 0} \mathcal{L}_z^m$, hence the surjectivity of $\phi^{\ast} \phi_{\ast} \mathcal{L}^m \to \mathcal{L}^m$ at a point $x \in |\mathcal{X}_z| \subset |\mathcal{X}|$ is equivalent to the surjectivity of $\phi_z^{\ast} (\phi_z)_{\ast} \mathcal{L}_z^m \to \mathcal{L}_z^m$ at $x$. It follows that the inclusion $(\mathcal{X}^{\mathrm{rig}})^{ss}(\mathcal{L})_z \subset (\mathcal{X}_z^{\mathrm{rig}})^{ss}(\mathcal{L}_z)$ is an isomorphism, and that the good moduli space of $(\mathcal{X}_z^{\mathrm{rig}})^{ss}(\mathcal{L}_z)$ coincides with the fiber over $z$ of the good moduli space $\pi_{\mathcal{L}} : (\mathcal{X}^{\mathrm{rig}})^{ss}(\mathcal{L}) \to X(\mathcal{L})$. In particular, the formation of the subsets $\mathcal{C}$ and $\pi_{\mathcal{L}}^{-1}(\pi_{\mathcal{L}}(\mathcal{C}))$ of $|(\mathcal{X}^{\mathrm{rig}})^{ss}(\mathcal{L})|$ are compatible with passing to the fiber over $z$, hence the inclusion $(\mathcal{X}^{\mathrm{rig}})^{s}(\mathcal{L})_z \subset (\mathcal{X}_z^{\mathrm{rig}})^{s}(\mathcal{L}_z)$ is also an isomorphism.
\end{proof}

\subsubsection{Characterizing the (Semi)stable Loci as $\Theta$-(Semi)stable Loci} \label{sec:ss-loci-as-theta-ss-loci}

With $k$ as above, we now show that the $k$-points of $(\mathcal{X}^{\mathrm{rig}})^{(s)s}(\mathcal{L})$ coincide with the $k$-points of $\mathcal{X}^{\mathrm{rig}}$ which are $\Theta_k$-(semi)stable with respect to $\mathcal{L}$. Our two main tools for doing this are the étale local structure theorem for good moduli spaces (Theorem \ref{thm:gms-etale-local-structure}) and the relative Hilbert--Mumford criterion of Gulbrandsen--Halle--Hulek \cite{GHH-Relative-HM}.

\begin{lem} \label{lem:fiber-local-structure-theorem}
    Let $z \in Z(k)$, where $k$ is an algebraically closed field. Then the fiber $\mathcal{X}_z^{\mathrm{rig}}$ is a Noetherian algebraic stack with good moduli space $\Spec k$, whence admits a unique closed point. Furthermore, if $G$ is the automorphism group of the unique closed point of the fiber $\mathcal{X}_z^{\mathrm{rig}}$, then $G$ is a linearly reductive $k$-group, and there exists an affine scheme $A$ of finite type over $k$ acted on by $G$ such that $\mathcal{X}_z^{\mathrm{rig}} \cong [\Spec A/G]$ and $\Spec A /\!\!/ G = \Spec k$.
\end{lem}

\begin{proof}
    As observed earlier, the morphism $\phi_z : \mathcal{X}_z^{\mathrm{rig}} \to \Spec k$ is a good moduli space. Since $\phi_z$ is quasi-compact, this implies that the fiber $\mathcal{X}_z^{\mathrm{rig}}$ (which a priori is locally Noetherian) is necessarily Noetherian. Since $\phi_z$ admits a bjection on the level of closed points, $\mathcal{X}_z^{\mathrm{rig}}$ must also admit a unique closed point; the associated automorphism group $G$ is necessarily linearly reductive. Finally, Theorem \ref{thm:gms-etale-local-structure} is applicable to the good moduli space $\phi_z : \mathcal{X}_z^{\mathrm{rig}} \to \Spec k$; the existence of an isomorphism $\mathcal{X}_z^{\mathrm{rig}} \cong [\Spec A/G]$ follows directly from the étale local structure theorem.
\end{proof}

Let $z \in Z(k)$, where $k$ is an algebraically closed field, and let $G$ be the automorphism group of the unique closed point of the fiber $\mathcal{X}_z^{\mathrm{rig}}$. The invertible sheaf $\mathcal{L}_z$ corresponds to a $G$-linearized invertible sheaf $M$ on $\Spec A$. By \cite[Remark 3.16 and Proposition 3.17]{Edidin-Rydh} (taking $f = \mathrm{id}_{\mathcal{X}_z^{\mathrm{rig}}}$ and $S = \mathcal{Y} = \Spec k$ to be the good moduli space of $\mathcal{X}_z^{\mathrm{rig}}$) and the discussion preceding \cite[Theorem 13.6]{Alper-GMS}, the isomorphism $\mathcal{X}_z^{\mathrm{rig}} \cong [\Spec A/G]$ induces an isomorphism of open substacks
$$ (\mathcal{X}_z^{\mathrm{rig}})^{ss}(\mathcal{L}_z) \cong [\Spec A/G]^{ss}(M) = [(\Spec A)^{ss}(M)/G] $$
where $(\Spec A)^{ss}(M)$ is as defined in \cite{MFK-GIT}. Since all geometric points of $(\mathcal{X}_z^{\mathrm{rig}})^{s}(\mathcal{L}_z)$ are regular and have a closed orbit in $(\mathcal{X}_z^{\mathrm{rig}})^{ss}(\mathcal{L}_z)$ of dimension $0$, it follows from \cite[Amplification 1.11]{MFK-GIT} that the above isomorphism restricts to an isomorphism
$$ (\mathcal{X}_z^{\mathrm{rig}})^{s}(\mathcal{L}_z) \cong [\Spec A/G]^{s}(M) = [(\Spec A)^{s}(M)/G], $$
where $(\Spec A)^s(M)$ denotes the GIT (properly) stable locus with respect to $M$.\footnote{In the notation of \cite{MFK-GIT}, this is $(\Spec A)_{(0)}^s(M)$.}

\begin{prop} \label{prop:ghh_hilbert_mumford}
    Let $p$ be a $k$-point of $\Spec A$. Then $p \in (\Spec A)^{(s)s}(M)(k)$ if and only if for all non-trivial cocharacters $\eta : \mathbb{G}_{m,k} \to G$ for which the limit $\lim_{t \to 0} \eta(t) \cdot p$ exists in $\Spec A$, the Hilbert--Mumford weight $\mu^M(p, \eta)$ satisfies $\mu^M(p, \eta) > (\geq) \ 0$.
\end{prop}

\begin{proof}
    Since $\Spec A$ is an affine scheme of finite type over $k$ which is projective over itself, and since $G$ is affine and linearly reductive, this follows from \cite[Theorem 3.3]{GHH-Relative-HM}.
\end{proof}

We are now ready to characterize the geometric points of $(\mathcal{X}^{\mathrm{rig}})^{(s)s}(\mathcal{L})$ as those points which are $\Theta$-(semi)stable with respect to $\mathcal{L}$.

\begin{prop} \label{prop:semistable-loci-are-theta-semistable-loci}
    Let $k$ be an algebraically closed field, and let $x \in \mathcal{X}^{\mathrm{rig}}(k)$. Then $x \in (\mathcal{X}^{\mathrm{rig}})^{(s)s}(\mathcal{L})(k)$ if and only if for all filtrations $f : \Theta_k \to \mathcal{X}^{\mathrm{rig}}$ with $\mathrm{ev}_1(f) \cong x$ for which $f$ does not factor through $\Theta_k \to \Spec k$, we have $\mathrm{wt}_0(f^{\ast} \mathcal{L}) < (\leq) \ 0$, i.e. $x$ is $\Theta_k$-(semi)stable with respect to $\mathcal{L}$.
\end{prop}

\begin{proof}
    Let $z = \phi(x) \in Z(k)$. The structure morphism $\Theta_k \to \Spec k$ is a good moduli space, and in particular is universal for morphisms from $\Theta_k$ to algebraic spaces. This implies that if $f : \Theta_k \to \mathcal{X}^{\mathrm{rig}}$ is a non-trivial filtration with $\mathrm{ev}_1(f) \cong x$, then the composite of $f$ with $\phi$ must factor through $\phi(\mathrm{ev}_1(f)) : \Spec k \to Z$. But $\phi(\mathrm{ev}_1(f)) = \phi(x) = z$, and so $f$ must factor through the fiber $\mathcal{X}_z^{\mathrm{rig}}$. In particular, we have $\mathrm{wt}_0(f^{\ast} \mathcal{L}) = \mathrm{wt}_0(f^{\ast} \mathcal{L}_z)$. 
    
    On the other hand, by combining Proposition \ref{prop:ghh_hilbert_mumford} with Example \ref{ex:theta-stability-for-quotient-stacks} we obtain that the $k$-points of the substacks $(\mathcal{X}_z^{\mathrm{rig}})^{(s)s}(\mathcal{L}_z)$ of $\mathcal{X}_z^{\mathrm{rig}}$ are characterized in terms of morphisms from $\Theta_k$: if $x \in \mathcal{X}_z^{\mathrm{rig}}(k)$, then $x \in (\mathcal{X}_z^{\mathrm{rig}})^{(s)s}(\mathcal{L})(k)$ if and only if for all non-trivial filtrations $f : \Theta_k \to \mathcal{X}_z^{\mathrm{rig}}$ with $\mathrm{ev}_1(f) \cong x$, $\mathrm{wt}_0(f^{\ast} \mathcal{L}_z) < (\leq) \ 0$. Combining this with Lemma \ref{lem:stability_is_fiberwise} then yields the result.
\end{proof}

\subsection{Lifting Filtrations Along the Rigidification}

In order to complete the proof of Theorem \ref{thm:main-theorem}, it remains to establish statement $(2)(\mathrm{iii})$, i.e. that the geometric points of the preimages of $(\mathcal{X}^{\mathrm{rig}})^{(s)s}(\mathcal{L})$ under the rigidification morphism $\rho : \mathcal{X} \to \mathcal{X}^{\mathrm{rig}}$ are themselves characterized in terms of morphisms $\Theta_k \to \mathcal{X}$. To do this, we first aim to show that for any algebraically closed field $k$ the gerbe $\rho$ is $\Theta_k$-surjective, in the sense of \cite[Section 3.4]{AHLH-Existence-GMS}; in other words, any commutative diagram 
\[\begin{tikzcd}[ampersand replacement=\&]
	{\Spec k} \&\& {\mathcal{X}} \\
	\\
	{\Theta_k} \&\& {\mathcal{X}^{\mathrm{rig}}}
	\arrow[from=1-1, to=1-3]
	\arrow["j"', from=1-1, to=3-1]
	\arrow["\rho", from=1-3, to=3-3]
	\arrow[dashed, from=3-1, to=1-3]
	\arrow[from=3-1, to=3-3]
\end{tikzcd}\]
of solid arrows can be filled in with a dashed arrow (here $j$ denotes the open $k$-point of $\Theta_k$, represented by $1 \in \mathbb{A}_k^1(k)$). 

\subsubsection{The Morphisms $\mathrm{Filt}(\rho)$ and $\mathrm{Grad}(\rho)$ are Surjective}

In order to prove that $\rho$ is $\Theta_k$-surjective, we first show  that the morphisms $\mathrm{Filt}(\mathcal{X}) \to \mathrm{Filt}(\mathcal{X}^{\mathrm{rig}})$ and $\mathrm{Grad}(\mathcal{X}) \to \mathrm{Grad}(\mathcal{X}^{\mathrm{rig}})$ induced by $\rho$ are surjective. Since $\rho$ is smooth, by \cite[Proposition 1.3.1]{Halpern-Leistner} both of these morphisms are smooth, and hence have open images (we used the smoothness of $\mathrm{Grad}(\rho)$ earlier in the proof of Lemma \ref{lem:centers-agree}). By \cite[Lemma 4.4.4]{Halpern-Leistner}, the surjectivity of $\mathrm{Filt}(\rho)$ is implied by the surjectivity of $\mathrm{Grad}(\rho)$. We establish the surjectivity of $\mathrm{Grad}(\rho)$ in the following lemmas.

\begin{lem} \label{lem:automorphism-group-split-filt-semidirect-product}
    Let $k$ be an algebraically closed field and let $z \in \mathcal{Z}(k)$. Then the automorphism group $\mathrm{Aut}_{\mathcal{X}}(\sigma(z))$ is given by the semidirect product $\mathcal{I}_{\mathrm{gr}, \sigma(z)} \rtimes \mathrm{Aut}_{\mathcal{X}^{\mathrm{rig}}}(\overline{\sigma}(z))$. 
\end{lem}

\begin{proof}
    We have a short exact sequence of $k$-groups
    \[\begin{tikzcd}[ampersand replacement=\&]
	1 \& {\mathcal{I}_{\mathrm{gr}, \sigma(z)}} \& {\mathrm{Aut}_{\mathcal{X}}(\sigma(z))} \& {\mathrm{Aut}_{\mathcal{X}^{\mathrm{rig}}}(\overline{\sigma}(z))} \& {1.}
	\arrow[from=1-1, to=1-2]
	\arrow[from=1-2, to=1-3]
	\arrow["{\rho_{\ast}}", from=1-3, to=1-4]
	\arrow[from=1-4, to=1-5]
    \end{tikzcd}\]
    On the other hand, by Theorem \ref{thm:affinity-grbar}, the morphism $\overline{\sigma} : \mathcal{Z} \to \mathcal{X}^{\mathrm{rig}}$ is a closed immersion; indeed, $\overline{\sigma}$ is a section of the induced (affine) retraction $\overline{\mathrm{gr}} : \mathcal{X}^{\mathrm{rig}} \to \mathcal{Z}$. In particular, the map $\overline{\mathrm{gr}}_{\ast} : \mathrm{Aut}_{\mathcal{X}^{\mathrm{rig}}}(\overline{\sigma}(z)) \to \mathrm{Aut}_{\mathcal{Z}}(z)$ is an isomorphism, with inverse $\overline{\sigma}_{\ast}$. Composing $\overline{\mathrm{gr}}_{\ast} : \mathrm{Aut}_{\mathcal{X}^{\mathrm{rig}}}(\overline{\sigma}(z)) \to \mathrm{Aut}_{\mathcal{Z}}(z)$ with $\sigma_{\ast} : \mathrm{Aut}_{\mathcal{Z}}(z) \to \mathrm{Aut}_{\mathcal{X}}(\sigma(z))$ then yields a splitting of the above short exact sequence.
\end{proof}

\begin{lem} \label{lem:lifting_graded_points_in_center}
    Let $k$ be an algebraically closed field, and let $g : B\mathbb{G}_{m,k} \to \mathcal{X}^{\mathrm{rig}}$ be a graded point whose underlying $k$-point $u(g)$ is isomorphic to $\overline{\sigma}(z)$ for some $z \in \mathcal{Z}(k)$. Then $g$ admits a lift to $\mathcal{X}$. In particular, $g$ lies in the image of $|\mathrm{Grad}(\mathcal{X})| \to |\mathrm{Grad}(\mathcal{X}^{\mathrm{rig}})|$.
\end{lem}

\begin{proof}
    Since $\rho$ is an fppf gerbe, each fiber of $\rho$ over a $k$-point of $\mathcal{X}^{\mathrm{rig}}$ admits a unique $k$-point. Moreover, if $x, x' \in \mathcal{X}(k)$ are such that $\rho(x) \cong \rho(x')$ as $k$-points of $\mathcal{X}^{\mathrm{rig}}$ then $x \cong x'$.\footnote{Since $\rho$ is an fppf gerbe then $|\rho|$ is a universal homeomorphism, so $\rho(x) \cong \rho(x')$ implies that $\mathrm{Isom}_{\mathcal{X}}(x, x')$ is non-empty. On the other hand, this is an affine $k$-scheme locally of finite type, hence must admit a $k$-point if non-empty.} It follows that we may pick a lift $x \in \mathcal{X}(k)$ of $u(g)$ with $x \cong \sigma(z)$. Without loss of generality we may assume that $x = \sigma(z)$, so that $u(g) = \rho(x) = \overline{\sigma}(z)$.

    By Proposition \ref{prop:alt-desc-grad}, specifying $g$ is equivalent to specifying a cocharacter $\eta : \mathbb{G}_{m,k} \to \mathrm{Aut}_{\mathcal{X}^{\mathrm{rig}}}(\overline{\sigma}(z)) = \mathrm{Aut}_{\mathcal{X}^{\mathrm{rig}}}(\rho(x))$; a lift of $g$ to $\mathcal{X}$ exists if there exists a lift of $\eta$ along the homomorphism $\mathrm{Aut}_{\mathcal{X}}(x) \to \mathrm{Aut}_{\mathcal{X}^{\mathrm{rig}}}(\rho(x))$. However, by Lemma \ref{lem:automorphism-group-split-filt-semidirect-product} this homomorphism admits a section, which immediately implies that there exists a lift of $\eta$. 
\end{proof}

\begin{lem} \label{lem:grad-gerbe-is-surjective}
    The morphism $\mathrm{Grad}(\rho) : \mathrm{Grad}(\mathcal{X}) \to \mathrm{Grad}(\mathcal{X}^{\mathrm{rig}})$ is surjective.
\end{lem}

\begin{proof}
    Let $k$ be an algebraically closed field, and let $g : B\mathbb{G}_{m,k} \to \mathcal{X}^{\mathrm{rig}}$ be a morphism whose underlying $k$-point $u(g)$ lies over $b \in B(k)$. By restricting the induced weak $\Theta_B$-action on $\mathcal{X}^{\mathrm{rig}}$, we have a weak $\Theta_b$-action $a_b : \Theta_k \times_k \mathcal{X}_b^{\mathrm{rig}} \to \mathcal{X}_b^{\mathrm{rig}}$. The composition
    \[\begin{tikzcd}[ampersand replacement=\&]
	{\Theta_k \times_k B\mathbb{G}_{m,k}} \&\& {\Theta_k \times_k \mathcal{X}_b^{\mathrm{rig}}} \&\& {\mathcal{X}_b^{\mathrm{rig}}}
	\arrow["{{{\mathrm{id}_{\Theta_k} \times g}}}", from=1-1, to=1-3]
	\arrow["{a_b}", from=1-3, to=1-5]
    \end{tikzcd}\]
    then exhibits a $\Theta_k$-specialization between $g$ and a morphism $g' = 0 \cdot g : B\mathbb{G}_{m,k} \to \mathcal{X}^{\mathrm{rig}}$ with $u(g') \cong \overline{\sigma}(z)$ for some $z \in \mathcal{Z}(k)$ lying over $b$ (here we use that $\overline{\sigma}$ identifies $\mathcal{Z}$ with the $\mathbb{G}_{m,B}$-fixed locus in $\mathcal{X}^{\mathrm{rig}}$, as in Theorem \ref{thm:affinity-grbar}). Lemma \ref{lem:lifting_graded_points_in_center} implies that $g'$ lies in the open image of $|\mathrm{Grad}(\mathcal{X})| \to |\mathrm{Grad}(\mathcal{X}^{\mathrm{rig}})|$. Consequently $g$ must also lie in the image of $|\mathrm{Grad}(\mathcal{X})| \to |\mathrm{Grad}(\mathcal{X}^{\mathrm{rig}})|$. This establishes the surjectivity of $\mathrm{Grad}(\mathcal{X}) \to \mathrm{Grad}(\mathcal{X}^{\mathrm{rig}})$.
\end{proof}

\begin{cor} \label{cor:filt-gerbe-is-surjective}
    The morphism $\mathrm{Filt}(\rho) : \mathrm{Filt}(\mathcal{X}) \to \mathrm{Filt}(\mathcal{X}^{\mathrm{rig}})$ is surjective.
\end{cor}

\begin{proof}
    Combine Lemma \ref{lem:grad-gerbe-is-surjective} with \cite[Lemma 4.4.4]{Halpern-Leistner}.
\end{proof}

\begin{rmk} \label{rmk:unipotent_relative_inertia_groups}
    Let $k$ be an algebraically closed field, and let $x \in \mathcal{X}(k)$ be such that $x \cong \sigma(\mathrm{gr}(x))$ and such that the $k$-group $\mathcal{I}_{\mathrm{gr}, x}$ is smooth. Then $\mathcal{I}_{\mathrm{gr},x}$ is a split unipotent $k$-group. 
    
    To see this, suppose $x$ lies over $b \in B(k)$. Then the weak $\Theta_B$-action (resp. $B\mathbb{G}_{m,B}$-action) on $\mathcal{X}$ (resp. $\mathcal{Z}$) restricts to a weak action of $\Theta_k$ (resp. $B\mathbb{G}_{m,k}$) on $\mathcal{X}_b = \mathcal{X} \times_B \Spec k$ (resp. $\mathcal{Z}_b$), for which the morphism $\mathrm{gr}_b : \mathcal{X}_b \to \mathcal{Z}_b$ (resp. $\sigma_b : \mathcal{Z}_b \to \mathcal{X}_b$) is invariant with respect to $\Theta_k \to B\mathbb{G}_{m,k}$ (resp. $B\mathbb{G}_{m,k} \to \Theta_k$). By arguing as in the proof of Proposition \ref{prop:induced-action-on-rigidification}, we obtain a weak action of the monoid $\mathbb{A}_k^1$ on $\mathcal{I}_{\mathcal{X}_b}$, under which the subgroup stack $\mathcal{I}_{\mathrm{gr}_b}$ is invariant. As $x \cong \sigma(\mathrm{gr}(x)) = 0 \cdot x$, this induces a (strict) action of $\mathbb{A}_k^1$ on the scheme $\mathcal{I}_{\mathrm{gr},x}$, for which acting by $0 \in \mathbb{A}_k^1$ sends everything to the identity. 
    
    Since $\mathcal{I}_{\mathrm{gr},x}$ is a smooth affine $k$-group, the Bia\l{}ynicki-Birula theorem of \cite{BB-Theorem} thus implies that the underlying scheme of $\mathcal{I}_{\mathrm{gr},x}$ is isomorphic to affine space over $k$. In turn, by a theorem of Lazard \cite[Chapitre IV, Théorème 4.4.1]{Demazure-Gabriel}, this implies that $\mathcal{I}_{\mathrm{gr},x}$ is a split unipotent $k$-group.
\end{rmk}

\subsubsection{The Gerbe $\rho$ is $\Theta$-Surjective}

We are now ready to prove that the morphism $\rho$ is $\Theta_k$-surjective, where $k$ is an algebraically closed field.

\begin{prop} \label{prop:lifting_filts_along_gerbe}
    Let $k$ be an algebraically closed field. Let $j : \Spec k \to \Theta_k$ denote the open $k$-point of $\Theta_k$. Then any commutative diagram 
\[\begin{tikzcd}[ampersand replacement=\&]
	{\Spec k} \&\& {\mathcal{X}} \\
	\\
	{\Theta_k} \&\& {\mathcal{X}^{\mathrm{rig}}}
	\arrow[from=1-1, to=1-3]
	\arrow["j"', from=1-1, to=3-1]
	\arrow["\rho", from=1-3, to=3-3]
	\arrow[dashed, from=3-1, to=1-3]
	\arrow[from=3-1, to=3-3]
\end{tikzcd}\]
of solid arrows can be filled in with a dashed arrow.
\end{prop}

\begin{proof}
    Suppose we are given the commutative diagram
    \[\begin{tikzcd}[ampersand replacement=\&]
	{\Spec k} \&\& {\mathcal{X}} \\
	\\
	{\Theta_k} \&\& {\mathcal{X}^{\mathrm{rig}}}
	\arrow["x", from=1-1, to=1-3]
	\arrow["j"', from=1-1, to=3-1]
	\arrow["\rho", from=1-3, to=3-3]
	\arrow["f"', from=3-1, to=3-3]
    \end{tikzcd}\]
    Since $\mathrm{Filt}(\mathcal{X}) \to \mathrm{Filt}(\mathcal{X}^{\mathrm{rig}})$ is a surjective smooth morphism, the fiber over any $k$-point of $\mathrm{Filt}(\mathcal{X}^{\mathrm{rig}})$ admits a $k$-point $f' : \Theta_k \to \mathcal{X}$. Since $x$ is the unique $k$-point of $\mathcal{X}$ lying over $\mathrm{ev}_1(f)$, $f'$ can be chosen such that the resulting diagram
    \[\begin{tikzcd}[ampersand replacement=\&]
	{\Spec k} \&\& {\mathcal{X}} \\
	\\
	{\Theta_k} \&\& {\mathcal{X}^{\mathrm{rig}}}
	\arrow["x", from=1-1, to=1-3]
	\arrow["j"', from=1-1, to=3-1]
	\arrow["\rho", from=1-3, to=3-3]
	\arrow["{f'}", from=3-1, to=1-3]
	\arrow["f"', from=3-1, to=3-3]
    \end{tikzcd}\]
    commutes.
\end{proof}

\subsection{Completing the Proof of the Main Theorem}

We are now in a position to prove the remaining statement $(2)(\mathrm{iii})$ of Theorem \ref{thm:main-theorem}.

\begin{prop} \label{prop:semistable_locus_upstairs}
    Let $\mathcal{L}$ be an invertible sheaf on $\mathcal{X}^{\mathrm{rig}}$. Then, for any algebraically closed field $k$, a $k$-point $x \in \mathcal{X}(k)$ is a $k$-point of the open substack $\rho^{-1}((\mathcal{X}^{\mathrm{rig}})^{(s)s}(\mathcal{L}))$ if and only if $x$ is $\Theta_k$-(semi)stable with respect to $\rho^{\ast} \mathcal{L}$.
\end{prop}

\begin{proof}
    For any filtration $f : \Theta_k \to \mathcal{X}$, the weight of $f$ with respect to $\rho^{\ast} \mathcal{L}$ is equal to the weight of $\rho \circ f$ with respect to $\mathcal{L}$ (both are equal to $\mathrm{wt}_0(f^{\ast}\rho^{\ast}\mathcal{L})$). In particular, any $k$-point of $\rho^{-1}((\mathcal{X}^{\mathrm{rig}})^{(s)s}(\mathcal{L}))$ is $\Theta_k$-(semi)stable with respect to $\rho^{\ast} \mathcal{L}$. Conversely, suppose $x \in \mathcal{X}(k)$ is $\Theta_k$-(semi)stable, and $\overline{x} = \rho(x)$. Let $f : \Theta_k \to \mathcal{X}^{\mathrm{rig}}$ be a non-trivial filtration with $\mathrm{ev}_1(f) \cong \overline{x}$. By Proposition \ref{prop:lifting_filts_along_gerbe}, there exists a lift $f' : \Theta_k \to \mathcal{X}$ of $f$ with $\mathrm{ev}_1(f') \cong x$; since $x$ is $\Theta_k$-(semi)stable, we have
    $$ 0 > (\geq) \ \mathrm{wt}_0((f')^{\ast} \rho^{\ast} \mathcal{L}) = \mathrm{wt}_0(f^{\ast} \mathcal{L}), $$
    and so $\overline{x}$ is $\Theta_k$-(semi)stable with respect to $\mathcal{L}$. This establishes the result.
\end{proof}

\begin{proof}[Proof of Theorem \ref{thm:main-theorem}]
    Theorem \ref{thm:main-theorem} follows from combining Lemma \ref{lem:existence-rigidification}, Proposition \ref{prop:induced-action-on-rigidification}, Theorem \ref{thm:affinity-grbar}, Lemma \ref{lem:centers-agree}, Propositions \ref{prop:gms-of-Xrig} and \ref{prop:gms-from-line-bundle}, Lemma \ref{lem:alt-desc-of-relative-proj} and Propositions \ref{prop:semistable-loci-are-theta-semistable-loci} and \ref{prop:semistable_locus_upstairs}.
\end{proof}

\section{Existence of Non-Reductive GIT Quotients}

In this section, we establish Theorem \ref{thm:nrgit-intro}, the existence of non-reductive GIT quotients for actions of internally graded groups, and Theorem \ref{thm:nrgit-external-intro}, the existence of non-reductive GIT quotients for actions of externally graded groups, as corollaries of Theorem \ref{thm:main-theorem}.

Throughout this section, we fix a field $k$ (not necessarily assumed to be algebraically closed) and take the base $B$ to be $\Spec k$.

\subsection{Existence of Non-Reductive GIT Quotients for Internally Graded Groups}

As in the statement of Theorem \ref{thm:nrgit-intro}, let $X$ be a separated $k$-scheme which is locally of finite type, acted on by a smooth linear algebraic $k$-group $H$ with a split maximal torus, such that there is a distinguished cocharacter $\lambda : \mathbb{G}_{m,k} \to H$ for which there are equalities $H = P_H(\lambda)$ and $X = X^{\lambda,+}$ (in the notation of Section \ref{section:filt-grad-quotient-stacks}); note that since $X$ is separated and $H$ is affine then $[X/H]$ has an affine diagonal. Set $L := L_H(\lambda)$ and $U := U_H(\lambda)$, so that $H = U \rtimes L$ is a split semidirect product of $U$ and $L$, and set $Z := X^{\lambda,0}$.

By Theorem \ref{theorem:attractor_loci_existence}, $Z$ is a closed subscheme of $X$, and the retraction $X \to Z$ induced by $\lambda$ is an affine morphism. Further, as explained in Section \ref{section:filt-grad-quotient-stacks}, the action of $H$ on $X$ induces an action of $L$ on $Z$, for which the retraction $X \to Z$ is equivariant with respect to the quotient $H \to L$, and the closed immersion $Z \hookrightarrow X$ is equivariant with respect to the inclusion $L \subset H$. Set $\mathcal{X} := [X/H]$ and $\mathcal{Z} := [Z/L]$, so that there are induced morphisms $\mathcal{X} \to \mathcal{Z}$ and $\mathcal{Z} \to \mathcal{X}$.

\subsubsection{Weak $\Theta_k$-Action on the Quotient $[X/H]$}

In the situation described above, the cocharacter $\lambda$ induces a weak $\Theta_k$-action on $\mathcal{X}$ with center $\mathcal{Z}$.

\begin{lem} \label{lem:quotient-stack-weak-theta-action-cocharacter}
    Let $\mathcal{X} = [X/H]$ and $\mathcal{Z} = [Z/L]$ be as described above. Then the cocharacter $\lambda$ induces a weak $\Theta_k$-action on $\mathcal{X}$ with center $\mathcal{Z}$, for which the retraction $\mathcal{X} \to \mathcal{Z}$ coincides with the morphism induced by the retraction $X \to Z$ induced by $\lambda$.
\end{lem}

\begin{proof}
    By Theorem \ref{thm:filt-grad-quotient-stack}, the morphisms $\mathcal{X} \to \mathcal{Z}$ and $\mathcal{Z} \to \mathcal{X}$ induced by the retraction $X \to Z$ and the inclusion $Z \hookrightarrow X$ fit into a Cartesian diagram
\[\begin{tikzcd}[ampersand replacement=\&]
	{\mathcal{X}} \&\& {\mathrm{Filt}(\mathcal{X})} \\
	\\
	{\mathcal{Z}} \&\& {\mathrm{Grad}(\mathcal{X})}
	\arrow[from=1-1, to=1-3]
	\arrow[from=1-1, to=3-1]
	\arrow["\lrcorner"{anchor=center, pos=0.125}, draw=none, from=1-1, to=3-3]
	\arrow["{\mathrm{gr}}", from=1-3, to=3-3]
	\arrow[curve={height=-18pt}, from=3-1, to=1-1]
	\arrow[from=3-1, to=3-3]
	\arrow["{\mathrm{\sigma}}"', curve={height=18pt}, from=3-3, to=1-3]
\end{tikzcd}\]
whose horizontal arrows are the morphisms $\mathcal{X} \to \mathrm{Filt}(\mathcal{X})$ and $\mathcal{Z} \to \mathrm{Grad}(\mathcal{X})$ induced by $\lambda$, as described in Section \ref{section:filt-grad-quotient-stacks}; moreover, these horizontal morphisms are both open and closed immersions. As such, by Proposition \ref{prop:filt-weak-actions} the weak $\Theta_k$-action on $\mathrm{Filt}(\mathcal{X})$ restricts to a weak $\Theta_k$-action on $\mathcal{X}$ with center $\mathcal{Z}$ and with retraction given by the aforementioned morphism $\mathcal{X} \to \mathcal{Z}$.
\end{proof}

Henceforth, we will denote the morphism $\mathcal{X} \to \mathcal{Z}$ (resp. $\mathcal{Z} \to \mathcal{X}$) by $\mathrm{gr}$ (resp. $\mathrm{\sigma}$). 

\subsubsection{Flatness of the Relative Inertia}

Let $\mathrm{Stab}_{U,X} \to X$ be the $U$-stabilizer group scheme for the restricted action of $U$ on $X$; by definition, this is the pullback of the action morphism $\mathrm{act}_U : U \times_k X \to X \times_k X$ along the diagonal $\Delta: X \to X \times_k X$:
\[\begin{tikzcd}[ampersand replacement=\&]
	{\mathrm{Stab}_{U,X}} \&\& {U \times_k X} \\
	\\
	X \&\& {X \times_k X}
	\arrow[from=1-1, to=1-3]
	\arrow[from=1-1, to=3-1]
	\arrow["\lrcorner"{anchor=center, pos=0.125}, draw=none, from=1-1, to=3-3]
	\arrow["{\mathrm{act}_U}", from=1-3, to=3-3]
	\arrow["\Delta"', from=3-1, to=3-3]
\end{tikzcd}\]

The flatness of the relative inertia $\mathcal{I}_{\mathrm{gr}} = \mathcal{I}_{\mathcal{X}/\mathcal{Z}} \to \mathcal{X}$ is governed by the $U$-stabilizer group scheme $\mathrm{Stab}_{U,X} \to X$.

\begin{lem} \label{lem:flat-stab-flat-relative-inertia}
    The relative inertia $\mathcal{I}_{\mathrm{gr}} \to \mathcal{X}$ of $\mathrm{gr} : \mathcal{X} \to \mathcal{Z}$ is flat if and only if $\mathrm{Stab}_{U,X} \to X$ is flat.
\end{lem}

\begin{proof}
    Since $U$ is the kernel of $H \to L$, the base change of $\mathrm{gr} : \mathcal{X} \to \mathcal{Z}$ along the quotient morphism $Z \to \mathcal{Z}$ is given by the quotient stack $[X/U]$. As such, by \cite[Tag 06PQ]{stacks-project} there is a Cartesian diagram
    \[\begin{tikzcd}[ampersand replacement=\&]
	{\mathcal{I}_{[X/U]}} \&\& {\mathcal{I}_{\mathrm{gr}}} \\
	\\
	{[X/U]} \&\& {\mathcal{X}} \\
	\\
	Z \&\& {\mathcal{Z}}
	\arrow[from=1-1, to=1-3]
	\arrow[from=1-1, to=3-1]
	\arrow["\lrcorner"{anchor=center, pos=0.125}, draw=none, from=1-1, to=3-3]
	\arrow[from=1-3, to=3-3]
	\arrow[from=3-1, to=3-3]
	\arrow[from=3-1, to=5-1]
	\arrow["\lrcorner"{anchor=center, pos=0.125}, draw=none, from=3-1, to=5-3]
	\arrow["{\mathrm{gr}}", from=3-3, to=5-3]
	\arrow[from=5-1, to=5-3]
    \end{tikzcd}\]

    On the other hand, by \cite[Tag 06PB]{stacks-project} there is a Cartesian diagram
    \[\begin{tikzcd}[ampersand replacement=\&]
	{\mathrm{Stab}_{U,X}} \&\& {\mathcal{I}_{[X/U]}} \\
	\\
	X \&\& {[X/U]}
	\arrow[from=1-1, to=1-3]
	\arrow[from=1-1, to=3-1]
	\arrow["\lrcorner"{anchor=center, pos=0.125}, draw=none, from=1-1, to=3-3]
	\arrow[from=1-3, to=3-3]
	\arrow[from=3-1, to=3-3]
    \end{tikzcd}\]
    where $X \to [X/U]$ is the quotient morphism. Consequently, by smooth descent for flat morphisms, the morphism $\mathcal{I}_{\mathrm{gr}} \to \mathcal{X}$ is flat if and only if $\mathrm{Stab}_{U,X} \to X$ is flat.
\end{proof}

\begin{rmk} \label{rmk:generalizing-constancy-unipotent-stabilizers}
    If the characteristic of $k$ is $0$, by Lemma \ref{lem:smoothness-gp-schemes-cotangent-bundle} we have that the following statements are equivalent:
    \begin{enumerate}
        \item $\mathrm{Stab}_{U,X} \to X$ is flat.
        \item $\mathrm{Stab}_{U,X} \to X$ is smooth.
        \item $\omega_{\mathrm{Stab}_{U,X}/X} := e^{\ast} \Omega^1_{\mathrm{Stab}_{U,X}/X}$ is a locally free $\mathcal{O}_X$-module (here $e : X \to \mathrm{Stab}_{U,X}$ denotes the identity section).
    \end{enumerate}
    Condition $(3)$ implies that the map $x \mapsto \dim \mathrm{Lie}(\mathrm{Stab}_U(x)) = \dim \mathrm{Stab}_U(x)$ is locally constant on $|X|$. If $X$ is reduced then by \cite[Tag 0FWH]{stacks-project} the converse holds: the local constancy of $x \mapsto \dim \mathrm{Stab}_U(x)$ implies that $\omega_{\mathrm{Stab}_{U,X}/X}$ is a locally free $\mathcal{O}_X$-module.
\end{rmk}

\begin{rmk} \label{rmk:generalizing-qiao-cdrs}
    By \cite[Tag 067L]{stacks-project} (along with the standard identification of sheaves of differentials as conormal sheaves), associated to the sequence of morphisms $X \stackrel{e \times \mathrm{id}_X}{\to} U \times_k X \stackrel{\mathrm{act}_U}{\to} X \times_k X$ is the right exact sequence
    \[\begin{tikzcd}[ampersand replacement=\&]
	{\Omega_{X/\Spec k}^1} \& {\mathrm{Lie}(U)^{\vee} \otimes_k \Omega_{X/\Spec k}^1} \& {(e \times \mathrm{id}_X)^{\ast} \Omega_{\mathrm{act}_U}^1} \& 0
	\arrow[from=1-1, to=1-2]
	\arrow[from=1-2, to=1-3]
	\arrow[from=1-3, to=1-4]
    \end{tikzcd}\]
    of coherent sheaves on $X$; in the notation of \cite[Section 2.1.2]{Qiao-Blowups}, the left hand morphism $\Omega_{X/\Spec k}^1 \to \mathrm{Lie}(U)^{\vee} \otimes_k \Omega_{X/\Spec k}^1$ is denoted as $\phi$. On the other hand, from the commutative diagram
    \[\begin{tikzcd}[ampersand replacement=\&]
	{\mathrm{Stab}_{U,X}} \&\& {U \times_k X} \\
	\\
	X \&\& {X \times_k X}
	\arrow[from=1-1, to=1-3]
	\arrow[from=1-1, to=3-1]
	\arrow["\lrcorner"{anchor=center, pos=0.125}, draw=none, from=1-1, to=3-3]
	\arrow["{{\mathrm{act}_U}}", from=1-3, to=3-3]
	\arrow["e", curve={height=-18pt}, from=3-1, to=1-1]
	\arrow["{e \times \mathrm{id}_X}", from=3-1, to=1-3]
	\arrow["\Delta"', from=3-1, to=3-3]
    \end{tikzcd}\]
    we may identify $(e \times \mathrm{id}_X)^{\ast} \Omega_{\mathrm{act}_U}^1$ with $\omega_{\mathrm{Stab}_{U,X}/X}$. As such, in the notation of \cite{Qiao-Blowups} the sheaf $\mathrm{coker}(\phi)$ coincides with the sheaf of differentials $\omega_{\mathrm{Stab}_{U,X}/X}$ which, if $k$ is of characteristic $0$, governs the flatness/smoothness of $\mathrm{Stab}_{U,X} \to X$.\footnote{Similarly, the cokernel sheaf $\mathcal{Q}(\pi)$ considered in \cite[Section 2.1]{qiao-graded-additive} is also of the form $\omega_{\mathrm{Stab}_{U,X}/X}$ for appropriate $U$ and $X$.} In this way, the condition that $\mathrm{Stab}_{U,X} \to X$ is flat is a generalization of \cite[Condition 2.26]{Qiao-Blowups}, i.e. Qiao's \emph{constant dimension of relative stabilizers} condition, which in turn is a generalization of the constancy of unipotent stabilizers condition imposed in \cite{BDHK-Projective-Completions, BDHK-git-graded-unipotent}.
\end{rmk}

\subsubsection{Existence of a Geometric Quotient for the Unipotent Radical}

From this point on, assume that the $U$-stabilizer group scheme $\mathrm{Stab}_{U,X} \to X$ is flat (equivalently, that $\mathcal{I}_{\mathrm{gr}} \to \mathcal{X}$ is flat). Since $\mathcal{I}_{\mathrm{gr}} \to \mathcal{X}$ is of finite presentation (as this is true for $\mathrm{Stab}_{U,X} \to X$), by the first statement of Theorem \ref{thm:main-theorem}, the rigidification $\mathcal{X}^{\mathrm{rig}} := \mathcal{X} \!\!\! \fatslash  \mathcal{I}_{\mathrm{gr}}$ is well-defined and inherits from $\mathcal{X}$ a weak $\Theta_k$-action with the same center $\mathcal{Z}$, and for which the gerbe $\rho : \mathcal{X} \to \mathcal{X}^{\mathrm{rig}}$ is $\Theta_k$-invariant. Furthermore, the induced morphism $\overline{\mathrm{gr}} : \mathcal{X}^{\mathrm{rig}} \to \mathcal{Z}$ is affine.

\begin{prop} \label{prop:existence-geometric-U-quotient}
    Assume that the $U$-stabilizer group scheme $\mathrm{Stab}_{U,X} \to X$ is flat. Then the following statements all hold:
    \begin{enumerate}
        \item There exists a geometric good $U$-quotient $X \to X/U$, and $X/U$ is affine over $Z$.
        \item If the action of $U$ on $X$ is free then $X \to X/U$ is a principal $U$-bundle.
        \item There is an isomorphism of algebraic stacks $\mathcal{X}^{\mathrm{rig}} \cong [(X/U)/L]$.
    \end{enumerate}
\end{prop}

\begin{proof}
    We once again consider base changing with respect to the quotient morphism $Z \to \mathcal{Z}$. This yields a Cartesian diagram
    \[\begin{tikzcd}[ampersand replacement=\&]
	{\mathcal{I}_{[X/U]}} \&\& {\mathcal{I}_{\mathrm{gr}}} \\
	\\
	{[X/U]} \&\& {\mathcal{X}} \\
	\\
	{\mathcal{X}^{\mathrm{rig}}_Z} \&\& {\mathcal{X}^{\mathrm{rig}}} \\
	\\
	Z \&\& {\mathcal{Z}}
	\arrow[from=1-1, to=1-3]
	\arrow[from=1-1, to=3-1]
	\arrow["\lrcorner"{anchor=center, pos=0.125}, draw=none, from=1-1, to=3-3]
	\arrow[from=1-3, to=3-3]
	\arrow[from=3-1, to=3-3]
	\arrow[from=3-1, to=5-1]
	\arrow["\lrcorner"{anchor=center, pos=0.125}, draw=none, from=3-1, to=5-3]
	\arrow["\rho", from=3-3, to=5-3]
	\arrow[from=5-1, to=5-3]
	\arrow[from=5-1, to=7-1]
	\arrow["\lrcorner"{anchor=center, pos=0.125}, draw=none, from=5-1, to=7-3]
	\arrow["{\overline{\mathrm{gr}}}", from=5-3, to=7-3]
	\arrow[from=7-1, to=7-3]
    \end{tikzcd}\]
    where $\mathcal{X}_Z^{\mathrm{rig}} := \mathcal{X}^{\mathrm{rig}} \times_{\mathcal{Z}} Z$. 
    
    Since $\overline{\mathrm{gr}}$ is an affine morphism, the morphism $\mathcal{X}_Z^{\mathrm{rig}} \to Z$ is also affine. In particular, $\mathcal{X}_Z^{\mathrm{rig}}$ is a scheme. The morphism $[X/U] \to \mathcal{X}_Z^{\mathrm{rig}}$ is then a gerbe whose base is a scheme, so there exists an isomorphism $\mathcal{X}_Z^{\mathrm{rig}} \cong [X/U] \!\! \fatslash \mathcal{I}_{[X/U]}$. It follows from the universal property of rigidification that the composite $X \to [X/U] \to \mathcal{X}_Z^{\mathrm{rig}}$ is a categorical $U$-quotient, whose geometric fibers are in bijection with the $U$-orbits in $X$. Since $X \to Z$ and $X/U \to Z$ are both affine morphisms, by \cite[Tag 01SG]{stacks-project} the quotient morphism $q : X \to X/U$ is also affine. Finally, since $[X/U] \to \mathcal{X}_Z^{\mathrm{rig}} \cong X/U$ is a gerbe, the adjunction map from $\mathcal{O}_{X/U}$ to the pushforward of $\mathcal{O}_{[X/U]}$ is an isomorphism (cf. proof of Lemma \ref{lem:alt-desc-of-relative-proj}). In other words, the adjunction map $\mathcal{O}_{X/U} \to q_{\ast} \mathcal{O}_X$ induces an isomorphism $\mathcal{O}_{X/U} \cong (q_{\ast} \mathcal{O}_X)^U$ of $\mathcal{O}_{X/U}$-modules. Consequently, $q$ is a geometric good $U$-quotient, proving statement $(1)$.

    For statement $(2)$, if the $U$-action on $X$ is free then $\mathcal{I}_{[X/U]}$ is isomorphic to the trivial group stack over $[X/U]$, so the morphism $[X/U] \to \mathcal{X}_Z^{\mathrm{rig}} \cong X/U$ is an isomorphism; since $X \to [X/U]$ is always a principal $U$-bundle, statement $(2)$ follows.

    For the final statement, from the Cartesian diagram
    \[\begin{tikzcd}[ampersand replacement=\&]
	{X/U} \&\& {\mathcal{X}^{\mathrm{rig}}} \\
	\\
	Z \&\& {\mathcal{Z} = [Z/L]}
	\arrow[from=1-1, to=1-3]
	\arrow[from=1-1, to=3-1]
	\arrow["\lrcorner"{anchor=center, pos=0.125}, draw=none, from=1-1, to=3-3]
	\arrow["{\overline{\mathrm{gr}}}", from=1-3, to=3-3]
	\arrow[from=3-1, to=3-3]
    \end{tikzcd}\]
    there is an isomorphism $\mathcal{X}^{\mathrm{rig}} \cong [(X/U)/L]$, with $\overline{\mathrm{gr}} : \mathcal{X}^{\mathrm{rig}} \to \mathcal{Z}$ coinciding with the morphism on quotient stacks induced by the $L$-equivariant affine morphism $X/U \to Z$.
\end{proof}

\subsubsection{Relatively Projective Moduli Spaces from Characters of the Levi} \label{sec:moduli-spaces-characters-levi}

Continue to assume that the $U$-stabilizer group scheme $\mathrm{Stab}_{U,X} \to X$ is flat. From now on, assume further that under the induced action of $L$ on $Z$, the quotient stack $[Z/L]$ admits a good moduli space $\pi : [Z/L] \to Z/\!\!/L$. We now turn to showing that we obtain good moduli spaces for appropriate open substacks of $\mathcal{X}^{\mathrm{rig}} \cong [(X/U)/L]$, coming from characters of the group $L$. 

\begin{lem} \label{lem:quotient-stack-gms-no-character}
    The composition $[(X/U)/L] \stackrel{\overline{\mathrm{gr}}}{\to} [Z/L] \stackrel{\pi}{\to} Z /\!\!/ L$ is a good moduli space (in particular, $X/U \to Z /\!\!/ L$ is a good $L$-quotient). Furthermore, the composition $X \to X/U \to Z/\!\!/L$ is a good $H$-quotient.
\end{lem}

\begin{proof}
    Since we have $\mathcal{X}^{\mathrm{rig}} \cong [(X/U)/L]$ by Proposition \ref{prop:existence-geometric-U-quotient}, the fact that the composition $[(X/U)/L] \stackrel{\overline{\mathrm{gr}}}{\to} [Z/L] \stackrel{\pi}{\to} Z /\!\!/ L$ is a good moduli space follows from the second statement of Theorem \ref{thm:main-theorem} (i.e. Proposition \ref{prop:gms-of-Xrig}). In particular, $X/U \to Z /\!\!/ L$ is a good $L$-quotient; since the composition of a good $U$-quotient with a good $L$-quotient is a good $H$-quotient, the  resulting composite morphism $X \to Z /\!\!/ L$ is a good $H$-quotient.
\end{proof}

The invertible sheaves on $\mathcal{X}^{\mathrm{rig}} \cong [(X/U)/L]$ we consider are defined as follows.

\begin{defi} \label{def:line-bundle-coming-from-character}
    Let $\chi : L \to \mathbb{G}_{m,k}$ be a character of $L$. Define $\mathcal{O}_{B_kL}(\chi)$ to be the invertible sheaf on the classifying stack $B_kL = [\Spec k/L]$ whose pullback along $\Spec k \to B_kL$ has on the total space $\mathbb{A}_k^1$ the $L$-action $l \cdot z = \chi(l)^{-1}z$.

    We define $\mathcal{O}_{[(X/U)/L]}(\chi)$ to be the pullback of $\mathcal{O}_{B_kL}(\chi)$ along the composition $\mathcal{X}^{\mathrm{rig}} \cong [(X/U)/L] \stackrel{\overline{\mathrm{gr}}}{\to} \mathcal{Z} = [Z/L] \to B_kL$, and define $\mathcal{O}_{[X/H]}(\chi)$ to the pullback of $\mathcal{O}_{[(X/U)/L]}(\chi)$ along the induced morphism $[X/H] \to [(X/U)/L]$.\footnote{Note that after identifying $\mathcal{X}^{\mathrm{rig}}$ with $[(X/U)/L]$, the morphism $[X/H] \to [(X/U)/L]$ is identified with the gerbe $\rho : \mathcal{X} \to \mathcal{X}^{\mathrm{rig}}$.}
\end{defi}

Given an algebraically closed field $K \supset k$ and a morphism $f : \Theta_K \to BL$ (i.e. a $K$-point of $\mathrm{Filt}(BL)$), by Example \ref{ex:theta-stability-for-quotient-stacks} there is an associated cocharacter $\eta : \mathbb{G}_{m,K} \to L_K$, with $\eta$ uniquely determined up to conjugation by the parabolic subgroup $P_{L_K}(\eta)$; we then have the formula
$$ \mathrm{wt}_0(f^{\ast} \mathcal{O}_{B_kL}(\chi)) = -\langle \chi, \eta \rangle, $$
where $\langle -,- \rangle$ denotes the natural pairing between characters and cocharacters of $L$. 

\begin{defi} \label{def:quotient-stack-semistable-locus}
    Fix once and for all an identification $\mathcal{X}^{\mathrm{rig}} = [(X/U)/L]$ of algebraic stacks. Let $\chi : L \to \mathbb{G}_{m,k}$ be a character of $L$. We define $(X/U)^{\chi-(s)s}$ (resp. $(X/U)^{\chi-ss}$) to be the unique $L$-invariant open subscheme of $X/U$ such that, under the identification $\mathcal{X}^{\mathrm{rig}} = [(X/U)/L]$, we have
    $$ (\mathcal{X}^{\mathrm{rig}})^{(s)s}(\mathcal{O}_{[(X/U)/L]}(\chi)) = [(X/U)/L]^{(s)s}(\mathcal{O}_{[(X/U)/L]}(\chi)) = [(X/U)^{\chi-(s)s}/L]. $$
    We define $X^{\chi-(s)s}$ to be the preimage of $(X/U)^{\chi-(s)s}$ under the geometric good $U$-quotient $X \to X/U$, so that we have 
    $$ \mathcal{X}^{(s)s}(\mathcal{O}_{[X/H]}(\chi)) := \rho^{-1}((\mathcal{X}^{\mathrm{rig}})^{(s)s}(\mathcal{O}_{[(X/U)/L]}(\chi))) = [X^{\chi-(s)s}/H] $$
\end{defi}

\begin{prop} \label{prop:quotient-stack-gms-character}
    The quotient stack $[(X/U)^{\chi-ss}/L]$ admits a good moduli space of the form
    $$ \pi_{\chi} : [(X/U)^{\chi-ss}/L] \to X /\!\!/_{\chi} H := \underline{\mathrm{Proj}}_{Z /\!\!/L} \bigoplus_{m \geq 0} (\pi \circ \mathrm{gr})_{\ast} (\mathcal{O}_{[X/H]}(\chi^m)), $$
    whose restriction to $[(X/U)^{\chi-s}/L]$ is a tame moduli space onto its image. In particular, the induced morphism $X^{\chi-ss} \to X /\!\!/_{\chi} H$ is a good $H$-quotient, whose restriction to $X^{\chi-s}$ is a geometric good $H$-quotient onto its image. Moreover, if $X$ is a $k$-scheme of finite type and if $Z /\!\!/ L$ is a scheme then the natural morphism $X /\!\!/_{\chi} H \to Z /\!\!/ L$ is projective.
\end{prop}

\begin{proof}
    By Lemma \ref{lem:alt-desc-of-relative-proj}, we have
    $$ X /\!\!/_{\chi} H := \underline{\mathrm{Proj}}_{Z /\!\!/L} \bigoplus_{m \geq 0} (\pi \circ \mathrm{gr})_{\ast} (\mathcal{O}_{[X/H]}(\chi^m)) = \underline{\mathrm{Proj}}_{Z /\!\!/L} \bigoplus_{m \geq 0} (\pi \circ \overline{\mathrm{gr}})_{\ast} \mathcal{O}_{[(X/U)/L]}(\chi^m). $$

    Consequently, the result that $\pi_{\chi}$ is a good moduli space whose restriction to $[(X/U)^{\chi-s}/L]$ is a tame moduli space onto its image follows immediately from the second statement of Theorem \ref{thm:main-theorem} (i.e. Proposition \ref{prop:gms-from-line-bundle}), after taking $\mathcal{L} := \mathcal{O}_{[(X/U)/L]}(\chi)$. Since $X \to X/U$ is a geometric good $U$-quotient, the same is true for $X^{\chi-ss} \to (X/U)^{\chi-ss}$ and $X^{\chi-s} \to (X/U)^{\chi-s}$; it follows that the induced morphism $X^{\chi-ss} \to X /\!\!/_{\chi} H$ is the composition of a good $U$-quotient with a good $L$-quotient, whence a good $H$-quotient, whose restriction to $X^{\chi-s}$ is a geometric good $H$-quotient onto its image. Finally, the projectivity of the morphism $X /\!\!/_{\chi} H \to Z /\!\!/ L$ in the case where $X$ is of finite type and where $Z /\!\!/ L$ is a scheme follows from Lemmas \ref{lem:when-moduli-space-is-projective} and \ref{lem:alt-desc-of-relative-proj}.
\end{proof}

\begin{rmk} \label{rmk:H-quotient-is-spectrum-of-semiinvariants}
    Denote by $\psi : X \to Z /\!\!/ L$ the good $H$-quotient arising from Lemma \ref{lem:quotient-stack-gms-no-character}. The invertible sheaf $\mathcal{O}_{[X/H]}(\chi)$ can be identified with the $H$-linearization $\mathcal{O}_X(\chi)$ obtained by twisting the trivial sheaf $\mathcal{O}_X$ by the character $\chi$ (here we view $\chi$ as a character of $H$ via the composition $H \to L \stackrel{\chi}{\to} \mathbb{G}_{m,k}$). Under this identification, we may write
    $$ X /\!\!/_{\chi} H = \underline{\mathrm{Proj}}_{Z /\!\!/ L} \bigoplus_{m \geq 0} (\psi_{\ast} \mathcal{O}_X(\chi^m))^H $$
    as the relative projective spectrum of the $H$-invariants of the graded $\mathcal{O}_{Z /\!\!/ L}$-algebra $\bigoplus_{m \geq 0} \psi_{\ast} \mathcal{O}_X(\chi^m)$. In other words, $X /\!\!/_{\chi} H$ is the relative projective spectrum of the algebra of $\chi$-twisted semi-invariants of $\psi_{\ast} \mathcal{O}_X$, in the sense of \cite[Definition 2.3]{HHJRelativeNRGIT}. If $X$ is of finite type over $k$ and if $Z /\!\!/ L$ is a scheme, it also follows from Lemmas \ref{lem:when-moduli-space-is-projective} and \ref{lem:alt-desc-of-relative-proj} that $\bigoplus_{m \geq 0} (\psi_{\ast} \mathcal{O}_X(\chi^m))^H$ is an $\mathcal{O}_{Z/\!\!/L}$-algebra of finite type.
\end{rmk}

\begin{rmk}
    In the case where $k$ is of characteristic $0$ and where $X$ is a $k$-scheme of finite type, one may easily check (e.g. using the characterization arising from \cite[Equation 5.1]{HHJRelativeNRGIT}) that the loci $X^{\chi-(s)s}$, as defined above, coincide with the loci $X^{(s)s}(X \to Z, Z \to Z/\!\!/L, \mathcal{O}_X(\chi))$, as defined in \cite[Definition 2.26]{HHJRelativeNRGIT}.
\end{rmk}

\subsubsection{The Hilbert--Mumford Criterion} \label{sec:nrgit-hm-criterion}

By Theorem \ref{thm:main-theorem}, the $K$-points (where $K \supset k$ is an algebraically closed field) of the open substacks $[(X/U)/L]^{(s)s}(\mathcal{O}_{[(X/U)/L]}(\chi)) = [(X/U)^{\chi-(s)s}/L]$ of $[(X/U)/L]$ are characterized in terms of stability with respect to $\mathcal{O}_{[(X/U)/L]}(\chi)$. Via Example \ref{ex:theta-stability-for-quotient-stacks}, this corresponds to a Hilbert--Mumford characterization of the $K$-points of $(X/U)^{\chi-(s)s}$ in terms of cocharacters $\eta : \mathbb{G}_{m,K} \to L_K$ for which the corresponding limit exists in $(X/U)_K$. Namely, we have $\overline{x} \in (X/U)^{\chi-(s)s}(K)$ if and only if for all cocharacters $\eta : \mathbb{G}_{m,K} \to L$ for which the limit $\lim_{t \to 0} \eta(t) \cdot \overline{x}$ exists in $(X/U)_K$, $\langle \chi, \eta \rangle > (\geq) \ 0$. In fact, we can be more precise about this.

\begin{prop} \label{prop:nrgit-hm-part-1}
    Let $K \supset k$ be an algebraically closed field, and let $\overline{x} \in (X/U)(K)$. Then there is unique polystable orbit contained in the $L_K$-orbit closure of $\overline{x}$ in $(X/U)_K$. Moreover, if $L_{\overline{x}} \subset L_K$ is the stabilizer associated to this polystable orbit, then $\overline{x} \in (X/U)^{\chi-(s)s}(K)$ if and only if for all non-trivial cocharacters $\eta : \mathbb{G}_{m,K} \to L_{\overline{x}}$ for which the limit $\lim_{t \to 0} \eta(t) \cdot \overline{x}$ exists in $(X/U)_K$, $\langle \chi, \eta \rangle > (\geq) \ 0$.
\end{prop}

\begin{proof}
    We make the following observations:
    \begin{itemize}
        \item Any filtration $\Theta_K \to \mathcal{X}^{\mathrm{rig}} = [(X/U)/L]$ must factor through a $K$-fiber of the good moduli space $\mathcal{X}^{\mathrm{rig}} \to Z /\!\!/L$ (cf. proof of Proposition \ref{prop:semistable-loci-are-theta-semistable-loci}).

        \item Let $\overline{z} \in (Z/\!\!/L)(K)$ be the image of $\overline{x}$. By Lemma \ref{lem:stability_is_fiberwise}, there are equalities
        $$ [(X/U)^{\chi-(s)s}/L]_{\overline{z}} = (\mathcal{X}^{\mathrm{rig}})^{(s)s}(\mathcal{O}_{[(X/U)/L]}(\chi))_{\overline{z}} = (\mathcal{X}_{\overline{z}}^{\mathrm{rig}})^{(s)s}(\mathcal{O}_{[(X/U)/L]}(\chi)_{\overline{z}}). $$

        \item By Lemma \ref{lem:fiber-local-structure-theorem}, the fiber $\mathcal{X}_{\overline{z}}^{\mathrm{rig}}$ has a unique closed point. In other words, there is a unique polystable orbit contained in the $L_K$-orbit closure of $\overline{x}$ in $(X/U)_K$. Moreover, if $L_{\overline{x}} \subset L_K$ is the associated (linearly reductive) stabilizer group, then there exists an affine $K$-scheme $\Spec A$ of finite type acted on by $L_{\overline{x}}$ such that $\mathcal{X}_{\overline{z}}^{\mathrm{rig}} \cong [\Spec A/L_{\overline{x}}]$, with this isomorphism being compatible with the inclusion $L_{\overline{x}} \subset L$. We fix the isomorphism $\mathcal{X}_{\overline{z}}^{\mathrm{rig}} \cong [\Spec A/L_{\overline{x}}]$.

        \item Let $M_{\chi}$ be the $L_{\overline{x}}$-linearization on $\Spec A$ obtained by pulling back $\mathcal{O}_{B_kL}(\chi)$ along the induced morphism $[\Spec A/L_{\overline{x}}] \cong \mathcal{X}_{\overline{z}}^{\mathrm{rig}} \to B_KL_K \to B_kL$. Then, as in Section \ref{sec:ss-loci-as-theta-ss-loci}, the isomorphism $\mathcal{X}_{\overline{z}}^{\mathrm{rig}} \cong [\Spec A/L_{\overline{x}}]$ induces isomorphisms
        $$ [(X/U)^{\chi-(s)s}/L]_{\overline{z}} = (\mathcal{X}_{\overline{z}}^{\mathrm{rig}})^{(s)s}(\mathcal{O}_{[(X/U)/L]}(\chi)_{\overline{z}}) \cong [(\Spec A)^{(s)s}(M_{\chi})/L_{\overline{x}}]. $$

        \item By Proposition \ref{prop:ghh_hilbert_mumford}, a $K$-point $p$ of $\Spec A$ is a $K$-point of $(\Spec A)^{(s)s}(M_{\chi})$ if and only if for all non-trivial cocharacters $\eta : \mathbb{G}_{m,K} \to L_{\overline{x}}$ for which the limit $\lim_{t \to 0} \eta(t) \cdot p$ exists in $\Spec A$, $\langle \chi, \eta \rangle > (\geq) \ 0$.
    \end{itemize}

    In conjunction with Example \ref{ex:theta-stability-for-quotient-stacks}, these observations imply that the $K$-point $\overline{x}$ of $X/U$ lying over $\overline{z}$ is Hilbert--Mumford (semi)stable with respect to cocharacters $\eta : \mathbb{G}_{m,K} \to L_K$ if and only if it is Hilbert--Mumford (semi)stable with respect to cocharacters $\eta : \mathbb{G}_{m,K} \to L_{\overline{x}}$. This proves the result.
\end{proof}

\begin{prop} \label{prop:nrgit-hm-part-2}
    Let $K \supset k$ be an algebraically closed field, and let $x \in X(K)$, with image $\overline{x} \in (X/U)(K)$. Let $L_{\overline{x}} \subset L_K$ be the stabilizer for the unique polystable orbit contained in the $L_K$-orbit closure of $\overline{x}$ in $(X/U)_K$. Then $x \in X(K)$ is contained in $X^{\chi-(s)s}(K)$ if and only if for all cocharacters $\eta : \mathbb{G}_{m,K} \to U_K \rtimes L_{\overline{x}}$ for which the limit $\lim_{t \to 0} \eta(t) \cdot x$ exists in $X_K$, $\langle \chi, \eta \rangle > (\geq) \ 0$.
\end{prop}

\begin{proof}
    By Proposition \ref{prop:semistable_locus_upstairs}, the $K$-point $\overline{x}$ of $\mathcal{X} = [X/H]$ corresponding to $x \in X(K)$ is a $K$-point of $\mathcal{X}^{(s)s}(\mathcal{O}_{[X/H]}(\chi)) = [X^{\chi-(s)s}/H]$ if and only if $\overline{x}$ is $\Theta_K$-(semi)stable, i.e. if and only if $x$ is Hilbert--Mumford (semi)stable with respect to cocharacters $\eta : \mathbb{G}_{m,K} \to H$.
    
    On the other hand, since the formation of $(\mathcal{X}^{\mathrm{rig}})^{(s)s}(\mathcal{O}_{[(X/U)/L]}(\chi))$ commutes with passing to fibers over $K$-points of $Z/\!\!/L$, for any $K$-point $\overline{z} \in (Z /\!\!/ L)(K)$ there are equalities
    $$ [X^{\chi-(s)s}/H]_{\overline{z}} = \mathcal{X}^{(s)s}(\mathcal{O}_{[X/H]}(\chi))_{\overline{z}} =\rho^{-1}((\mathcal{X}_{\overline{z}}^{\mathrm{rig}})^{(s)s}(\mathcal{O}_{[(X/U)/L]}(\chi)_{\overline{z}})). $$

    Take $\overline{z}$ to be the image of $\overline{x}$ in $(Z /\!\!/ L)(K)$ and fix the isomorphism $\mathcal{X}_{\overline{z}}^{\mathrm{rig}} \cong [\Spec A/L_{\overline{x}}]$ considered in the proof of Proposition \ref{prop:nrgit-hm-part-1}. Write $(\mathcal{X}_{\overline{z}}^{\mathrm{rig}})^{(s)s}(\mathcal{O}_{[(X/U)/L]}(\chi)_{\overline{z}}) = [(\Spec A)^{(s)s}(M_{\chi})/L_{\overline{x}}]$, again as in the proof of Proposition \ref{prop:nrgit-hm-part-1}. From the diagram
    \[\begin{tikzcd}[ampersand replacement=\&]
	\&\& X_K \&\& \\
	\\
	\&\& {[X_K/U_K]} \&\& {\mathcal{X} = [X/H]} \\
	\\
	{\Spec A} \&\& {(X/U)_K} \&\& {\mathcal{X}^{\mathrm{rig}} = [(X/U)/L]} \\
	\\
	\&\& {\mathcal{X}_{\overline{z}}^{\mathrm{rig}} = [\Spec A/L_{\overline{x}}]}
	\arrow[from=1-3, to=3-3]
	\arrow[from=3-3, to=3-5]
	\arrow[from=3-3, to=5-3]
	\arrow["\lrcorner"{anchor=center, pos=0.125}, draw=none, from=3-3, to=5-5]
	\arrow["\rho", from=3-5, to=5-5]
	\arrow[from=5-1, to=5-3]
	\arrow[curve={height=18pt}, from=5-1, to=7-3]
	\arrow[from=5-3, to=5-5]
	\arrow[curve={height=18pt}, from=7-3, to=5-5]
    \end{tikzcd}\]
    we may identify the preimage $\rho^{-1}(\mathcal{X}_{\overline{z}}^{\mathrm{rig}})$ with the quotient stack $[X_A/U_K \rtimes L_{\overline{x}}]$, where $X_A := X_K \times_{(X/U)_K} \Spec A$. Denoting by $N_{\chi}$ the pullback of the $L_{\overline{x}}$-linearized invertible sheaf $M_{\chi}$ along the projection $X_A \to \Spec A$, by viewing $\chi$ as a character of $H$ in the usual way the invertible sheaf $N_{\chi}$ inherits the structure of a $U_K \rtimes L_{\overline{x}}$-linearization. Since the $K$-points of $(\Spec A)^{(s)s}(M_{\chi})$ are characterized in terms of the Hilbert--Mumford criterion (with respect to $M_{\chi}$ and cocharacters of $L_{\overline{x}}$), by Example \ref{ex:theta-stability-for-quotient-stacks} and Proposition \ref{prop:semistable_locus_upstairs} the $K$-points of $\rho^{-1}([(\Spec A)^{(s)s}(M_{\chi})/L_{\overline{x}}])$ are also characterized by the Hilbert--Mumford criterion, albeit with respect to $N_{\chi}$ and cocharacters of $U_K \rtimes L_{\overline{x}}$. As such, the isomorphism $[X/H]_{\overline{z}} = \rho^{-1}(\mathcal{X}_{\overline{z}}^{\mathrm{rig}}) \cong [X_A/U_K \rtimes L_{\overline{x}}]$ induces an isomorphism
    $$ [X^{\chi-(s)s}/H]_{\overline{z}} = \rho^{-1}([(\Spec A)^{(s)s}(M_{\chi})/L_{\overline{x}}]) \cong [X_A^{(s)s}(N_{\chi})/U_K \rtimes L_{\overline{x}}] $$
    of $[X^{\chi-(s)s}/H]_{\overline{z}}$ with the quotient of the (open) locus of $X_A$ consisting of points which are Hilbert--Mumford (semi)stable. As in the proof of Proposition \ref{prop:nrgit-hm-part-1}, it follows that the $K$-point $x \in X(K)$ is a $K$-point of $X^{\chi-(s)s}$ if and only if $x$ is Hilbert--Mumford (semi)stable with respect to cocharacters $\eta : \mathbb{G}_{m,K} \to U_K \rtimes L_{\overline{x}}$. This yields the result.
\end{proof}

\begin{proof}[Proof of Theorem \ref{thm:nrgit-intro}]
    Theorem \ref{thm:nrgit-intro} follows from combining Lemma \ref{lem:quotient-stack-weak-theta-action-cocharacter}, Proposition \ref{prop:existence-geometric-U-quotient}, Lemma \ref{lem:quotient-stack-gms-no-character} and Propositions \ref{prop:quotient-stack-gms-character}, \ref{prop:nrgit-hm-part-1} and \ref{prop:nrgit-hm-part-2}.
\end{proof}

\subsection{Existence of Non-Reductive GIT Quotients for Externally Graded Groups}

We now turn to proving Theorem \ref{thm:nrgit-external-intro} using Theorem \ref{thm:main-theorem}. This proof proceeds along similar lines to the proof of Theorem \ref{thm:nrgit-intro}, albeit with a few modifications to take into account that the group $H$ for which we wish to quotient by is externally graded instead of internally graded.

As in the statement of Theorem \ref{thm:nrgit-external-intro}, let $X$ be a separated $k$-scheme locally of finite type, acted on by a smooth linear algebraic group $H = U \rtimes L$ over $k$ with a split maximal torus. Assume that there exists a cocharacter $\lambda : \mathbb{G}_{m,k} \to \mathrm{Aut}(H)$ such that $\lambda(\mathbb{G}_{m,k})$ commutes with $L$, and such that for the group $\widehat{H} := H \rtimes_{\lambda} \mathbb{G}_{m,k}$, we have $P_{\widehat{H}}(\widehat{\lambda}) = \widehat{H}$, $U_{\widehat{H}}(\widehat{\lambda}) = U$ and $L_{\widehat{H}}(\widehat{\lambda}) = L \rtimes_{\lambda} \mathbb{G}_{m,k} \cong L \times \mathbb{G}_{m,k}$, where $\widehat{\lambda} : \mathbb{G}_{m,k} \to \widehat{H}$ corresponds to the external $\mathbb{G}_{m,k}$. Assume further that the action of $H$ on $X$ extends to an action of $\widehat{H}$, that the action of the external $\mathbb{G}_{m,k} \subset \widehat{H}$ extends to an action of the monoid $\mathbb{A}_k^1$, and that the $U$-stabilizer group scheme $\mathrm{Stab}_{U,X} \to X$ is flat. Let $Z := X^{\widehat{\lambda},0}$ be the fixed locus for the action of the external $\mathbb{G}_{m,k} \subset \widehat{H}$.

\begin{proof}[Proof of Theorem \ref{thm:nrgit-external-intro}]
    Set $\widehat{L} := L \rtimes_{\lambda} \mathbb{G}_{m,k}$, $\mathcal{X} := [X/H]$, $\widehat{\mathcal{X}} := [X/\widehat{H}]$, $\mathcal{Z} := [Z/L]$ and $\widehat{\mathcal{Z}} := [Z/\widehat{L}]$. This is a quotient stack associated to the action of the internally graded group $\widehat{H}$. In particular, by Lemma \ref{lem:quotient-stack-weak-theta-action-cocharacter} the cocharacter $\widehat{\lambda}$ induces a weak $\Theta_k$-action on $\widehat{\mathcal{X}}$ with center $\widehat{\mathcal{Z}}$ and retraction $\mathrm{gr} : \widehat{\mathcal{X}} \to \widehat{\mathcal{Z}}$ induced by the $\widehat{\lambda}$-retraction $X \to Z$. Since $\mathrm{Stab}_{U,X} \to X$ is flat, the relative inertia $\mathcal{I}_{\widehat{\mathcal{X}}/\widehat{\mathcal{Z}}} \to \widehat{\mathcal{X}}$ is flat and of finite presentation by Lemma \ref{lem:flat-stab-flat-relative-inertia}; since the pullback of $\mathcal{I}_{\widehat{\mathcal{X}}/\widehat{\mathcal{Z}}} \to \widehat{\mathcal{X}}$ along the natural morphism $\mathcal{X} \to \widehat{\mathcal{X}}$ induced by the inclusion $H \subset \widehat{H}$ coincides with $\mathcal{I}_{\mathcal{X}/\mathcal{Z}}$, the relative inertia $\mathcal{I}_{\mathcal{X}/\mathcal{Z}} \to \mathcal{X}$ is also flat and of finite presentation. 
    
    Theorem \ref{thm:main-theorem} is then applicable to the weak $\Theta_k$-action on $\widehat{\mathcal{X}}$ arising from $\widehat{\lambda}$. In particular, the rigidifications $\mathcal{X}^{\mathrm{rig}} := \mathcal{X} \!\!\! \fatslash  \mathcal{I}_{\mathcal{X}/\mathcal{Z}}$ and $\widehat{\mathcal{X}}^{\mathrm{rig}} := \widehat{\mathcal{X}} \!\!\! \fatslash  \mathcal{I}_{\widehat{\mathcal{X}}/\widehat{\mathcal{Z}}}$ are well-defined. These fit into a Cartesian diagram
    \[\begin{tikzcd}[ampersand replacement=\&]
	{\mathcal{X} = [X/H]} \&\& {\widehat{\mathcal{X}} = [X/\widehat{H}]} \\
	\\
	{\mathcal{X}^{\mathrm{rig}} = \mathcal{X} \!\!\! \fatslash  \mathcal{I}_{\mathcal{X}/\mathcal{Z}}} \&\& {\widehat{\mathcal{X}}^{\mathrm{rig}} = \widehat{\mathcal{X}} \!\!\! \fatslash  \mathcal{I}_{\widehat{\mathcal{X}}/\widehat{\mathcal{Z}}}} \\
	\\
	{\mathcal{Z} = [Z/L]} \&\& {\widehat{\mathcal{Z}} = [Z/\widehat{L}]}
	\arrow[from=1-1, to=1-3]
	\arrow[from=1-1, to=3-1]
	\arrow["\lrcorner"{anchor=center, pos=0.125}, draw=none, from=1-1, to=3-3]
	\arrow[from=1-3, to=3-3]
	\arrow[from=3-1, to=3-3]
	\arrow[from=3-1, to=5-1]
	\arrow["\lrcorner"{anchor=center, pos=0.125}, draw=none, from=3-1, to=5-3]
	\arrow[from=3-3, to=5-3]
	\arrow[from=5-1, to=5-3]
    \end{tikzcd}\]
    whose horizontal morphisms are all principal $\mathbb{G}_{m,k}$-bundles (in particular, these morphisms are affine). Moreover, the morphisms $\mathcal{X}^{\mathrm{rig}} \to \mathcal{Z}$ and $\widehat{\mathcal{X}}^{\mathrm{rig}} \to \widehat{\mathcal{Z}}$ are both affine. 
    
    For the first statement of Theorem \ref{thm:nrgit-external-intro}, we may apply Proposition \ref{prop:existence-geometric-U-quotient} to $\widehat{\mathcal{X}}$; this yields the existence of a geometric good quotient $X \to X/U$ with $X/U$ affine over $Z$, with the property that $X \to X/U$ is a principal bundle if the $U$-action is free. Moreover, we have $\widehat{\mathcal{X}}^{\mathrm{rig}} \cong [(X/U)/\widehat{L}]$; from the above diagram, the isomorphism $\widehat{\mathcal{X}}^{\mathrm{rig}} \cong [(X/U)/\widehat{L}]$ induces an isomorphism $\mathcal{X}^{\mathrm{rig}} \cong [(X/U)/L]$.

    For the remaining statements of Theorem \ref{thm:nrgit-external-intro}, assume further that there exists a good moduli space $\pi : [Z/L] \to Z /\!\!/ L$. Since the action of the external $\mathbb{G}_{m,k}$ on $Z$ is trivial, the evident morphism $\widehat{\mathcal{Z}} \to \mathcal{Z}$ is a trivial $\mathbb{G}_{m,k}$-gerbe (in particular, a good moduli space morphism); the composition of $\pi$ with this gerbe yields a good moduli space $\widehat{\pi} : \widehat{\mathcal{Z}} \to Z$ (cf. proof of Lemma \ref{lem:gms-Zrig}). By Lemma \ref{lem:quotient-stack-gms-no-character}, the induced morphism $\widehat{\mathcal{X}}^{\mathrm{rig}} \cong [(X/U)/\widehat{L}] \to [Z/\widehat{L}] \stackrel{\widehat{\pi}}{\to} Z /\!\!/ L$ is then a good moduli space. Since $\mathcal{X}^{\mathrm{rig}} \to \widehat{\mathcal{X}}^{\mathrm{rig}}$ is affine, by \cite[Lemma 4.14]{Alper-GMS} there exists a good moduli space $\mathcal{X}^{\mathrm{rig}} \cong [(X/U)/L] \to X /\!\!/ H$, with $X /\!\!/ H$ affine over $Z /\!\!/ L$. In particular, there exists a good $L$-quotient $X/U \to X /\!\!/ H$ and a good $H$-quotient $X \to X /\!\!/ H$.

    The proofs of the remaining statements of Theorem \ref{thm:nrgit-external-intro} follow after making minor modifications to the proofs of the analogous statements of Theorems \ref{thm:main-theorem} and \ref{thm:nrgit-intro}, to take into account the need to work relative to $X /\!\!/ H$ instead of $Z /\!\!/ L$:

    \begin{itemize}
        \item After fixing an identification $\mathcal{X}^{\mathrm{rig}} = [(X/U)/L]$, we modify Definition \ref{def:quotient-stack-semistable-locus} by taking $(\mathcal{X}^{\mathrm{rig}})^{ss}(\mathcal{O}_{[(X/U)/L]}(\chi))$ to be the saturated projective spectrum of $\mathcal{A} := \bigoplus_{m \geq 0} \mathcal{O}_{[(X/U)/L]}(\chi^m)$ relative to the good moduli space $\psi : [(X/U)/L] \to X /\!\!/H$ and setting
        $$ X /\!\!/_{\chi} H := \underline{\mathrm{Proj}}_{X /\!\!/ H} \bigoplus_{m \geq 0} \psi_{\ast}(\mathcal{O}_{[(X/U)/L]}(\chi^m)) \left( = \underline{\mathrm{Proj}}_{X /\!\!/ H} \bigoplus_{m \geq 0} (\psi \circ \rho)_{\ast}(\mathcal{O}_{[X/H]}(\chi^m)) \right). $$
        The proof of Proposition \ref{prop:gms-from-line-bundle} carries over to yield that the natural morphism
        $$ \pi_{\chi} : (\mathcal{X}^{\mathrm{rig}})^{ss}(\mathcal{O}_{[(X/U)/L]}(\chi)) \to X /\!\!/_{\chi} H $$
        is a good moduli space, whose restriction to $(\mathcal{X}^{\mathrm{rig}})^{s}(\mathcal{O}_{[(X/U)/L]}(\chi))$ is a tame moduli space onto its image. If $X$ is of finite type as a $k$-scheme and if $Z /\!\!/ L$ is a scheme, the projectivity of $X /\!\!/_{\chi} H \to X /\!\!/ H$ follows from Lemma \ref{lem:when-moduli-space-is-projective}.

        \item As in Definition \ref{def:quotient-stack-semistable-locus}, $(X/U)^{\chi-s} \subset (X/U)^{\chi-ss}$ are defined such that there are equalities $[(X/U)^{\chi-(s)s}/L] = (\mathcal{X}^{\mathrm{rig}})^{(s)s}(\mathcal{O}_{[(X/U)/L]}(\chi))$, and $X^{\chi-(s)s}$ is defined to be the preimage of $(X/U)^{\chi-(s)s}$ under the geometric good $U$-quotient $X \to X/U$.

        \item By working relative to $X /\!\!/ H$ instead of $Z /\!\!/ L$, and noting that the same argument used to prove Lemma \ref{lem:stability_is_fiberwise} yields that the formation of $(\mathcal{X}^{\mathrm{rig}})^{(s)s}(\mathcal{O}_{[(X/U)/L]}(\chi))$ commutes with passing to fibers over $K$-points of $X/\!\!/H$ (where $K \supset k$ is an algebraically closed field), the proofs of Propositions \ref{prop:nrgit-hm-part-1} and \ref{prop:nrgit-hm-part-2} carry over \emph{mutatis mutandis} to prove the Hilbert--Mumford criterion (statements 2(i) and 2(ii) of Theorem \ref{thm:nrgit-external-intro}).
    \end{itemize}

    This completes the proof of Theorem \ref{thm:nrgit-external-intro}.
\end{proof}

\section{Application to Moduli Spaces of Principal Bundles over Curves} \label{sec:principal-bundles}

In this section, we use Theorem \ref{thm:main-theorem} to construct new examples of quasi-projective moduli spaces of principal bundles over a smooth projective curve with non-reductive structure groups (cf. Theorem \ref{thm:moduli-principal-bundles-intro}).

Throughout this section, we fix an algebraically closed field $k$ of characteristic $0$, and work relative to $B = \Spec k$. To simplify notation, set $\Theta := \Theta_k$, $B\mathbb{G}_m := B\mathbb{G}_{m,k}$ and so on. We fix once and for all a smooth connected projective curve $C$ of genus $p_a(C) \geq 2$.

\begin{rmk}
    For simplicity of exposition, we work with a single curve $C \to \Spec k$. However, the results of this section extend to the case where $C \to \Spec k$ is replaced with a flat projective morphism $C \to S$ of finite presentation over a $k$-scheme $S$ (taking all structure groups to be constant relative to $S$). We leave it to the interested reader to make the necessary modifications.
\end{rmk}

\subsubsection{Additional Conventions, Notation and Terminology}

In this section we adopt the following additional conventions, notation and terminology:

\begin{enumerate}
    \item For an algebraic group $G$, $\mathfrak{X}^{\ast}(G)$ (resp. $\mathfrak{X}_{\ast}(G)$) denotes the set of characters (resp. cocharacters) of $G$, and $\mathfrak{X}^{\ast}(G)_{\Q}$ (resp. $\mathfrak{X}_{\ast}(G)_{\Q}$) denotes the set of rational characters (resp. cocharacters) of $G$. We denote the natural pairing between (rational) characters and cocharacters by $\langle -,-\rangle$. We denote by $G^{\circ}$ the identity component of $G$.

    \item By an \emph{isogeny} of algebraic groups, we mean a surjective homomorphism with finite kernel.

    \item If $G$ is a reductive algebraic group, then by a \emph{Borel pair} of $G$, we mean a pair $(B,T)$, where $T$ is a maximal torus of $G$ and where $B$ is a Borel subgroup of $G$ containing $T$. We assume that the roots associated to this pair are all trivial on $Z(G)^{\circ} \subset T$.

    \item If $S$ is a scheme and if $G \to S$ is an $S$-group scheme, we denote by $\omega_{G/S}$ the pullback of the sheaf of differentials $\Omega^1_{G/S}$ under the identity section $e : S \to G$.

    \item If $X$ is a right $G$-space, if $Y$ is a left $G$-space and if $f : X \to Y$ is a map, we say $f$ is \emph{$G$-equivariant} if $f(x \cdot g) = g^{-1} \cdot f(x)$ for all $x \in X$ and $g \in G$.

    \item All principal bundles are assumed to admit right actions. If $E$ is a principal $G$-bundle and if $X$ is a left $G$-space, we denote by $E \times^G X$ the associated Borel quotient, i.e. the geometric quotient of $E \times X$ by the right action $(e,x) \cdot g = (eg, g^{-1}x)$.

    \item If $E \to S$ is a principal $G$-bundle and if $H \subset G$ is a subgroup, by a \emph{reduction of structure group} of $E$ from $G$ to $H$, we mean a section $\theta : S \to E/H$ of the quotient $E/H \to S$; the corresponding principal $H$-bundle is given by $E_H = E \times_{E/H} S$. If $\phi : G \to K$ is a homomorphism of groups, we denote the corresponding extension of structure group of $E$ by $\phi_{\ast} E$ or by $E \times^G K$.

    \item If $E$ is a principal $G$-bundle, we denote the bundle $E \times^G G$ associated to the conjugation action of $G$ on itself by $c_G(E)$, and denote the bundle $E \times^G \mathfrak{g}$ associated to the adjoint representation of $G$ on $\mathfrak{g} = \mathrm{Lie}(G)$ by $\mathrm{Ad}_G(E)$.

    \item For a morphism $X \to Y$, we denote the sheaf of sections of this morphism by $\Gamma_Y(X)$. Where likely to not cause confusion, we also use the shorthand notation $H^0(X) = H^0(\Gamma_Y(X))$ to denote the space of globally defined sections $Y \to X$ of $X \to Y$.

    \item Given a maximal torus $T \subset G$, the \emph{type} of a principal $G$-bundle $E$ over $C$ (with respect to $T$) is the unique rational cocharacter $d \in \mathfrak{X}_{\ast}(T)_{\Q}$ such that $\langle \chi, d \rangle = \deg(\chi_{\ast}E)$ for all $\chi \in \mathfrak{X}^{\ast}(G)$, where by $\deg(\chi_{\ast}E)$ we mean the degree of the sheaf of sections of the principal $\mathbb{G}_m$-bundle $\chi_{\ast}E$.
\end{enumerate}

\subsection{Preliminaries on Principle Bundles over Curves and their Moduli}

Let $G$ be a connected linear algebraic group over $k$, not assumed to be reductive. We denote by
$$ \mathcal{B}un_G := \underline{\mathrm{Map}}(C, BG) $$
the moduli stack of principal $G$-bundles over $C$; if $S$ is a scheme, a morphism $S \to \mathcal{B}un_G$ is the same data as a morphism $C \times S \to BG$, or equivalently an (fppf locally trivial) principal $G$-bundle $E_S \to C \times S$. The stack $\mathcal{B}un_G$ is algebraic, smooth and locally of finite type, and has an affine diagonal of finite presentation.\footnote{An exposition of this result is given in \cite{Wang-Bundles}.} In particular, by Theorem \ref{thm:algebraicity-filt-grad} we have that $\mathrm{Filt}(\mathcal{B}un_G)$ and $\mathrm{Grad}(\mathcal{B}un_G)$ are well-defined algebraic stacks. For each type $d \in \mathfrak{X}_{\ast}(T)_{\Q}$ (where $T$ is a maximal torus of $G$), there is an open and closed substack $\mathcal{B}un_G^d \subset \mathcal{B}un_G$ parameterizing families of principal $G$-bundles $E_S \to C \times S$ whose fiber over any $k$-point $s \in S(k)$ is of type $d$.

\subsubsection{Automorphisms of Principal Bundles}

Let $S$ be a scheme and let $p : E_S \to C \times S$ be a principal $G$-bundle. By definition, an automorphism of the principal $G$-bundle $E_S$ is a $G$-equivariant morphism $\gamma : E_S \to E_S$ with $p \circ \gamma = p$. Recall that the automorphisms of $E_S$ are in one-to-one correspondence with equivariant morphisms $\xi : E_S \to G$, where $G$ acts on itself on the left by conjugation, by associating to $\xi$ the automorphism
$$ x \mapsto x \cdot \xi(x) $$
of $E_S$. In turn, specifying a $G$-equivariant morphism $E_G \to G$ is equivalent to specifying a section of the associated bundle $c_G(E_S) = E_S \times^G G \to C \times S$. In particular, the relative automorphism group
$$ \underline{\mathrm{Aut}}(E_S) := \mathcal{I}_{\mathcal{B}un_G} \times_{\mathcal{B}un_G} S $$
can be identified with $(\mathrm{pr}_S)_{\ast} \Gamma_{C \times S}(c_G(E_S))$, the pushforward of the sheaf of sections of $c_G(E_S)$ along the projection $C \times S \to S$.

Taking $S = \Spec k$, if $E \to C$ is a principal $G$-bundle then any element $z \in Z(G)$ defines an automorphism of $E$, given by right multiplication by $z$ (or, in terms of $G$-equivariant morphisms $E \to G$, the constant morphism to $z$). This defines an inclusion $Z(G) \subset \mathrm{Aut}(E)$ for which $Z(G)$ is a central subgroup.

In the sequel, we make use the following result concerning the sheaf of differentials $\omega_{\underline{\mathrm{Aut}}(E_S)/S}$ associated to a family of principal bundles over a scheme $S$.

\begin{lem} \label{lem:lie-alg-aut-gp}
    Let $S$ be a scheme, and let $E_S \to C \times S$ be a principal $G$-bundle. Then the $\mathcal{O}_S$-module dual $\omega_{\underline{\mathrm{Aut}}(E_S)/S}^{\vee} = \mathcal{H}om_{\mathcal{O}_S}(\omega_{\underline{\mathrm{Aut}}(E_S)/S}, \mathcal{O}_S)$ of the sheaf of differentials $\omega_{\underline{\mathrm{Aut}}(E_S)/S} = e^{\ast} \Omega_{\underline{\mathrm{Aut}}(E_S)/S}^1$ is isomorphic to $(\mathrm{pr}_S)_{\ast} \Gamma_{C \times S}(\mathrm{Ad}_G(E_S))$.
\end{lem}

\begin{proof}
    Let $S' \to S$ be a flat $S$-scheme, and let $E_{S'} \to C \times S'$ denote the pullback of $E_S$ along $S' \to S$. By \cite[Exposé II, Proposition 3.3 and Scholie 4.11.3]{SGA3}, a section of $\omega^{\lor}_{\underline{\mathrm{Aut}}(E_S)/S}$ over $S'$ is the same as a morphism
    $$ f : S'[\varepsilon] := S' \times_k \Spec k[\varepsilon]/\varepsilon^2 \to \underline{\mathrm{Aut}}(E_{S'}) $$
    for which the restriction of $f$ along $S' \stackrel{\varepsilon = 0}{\hookrightarrow} S'[\varepsilon]$ is the identity section $S' \to \underline{\mathrm{Aut}}(E_{S'})$. In turn, specifying $f$ is equivalent to specifying a $G$-equivariant morphism
    $$ f' : E_{S'[\varepsilon]} = E_{S'}[\varepsilon] \to G $$
    for which the restriction of $f'$ along $E_{S'} \stackrel{\varepsilon = 0}{\hookrightarrow} E_{S'}[\varepsilon]$ is the constant map to the identity element of $G$. By \cite[Exposé II, Proposition 3.3 and Proposition 4.1.1]{SGA3}, specifying $f'$ is equivalent to specifying a $G$-equivariant morphism $E_{S'} \to \mathfrak{g} = \mathrm{Lie}(G)$, or in other words a global section of $(\mathrm{pr}_{S'})_{\ast}\Gamma_{C \times S'}(\mathrm{Ad}_G(E_{S'}))$. By flat base change, there is a natural isomorphism $(\mathrm{pr}_{S'})_{\ast}\Gamma_{C \times S'}(\mathrm{Ad}_G(E_{S'})) \cong (\mathrm{pr}_S)_{\ast} \Gamma_{C \times S}(\mathrm{Ad}_G(E_S)) \otimes_{\mathcal{O}_S} \mathcal{O}_{S'}$, and so $\omega_{\underline{\mathrm{Aut}}(E_S)/S}^{\vee}$ and $(\mathrm{pr}_S)_{\ast} \Gamma_{C \times S}(\mathrm{Ad}_G(E_S))$ are isomorphic as $\mathcal{O}_S$-modules.
\end{proof}

\subsubsection{Filtered and Graded Objects of $\mathcal{B}un_G$} \label{sec:filt-grad-bung}

Fix a maximal torus $T \subset G$, and let $W$ be the corresponding Weyl group. By Theorem \ref{thm:filt-grad-quotient-stack}, there are isomorphisms
$$ \mathrm{Filt}(BG) \cong \bigsqcup_{\lambda \in \mathfrak{X}_{\ast}(T)/W} BP_{G}(\lambda), \quad \mathrm{Grad}(BG) \cong \bigsqcup_{\lambda \in \mathfrak{X}_{\ast}(T)/W} BL_{G}(\lambda). $$

By functoriality of the mapping stack construction, we have
\begin{align*}
    \mathrm{Filt}(\mathcal{B}un_G) &= \underline{\mathrm{Map}}(\Theta, \underline{\mathrm{Map}}(C, BG)) \\ &= \underline{\mathrm{Map}}(C, \underline{\mathrm{Map}}(\Theta, BG)) \\ &= \underline{\mathrm{Map}}(C, \mathrm{Filt}(BG)) \\
    &\cong \bigsqcup_{\lambda \in \mathfrak{X}_{\ast}(T)/W} \mathcal{B}un_{P_G(\lambda)},
\end{align*}

and similarly

\begin{align*}
\mathrm{Grad}(\mathcal{B}un_G) &= \underline{\mathrm{Map}}(C, \mathrm{Grad}(BG)) \\
&\cong \bigsqcup_{\lambda \in \mathfrak{X}_{\ast}(T)/W} \mathcal{B}un_{L_G(\lambda)}
\end{align*}
In particular, if $E$ is a principal $G$-bundle, then specifying a filtration $f : \Theta \to \mathcal{B}un_G$ with $\mathrm{ev}_1(f) \cong E$ is equivalent to specifying a cocharacter $\lambda \in \mathfrak{X}_{\ast}(T)$ (up to equivalence), a principal $P_G(\lambda)$-bundle $E_{P_G(\lambda)}$ and an isomorphism $E \cong E_{P_G(\lambda)} \times^{P_G(\lambda)} G$ of principal $G$-bundles. In other words, filtrations of $E$ are the same as reductions of structure groups to subgroups of $G$ of the form $P_G(\lambda)$.

It similarly follows from Theorem \ref{thm:filt-grad-quotient-stack} and the functoriality of the mapping stack construction that, given a cocharacter $\lambda \in \mathfrak{X}_{\ast}(T)$, there is a Cartesian diagram
\[\begin{tikzcd}[ampersand replacement=\&,cramped]
	{\mathcal{B}un_{P_G(\lambda)}} \&\& {\mathrm{Filt}(\mathcal{B}un_G)} \\
	\\
	{\mathcal{B}un_{L_G(\lambda)}} \&\& {\mathrm{Grad}(\mathcal{B}un_G)}
	\arrow[from=1-1, to=1-3]
	\arrow["{\mathrm{gr}}"', from=1-1, to=3-1]
	\arrow["\lrcorner"{anchor=center, pos=0.125}, draw=none, from=1-1, to=3-3]
	\arrow["{\mathrm{gr}}", from=1-3, to=3-3]
	\arrow["\sigma", curve={height=-18pt}, from=3-1, to=1-1]
	\arrow[from=3-1, to=3-3]
	\arrow["\sigma"', curve={height=18pt}, from=3-3, to=1-3]
\end{tikzcd}\]
whose horizontal arrows are open and closed immersions. The morphism $\mathrm{gr}: \mathcal{B}un_{P_G(\lambda)} \to \mathcal{B}un_{L_G(\lambda)}$ sends a principal $P_G(\lambda)$-bundle $E$ to $E \times^{P_G(\lambda)} L_G(\lambda) = E/U_G(\lambda)$, the extension of structure group of $E$ with respect to the quotient $P_G(\lambda) \to L_G(\lambda)$, and the morphism $\sigma : \mathcal{B}un_{L_G(\lambda)} \to \mathcal{B}un_{P_G(\lambda)}$ sends a principal $L_G(\lambda)$-bundle $F$ to $F \times^{L_G(\lambda)} P_G(\lambda)$, the extension of structure group of $F$ with respect to the inclusion $L_G(\lambda) \hookrightarrow P_G(\lambda)$.

\subsubsection{Ramanathan Semistability and Existence of Good Moduli Spaces}

Suppose in addition that the group $G$ is reductive. We recall the notion of semistability for principal $G$-bundles of Ramanathan \cite{Ramanathan-Riem-Surface}, in terms of reductions of structure groups to parabolic subgroups of $G$.

Fix a Borel pair $(B, T)$ of $G$. Let $\mathfrak{g}$ (resp. $\mathfrak{t}$) denote the Lie algebra of $G$ (resp. $T$), and consider the associated weight space decomposition
$$ \mathfrak{g} = \mathfrak{t} \oplus \bigoplus_{\alpha \in \Phi} \mathfrak{g}_{\alpha}. $$
Let $\Phi^+ \subset \Phi$ be the set of positive roots and let $\Delta \subset \Phi^+$ be the set of simple roots determined by $(B,T)$. Recall that the parabolic subgroups of $G$ which contain $B$ are in one-to-one correspondence with the subsets of $\Delta$.

\begin{defi}
    Let $P \subset G$ be a proper parabolic subgroup. A character $\chi : P \to \mathbb{G}_m$ of $P$ is said to be \emph{dominant} if $\chi$ is trivial on $Z(G)^{\circ}$ and if there exists a Borel pair $(B,T)$ with $B \subset P$ such that, if $I_P \subset \Delta$ is the corresponding subset of simple roots, then:
    \begin{enumerate}
        \item $\langle \chi, \beta^{\vee} \rangle = 0$ for all $\beta \in I_P$, where $\beta^{\vee}$ is the coroot associated to $\beta$; and
        \item $\langle \chi, \alpha^{\vee} \rangle \geq 0$ for all $\alpha \in \Delta \setminus I_P$.
    \end{enumerate}
\end{defi}

\begin{defi}
    Let $E$ be a principal $G$-bundle over $C$. We say that $E$ is \emph{stable} (resp. \emph{semistable}) if, whenever $E$ admits a reduction of structure group $(P, \theta : C \to E/P)$ to a proper parabolic subgroup $P \subset G$, then for all non-trivial dominant characters $\chi : P \to \mathbb{G}_m$ one has $\deg(\chi_{\ast}E_P) < 0$ (resp. $\deg(\chi_{\ast}E_P) \leq 0$), where $E_P$ is the principal $P$-bundle corresponding to $\theta$.
\end{defi}

\begin{rmk}
By \cite[Proposition 3.2]{Ramanathan-Riem-Surface}, if $E$ is a stable principal $G$-bundle then the quotient $\mathrm{Aut}(E)/Z(G)$ is a finite group. By \cite[Corollary 3.18]{Ramanathan-Moduli-1}, a principal $G$-bundle $E$ is semistable if and only if the sheaf of sections $\Gamma_C(\mathrm{Ad}_G(E))$ of the adjoint bundle $\mathrm{Ad}_G(E) = E \times^G \mathfrak{g}$ is a slope semistable vector bundle.
\end{rmk}

\begin{teo} \label{thm:good-moduli-spaces=semistable-bundles}
Let $G$ be a reductive linear algebraic group over $k$. Fix a type $d \in \mathfrak{X}_{\ast}(T)_{\Q}$, where $T$ is a maximal torus of $G$. Then there exists a finite type open substack $\mathcal{B}un_G^{d, ss} \subset \mathcal{B}un^d_G$ with the property that a morphism from a scheme $S \to \mathcal{B}un_G$, corresponding to the principal $G$-bundle $E_S \to C \times S$, factors through $\mathcal{B}un_G^{d, ss} \subset \mathcal{B}un_G$ if and only if for all $s \in S(k)$, the restricted $G$-bundle $E_s \to C$ is of type $d$ and is semistable. In turn, there exists an open substack $\mathcal{B}un_G^{d, s} \subset \mathcal{B}un_G^{d, ss}$ with the property that a morphism from a scheme $S \to \mathcal{B}un_G$ factors through $\mathcal{B}un_G^{d, s} \subset \mathcal{B}un_G$ if and only if in addition each $E_s$ is stable.

Moreover, the open substack $\mathcal{B}un_G^{d,ss}$ admits a projective good moduli space $\pi : \mathcal{B}un_G^{d,ss} \to M_G^d$. If $M_G^{d,s} = \pi(\mathcal{B}un_G^{d,s})$ then $M_G^{d,s}$ is an open subscheme of $M_G^d$, and $\pi : \mathcal{B}un_G^{d,s} \to M_G^{d,s}$ is a quasi-projective tame moduli space.
\end{teo}

\begin{proof}
    This is the main result of Ramanathan's doctoral thesis, published posthumously as \cite{Ramanathan-Moduli-1, Ramanathan-Moduli-2}.
\end{proof}

The closed $k$-points of the moduli stack $\mathcal{B}un_G^{d,ss}$ and the $k$-points of the good moduli space $M_G^d$ are both in bijection with the set of polystable principal $G$-bundles of type $d$; by definition, a semistable principal $G$-bundle $E \to C$ is \emph{polystable} if there exists a parabolic subgroup $P \subset G$ (with Levi subgroup $L \subset P$) such that $E \cong F \times^L G$ for some stable principal $L$-bundle $F$, and such that for every character $\chi$ of $P$ trivial on $Z(G)^{\circ}$, the line bundle $\chi_{\ast}(F \times^L P)$ has degree $0$. All stable principal $G$-bundles are polystable.

\subsubsection{Canonical Reductions}

Continue to assume that the group $G$ is reductive. We recall the notion of a canonical reduction of a principal $G$-bundle, as originally defined in \cite{Ramanathan-Proceedings}; the canonical reduction of a principal bundle is the principal bundle-theoretic analog of the Harder--Narasimhan filtration of a vector bundle over a curve.

\begin{defi}
    Let $E$ be a principal $G$-bundle over $C$. A \emph{canonical reduction} of $E$ is a reduction of structure group $(P, \theta : C \to E/P)$ of $E$ to a parabolic subgroup $P \subset G$ satisfying the following two properties:
    \begin{enumerate}
        \item If $L$ is the Levi subgroup of $P$, then the principal $L$-bundle $E_P \times^P L$ is semistable, where $E_P$ is the principal $P$-bundle corresponding to $\theta$.
        \item For all non-trivial dominant characters $\chi : P \to \mathbb{G}_m$, $\deg(\chi_{\ast}E_P) > 0$.
    \end{enumerate}
\end{defi}

As shown by Atiyah--Bott \cite{Atiyah-Bott-YM} and Behrend \cite{Behrend-Bundles}, after fixing a Borel pair $(B, T)$ any principal $G$-bundle $E \to G$ admits a unique canonical reduction to a unique parabolic subgroup $P \subset G$ containing $B$. If $E$ is semistable then the canonical reduction of $E$ is $E$ itself; if $E$ is unstable then the group $P$ is a proper parabolic subgroup of $G$. In the case $G = \mathrm{GL}_r$, the canonical reduction of a principal $\mathrm{GL}_r$-bundle is given by the Harder--Narasimhan filtration of the corresponding (locally free) sheaf of sections.

In the following discussion, we fix a Borel pair $(B,T)$, and consider canonical reductions with respect to $(B,T)$.

\begin{defi}[\cite{Gurjar-Nitsure-Stratification}]
    Let $E$ be a principal $G$-bundle over $C$. The \emph{Harder--Narasimhan type} of $E$ is the unique rational cocharacter $\tau_E \in \mathfrak{X}_{\ast}(T)_{\Q}$ such that
    $$ \langle \chi, \tau_E \rangle = \begin{cases} \deg(\chi_{\ast} E_P) & \text{if } \chi \text{ is a character of } P, \\ 0 & \text{if } \chi \in I_P. \end{cases} $$
\end{defi}

As shown by Gurjar--Nitsure \cite{Gurjar-Nitsure-Stratification}, the Harder--Narasimhan type $\tau_E$ of $E$ uniquely determines the parabolic structure group $P$ of the canonical reduction of $E$. Moreover, $\tau_E$ is an element of the closed positive Weyl chamber $\overline{\mathcal{C}}^+$ determined by $(B,T)$.

\begin{teo} \label{thm:gurjar-nitsure-hn-strata}
    Fix $\tau \in \overline{\mathcal{C}}^+$. Then there exists a locally closed substack $\mathcal{B}un_G^{\tau} \subset \mathcal{B}un_G$ such that a morphism from a scheme $S \to \mathcal{B}un_G$, corresponding to the principal $G$-bundle $E_S \to C \times S$, factors through $\mathcal{B}un_G^{\tau} \subset \mathcal{B}un_G$ if and only if there exists a parabolic subgroup $P \subset G$ containing $B$ and a (necessarily unique) reduction of structure group $(P, \theta: C \times S \to E_S/P)$ such that, for all $s \in S(k)$, the reduction of structure group $(P, \theta_s : C \to E_s/P)$ is the canonical reduction of $E_s$.
\end{teo}

\begin{proof}
    This is \cite[Theorem 3.10]{Gurjar-Nitsure-Stratification}.
\end{proof}

\begin{rmk}
    Setting $\mathcal{B}un_G^{d, \tau} := \mathcal{B}un_G^d \cap \mathcal{B}un_G^{\tau}$, the stratification $\mathcal{B}un^d_G = \bigsqcup_{\tau} \mathcal{B}un_G^{d,\tau}$ by Harder--Narasimhan type defines a $\Theta$-stratification (in the sense of \cite{Halpern-Leistner}) of the stack $\mathcal{B}un^d_G$, whose open $\Theta$-stratum is $\mathcal{B}un_G^{d,ss}$ (see also \cite[Lemma 8.2]{AHLH-Existence-GMS}).
\end{rmk}

\subsection{Conditions on the Structure Group}

In order to apply Theorem \ref{thm:main-theorem} to construct moduli spaces of principal bundles with non-reductive structure groups, we first consider moduli stacks of principal bundles $\mathcal{B}un_H$ which admit natural weak $\Theta$-actions and for which the center is isomorphic to a moduli stack of principal bundles with a reductive structure group, whence contains finite type open substacks admitting good moduli spaces (given by the semistable bundles of fixed type). In order to obtain such stacks $\mathcal{B}un_H$, we impose conditions on the structure group $H$ which ensure that the classifying stack $BH$ admits a weak $\Theta$-action with center $BL$ for a reductive group $L$.

\begin{ass} \label{ass:structure-group-bundles}
    This refers to the following set of assumptions on the structure group $H$.

    \begin{enumerate}
        \item $H$ is a connected non-reductive linear algebraic group, with Levi decomposition $H = U \rtimes L$, where $U \neq \{e\}$ is the unipotent radical of $H$ and $L$ is a reductive Levi factor, with a fixed maximal torus $T \subset L$.

        \item There is a distinguished non-trivial cocharacter $\lambda : \mathbb{G}_m \to Z(L)^{\circ}$ such that $H = P_H(\lambda)$, $L = L_H(\lambda)$ and $U = U_H(\lambda)$. In particular, for all $h \in H$ the limit $\lim_{t \to 0} \lambda(t)h\lambda(t)^{-1}$ exists in $H$, and this limit is the identity $e$ if and only if $h \in U$.
    \end{enumerate}
\end{ass}

In Section \ref{sec:stability-adapted-characters}, we seek stability conditions for which the associated notion of stability can be described fully explicitly\footnote{In our case, this will be in terms of the non-existence of reductions of structure groups with respect to the resulting decomposition \eqref{eq:iterated-semidirect-product}. In particular, the order of the $\lambda_i$ in Assumption \ref{ass:structure-group-bundles-strong} matters, since the decomposition \eqref{eq:iterated-semidirect-product} depends on the order of the $\lambda_i$.} and for which all $\Theta$-semistable bundles are $\Theta$-stable, up to filtrations trivialized by composition with the gerbe $\mathcal{B}un_H \to \mathcal{B}un_H \! \! \fatslash Z(H)'$ (where $Z(H)' = Z(L)^{\circ} \cap C_L(U)$), so that the resulting moduli spaces are coarse moduli spaces. Consequently, in Section \ref{sec:stability-adapted-characters} we will impose the following stronger set of assumptions on the structure group $H$.

\begin{ass} \label{ass:structure-group-bundles-strong}
    This refers to the following set of assumptions on the structure group $H$.

    \begin{enumerate}
        \item $H$ is a connected non-reductive linear algebraic group, with Levi decomposition $H = U \rtimes L$, where $U \neq \{e\}$ is the unipotent radical of $H$ and $L$ is a reductive Levi factor, with a fixed maximal torus $T \subset L$.

        \item Set $Z(H)' := Z(L)^{\circ} \cap C_L(U)$, where $C_L(U)$ is the $L$-centralizer of $U$, and $n := \dim Z(L)^{\circ} - \dim Z(H)'$; we assume that $n > 0$.
        
        \item There is a distinguished non-trivial homomorphism $\lambda : \mathbb{G}_m^n \to Z(L)^{\circ}$ satisfying the following properties:
        \begin{enumerate}
            \item[(i)] For each $i$, set $\lambda_i : \mathbb{G}_m \to Z(L)^{\circ}$ to be the restriction of $\lambda$ to the $i$th factor $\mathbb{G}_m \subset \mathbb{G}_m^n$. Then for all $i = 1, \dots, n$, we assume that the limit $\lim_{t \to 0} \lambda_i(t)h\lambda_i(t)^{-1}$ exists in $H$, so that we have an iterated semidirect product decomposition
            $$ H = U_1 \rtimes \left( U_2 \rtimes \left( \cdots \rtimes \left( U_n \rtimes H_n \right) \cdots \right) \right) $$
            defined inductively by $H_0 := H$ and $H_{i-1} := U_i \rtimes H_i$ for all $i = 1, \dots, n$, where $U_i = U_{H_{i-1}}(\lambda_i)$ and $H_i = L_{H_{i-1}}(\lambda_i)$. We further assume that $H_n = L$, so that $H$ is expressible as the iterated semidirect product
            \begin{equation} \label{eq:iterated-semidirect-product}
                H = U_1 \rtimes \left( U_2 \rtimes \left( \cdots \rtimes \left( U_n \rtimes L \right) \cdots \right) \right). \tag{$\dagger$}
        \end{equation}

        \item[(ii)] We assume that the homomorphism $\mathbb{G}_m^n \times Z(H)' \to Z(L)^{\circ}$ given by $(t, l) \mapsto \lambda(t)l$ is an isogeny.
        \end{enumerate}
    \end{enumerate}
\end{ass}

\begin{rmk}
    Assume Assumption \ref{ass:structure-group-bundles-strong} holds for $H$. Consider the cocharacter $\widehat{\lambda} : \mathbb{G}_m \to H$ defined by $\widehat{\lambda}(t) := \prod_{i=1}^n \lambda_i(t)$. We then have $H = P_H(\widehat{\lambda})$, $L = L_H(\widehat{\lambda})$ and $U = U_H(\widehat{\lambda})$, so Assumption \ref{ass:structure-group-bundles} holds for $H$. Moreover, since $L$ is connected and reductive, then $Z(L)^{\circ}$ is a torus, and there is an induced central isogeny $\mathbb{G}_m^n \times Z(H)' \times [L,L] \to L$, where $[L,L]$ is the derived subgroup of $L$ (cf. \cite[Exposé XXII, Proposition 6.2.4]{SGA3}).
\end{rmk}

\subsubsection{Induced Weak $\Theta$-Action on $\mathcal{B}un_H$}

Between here and Section \ref{sec:stability-adapted-characters}, we assume Assumption \ref{ass:structure-group-bundles} applies to $H$.

\begin{lem} \label{lem:weak-theta-action-bunh}
    Assume Assumption \ref{ass:structure-group-bundles} holds for $H$. Then there exists a weak action of $\Theta$ on $\mathcal{B}un_H$ with center $\mathcal{B}un_L$, for which the associated morphisms $\mathrm{gr} : \mathcal{B}un_H \to \mathcal{B}un_L$ and $\sigma : \mathcal{B}un_L \to \mathcal{B}un_H$ are given by extension of structure groups with respect to $H \to L$ and $L \subset H$ respectively.
\end{lem}

\begin{proof}
    Recall that under Assumption \ref{ass:structure-group-bundles} we have $H = P_H(\lambda)$, $L = L_H(\lambda)$ and $U = U_H(\lambda)$. Consequently, as explained in Section \ref{sec:filt-grad-bung}, there is a Cartesian diagram
    \[\begin{tikzcd}[ampersand replacement=\&,cramped]
	{\mathcal{B}un_{H}} \&\& {\mathrm{Filt}(\mathcal{B}un_H)} \\
	\\
	{\mathcal{B}un_{L}} \&\& {\mathrm{Grad}(\mathcal{B}un_H)}
	\arrow[from=1-1, to=1-3]
	\arrow["{\mathrm{gr}}"', from=1-1, to=3-1]
	\arrow["\lrcorner"{anchor=center, pos=0.125}, draw=none, from=1-1, to=3-3]
	\arrow["{\mathrm{gr}}", from=1-3, to=3-3]
	\arrow["\sigma", curve={height=-18pt}, from=3-1, to=1-1]
	\arrow[from=3-1, to=3-3]
	\arrow["\sigma"', curve={height=18pt}, from=3-3, to=1-3]
    \end{tikzcd}\]
    whose horizontal arrows are open and    closed immersions, such that the morphisms $\mathrm{gr} : \mathcal{B}un_H \to \mathcal{B}un_L$ and $\sigma : \mathcal{B}un_L \to \mathcal{B}un_H$ are given by extension of structure groups. The weak $\Theta$-action on $\mathrm{Filt}(\mathcal{B}un_H)$ arising from Proposition \ref{prop:filt-weak-actions} then restricts to a weak $\Theta$-action on $\mathcal{B}un_H$ with center $\mathcal{B}un_L$.
\end{proof}

\begin{rmk}
    The weak $\Theta$-action described in the proof of Lemma \ref{lem:weak-theta-action-bunh} admits the following alternative description: the cocharacter $\lambda$ induces via Theorem \ref{thm:filt-grad-quotient-stack} an open and closed immersion of $BH$ into $\mathrm{Filt}(BH)$, and hence a weak $\Theta$-action on $BH$, whose center is $BL$. By the functoriality of the mapping stack construction (and via the canonical open and closed immersion $\Theta \to \mathrm{Filt}(\Theta)$), this yields a weak $\Theta$-action on $\mathcal{B}un_H$ with center $\mathcal{B}un_L$; this weak action is equivalent to the weak action described in the proof of Lemma \ref{lem:weak-theta-action-bunh}.
\end{rmk}

\begin{rmk}
    One can show using \cite[Theorem 1.4.8]{Halpern-Leistner} that if Assumption \ref{ass:structure-group-bundles-strong} holds for $H$ then the weak $\Theta$-action described above (using the product cocharacter $\widehat{\lambda}$) is the restriction of a weak $\Theta^n$-action on $\mathcal{B}un_H$ induced by the homomorphism $\lambda : \mathbb{G}_m^n \to Z(L)^{\circ}$, however for our purposes we do not need to make use of this.
\end{rmk}

It will be convenient to have an explicit description of the fibers of $\mathcal{I}_{\mathrm{gr}} \to \mathcal{B}un_H$ over $k$-points, along with their Lie algebras.

\begin{lem} \label{lem:explicit-descr-fibers-Igr}
    Let $E \to C$ be a principal $H$-bundle, and let $\mathcal{I}_{\mathrm{gr},E} := \mathcal{I}_{\mathrm{gr}} \times_{\mathcal{B}un_H, E} \Spec k$. Denote $\mathfrak{u} = \mathrm{Lie}(U)$. Then there are isomorphisms
    $$ \mathcal{I}_{\mathrm{gr},E} \cong H^0(E \times^H U), \quad \mathrm{Lie}(\mathcal{I}_{\mathrm{gr},E}) \cong H^0(E \times^H \mathfrak{u}), $$
    where $H$ acts on $U$ by conjugation and on $\mathfrak{u}$ by the adjoint representation.
\end{lem}

\begin{proof}
    Let $E_L := E \times^H L$ be the extension of structure group along the quotient $H \to L$. Identifying $\mathrm{Aut}(E) = H^0(c_H(E))$ and $\mathrm{Aut}(E_L) = H^0(c_L(E_L))$, the homomorphism $\mathrm{Aut}(E) \to \mathrm{Aut}(E_L)$ induced by $\mathrm{gr} : \mathcal{B}un_H \to \mathcal{B}un_L$ corresponds to the homomorphism $H^0(c_H(E)) \to H^0(c_L(E_L))$ induced by composing an $H$-equivariant morphism $E \to H$ with the quotient $H \to L$. But this has kernel $H^0(E \times^H U)$, and so $\mathcal{I}_{\mathrm{gr},E} \cong H^0(E \times^H U)$. Since $\mathrm{Lie}(\mathcal{I}_{\mathrm{gr},E})$ is the vector space dual of $\omega_{\mathcal{I}_{\mathrm{gr},E}/\Spec k}$ (cf. \cite[Exposé II, Scholie 4.11.3]{SGA3}), the isomorphism $\mathrm{Lie}(\mathcal{I}_{\mathrm{gr},E}) \cong H^0(E \times^H \mathfrak{u})$ then follows by the same argument used to prove Lemma \ref{lem:lie-alg-aut-gp}.
\end{proof}

\subsubsection{Relation to Parabolic Subgroups of Reductive Groups and Unstable Harder--Narasimhan Strata}

An example of when the stronger Assumption \ref{ass:structure-group-bundles-strong} is applicable is when $H$ is a parabolic subgroup of a reductive group $G$.

\begin{ex} \label{ex:setting-parabolic-subgroup}
    Let $G$ be a connected reductive linear algebraic group. Fix a Borel pair $(B,T)$ of $G$, and let
    $$ \mathfrak{g} = \mathfrak{t} \oplus \bigoplus_{\alpha \in \Phi} \mathfrak{g}_{\alpha} $$
    be the corresponding weight space decomposition of $\mathfrak{g} = \mathrm{Lie}(G)$. Let $\Delta \subset \Phi^+$ be the corresponding set of simple roots. Fix a proper subset $I \subset \Delta$, and let $H = P_I$ be the corresponding parabolic subgroup of $G$. Let $n := |\Delta \setminus I| > 0$. Denoting by $\Phi_I$ the span of $I$ in $\Phi$, we have $H = U \rtimes L$, with
    $$ \mathrm{Lie}(L) = \mathfrak{t} \oplus \bigoplus_{\alpha \in \Phi_I} \mathfrak{g}_{\alpha}, \quad \mathrm{Lie}(U) = \bigoplus_{\beta \in \Phi^+ \setminus \Phi_I} \mathfrak{g}_{\beta}. $$
    We also have $Z(H)' = Z(G)^{\circ}$ and
    $$ Z(L)^{\circ} = \left( \bigcap_{\alpha \in I} \ker(\alpha : T \to \mathbb{G}_m) \right)^{\circ} \cong \mathbb{G}_m^{\dim Z(G)^{\circ} + n}. $$

    Fix an ordering $\alpha_1, \dots, \alpha_n, \dots, \alpha_{n+n'}$ of the simple roots with $\Delta \setminus I = \{ \alpha_1, \dots, \alpha_n \}$. For each $i = 1, \dots, n+n'$, let $\omega_i \in \mathfrak{X}_{\ast}(T)_{\Q}$ be the corresponding fundamental coweight, defined such that $\langle \alpha_i, \omega_j \rangle = \delta_{i,j}$. Choose positive integers $c_1, \dots, c_n \in \Z^{>0}$ such that for all $i = 1, \dots, n$, $\lambda_i := c_i \omega_i$ is an integral cocharacter of $T$. Define $\lambda : \mathbb{G}_m^n \to T \subset H$ by $\lambda(t_1, \dots, t_n) = \prod_{i=1}^n \lambda_i(t_i)$. The following statements then hold:
    \begin{enumerate}
        \item We have $\lambda(\mathbb{G}_m^n) \subset Z(L)^{\circ}$. Moreover, for all $h \in H$, the limit $\lim_{t \to 0} \lambda_i(t)h\lambda_i(t)^{-1}$ exists in $H$, and the iterated limit of $h$ under $\lambda_1, \dots, \lambda_n$ (in this order) is equal to $h$ if and only if $h \in L$. All of these assertions follow from the observation that if $\beta \in \Phi$ is such that $\beta \equiv_{\Phi_I} \sum_{j=1}^n b_j \alpha_j$ then (via the homomorphism $\lambda$) the torus $\mathbb{G}_m^n$ acts on the root space $\mathfrak{g}_{\beta}$ with $\Z^n$-weight $(c_1b_1, \dots, c_nb_n)$.

        \item The map $\mathbb{G}_m^n \times Z(G)^{\circ} \to Z(L)^{\circ}$ given by $(t, g) \mapsto \lambda(t)g$ is an isogeny.\footnote{Indeed, this is a homomorphism between tori of the same dimension with a $0$-dimensional kernel, hence is necessarily surjective.}
    \end{enumerate}

    Consequently, Assumption \ref{ass:structure-group-bundles-strong} applies to $H$. For each $i = 1, \dots, n$, the corresponding subgroup $U_i \subset U$ is generated by root subgroups $U_{\alpha}$ of $G$:
    $$ U_i = \left\langle U_{\alpha} : \alpha \equiv_{\Phi_I} \sum_{j=i}^n b_j \alpha_j, \ b_i > 0, b_{i+1}, \dots, b_n \geq 0 \right\rangle. $$
    We have that $U_{\alpha_i} \subset U_j$ if and only if $i = j$; in particular, each $U_i$ is a non-trivial subgroup of $U$.
\end{ex}

\begin{rmk} \label{rmk:hj24-4.1}
    In the setting of Example \ref{ex:setting-parabolic-subgroup}, take $G = \mathrm{SL}_r$, and let $H$ be the parabolic subgroup of $G$ consisting of block-upper triangular matrices whose diagonal blocks are of sizes $r_1 \times r_1, \dots, r_{\ell} \times r_{\ell}$ (in this order). Write $\gamma_i = \varepsilon_i - \varepsilon_{i+1}$, so that $\Delta = \{\gamma_1, \dots, \gamma_{r-1}\}$ and $\Delta \setminus I = \{\alpha_1, \dots, \alpha_{\ell-1}\}$, where $\alpha_i = \gamma_{s_i}$, $s_i := \sum_{j=1}^i r_j$. Take $\alpha_1, \dots, \alpha_{\ell-1}$ as the ordering of the roots in $\Delta \setminus I$. 
    
    Since the root space of $\mathfrak{sl}_r$ corresponding to the weight $\varepsilon_i - \varepsilon_j$ is given by $I + k \cdot E_{i,j}$ for all $i < j$, and since $\varepsilon_i - \varepsilon_j \equiv_{\Phi_I} \alpha_p + \cdots + \alpha_{q-1}$ whenever $s_{p-1} + 1 \leq i \leq s_p$ and $s_{q-1} + 1 \leq j \leq s_q$, it follows that the corresponding subgroup $U_i$ of the unipotent radical $U$ of $H$ is the subgroup given by the matrices in $U$ whose off-diagonal blocks are all $0$, except for those blocks in row $i$.
    
    If $U^{(0)} = \{I\}, U^{(1)}, \dots, U^{(\ell-1)}$ are as defined in \cite[Definition 4.4]{Hoskins-Jackson-QIS}, then $U_i \cong U^{(i)}/U^{(i-1)}$ for all $i = 1, \dots, \ell-1$. Consequently, from the iterated semidirect product
    $$ P = U_1 \rtimes \left( U_2 \rtimes \left( \cdots \rtimes \left( U_{\ell-1} \rtimes L \right) \cdots \right) \right) $$
    we recover the ``row filtration'' of the unipotent radical $U$ of $H \subset \mathrm{SL}_r$ considered in \cite[Section 4.1]{Hoskins-Jackson-QIS}. We consider this further in Example \ref{ex:comparing-adapt-tau-stable}.
\end{rmk}

By combining Lemma \ref{lem:weak-theta-action-bunh} with Example \ref{ex:setting-parabolic-subgroup}, our assumptions allow us to handle unstable Harder--Narasimhan strata in $\mathcal{B}un_G$, where $G$ is a reductive group.

\begin{ex} \label{ex:unstable-hn-strata-in-our-framework}
    Let $G$ be a connected reductive linear algebraic group, and fix a Borel pair $(B,T)$ of $G$. Let $\overline{\mathcal{C}}^+$ be the corresponding closed positive Weyl chamber. Fix an element $\tau \in \overline{\mathcal{C}}^+$ which is the Harder--Narasimhan type of an unstable principal $G$-bundle, and denote the corresponding Harder--Narasimhan stratum of $\mathcal{B}un_G$ by $\mathcal{B}un_G^{\tau}$, as in Theorem \ref{thm:gurjar-nitsure-hn-strata}.

    Let $H \subset G$ be the parabolic subgroup of $G$ containing $B$ determined by $\tau$, and let $L \subset H$ be the Levi subgroup of $H$. In the obvious way, we may view $\tau$ as a type for both principal $H$-bundles and principal $L$-bundles; denote by $\mathcal{B}un_H^{\tau}$ (resp. $\mathcal{B}un_L^{\tau}$) the corresponding open and closed substack of $\mathcal{B}un_H$ (resp. $\mathcal{B}un_L$). In light of Example \ref{ex:setting-parabolic-subgroup}, by Lemma \ref{lem:weak-theta-action-bunh} the stack $\mathcal{B}un_H^{\tau}$ admits a weak $\Theta$-action with center $\mathcal{B}un_L^{\tau}$.

    There is a natural morphism $\mathcal{B}un_H^{\tau} \to \mathcal{B}un_G^{\tau}$ given by extension of structure group. This is easily seen to be an isomorphism of algebraic stacks, whose inverse on $S$-points is given by associating to the unique reduction of structure group $(H, \theta: C \times S \to E_S/H)$ the principal $H$-bundle $F_S \to C \times S$ determined by $\theta$.
\end{ex}

\subsection{Substacks of $\mathcal{B}un_H$ Satisfying the Hypotheses of the Main Theorem}

Recall that to apply Theorem \ref{thm:main-theorem} to an algebraic stack $\mathcal{X}$ locally of finite presentation over $\Spec k$ with an affine diagonal, $\mathcal{X}$ is required to admit a weak $\Theta$-action with the following additional requirements:
\begin{enumerate}
    \item If $\mathcal{Z}$ is the center of the weak $\Theta$-action, then $\mathcal{Z}$ is required to admit a good moduli space.
    \item If $\mathrm{gr} : \mathcal{X} \to \mathcal{Z}$ is the retraction of $\mathcal{X}$ onto the center, then $\mathcal{I}_{\mathrm{gr}} \to \mathcal{X}$ is required to be flat and of finite presentation.
\end{enumerate}

Continue to assume Assumption \ref{ass:structure-group-bundles} holds for $H = U \rtimes L$, so that by Lemma \ref{lem:weak-theta-action-bunh} $\mathcal{B}un_H$ admits a weak $\Theta$-action with center $\mathcal{B}un_L$. In the context of this weak action, the first requirement can easily be handled by fixing a type $d \in \mathfrak{X}_{\ast}(T)$, and restricting to the preimage
$$ \mathcal{B}un_H^{d,ss} := \mathrm{gr}^{-1}(\mathcal{B}un_L^{d,ss}) $$
of the open substack $\mathcal{B}un_L^{d,ss}$ of $\mathcal{B}un_L$ parameterizing semistable $L$-bundles of type $d$; by Theorem \ref{thm:good-moduli-spaces=semistable-bundles}, the stack $\mathcal{B}un_L^{d,ss}$ admits a projective good moduli space. 

The second requirement requires more care to satisfy, since in general the relative inertia of $\mathrm{gr} : \mathcal{B}un_H^{d,ss} \to \mathcal{B}un_L^{d,ss}$ fails to be flat. We overcome this by passing to appropriate $\Theta$-invariant locally closed substacks of $\mathcal{B}un_H^{d,ss}$, defined in terms of vanishing loci of Fitting ideals of the relative sheaves of differentials associated to $\mathcal{I}_{\mathrm{gr}} \to \mathcal{B}un_H^{d,ss}$ and $\mathcal{I}_{\mathrm{gr}} \times_{\mathcal{B}un_H^{d,ss}} \mathcal{B}un_L^{d,ss} \to \mathcal{B}un_L^{d,ss}$.

\subsubsection{Flatness of the Relative Inertia in Terms of Sheaves of Differentials}

We require the following result characterising the flatness of a group stack $\Gamma \to \mathcal{S}$ representable by schemes, in terms of the sheaf of differentials $\omega_{\Gamma/\mathcal{S}} = e^{\ast}\Omega_{\Gamma/\mathcal{S}}^1$.

\begin{lem} \label{lem:smoothness-gp-schemes-cotangent-bundle}
    Let $\mathcal{S}$ be an algebraic stack over $\Spec k$, and let $\Gamma \to \mathcal{S}$ be an $\mathcal{S}$-group stack which is representable by schemes. Assume further that $\Gamma \to \mathcal{S}$ locally of finite presentation and has geometrically connected fibers. Then $\Gamma \to \mathcal{S}$ is flat if and only if $\Gamma \to \mathcal{S}$ is smooth, if and only if $\omega_{\Gamma/\mathcal{S}} = e^{\ast}\Omega_{\Gamma/\mathcal{S}}^1$ is a locally free $\mathcal{O}_S$-module.
\end{lem}

\begin{proof}
    By descent, it suffices to consider the case where $\mathcal{S}$ (and hence $\Gamma$) is a $k$-scheme. 
    
    We first observe that $\Gamma \to \mathcal{S}$ is smooth if and only if $\Gamma \to \mathcal{S}$ is flat, since group schemes locally of finite type over an algebraically closed field of characteristic $0$ are smooth \cite[Exposé VI-B, Corollaire 1.6.1]{SGA3}, and a flat morphism locally of finite presentation with smooth geometric fibers is smooth \cite[Tag 01V8]{stacks-project}.

    The hypotheses on $\Gamma$ imply that $\Gamma = \Gamma^{\circ}$ coincides with it's own identity component over $\mathcal{S}$; as such, the result that $\Gamma \to \mathcal{S}$ is smooth if and only if $\omega_{\Gamma/\mathcal{S}}$ is locally free follows from \cite[Exposé VI-B, Proposition 1.6 and Théorème 3.10]{SGA3}.
\end{proof}

\begin{cor} \label{cor:flatness-Igr-differentials}
    Let $\mathcal{S} \to \mathcal{B}un_H$ be a morphism from an algebraic stack. Then the pullback $\mathcal{I}_{\mathrm{gr},\mathcal{S}} = \mathcal{I}_{\mathrm{gr}} \times_{\mathcal{B}un_H} \mathcal{S} \to \mathcal{S}$ is flat and of finite presentation over $\mathcal{S}$ if and only if $\mathcal{I}_{\mathrm{gr},\mathcal{S}} \to \mathcal{S}$ is smooth, if and only if $\omega_{\mathcal{I}_{\mathrm{gr},\mathcal{S}}/\mathcal{S}}$ is locally free.
\end{cor}

\begin{proof}
    We will apply Lemma \ref{lem:smoothness-gp-schemes-cotangent-bundle}. Since $\mathcal{B}un_H$ has an affine diagonal of finite presentation, $\mathcal{I}_{\mathrm{gr},\mathcal{S}} \to \mathcal{S}$ is affine and of finite presentation, so in order to apply Lemma \ref{lem:smoothness-gp-schemes-cotangent-bundle} it remains to show that $\mathcal{I}_{\mathrm{gr},\mathcal{S}} \to \mathcal{S}$ has geometrically connected fibers.
    
    Without loss of generality, we may assume $\mathcal{S} = \Spec k$ and corresponds to a principal $H$-bundle $E \to C$. In this case, by Lemma \ref{lem:explicit-descr-fibers-Igr} the group $\mathcal{I}_{\mathrm{gr},E} = \mathcal{I}_{\mathrm{gr}} \times_{\mathcal{B}un_H, E} \Spec k$ is isomorphic to $H^0(E \times^H U).$ Since $U$ is unipotent and $k$ is of characteristic $0$, the exponential map $\exp : \mathfrak{u} = \mathrm{Lie}(U) \to U$ is an $H$-equivariant isomorphism of schemes (cf. \cite[Chapitre IV, Proposition 2.4.1]{Demazure-Gabriel}), so $H^0(E \times^H U)$ is isomorphic to the affine space of global sections of a vector bundle over $C$, and in particular is geometrically connected as a scheme.
\end{proof}

\subsubsection{Passing to Appropriate Fitting Ideal Loci I} \label{sec:passing-fitting-ideals-1}

Recall (cf. \cite[Tag 0C3C]{stacks-project}) that if $S$ is a scheme and if $\mathcal{F}$ is a finite type quasi-coherent $\mathcal{O}_S$-module, then there is a sequence of Fitting ideals
$$ 0 = \mathrm{Fit}_{-1}(\mathcal{F}) \subset \mathrm{Fit}_0(\mathcal{F}) \subset \mathrm{Fit}_1(\mathcal{F}) \subset \cdots \subset \mathcal{O}_S $$
with the property that, if $U \subset S$ is an open subscheme and if
\[\begin{tikzcd}[ampersand replacement=\&]
	{\bigoplus_{i \in I} \mathcal{O}_U} \& {\mathcal{O}_U^{\oplus m}} \& {\mathcal{F}|_U} \& 0
	\arrow["\psi", from=1-1, to=1-2]
	\arrow[from=1-2, to=1-3]
	\arrow[from=1-3, to=1-4]
\end{tikzcd}\]
is a presentation of $\mathcal{F}|_U$, then $\mathrm{Fit}_r(\mathcal{F})|_U$ is the ideal of $\mathcal{O}_U$ generated by the $(m-r)\times(m-r)$-minors of $\psi$. If $f : S' \to S$ is a morphisms of schemes, then for all $r \geq -1$ there is an equality
$$ f^{-1} \mathrm{Fit}_r(\mathcal{F}) \cdot \mathcal{O}_{S'} = \mathrm{Fit}_r(f^{\ast} \mathcal{F}). $$
This allows us to extend the Fitting ideal construction to the case where $S$ is replaced with an algebraic stack $\mathcal{S}$, by passing to a smooth presentation $S \to \mathcal{S}$ and applying descent. Smooth descent also implies that the Fitting ideals $\mathrm{Fit}_r(\mathcal{F})$ satisfy the following properties, extending the analogous properties for schemes \cite[Tags 05P8, 05P9]{stacks-project}:

\begin{enumerate}
    \item Let $f : S \to \mathcal{S}$ be a morphism from a scheme. Then $f^{\ast}\mathcal{F}$ is locally free of rank $r$ if and only if $f$ factors through the locally closed substack $\mathbb{V}(\mathrm{Fit}_{r-1}(\mathcal{F})) \setminus \mathbb{V}(\mathrm{Fit}_r(\mathcal{F}))$ of $\mathcal{S}$.

    \item Let $f : S \to \mathcal{S}$ be a morphism from a scheme. Then $f^{\ast}\mathcal{F}$ is flat if and only if $f$ factors through the natural morphism
    $$ \bigsqcup_{r=0}^{\infty} \left( \mathbb{V} ( \mathrm{Fit}_{r-1} ( \mathcal{F} ) ) \setminus \mathbb{V} ( \mathrm{Fit}_r ( \mathcal{F} ) ) \right) \to \mathcal{S}. $$

    \item Let $s : \Spec k \to \mathbb{V}(\mathrm{Filt}_r(\mathcal{F}))$ be a $k$-point, and let $\mathcal{F}|_s$ be the fiber of $\mathcal{F}$ over $s$. Then $\dim_k \mathcal{F}|_s > r$.
\end{enumerate}

\begin{cor} \label{cor:flatness-Igr-fitting-ideals}
    Let $\mathcal{S}$ be a locally closed substack of $\mathcal{B}un_H$. Then $\mathcal{I}_{\mathrm{gr},\mathcal{S}} \to \mathcal{S}$ is flat and of finite presentation if and only if the natural morphism
    $$ \bigsqcup_{r=0}^{\infty} \left( \mathbb{V} ( \mathrm{Fit}_{r-1} ( \omega_{\mathcal{I}_{\mathrm{gr}, \mathcal{S}}/\mathcal{S}} ) ) \setminus \mathbb{V} ( \mathrm{Fit}_r ( \omega_{\mathcal{I}_{\mathrm{gr}, \mathcal{S}}/\mathcal{S}} ) ) \right) \to \mathcal{S} $$
    is an isomorphism. Moreover, the restriction of $\mathcal{I}_{\mathrm{gr},\mathcal{S}} \to \mathcal{S}$ to the locally closed substack $\mathbb{V} ( \mathrm{Fit}_{r-1} ( \omega_{\mathcal{I}_{\mathrm{gr}, \mathcal{S}}/\mathcal{S}} ) ) \setminus \mathbb{V} ( \mathrm{Fit}_r ( \omega_{\mathcal{I}_{\mathrm{gr}, \mathcal{S}}/\mathcal{S}} ) )$ of $\mathcal{S}$ is smooth of relative dimension $r$.
\end{cor}

\begin{proof}
    Taking $\mathcal{F} = \omega_{\mathcal{I}_{\mathrm{gr}, \mathcal{S}}/\mathcal{S}}$, this follows from Corollary \ref{cor:flatness-Igr-differentials} and properties $(1)$, $(2)$ above.
\end{proof}

\subsubsection{Necessary Conditions on Unipotent Automorphisms} \label{sec:necessary-conditions-unipotent-automorphisms-bundles}

Fix a non-negative integer $\delta \in \Z^{\geq 0}$. Let
$$ \mathcal{S}_{\delta} := \mathbb{V} ( \mathrm{Fit}_{\delta-1} ( \omega_{\mathcal{I}_{\mathrm{gr}}/\mathcal{B}un_H^{d,ss}} ) ) \setminus \mathbb{V} ( \mathrm{Fit}_{\delta} ( \omega_{\mathcal{I}_{\mathrm{gr}}/\mathcal{B}un_H^{d,ss}} ) ) \subset \mathcal{B}un_H^{d,ss}. $$
By Corollary \ref{cor:flatness-Igr-fitting-ideals}, the restriction of $\mathcal{I}_{\mathrm{gr}} \to \mathcal{B}un_H^{d,ss}$ to the locally closed substack $\mathcal{S}_{\delta}$ is smooth of relative dimension $\delta$, and in particular is flat. However, the weak $\Theta$-action on $\mathcal{B}un_H^{d,ss}$ in general does not restrict to a weak action on $\mathcal{S}_{\delta}$, since the set of $k$-points of $\mathcal{S}_{\delta}$ is in general not closed under the operation of replacing a principal $H$-bundle $E$ over $C$ with the associated graded $H$-bundle $F = E_L \times^L H$, where $E_L = E \times^H L$. Indeed, by Lemma \ref{lem:explicit-descr-fibers-Igr} we have
$$ \mathrm{Lie}(\mathcal{I}_{\mathrm{gr},E}) \cong H^0(E \times^H \mathfrak{u}), \quad \mathrm{Lie}(\mathcal{I}_{\mathrm{gr},F}) \cong H^0(F \times^H \mathfrak{u}) = H^0(E_L \times^L \mathfrak{u}). $$
Consequently, $E$ and $F$ are both $k$-points of $\mathcal{S}_{\delta}$ only if $\dim H^0(E \times^H \mathfrak{u}) = \dim H^0(E_L \times^L \mathfrak{u})$. Unless $U$ is abelian, this equality may fail to hold, as indicated by the following remark.

\begin{rmk} \label{rmk:unipotent-auts-bundles-non-abelian}
     Let $\{e\} = U^{(q+1)} \lhd U^{(q)} \lhd \cdots \lhd U^{(1)} = U$ be any central series such that $\{e\} = U^{(q+1)} \lhd U^{(q)} \lhd \cdots \lhd U^{(1)} = U \lhd H$ is a normal series (for example, the lower central series $U^{(i+1)} = [U^{(i)}, U]$). For each $i = 1, \dots, q, q+1$, write $\mathfrak{u}^{(i)} = \mathrm{Lie}(U^{(i)})$.

     Let $E$ and $F = E_L \times^L H$ be as before. For each $i = 1, \dots, q$, we obtain from $U^{(i)}$ a subgroup $H^0(E \times^H U^{(i)}) \subset H^0(E \times^H U) \cong \mathcal{I}_{\mathrm{gr},E}$, with Lie algebra $H^0(E \times^H \mathfrak{u}^{(i)})$. Viewed as points of $|\mathcal{B}un_H|$, $F$ lies in the closure of $E$ (since $F \cong 0 \cdot E$ under the weak $\Theta$-action on $\mathcal{B}un_H$), so upper semi-continuity for the dimension of group schemes \cite[Exposé VI-B, Proposition 4.1]{SGA3} implies that $\dim H^0(E \times^H \mathfrak{u}^{(i)}) \leq \dim H^0(F \times^H \mathfrak{u}^{(i)}) = \dim H^0(E_L \times^L \mathfrak{u}^{(i)})$ for all $i = 1, \dots, q$.

     On the other hand, since the adjoint action of $U$ on $\mathfrak{u}^{(i)}/\mathfrak{u}^{(i+1)}$ is trivial, there is an equality of associated bundles $E \times^H \mathfrak{u}^{(i)}/\mathfrak{u}^{(i+1)} = E_L \times^L \mathfrak{u}^{(i)}/\mathfrak{u}^{(i+1)}$. Moreover, as $L$ is linearly reductive, there is an isomorphism $\mathfrak{u}^{(i)} \cong \mathfrak{u}^{(i+1)} \oplus \mathfrak{u}^{(i)}/\mathfrak{u}^{(i+1)}$ of $L$-modules. As such, the composition of the quotient map $E \to E/U \cong E_L$ with the natural inclusion $E_L \subset E_L \times^L H$ induces commutative diagrams
    \[\begin{tikzcd}[ampersand replacement=\&]
	0 \& {H^0(E \times^H \mathfrak{u}^{(i+1)})} \& {H^0(E \times^H \mathfrak{u}^{(i)})} \& {H^0(E \times^H \mathfrak{u}^{(i)}/\mathfrak{u}^{(i+1)})} \& \\
	\\
	0 \& {H^0(E_L \times^L \mathfrak{u}^{(i+1)})} \& {H^0(E_L \times^L \mathfrak{u}^{(i)})} \& {H^0(E_L \times^L \mathfrak{u}^{(i)}/\mathfrak{u}^{(i+1)})} \& 0
	\arrow[from=1-1, to=1-2]
	\arrow[from=1-2, to=1-3]
	\arrow[from=1-2, to=3-2]
	\arrow[from=1-3, to=1-4]
	\arrow[from=1-3, to=3-3]
	\arrow["{{\cong}}", from=1-4, to=3-4]
	\arrow[from=3-1, to=3-2]
	\arrow[from=3-2, to=3-3]
	\arrow[from=3-3, to=3-4]
	\arrow[from=3-4, to=3-5]
    \end{tikzcd}\]
    whose rows are exact. Moreover, when $i + 1 = q$, the map $H^0(E \times^H \mathfrak{u}^{(q)}) \to H^0(E_L \times^L \mathfrak{u}^{(q)})$ is an isomorphism. From this, one may deduce by induction on $i$ (and the snake lemma) that the following are equivalent:

    \begin{enumerate}
    \item There is an equality $\dim H^0(E \times^H \mathfrak{u}) = \dim H^0(E_L \times^L \mathfrak{u})$.
    \item For all $i = 1, \dots, q$, there is an equality $\dim H^0(E \times^H \mathfrak{u}^{(i)}) = \dim H^0(E_L \times^L \mathfrak{u}^{(i)})$, and the natural map $H^0(E \times^H \mathfrak{u}^{(i)}) \to H^0(E \times^H \mathfrak{u}^{(i)}/\mathfrak{u}^{(i+1)})$ is surjective.
\end{enumerate}

The latter condition $(2)$ is implied by the vanishing of the sheaf cohomology groups $H^1(\Gamma_C(E \times^{H} \mathfrak{u}^{(i)}))$ for all $i \geq 2$. By induction on $i$, this is equivalent to the vanishing of $H^1(\Gamma_C(E \times^{H} \mathfrak{u}^{(i)}/\mathfrak{u}^{(i+1)})) = H^1(\Gamma_C(E_L \times^{L} \mathfrak{u}^{(i)}/\mathfrak{u}^{(i+1)}))$ for all $i \geq 2$. The associated locally free coherent sheaf $\Gamma_C(E_L \times^{L} \mathfrak{u}^{(i)}/\mathfrak{u}^{(i+1)})$ decomposes as a direct sum
$$ \Gamma_C(E_L \times^{L} \mathfrak{u}^{(i)}/\mathfrak{u}^{(i+1)}) \cong \bigoplus_{j=1}^{n_i} \Gamma_C(E_L \times^L V_{i,j}), $$
where $V_{i,1}, \dots, V_{i,n_i}$ are the irreducible $L$-subrepresentations of $\mathfrak{u}^{(i)}/\mathfrak{u}^{(i+1)}$. Invoking Serre duality, we obtain that a sufficient condition for condition $(2)$ to hold is that for any irreducible $L$-representation $V$ appearing as a subrepresentation of $\mathfrak{u}^{(i)}/\mathfrak{u}^{(i+1)}$ for some $i \geq 2$, we have
\begin{equation} \label{eq:generalizing-qiao-6.6.2}
    \mu(\Gamma_C(E_L \times^L V)) > 2p_a(C) - 2,
\end{equation}
where $\mu(\mathcal{E}) = \deg(\mathcal{E})/\mathrm{rank}(\mathcal{E})$ denotes the Mumford slope of a vector bundle. Since the degree of the sheaf of sections of $E_L \times^L V$ is determined by the type of $E_L$, this is a condition which depends only on the type of the associated $L$-bundle $E_L$.
\end{rmk}

\begin{rmk}
    Condition \eqref{eq:generalizing-qiao-6.6.2} is further considered in the case where $H$ is a standard parabolic subgroup of $\mathrm{GL}_r$ in Example \ref{ex:qiao-cohomological-condition}.
\end{rmk}
\subsubsection{Passing to Appropriate Fitting Ideal Loci II}

We return to the problem of exhibiting locally closed substacks $\mathcal{X} \subset \mathcal{B}un_H^{d,ss}$ to which we can apply Theorem \ref{thm:main-theorem}. As indicated by both Corollary \ref{cor:flatness-Igr-fitting-ideals} and the discussion of Section \ref{sec:necessary-conditions-unipotent-automorphisms-bundles}, the candidate substacks $\mathcal{X}$ are substacks of one of the stacks
$$ \mathcal{S}_{\delta} = \mathbb{V} ( \mathrm{Fit}_{\delta-1} ( \omega_{\mathcal{I}_{\mathrm{gr}}/\mathcal{B}un_H^{d,ss}} ) ) \setminus \mathbb{V} ( \mathrm{Fit}_{\delta} ( \omega_{\mathcal{I}_{\mathrm{gr}}/\mathcal{B}un_H^{d,ss}} ) ) \subset \mathcal{B}un_H^{d,ss} $$
which are additionally closed under the operation $E \mapsto E_L \times^L H$ of replacing a principal $H$-bundle over $C$ with its associated graded $H$-bundle. 

We proceed by first defining the centers $\mathcal{Z} = \mathcal{Z}_{\delta} \subset \mathcal{B}un_L^{d,ss}$ in terms of the Fitting ideals of the pullback of $\omega_{\mathcal{I}_{\mathrm{gr}}}$ along $\sigma : \mathcal{B}un_L^{d,ss} \to \mathcal{B}un_H^{d,ss}$. Since we still need $\mathcal{Z}$ to admit a good moduli space, and since only \emph{saturated} substacks of $\mathcal{B}un_L^{d,ss}$ inherit from $\mathcal{B}un_L^{d,ss}$ a good moduli space, we impose the additional condition that the principal $L$-bundles $E_L$ appearing in $\mathcal{Z}$ are all stable, i.e. that $\mathcal{Z} \subset \mathcal{B}un_L^{d,s}$. Since $\mathcal{B}un_L^{d,s}$ admits a tame moduli space, any locally closed substack of $\mathcal{B}un_L^{d,s}$ is saturated and hence inherits from $\mathcal{B}un_L^{d,s}$ a tame moduli space.

After defining $\mathcal{Z}$, we then define the corresponding stack $\mathcal{X} = \mathcal{X}_{\delta}$ by passing to an appropriate substack of $\mathrm{gr}^{-1}(\mathcal{Z})$, defined in terms of the Fitting ideals of the restriction of $\mathcal{I}_{\mathrm{gr}} \to \mathcal{B}un_H^{d,s} := \mathrm{gr}^{-1}(\mathcal{B}un_L^{d,s})$ to $\mathcal{Z}$, in a way which ensures that both $\mathcal{X} \to \mathcal{Z}$ has flat relative inertia and that $\mathcal{X}$ is $\Theta$-invariant with center $\mathcal{Z}$.

Fix $\delta \in \mathbb{Z}^{\geq 0}$. We consider the group stack $\mathcal{J}_{\mathrm{gr}} \to \mathcal{B}un_L^{d,s}$ obtained by taking the pullback of $\mathcal{I}_{\mathrm{gr}} \to \mathcal{B}un_H^{d,s}$ along the morphism $\sigma : \mathcal{B}un_L^{d,s} \to \mathcal{B}un_H^{d,s}$:
$$ \mathcal{J}_{\mathrm{gr}} := \mathcal{I}_{\mathrm{gr}} \times_{\mathcal{B}un_H^{d,s}} \mathcal{B}un_L^{d,s} \to \mathcal{B}un_L^{d,s}. $$
We then define $\mathcal{Z}$ to be the locally closed substack
$$ \mathcal{Z} = \mathcal{Z}_{\delta} := \mathbb{V} ( \mathrm{Fit}_{\delta-1} ( \omega_{\mathcal{J}_{\mathrm{gr}}/\mathcal{B}un_L^{d,s}} ) ) \setminus \mathbb{V} ( \mathrm{Fit}_{\delta} ( \omega_{\mathcal{J}_{\mathrm{gr}}/\mathcal{B}un_L^{d,s}} ) ) \subset \mathcal{B}un_L^{d,s}. $$
By construction, the restriction $\mathcal{J}_{\mathrm{gr}, \mathcal{Z}} = \mathcal{J}_{\mathrm{gr}} \times_{\mathcal{B}un_L^{d,s}} \mathcal{Z}$ of $\mathcal{J}_{\mathrm{gr}}$ to $\mathcal{Z}$ is smooth of relative dimension $\delta$.

\begin{lem}
    The weak $B\mathbb{G}_m$-action on $\mathcal{B}un_L$ restricts to a weak $B\mathbb{G}_m$-action on $\mathcal{Z}$.
\end{lem}

\begin{proof}
    The composition of the morphism $\mathcal{Z} = \Spec k \times \mathcal{Z} \stackrel{q \times \mathrm{id}_{\mathcal{Z}}}{\to} B\mathbb{G}_m \times \mathcal{Z}$ with the weak action $B\mathbb{G}_m \times \mathcal{Z} \to \mathcal{B}un_L$ factors through the locally closed substack $\mathcal{Z} \subset \mathcal{B}un_L$. Since $\mathcal{Z} \to \mathcal{Z} \times B\mathbb{G}_m$ is a principal bundle, and in particular admits sections locally in the étale topology, \cite[Proposition 3.7]{champs-algebriques} implies that $B\mathbb{G}_m \times \mathcal{Z} \to \mathcal{B}un_L$ must factor through $\mathcal{Z}$.
\end{proof}

Let $\mathcal{X}' := \mathrm{gr}^{-1}(\mathcal{Z})$ be the preimage of $\mathcal{Z}$ in $\mathcal{B}un_H^{d,s}$. We then define
$$ \mathcal{X} = \mathcal{X}_{\delta} := \mathbb{V} ( \mathrm{Fit}_{\delta-1} ( \omega_{\mathcal{I}_{\mathrm{gr}, \mathcal{X}'}/\mathcal{X}'} ) ) \subset \mathcal{X}' \subset \mathcal{B}un_H^{d,s}. $$

\begin{prop}
    The weak $\Theta$-action on $\mathcal{B}un_H$ restricts to a weak $\Theta$-action on $\mathcal{X}$, for which $\mathcal{X}$ has center $\mathcal{Z}$ and for which the retraction $\mathcal{X} \to \mathcal{Z}$ is given by the restriction of $\mathrm{gr} : \mathcal{B}un_H \to \mathcal{B}un_L$. Moreover, the relative inertia $\mathcal{I}_{\mathcal{X}/\mathcal{Z}} \to \mathcal{X}$ is given by $\mathcal{I}_{\mathrm{gr},\mathcal{X}} = \mathcal{I}_{\mathrm{gr}} \times_{\mathcal{B}un_H^{d,s}} \mathcal{X}$, and this is smooth of relative dimension $\delta$.
\end{prop}

\begin{proof}
    Since $\mathcal{Z}$ is $B\mathbb{G}_m$-invariant, the composition of the weak $\Theta$-action $\Theta \times \mathcal{X}' \to \mathcal{B}un_H$ with $\mathrm{gr} : \mathcal{B}un_H \to \mathcal{B}un_L$ must factor through $\mathcal{Z} \subset \mathcal{B}un_L$, and so $\Theta \times \mathcal{X}' \to \mathcal{B}un_H$ must factor through $\mathrm{gr}^{-1}(\mathcal{Z}) = \mathcal{X}'$. Therefore $\mathcal{X}'$ is $\Theta$-invariant. It similarly follows that the restriction of $\sigma : \mathcal{B}un_L \to \mathcal{B}un_H$ to $\mathcal{Z}$ must factor through $\mathcal{X}'$. Since $\mathcal{X}$ is a closed substack of $\mathcal{X}'$, and since the morphism $\mathcal{X} \stackrel{1 \times \mathrm{id}_{\mathcal{X}}}{\to} \Theta \times \mathcal{X}$ is an open immersion with schematically dense image, the action morphism $\Theta \times \mathcal{X} \to \mathcal{X}'$ must factor through the closed immersion $\mathcal{X} \subset \mathcal{X}'$, and so $\mathcal{X}$ is also $\Theta$-invariant.

    Next, we show that $\mathcal{I}_{\mathrm{gr}, \mathcal{X}} \to \mathcal{X}$ is smooth of relative dimension $\delta$. By the properties of Fitting ideal loci mentioned in Section \ref{sec:passing-fitting-ideals-1}, it suffices to show that the closed substack $\mathbb{V} ( \mathrm{Fit}_{\delta} ( \omega_{\mathcal{I}_{\mathrm{gr}, \mathcal{X}'}/\mathcal{X}'} ) )$ of $\mathcal{X}'$ is empty, or equivalently (since $\mathcal{X}'$ is of finite type over $k$) that it admits no $k$-points. Let $x \in \mathcal{X}'(k)$ be a point, corresponding to the principal $H$-bundle $E$; the point $z = \mathrm{gr}(x) \in \mathcal{Z}(k)$ corresponds to the principal $L$-bundle $E_L = E \times^H L$. As in Remark \ref{rmk:unipotent-auts-bundles-non-abelian}, we have
    $$ \dim \mathrm{Lie}(\mathcal{I}_{\mathrm{gr},E}) = \dim H^0(E \times^H \mathfrak{u}) \leq \dim H^0(E_L \times^L \mathfrak{u}) = \dim \mathrm{Lie}(\mathcal{J}_{\mathrm{gr}, E_L}). $$
    On the other hand, since $\mathcal{J}_{\mathrm{gr},\mathcal{Z}} \to \mathcal{Z}$ is smooth of relative dimension $\delta$, we have $\dim \mathrm{Lie}(\mathcal{J}_{\mathrm{gr}, E_L}) = \delta$, and so $\dim \mathrm{Lie}(\mathcal{I}_{\mathrm{gr},E}) \leq \delta$. It follows that $x \not \in \mathbb{V} ( \mathrm{Fit}_{\delta} ( \omega_{\mathcal{I}_{\mathrm{gr}, \mathcal{X}'}/\mathcal{X}'} ) )(k)$, and so $\mathbb{V} ( \mathrm{Fit}_{\delta} ( \omega_{\mathcal{I}_{\mathrm{gr}, \mathcal{X}'}/\mathcal{X}'} ) ) = \emptyset$. Consequently $\mathcal{I}_{\mathrm{gr}, \mathcal{X}} \to \mathcal{X}$ is smooth of relative dimension $\delta$ as claimed.

    We now show that $\mathcal{X}$ has center $\mathcal{Z}$, and that the retraction $\mathcal{X} \to \mathcal{Z}$ is given by the restriction of $\mathrm{gr} : \mathcal{X}' \to \mathcal{Z}$. By construction, the restriction of $\sigma$ to $\mathcal{Z}$ factors through $\mathcal{X}' \subset \mathcal{B}un_L^{d,s}$. However, since $\mathcal{J}_{\mathrm{gr}, \mathcal{Z}} \to \mathcal{Z}$ is smooth of relative dimension $\delta$, the restriction $\sigma|_{\mathcal{Z}}$ further factors through
    $$ \mathbb{V} ( \mathrm{Fit}_{\delta-1} ( \omega_{\mathcal{I}_{\mathrm{gr}, \mathcal{X}'}/\mathcal{X}'} ) ) \setminus \mathbb{V} ( \mathrm{Fit}_{\delta} ( \omega_{\mathcal{I}_{\mathrm{gr}, \mathcal{X}'}/\mathcal{X}'} ) ) = \mathcal{X}. $$
    As such, we have $\sigma^{-1}(\mathcal{X}) = \mathcal{Z}$. On the other hand, since both $\mathrm{Filt}$ and $\mathrm{Grad}$ preserve both open immersions and closed immersions (cf. \cite[Corollary 1.1.8, Proposition 1.3.1]{Halpern-Leistner}), the induced morphism $\mathrm{Filt}(\mathcal{X}) \to \mathrm{Filt}(\mathcal{B}un_H)$ is a locally closed immersion, for which the composition with the weak action classifying morphism $\mathcal{X} \to \mathrm{Filt}(\mathcal{X})$ factors through the open and closed immersion $\mathcal{B}un_H \to \mathrm{Filt}(\mathcal{B}un_H)$ considered in the proof of Lemma \ref{lem:weak-theta-action-bunh}. Similarly, the induced morphism $\mathrm{Grad}(\mathcal{X}) \to \mathrm{Grad}(\mathcal{B}un_H)$ is a locally closed immersion, factoring through the open and closed immersion $\mathcal{B}un_L \to \mathrm{Grad}(\mathcal{B}un_H)$. Consequently there is a Cartesian diagram
    \[\begin{tikzcd}
	{\mathrm{Grad}(\mathcal{X})} && {\mathcal{B}un_L} \\
	\\
	{\mathrm{Filt}(\mathcal{X})} && {\mathcal{B}un_H}
	\arrow[from=1-1, to=1-3]
	\arrow[from=1-1, to=3-1]
	\arrow["\lrcorner"{anchor=center, pos=0.125}, draw=none, from=1-1, to=3-3]
	\arrow["\sigma", from=1-3, to=3-3]
	\arrow[from=3-1, to=3-3]
    \end{tikzcd}\]
    Since $\mathcal{Z} = \sigma^{-1}(\mathcal{X})$, this extends to a Cartesian diagram
    \[\begin{tikzcd}
	{\mathcal{Z}} && {\mathrm{Grad}(\mathcal{X})} && {\mathcal{B}un_L} \\
	\\
	{\mathcal{X}} && {\mathrm{Filt}(\mathcal{X})} && {\mathcal{B}un_H}
	\arrow[from=1-1, to=1-3]
	\arrow[from=1-1, to=3-1]
	\arrow["\lrcorner"{anchor=center, pos=0.125}, draw=none, from=1-1, to=3-3]
	\arrow[from=1-3, to=1-5]
	\arrow[from=1-3, to=3-3]
	\arrow["\lrcorner"{anchor=center, pos=0.125}, draw=none, from=1-3, to=3-5]
	\arrow["\sigma", from=1-5, to=3-5]
	\arrow["\phi"', from=3-1, to=3-3]
	\arrow[from=3-3, to=3-5]
    \end{tikzcd}\]
    where $\phi$ is the morphism classifying the restricted weak $\Theta$-action on $\mathcal{X}$. This shows that $\mathcal{Z}$ is the center for the weak $\Theta$-action on $\mathcal{X}$. Since the retraction $\mathrm{Filt}(\mathcal{X}) \to \mathrm{Grad}(\mathcal{X})$ is given by the restriction of the extension of structure group morphism $\mathrm{gr} : \mathcal{B}un_H \to \mathcal{B}un_L$, the same is true for the retraction $\mathcal{X} \to \mathcal{Z}$. This also implies that $\mathcal{I}_{\mathcal{X}/\mathcal{Z}} \cong \mathcal{I}_{\mathrm{gr}, \mathcal{X}}$ as group stacks over $\mathcal{X}$.
\end{proof}

\begin{cor} \label{cor:moduli-space-h-bundles-no-character}
    Let $\pi : \mathcal{Z} \to M_L^{d,\delta}$ be the quasi-projective tame moduli space of $\mathcal{Z} \subset \mathcal{B}un_L^{d,s}$. Let $\mathcal{X}^{\mathrm{rig}}$ be the rigidification of $\mathcal{X}$ by $\mathcal{I}_{\mathrm{gr}, \mathcal{X}} = \mathcal{I}_{\mathcal{X}/\mathcal{Z}}$. Then the composition of $\pi$ with $\mathrm{gr} : \mathcal{X} \to \mathcal{Z}$ factors as a gerbe $\rho : \mathcal{X} \to \mathcal{X}^{\mathrm{rig}}$ followed by a good moduli space $\pi \circ \overline{\mathrm{gr}} : \mathcal{X}^{\mathrm{rig}} \to M_L^{d,\delta}$.
\end{cor}

\begin{proof}
    This follows by combining the previous lemma with the first two statements of Theorem \ref{thm:main-theorem}.
\end{proof}

\subsubsection{Functors of Points of the Stacks}

Our next task is to give an explicit description of the groupoid $\mathcal{X}(S)$ for a scheme $S$, in terms of families of principal bundles $E_S \to C \times S$. To do this, we will make use of the following lemma, which provides a converse to the standard result that the dual of a locally free $\mathcal{O}_S$-module of finite rank is always locally free of the same rank.

\begin{lem} \label{lem:dual-sheaf-locally-free}
    Let $\mathcal{F}$ be a finite type quasi-coherent sheaf on a scheme $S$. The following are equivalent:

    \begin{enumerate}
        \item $\mathcal{F}$ is locally free of rank $r$.
        \item For all closed subschemes $S' \subset S$, $\mathcal{F}|_{S'}^{\vee} = \mathcal{H}om_{\mathcal{O}_{S'}}(\mathcal{F}|_{S'}, \mathcal{O}_{S'})$ is locally free of rank $r$.
        \item For all locally closed subschemes $S' \subset S$, $\mathcal{F}|_{S'}^{\vee} = \mathcal{H}om_{\mathcal{O}_{S'}}(\mathcal{F}|_{S'}, \mathcal{O}_{S'})$ is locally free of rank $r$.
    \end{enumerate}
\end{lem}

\begin{proof}
    This is essentially \cite[Proposition 2.2.9]{Qiao-Thesis}; we repeat the proof here for completeness.

    The only non-straightforward implication is $(3) \Rightarrow (1)$. For $m \geq -1$, let $Z_m := \mathbb{V}(\mathrm{Fit}_m(\mathcal{F}))$ be the closed subscheme of $S$ defined by the $m$th Fitting ideal of $\mathcal{F}$, and for $m \geq 0$ let $T_m := Z_{m-1} \setminus Z_m$. Then, for each $m$ such that $T_m \neq \emptyset$, the restriction of $\mathcal{F}|_{T_m}$ to $T_m$ is locally free of rank $m$. But $(3)$ implies that the only non-empty $T_m$ is $T_r$, so $Z_r = \emptyset$ and $\mathrm{Fit}_r(\mathcal{F}) = \mathcal{O}_S$. By \cite[Tag 0C3F]{stacks-project} and Nakayama's lemma, this implies that there exists a Zariski open cover of $S$ with the property that over every open $U$ in the cover, $\mathcal{F}|_U$ can be generated by $r$ sections.

    Without loss of generality, we may therefore assume that $\mathcal{F}$ can be generated by $r$ sections. Pick a presentation of the form
    \[\begin{tikzcd}[ampersand replacement=\&,cramped]
	   {\mathcal{O}_S^{\oplus J}} \&\& {\mathcal{O}_S^{\oplus r}} \&\& {\mathcal{F}} \&\& {0.}
	   \arrow["\psi", from=1-1, to=1-3]
	   \arrow["\phi", from=1-3, to=1-5]
	   \arrow[from=1-5, to=1-7]
    \end{tikzcd}\]

    Since we assume $(3)$ holds, the sheaf $\mathcal{F}^{\vee}$ is locally free of rank $r$. As such, $\phi^{\vee}$ is an injective homomorphism between locally free sheaves of rank $r$, so $\mathrm{coker}(\phi^{\vee})$ is a torsion $\mathcal{O}_S$-module. On the other hand, $\mathrm{coker}(\phi^{\vee})$ is a submodule of a free $\mathcal{O}_S$-module; this implies that $\mathrm{coker}(\phi^{\vee}) = 0$, hence $\phi^{\vee}$ is an isomorphism and $\psi^{\vee} = 0$. As $\psi$ is a homomorphism of free $\mathcal{O}_S$-modules whose dual $\psi^{\vee} = 0$, we obtain $\psi = 0$, and hence $\phi$ is an isomorphism, which yields $(1)$.
\end{proof}

\begin{lem} \label{lem:univ-prop-X-bund-given-Z}
    Let $S$ be a scheme and let $S \to \mathcal{B}un_H^{d,s}$ be a morphism, corresponding to the principal $H$-bundle $E_S \to C \times S$ Assume that $S \to \mathcal{B}un_H^{d,s}$ factors through $\mathcal{X}' = \mathrm{gr}^{-1}(\mathcal{Z})$. Set $\mathcal{I}_{\mathrm{gr},S} := \mathcal{I}_{\mathrm{gr}} \times_{\mathcal{B}un_H^{d,s}} S$. Then the following are equivalent:
    \begin{enumerate}
        \item $S \to \mathcal{X}'$ factors through $\mathcal{X} \subset \mathcal{B}un_H^{d,s}$.
        \item The $\mathcal{O}_S$-module $\omega_{\mathcal{I}_{\mathrm{gr},S}/S}$ is locally free of rank $\delta$.
        \item The group scheme $\mathcal{I}_{\mathrm{gr},S} \to S$ is smooth of relative dimension $\delta$.
        \item For all closed subschemes $S' \subset S$, the $\mathcal{O}_{S'}$-module $(\mathrm{pr}_{S'})_{\ast} \Gamma_{C \times S'}(E_{S'} \times^H \mathfrak{u})$ is locally free of rank $\delta$.
    \end{enumerate}
\end{lem}

\begin{proof}
    In light of Lemma \ref{lem:univ-prop-X-bund-given-Z} and Corollary \ref{cor:flatness-Igr-differentials}, it remains to show $(2) \iff (4)$. However, for any $S$-scheme $S'$, we have 
    $$ \mathcal{H}om_{\mathcal{O}_{S'}}(\omega_{\underline{\mathrm{Aut}}(E_{S'})/S'}, \mathcal{O}_{S'}) \cong (\mathrm{pr}_{S'})_{\ast} \Gamma_{C \times S'}(\mathrm{Ad}_H(E_{S'})) $$ 
    by Lemma \ref{lem:lie-alg-aut-gp}. Passing to the subgroup $\mathcal{I}_{\mathrm{gr},S} \subset \underline{\mathrm{Aut}}(E_{S'})$, this restricts to give an isomorphism
    $$ \mathcal{H}om_{\mathcal{O}_{S'}}(\omega_{\mathcal{I}_{\mathrm{gr},S'}/S'}, \mathcal{O}_{S'}) \cong (\mathrm{pr}_{S'})_{\ast} \Gamma_{C \times S'}(E_{S'} \times^H \mathfrak{u}). $$
    As such, the equivalence $(2) \iff (4)$ follows from Lemma \ref{lem:dual-sheaf-locally-free}.
\end{proof}

In order to characterize when $S \to \mathcal{B}un_H^{d,s}$ factors through $\mathcal{X}' = \mathrm{gr}^{-1}(\mathcal{Z})$, we use the following lemma.

\begin{lem} \label{lem:univ-prop-Z-bund}
    Let $S$ be a scheme and let $S \to \mathcal{B}un_L^{d,s}$ be a morphism, corresponding to the principal $L$-bundle $F_S \to C \times S$. Set $\mathcal{J}_{\mathrm{gr},S} := \mathcal{J}_{\mathrm{gr}} \times_{\mathcal{B}un_L^{d,s}} S$. Then the following are equivalent:
    \begin{enumerate}
        \item $S \to \mathcal{B}un_L^{d,s}$ factors through $\mathcal{Z} \subset \mathcal{B}un_L^{d,s}$.
        \item The $\mathcal{O}_S$-module $\omega_{\mathcal{J}_{\mathrm{gr},S}/S}$ is locally free of rank $\delta$.
        \item The group scheme $\mathcal{J}_{\mathrm{gr},S} \to S$ is smooth of relative dimension $\delta$.
        \item For all closed subschemes $S' \subset S$, the $\mathcal{O}_{S'}$-module $(\mathrm{pr}_{S'})_{\ast} \Gamma_{C \times S'}(F_{S'} \times^L \mathfrak{u})$ is locally free of rank $\delta$.
    \end{enumerate}
\end{lem}

\begin{proof}
    The proof is analogous to the proof of Lemma \ref{lem:univ-prop-X-bund-given-Z}; we leave it to the interested reader to make the necessary modifications.
\end{proof}

\begin{cor} \label{cor:univ-prop-X-bund}
    Let $S$ be a scheme. Then a morphism $S \to \mathcal{X} = \mathcal{X}_{\delta} \subset \mathcal{B}un_H^{d,s}$ is the same data as a principal $H$-bundle $E_S \to C \times S$ for which the following properties hold:
    \begin{enumerate}
        \item For all $s \in S(k)$, the restricted $H$-bundle $E_s \to C$ is of type $d$, and the extension of structure group $E_s \times^H L$ is a stable principal $L$-bundle.
        \item For all closed subschemes $S' \subset S$, the $\mathcal{O}_{S'}$-module $(\mathrm{pr}_{S'})_{\ast} \Gamma_{C \times S'}(E_{S'} \times^H \mathfrak{u})$ is locally free of rank $\delta$.
        \item For all closed subschemes $S' \subset S$, the $\mathcal{O}_{S'}$-module $(\mathrm{pr}_{S'})_{\ast} \Gamma_{C \times S'}((E_{S'} \times^H L) \times^L \mathfrak{u})$ is locally free of rank $\delta$.
    \end{enumerate}
\end{cor}

\begin{proof}
    Combine Lemmas \ref{lem:univ-prop-X-bund-given-Z} and \ref{lem:univ-prop-Z-bund}.
\end{proof}

\begin{rmk} \label{rmk:functor-of-points-of-reduction}
    If $S$ is reduced, then by Grauert's theorem (as stated in \cite[Theorem 23.140]{Gortz-Wedhorn-2}) we have that condition $(2)$ is equivalent to $s \mapsto \dim_{\kappa(s)} H^0(E_s \times^H \mathfrak{u})$ being a constant function on $|S|$ with value $\delta$, and similarly for condition $(3)$.
\end{rmk}

\begin{rmk} \label{rmk:dropping-conditions-L-stabs}
    As in Remark \ref{rmk:unipotent-auts-bundles-non-abelian}, let $\{e\} = U^{(q+1)} \lhd U^{(q)} \lhd \cdots \lhd U^{(1)} = U$ be any central series such that $\{e\} = U^{(q+1)} \lhd U^{(q)} \lhd \cdots \lhd U^{(1)} = U \lhd H$ is a normal series, and write $\mathfrak{u}^{(i)} = \mathrm{Lie}(U^{(i)})$. If $S$ is a scheme and if $E_S \to C \times S$ is a principal $H$-bundle, we obtain commutative diagrams with exact rows as in Remark \ref{rmk:unipotent-auts-bundles-non-abelian}, with $(\mathrm{pr}_S)_{\ast} \Gamma_{C \times S}(E_S \times^H \mathfrak{u}^{(i)})$ in place of $H^0(E \times^H \mathfrak{u}^{(i)})$ etc.

    If each of the maps $(\mathrm{pr}_{S})_{\ast} \Gamma_{C \times S}(E_{S} \times^H \mathfrak{u}^{(i)}) \to (\mathrm{pr}_{S})_{\ast} \Gamma_{C \times S}(E_{S} \times^H \mathfrak{u}^{(i)} / \mathfrak{u}^{(i+1)})$ are surjective, by downward induction on $i$ and the snake lemma each of the vertical maps $(\mathrm{pr}_{S})_{\ast} \Gamma_{C \times S}(E_{S} \times^H \mathfrak{u}^{(i)}) \to (\mathrm{pr}_{S})_{\ast} \Gamma_{C \times S}((E_S \times^H L) \times^L \mathfrak{u}^{(i)})$ is an isomorphism; taking $i = 1$ then gives that $(\mathrm{pr}_{S})_{\ast} \Gamma_{C \times S}(E_{S} \times^H \mathfrak{u}) \to (\mathrm{pr}_{S})_{\ast} \Gamma_{C \times S}((E_S \times^H L) \times^L \mathfrak{u})$ is an isomorphism. In particular, if $(\mathrm{pr}_{S})_{\ast} \Gamma_{C \times S}(E_{S} \times^H \mathfrak{u})$ is locally free of rank $\delta$ then so is $(\mathrm{pr}_{S})_{\ast} \Gamma_{C \times S}((E_S \times^H L) \times^L \mathfrak{u})$.

    It follows that if the type $d$ of the fibers of $E_S \to C \times S$ is chosen such that the inequality \eqref{eq:generalizing-qiao-6.6.2} in Remark \ref{rmk:unipotent-auts-bundles-non-abelian} holds, so that for any $s \in S(k)$ and for any $i \geq 2$ we have $H^1(\Gamma_C(E_s \times^H \mathfrak{u}^{(i)})) = 0$, then by a corollary of Grothendieck's cohomology and base change theorem (cf. \cite[Corollary 23.144]{Gortz-Wedhorn-2}) we have $R^1(\mathrm{pr}_{S'})_{\ast} \Gamma_{C \times S'}(E_{S'} \times^H \mathfrak{u}^{(i)}) = 0$ for all $i \geq 2$ and all closed subschemes $S' \subset S$. Consequently, for any closed subscheme $S' \subset S$, each of the maps $(\mathrm{pr}_{S'})_{\ast} \Gamma_{C \times S'}(E_{S'} \times^H \mathfrak{u}^{(i)}) \to (\mathrm{pr}_{S'})_{\ast} \Gamma_{C \times S'}(E_{S'} \times^H \mathfrak{u}^{(i)} / \mathfrak{u}^{(i+1)})$ are surjective. As such, in this case condition $(3)$ of Corollary \ref{cor:univ-prop-X-bund} is automatically implied by condition $(2)$.
\end{rmk}

\subsection{Relatively Projective Moduli Spaces from Characters of the Levi}

Let $\mathcal{X} = \mathcal{X}_{\delta}$ and $\mathcal{Z} = \mathcal{Z}_{\delta}$ be as in the previous subsection. We also continue to assume Assumption \ref{ass:structure-group-bundles} holds for $H$. From now on, we take $\mathrm{gr} : \mathcal{X} \to \mathcal{Z}$ to be the retraction of $\mathcal{X}$ onto the center $\mathcal{Z}$ (given by extension of structure group with respect to $H \to L$) and write $\mathcal{I}_{\mathrm{gr}}$ in place of $\mathcal{I}_{\mathrm{gr}, \mathcal{X}}$, so that now $\mathcal{I}_{\mathrm{gr}} = \mathcal{I}_{\mathcal{X}/\mathcal{Z}}$ is the relative inertia of $\mathrm{gr} : \mathcal{X} \to \mathcal{Z}$. As in the statement of Corollary \ref{cor:moduli-space-h-bundles-no-character}, write $\mathcal{X}^{\mathrm{rig}} := \mathcal{X} \!\!\! \fatslash \mathcal{I}_{\mathrm{gr}}$, and write $\rho : \mathcal{X} \to \mathcal{X}^{\mathrm{rig}}$ for the rigidification morphism. Also write $\overline{\mathrm{gr}} : \mathcal{X}^{\mathrm{rig}} \to \mathcal{Z}$ for the induced morphism on the rigidification. By Corollary \ref{cor:moduli-space-h-bundles-no-character}, the morphism $\mathcal{X} \to M_L^{d, \delta}$ factors as $\rho$ followed by a good moduli space. However, this is far from being a coarse moduli space, since any principal $H$-bundle $E$ corresponding to a point of $\mathcal{X}$ is identified with the principal $L$-bundle $E_L = E \times^H L$. To rectify this we introduce stability conditions on $\mathcal{X}$, in our case arising from characters of $L$.

\subsubsection{Stability with Respect to an Arbitrary Character}

Fix a base point $p_0 \in C(k)$ (the choice of $p_0$ does not matter for our purposes). This defines an an evaluation morphism $\mathrm{ev}_{p_0} : \mathcal{B}un_L \to BL$, given by the composition
\[\begin{tikzcd}[ampersand replacement=\&]
	{\mathcal{B}un_L = \underline{\mathrm{Map}}(C, BL)} \&\& {C \times \underline{\mathrm{Map}}(C, BL)} \&\& BL,
	\arrow["{p_0 \times \mathrm{id}_{\underline{\mathrm{Map}}(C, BL)}}", from=1-1, to=1-3]
	\arrow["{\mathrm{ev}}", from=1-3, to=1-5]
\end{tikzcd}\]
where $\mathrm{ev} : C \times \underline{\mathrm{Map}}(C, BL) \to BL$ is the universal morphism. Given a character $\chi \in \mathfrak{X}^{\ast}(L)$, there is a corresponding invertible sheaf $\mathcal{O}_{BL}(\chi)$ on the classifying stack $BL$, as defined in Definition \ref{def:line-bundle-coming-from-character}.

\begin{defi}
    Let $\chi \in \mathfrak{X}^{\ast}(L)$ be a character of $L$. We define $\mathcal{L}_{\chi}$ to be the invertible sheaf on $\mathcal{X}^{\mathrm{rig}}$ obtained by pulling back $\mathcal{O}_{BL}(\chi)$ along the composition \[\begin{tikzcd}[ampersand replacement=\&]
	{\mathcal{X}^{\mathrm{rig}}} \& {\mathcal{Z}} \& {\mathcal{B}un_L} \& BL.
	\arrow["{\overline{\mathrm{gr}}}", from=1-1, to=1-2]
	\arrow["\subset", from=1-2, to=1-3]
	\arrow["{\mathrm{ev}_{p_0}}", from=1-3, to=1-4]
    \end{tikzcd}\]
\end{defi}

\begin{lem} \label{lem:principal-bundle-semistability-general-character}
    Let $E$ be a principal $H$-bundle corresponding to a $k$-point of $\mathcal{X}$. Then $E$ is $\Theta$-(semi)stable with respect to $\rho^{\ast}\mathcal{L}_{\chi}$ if and only if for all non-trivial cocharacters $\eta : \mathbb{G}_m \to Z(L)^{\circ}$ for which $E$ admits a reduction of structure group to the subgroup $P_H(\eta) \subset H$, we have that $\langle \chi, \eta \rangle > (\geq) \ \ 0$.
\end{lem}

\begin{proof}
    Since $\rho$ is $\Theta$-surjective (cf. Corollary \ref{cor:filt-gerbe-is-surjective}), it is possible to prove the result by applying Theorem \ref{thm:gms-etale-local-structure} to the fibers of $\mathcal{X}^{\mathrm{rig}} \to M_L^{d, \delta}$ to reduce the problem to GIT. However, we will instead provide a more direct argument.

    Let $f : \Theta \to \mathcal{X}$ be a non-trivial filtration with $\mathrm{ev}_1(f) \cong E$. Let $g : \Theta \to \mathcal{Z}$ be the induced filtration of the stable principal $L$-bundle $E_L = E \times^H L$ given by composing $f$ with $\mathrm{gr} : \mathcal{X} \to \mathcal{Z}$. Corresponding to $f$ and $g$ is a non-trivial cocharacter $\eta : \mathbb{G}_m \to T \subset L$ and isomorphisms $E \cong E_{P} \times^{P} H$, $E_L \cong E_{\overline{P}} \times^{\overline{P}} L$, where $E_P$ is a principal bundle with structure group $P := P_H(\eta)$ and $E_{\overline{P}}$ is a principal bundle with structure group $\overline{P} := P_L(\eta)$. The Weyl-conjugacy class of the cocharacter $\eta$ is the same conjugacy class corresponding to the composition of $g$ with $\mathrm{ev}_{p_0} : \mathcal{B}un_L \to BL$ afforded by Theorem \ref{thm:filt-grad-quotient-stack}; as such, we have
    $$ \mathrm{wt}_0(f^{\ast}\rho^{\ast}\mathcal{L}_{\chi}) = -\langle \chi, \eta \rangle. $$

    We claim that $\eta$ must factor through $Z(L)^{\circ}$. Set $\overline{L} := L_L(\eta)$ and set $E_{\overline{L}} := E_{\overline{P}} \times^{\overline{P}} \overline{L}$; then $F := E_{\overline{L}} \times^{\overline{L}} L$ is isomorphic to $\mathrm{ev}_1(g)$. As $\mathrm{ev}_1(g)$ is a $k$-point of $\mathcal{Z} \subset \mathcal{B}un_L^{d,s}$, $F$ must be a stable $L$-bundle.

    Assume for a contradiction that $\overline{P}$ is a proper parabolic subgroup of $L$. Then there exists a non-trivial dominant character $\xi$ of $\overline{P}$ (relative to $L$). Since $F$ is stable and admits the reduction of structure group $E_{\overline{L}} \times^{\overline{L}} \overline{P}$ to $\overline{P}$, we have $0 > \deg(\xi_{\ast}(E_{\overline{L}} \times^{\overline{L}} \overline{P})) = \deg(\xi_{\ast} E_{\overline{L}})$. Similarly, $F$ admits a reduction of structure group to $\overline{P}^{\mathrm{opp}} := P_L(\eta^{-1})$, for which the character $\xi^{-1}$ is dominant, so $0 > \deg(\xi_{\ast}^{-1} E_{\overline{L}})$. But the sheaves of sections of $\xi_{\ast}E_{\overline{L}}$ and $\xi_{\ast}^{-1}E_{\overline{L}}$ are dual invertible sheaves, so $\deg(\xi_{\ast}^{-1} E_{\overline{L}}) = -\deg(\xi_{\ast} E_{\overline{L}})$; this gives a contradiction.

    Consequently, we must have $\overline{P} = L$. Since $L$ is reductive, the unipotent radical $U_L(\eta)$ of $\overline{P} = P_L(\eta)$ must be trivial, so $L = \overline{L} = L_L(\eta)$, which implies that $\eta$ is a cocharacter of $Z(L)^{\circ}$. This proves the lemma.
\end{proof}

\begin{rmk}
    If $\eta$ is any cocharacter of $Z(H)'$, then $H = P_H(\eta^{\pm1})$, so any principal $H$-bundle trivially admits a reduction of structure group to both $P_H(\eta)$ and $P_H(\eta^{-1})$. As such, there are no $\Theta$-semistable principal $H$-bundles unless $\langle \chi, \eta \rangle = 0$ for all cocharacters $\eta \in \mathfrak{X}^{\ast}(Z(H)')$. In any case, if $\mathfrak{X}^{\ast}(Z(H)')$ contains a non-trivial cocharacter then any $\Theta$-semistable principal $H$-bundle is strictly $\Theta$-semistable, no matter the choice of $\chi$.
\end{rmk}

\subsubsection{Existence of Moduli Spaces}

From now on, we impose the following non-degeneracy condition on $\chi$: $\langle \chi, \eta \rangle = 0$ for all cocharacters $\eta \in \mathfrak{X}^{\ast}(Z(H)')$. Set $\mathcal{X}^{\chi-ss} := \mathcal{X}^{ss}(\rho^{\ast}\mathcal{L}_{\chi})$ to be the open substack of $\mathcal{X}$ consisting of principal $H$-bundles which are $\Theta$-semistable with respect to $\mathcal{L}_{\chi}$.

\begin{cor} \label{cor:moduli-space-h-bundles-general-character}
    Let $S$ be a scheme. Then a morphism $S \to \mathcal{X}^{\chi-ss} = \mathcal{X}_{\delta}^{\chi-ss} \subset \mathcal{B}un_H^{d,s}$ is the same data as a principal $H$-bundle $E_S \to C \times S$ for which the following properties hold:
    \begin{enumerate}
        \item For all $s \in S(k)$, the restricted $H$-bundle $E_s \to C$ is of type $d$, and the extension of structure group $E_s \times^H L$ is a stable principal $L$-bundle.
        \item For all closed subschemes $S' \subset S$, the $\mathcal{O}_{S'}$-module $(\mathrm{pr}_{S'})_{\ast} \Gamma_{C \times S'}(E_{S'} \times^H \mathfrak{u})$ is locally free of rank $\delta$.
        \item For all closed subschemes $S' \subset S$, the $\mathcal{O}_{S'}$-module $(\mathrm{pr}_{S'})_{\ast} \Gamma_{C \times S'}((E_{S'} \times^H L) \times^L \mathfrak{u})$ is locally free of rank $\delta$.
        \item For all $s \in S(k)$, for all non-trivial cocharacters $\eta : \mathbb{G}_m \to Z(L)^{\circ}$ for which $E_s$ admits a reduction of structure group to the subgroup $P_H(\eta) \subset H$ we have $\langle \chi, \eta \rangle \geq 0$.
    \end{enumerate}
\end{cor}

\begin{proof}
    Since any open substack of $\mathcal{X}$ is uniquely determined by its $k$-points, this follows by combining Lemma \ref{lem:principal-bundle-semistability-general-character} with Corollary \ref{cor:univ-prop-X-bund}.
\end{proof}

\begin{teo} \label{thm:moduli-principal-bundles-general-character}
    Assume Assumption \ref{ass:structure-group-bundles} holds for $H$. Let $d \in \mathfrak{X}_{\ast}(T)$ and let $\delta \in \Z^{\geq 0}$. Let $\chi$ be a character of $L$ with $\langle \chi, \eta \rangle = 0$ for all cocharacters $\eta \in \mathfrak{X}^{\ast}(Z(H)')$. Let $\pi : \mathcal{Z} \to M_L^{d,\delta}$ be the quasi-projective tame moduli space of the center $\mathcal{Z} \subset \mathcal{B}un_L^{d,s}$ of $\mathcal{X}$, where $\mathcal{X} = \mathcal{X}_{\delta}$ is as in Corollary \ref{cor:univ-prop-X-bund}. Then the morphism
    $$ \pi_{\chi} : \mathcal{X}^{\chi-ss} \to M_H^{d,\delta,\chi} := \underline{\mathrm{Proj}}_{M_L^{d,\delta}} \bigoplus_{m \geq 0} (\pi \circ \overline{\mathrm{gr}})_{\ast}(\mathcal{L}_{\chi}) $$
    factors as a gerbe $\rho : \mathcal{X}^{\chi-ss} \to \rho(\mathcal{X}^{\chi-ss}) = (\mathcal{X}^{\mathrm{rig}})^{ss}(\mathcal{L}_{\chi})$ followed by a good moduli space $(\mathcal{X}^{\mathrm{rig}})^{ss}(\mathcal{L}_{\chi}) \to M_H^{d,\delta,\chi}$. Moreover $M_H^{d,\delta,\chi}$ is projective over $M_L^{d,\delta}$.
\end{teo}

\begin{proof}
    All assertions except for the relative projectivity of $M_H^{d,\delta,\chi}$ follow from the third statement of Theorem \ref{thm:main-theorem}. The projectivity of $M_H^{d,\delta,\chi} \to M_L^{d,\delta}$ is a consequence of Lemma \ref{lem:when-moduli-space-is-projective}.
\end{proof}

\subsubsection{Stability with Respect to Adapted Characters} \label{sec:stability-adapted-characters}

From this point on, we assume that the stronger Assumption \ref{ass:structure-group-bundles-strong} holds in place of Assumption \ref{ass:structure-group-bundles}. To obtain coarse moduli spaces of principal $H$-bundles, we consider stability conditions on $\mathcal{X}^{\mathrm{rig}}$ arising from characters of $L$ which are adapted to the decomposition \eqref{eq:iterated-semidirect-product} arising from Assumption \ref{ass:structure-group-bundles-strong}, in the sense of Definition \ref{def:adapted-characters}. For this, we will make use of the following lemma, concerning GIT stability for actions of tori on vector spaces.

\begin{lem} \label{lem:finding-adapted-characters}
    Let $V$ be a representation of the torus $\mathbb{G}_m^n$ of finite dimension with weight space decomposition
    $$ V = \bigoplus_{\omega \in \Omega} V_{\omega}, $$
    where $\Omega \subset (\mathbb{Z}^{\geq 0})^n$. Assume that $V_{\omega} \neq 0$ for all $\omega \in \Omega$. For each $m = 0, 1, \dots, n$, let $p_m : V \to V_m$ be the projection onto the subspace $V_m \subset V$ given by
    $$ V_m := \bigoplus_{\stackrel{\omega \in \Omega}{\omega_1 = \cdots = \omega_m = 0}} V_{\omega}. $$
    Then, for the character $\chi : \mathbb{G}_m^n \to \mathbb{G}_m$ given by $(t_1, \dots, t_n) \mapsto \prod_{i=1}^n t_i^{\chi_i}$, the following holds: if $0 < \chi_1 \ll \chi_2 \ll \cdots \ll \chi_n$ then a vector $v \in V$ is semistable with respect to $\chi$ if and only if $v$ is stable with respect to $\chi$, if and only if for all $m = 1, \dots, n$, $p_{m-1}(v) \neq p_m(v)$. Moreover, after fixing $\chi$, for each $m = 1, \dots, n$ there exists a cocharacter $\eta_m : \mathbb{G}_m \to \mathbb{G}_m^m \subset \mathbb{G}_m^n$ (where the inclusion $\mathbb{G}_m^m \subset \mathbb{G}_m^n$ corresponds to the first $m$ factors of $\mathbb{G}_m^n$) such that $\eta_m$ destabilizes all $v \in V$ with $p_{m-1}(v) = p_m(v)$ (i.e. $\langle \chi, \eta_m \rangle < 0$, and for all such $v$ we have that the limit $\lim_{t \to 0} \eta_m(t) \cdot v$ exists in $V$).
\end{lem}

\begin{proof}
    The proof we give is similar to the proof of the related result \cite[Lemma 3.4.1]{Qiao-Thesis}, albeit we make use of the weight polytope criterion of \cite[Lemma 2.11]{HHJ-Affine-NRGIT} for twisted affine GIT with respect to torus actions.

    We introduce the following notation: 
    \begin{itemize}
        \item For a vector $v \in V$, set $\mathrm{wt}(v) := \{ \omega \in \Omega : v_{\omega} \neq 0 \}$.
        \item For a subset $A \subset \mathbb{R}^m$, denote by $\mathrm{conv}_{\R^m}(A)$ the convex hull of $A$ in $\R^m$.
        \item For each $m = 1, \dots, n$, set
        $$ \Omega_m := \left( \bigcup_{j=1}^{m-1} \{ \omega \in \Omega : \omega_i = 0 \text{ for all } i < j, \ \omega_j > 0 \} \right) \cup \{ \omega \in \Omega : \omega_1 = \cdots = \omega_m = 0 \}. $$
        \item For each $m = 1, \dots, n$ let $q_m : \R^n \to \R^m$ be the projection map $(x_1, \dots, x_n) \mapsto (x_1, \dots, x_m)$.
    \end{itemize}

    The weight polytope criterion of \cite[Lemma 2.11]{HHJ-Affine-NRGIT} asserts that if $\chi = (\chi_1, \dots, \chi_n) \in \mathfrak{X}^{\ast}(\mathbb{G}_m^n) \cong \Z^n$ is a non-trivial character and if $v \in V$, then the vector $v$ is semistable with respect to $\chi$ if and only if $\mathrm{conv}_{\R^n}(\mathrm{wt}(v))$ contains a point along the ray $\R^{>0} \cdot \chi$, and $v$ is stable with respect to $\chi$ if and only if the interior of $\mathrm{conv}_{\R^n}(\mathrm{wt}(v))$ contains a point along the ray $\R^{>0} \cdot \chi$.

    We make the following observations:
    \begin{enumerate}
        \item[(i)] As each $V_{\omega} \neq 0$, for any non-empty subset $\Omega' \subset \Omega$ there exists $v \in V$ with $\mathrm{wt}(v) = \Omega'$. In particular, if $\Omega_m \neq \emptyset$, there exists a vector $v_m \in V$ with $\mathrm{wt}(v_m) = \Omega_m$. For this vector $v_m$, we have $p_{m-1}(v_m) = p_m(v_m)$.

        \item[(ii)] For any $v \in V$ with $p_{m-1}(v) = p_m(v)$, we have $\mathrm{wt}(v) \subset \Omega_m$.
    \end{enumerate}

    Let $\Xi$ be the set of vectors $(\chi_1, \dots, \chi_n) \in \mathbb{Z}^n$ satisfying the following (finite) set of conditions:
    \begin{enumerate}
        \item For each $m = 1, \dots, n$, 
        $$ \mathrm{conv}_{\R^m}(q_m(\Omega_m)) \cap (\R^{>0} \cdot (\chi_1, \dots, \chi_m)) = \emptyset. $$

        \item For any non-singular, lower triangular $n \times n$ matrix $M$ whose columns are elements of $\Omega$, all coordinates of $M^{-1} (\chi_1, \dots, \chi_n)^T$ are strictly positive.
    \end{enumerate}
    One easily checks that if $0 < \chi_1 \ll \cdots \ll \chi_n$ then $(\chi_1, \dots, \chi_n)$ is an element of $\Xi$. In particular, $\Xi \neq \emptyset$. We claim that for any element $(\chi_1, \dots, \chi_n) \in \Xi$, the resulting character $\chi : \mathbb{G}_m^n \to \mathbb{G}_m$ has the property that a vector $v \in V$ is semistable with respect to $\chi$ if and only if $v$ is stable with respect to $\chi$, if and only if for each $m = 1, \dots, n$, $p_{m-1}(v) \neq p_m(v)$.

    Suppose first that $p_{m-1}(v) \neq p_m(v)$ for all $m = 1, \dots, n$. Then for each $i = 1, \dots, n$, there exists a weight $\omega_i = (\omega_{i,1}, \dots, \omega_{i,n}) \in \Omega$ such that $\omega_{i,1} = \cdots = \omega_{i,i-1} = 0$, $\omega_{i,i} > 0$ and such that $v_{\omega_i} \neq 0$. Let $M$ be the $n \times n$ lower triangular matrix whose $i$th column is $\omega_i$. By condition (2), the vector $b := M^{-1}(\chi_1, \dots, \chi_n)^T$ has all coordinates $b_i > 0$. As $(\chi_1, \dots, \chi_n)^T = Mb$, we have a relation of the form $(\chi_1, \dots, \chi_n) = \sum_{i=1}^n b_i \omega_i$
    where each $b_i > 0$. It follows that the interior of $\mathrm{conv}_{\R^n}(\mathrm{wt}(v))$ contains a non-zero point along the ray $\R^{>0} \cdot (\chi_1, \dots, \chi_n)$, which implies that $v$ is stable with respect to $\chi$.

    We now show that any vector $v$ with $p_{m-1}(v) = p_m(v)$ for some $m$ is unstable with respect to $\chi$, and that (after fixing $\chi$) for each $m$ there exists a cocharacter $\eta_m$ of $\mathbb{G}_m^m \subset \mathbb{G}_m^n$ which destabilizes all $v \in V$ with $p_{m-1}(v) = p_m(v)$. Fix $m \in \{1, \dots, n\}$ such that $p_{m-1}(v) = p_m(v)$ for some non-zero $v \in V$ (so that $\Omega_m \neq \emptyset$). By condition (1) and the weight polytope criterion for stability, the vector $v_m$ is unstable with respect to the action of the subtorus $\mathbb{G}_m^m \subset \mathbb{G}_m^n$ given by the first $m$ factors of $\mathbb{G}_m^n$. The affine Hilbert--Mumford criterion (cf. Proposition \ref{prop:ghh_hilbert_mumford}) then implies that there exists a cocharacter $\eta_m \in \mathfrak{X}_{\ast}(\mathbb{G}_m^n) \cong \Z^n$ of the form $\eta_m = (\eta_{m,1}, \dots, \eta_{m,m}, 0, \dots, 0)$ such that $\langle \omega, \eta_m \rangle \geq 0$ for all $\omega \in \Omega_m = \mathrm{wt}(v_m)$, so that $\lim_{t \to 0} \eta_m(t) \cdot v_m$ exists in $V$, and such that $\langle \chi, \eta_m \rangle < 0$. Then, for any other vector $v \in V$ with $p_{m-1}(v) = p_m(v)$, since $\mathrm{wt}(v) \subset \Omega_m$, the limit $\lim_{t \to 0} \eta_m(t) \cdot v$ also exists in $V$, so $\eta_m$ also destabilizes $v$. This proves the lemma.
\end{proof}

Condition $(3)$ of Assumption \ref{ass:structure-group-bundles-strong} implies that the set of weights for the action of $\mathbb{G}_m^n$ on $\mathfrak{u} = \mathrm{Lie}(U)$ given by composing the homomorphism $\lambda : \mathbb{G}_m^n \to Z(L)^{\circ}$ with the adjoint action of $Z(L)^{\circ}$ on $\mathfrak{u}$ is contained in $(\mathbb{Z}^{\geq0})^n$. The decomposition \eqref{eq:iterated-semidirect-product} of Assumption \ref{ass:structure-group-bundles-strong} yields a decomposition
$$ \mathfrak{u} = \bigoplus_{i=1}^n \mathfrak{u}_i, \quad \mathfrak{u}_i := \mathrm{Lie}(U_i), $$
with respect to which $p_m$ is given by the projection
$$ p_m: \mathfrak{u} \to \bigoplus_{i > m} \mathfrak{u}_i. $$
Lemma \ref{lem:finding-adapted-characters} then implies the existence of a character $\chi' \in \mathfrak{X}^{\ast}(\mathbb{G}_m^n)$ and cocharacters $\eta_1, \dots, \eta_n \in \mathfrak{X}_{\ast}(\mathbb{G}_m^n)$ with the following properties:
\begin{enumerate}
    \item Each $\eta_m \in \mathfrak{X}_{\ast}(\mathbb{G}_m^m) \subset \mathfrak{X}_{\ast}(\mathbb{G}_m^n)$, and $\langle \chi', \eta_m \rangle < 0$.\footnote{If the only $v \in \mathfrak{u}$ for which $p_{m-1}(v) = p_m(v)$ is $v = 0$, we may take $\eta_m \in \mathfrak{X}_{\ast}(\mathbb{G}_m^m)$ to be any cocharacter with negative pairing $\langle \chi', \eta_m \rangle < 0$.}
    \item If $v \in \mathfrak{u}$ and $p_{m-1}(v) \neq p_m(v)$ for all $m = 1, \dots, n$, then $v$ is $\chi'$-stable.
    \item If $p_{m-1}(v) = p_m(v)$ then $v$ is $\chi'$-unstable and is destabilized by $\eta_m$.
\end{enumerate}

By condition $(3)$ of Assumption \ref{ass:structure-group-bundles-strong}, we have an isogeny $\mathbb{G}_m^n \times Z(H)' \times [L,L] \to L$. Therefore, by replacing $\chi'$ with a positive multiple if necessary, there exists a character $\chi \in \mathfrak{X}^{\ast}(L)$ such that $\chi' = \chi \circ \lambda$ and such that $\chi|_{Z(H)'}$ is the trivial character.

\begin{defi} \label{def:adapted-characters}
    Assume Assumption \ref{ass:structure-group-bundles-strong} holds for $H$. We say $\chi \in \mathfrak{X}^{\ast}(L)$ is an \emph{adapted} character if the following conditions hold:
    \begin{enumerate}
        \item $\chi|_{Z(H)'}$ is the trivial character.
        \item Setting $\chi' := \chi \circ \lambda \in \mathfrak{X}^{\ast}(\mathbb{G}_m^n)$, there exist cocharacters $\eta_m \in \mathfrak{X}_{\ast}(\mathbb{G}_m^m) \subset \mathfrak{X}_{\ast}(\mathbb{G}_m^n)$ ($m = 1, \dots, n$) such that for the action of $\mathbb{G}_m^n$ on $\mathfrak{u}$ described above:
        \begin{enumerate}
            \item[(i)] For each $m = 1, \dots, n$, $\langle \chi', \eta_m \rangle = \langle \chi, \lambda \circ \eta_m \rangle < 0$.
            \item[(ii)] if $v \in \mathfrak{u}$ and $p_{m-1}(v) \neq p_m(v)$ for all $m = 1, \dots, n$, then $v$ is $\chi'$-stable; and
            \item[(iii)] if $p_{m-1}(v) = p_m(v)$ then the limit $\lim_{t \to 0} \eta_m(t) \cdot v$ exists in $\mathfrak{u}$ (and consequently that $v$ is $\chi'$-unstable).
        \end{enumerate}
    \end{enumerate}
\end{defi}

By the discussion preceding Definition \ref{def:adapted-characters}, the set of adapted characters of $L$ is non-empty.

\begin{prop} \label{prop:semistability-adapted-character}
    Assume Assumption \ref{ass:structure-group-bundles-strong} holds for $H$. Let $E$ be a principal $H$-bundle corresponding to a $k$-point of $\mathcal{X} = \mathcal{X}_{\delta}$. Let $\chi \in \mathfrak{X}^{\ast}(L)$ be an adapted character. For each $m = 1, \dots, n$, set $F_m := ((E \times^{H_0} H_1) \times^{H_1} \cdots) \times^{H_{m-1}} H_m$ and $E_m := F_m \times^{H_m} H$, where $H_0 = H, H_1 \dots, H_{n-1}, H_n = L$ are as defined in Assumption \ref{ass:structure-group-bundles-strong}.
    
    Then $E$ is $\Theta$-semistable with respect to $\rho^{\ast}\mathcal{L}_{\chi}$ if and only if for all $m = 1, \dots, n$, $E_{m-1} \not \cong E_m$ as principal $H$-bundles.
\end{prop}

\begin{proof}
    Let $\eta_1, \dots, \eta_n \in \mathfrak{X}_{\ast}(\mathbb{G}_m^n)$ be as in Definition \ref{def:adapted-characters}(2). Since $\eta_m$ factors through the inclusion $\mathbb{G}_m^m \subset \mathbb{G}_m^n$, $\lambda \circ \eta_m$ commutes with all elements of $H_m$.

    We first show that if $E_{m-1} \cong E_m$ for some $m$, then $E$ is $\Theta$-unstable with respect to $\rho^{\ast} \mathcal{L}_{\chi}$. The existence of an isomorphism $E_{m-1} \cong E_m$ implies that $E_{m-1}$ admits a reduction of structure group $C \to E_{m-1}/H_m$. However for each $i = 1, \dots, n$ we have $H_{i-1} = U_i \rtimes H_i$ by Assumption \ref{ass:structure-group-bundles-strong}(4), so $E_i \cong E_{i-1}/U_i$ as varieties over $C$. It follows by induction that
    $$ E_{m-1}/H_m \cong E/(U_1 \rtimes (\cdots \rtimes ( U_{m-1} \rtimes H_m ) \cdots)) $$
    as varieties over $C$, and so we have a reduction of structure group of $E$ to the subgroup $U_1 \rtimes (\cdots \rtimes ( U_{m-1} \rtimes H_m ) \cdots)$ of $H$. Set
    $$ \mathfrak{v}_m := \mathrm{Lie}(U_1 \rtimes (\cdots (U_{m-2} \rtimes U_{m-1})  \cdots)) = \bigoplus_{i \leq m-1} \mathfrak{u}_i \subset \mathfrak{u}. $$
    For each $v \in \mathfrak{v}_m$, we have $p_{m-1}(v) = p_m(v)$, so the limit $\lim_{t \to 0} \eta_m(t) \cdot v$ exists in $\mathfrak{u}$. Exponentiating $\mathfrak{v}_m$ and using that $\lambda \circ \eta_m$ commutes with all elements of $H_m$, this implies that for any element $h \in U_1 \rtimes (\cdots \rtimes ( U_{m-1} \rtimes H_m ) \cdots)$, $\lim_{t \to 0} (\lambda \circ \eta_m)(t)h(\lambda \circ \eta_m)(t)^{-1}$ exists in $U_1 \rtimes (\cdots \rtimes ( U_{m-1} \rtimes H_m ) \cdots)$. Therefore $E$ admits a reduction of structure group to the subgroup $P_H(\lambda \circ \eta_m)$ of $H$. Since
    $$ \langle \chi, \lambda \circ \eta_m \rangle = \langle \chi \circ \lambda, \eta_m \rangle = \langle \chi', \eta_m \rangle < 0, $$
    Lemma \ref{lem:principal-bundle-semistability-general-character} then implies that $E$ is $\Theta$-unstable with respect to $\rho^{\ast}\mathcal{L}_{\chi}$.

    We now show that if $E$ is $\Theta$-unstable with respect to $\rho^{\ast} \mathcal{L}_{\chi}$ then $E_{m-1} \cong E_m$ for some $m$. By Lemma \ref{lem:principal-bundle-semistability-general-character}, there exists a non-trivial cocharacter $\eta \in \mathfrak{X}_{\ast}(Z(L)^{\circ})$ such that $E$ admits a reduction of structure group to the subgroup $P_H(\eta)$ of $H$ and such that $\langle \chi, \eta \rangle < 0$. By replacing $\eta$ with a positive multiple if necessary, we may assume that $\eta$ lifts along the isogeny $\mathbb{G}_m^n \times Z(H)' \to Z(L)^{\circ}$ of Assumption \ref{ass:structure-group-bundles-strong}(3)(ii). Since $\chi$ is orthogonal to any cocharacter of $Z(H)'$, we may further assume $\eta$ is of the form $\eta = \lambda \circ \zeta$ for some cocharacter $\zeta : \mathbb{G}_m \to \mathbb{G}_m^n$, so that $\langle \chi', \zeta \rangle = \langle \chi, \eta \rangle < 0$.

    Set $\mathfrak{v} := \mathrm{Lie}(U \cap P_H(\eta)) = \mathrm{Lie}(U \cap P_H(\lambda \circ \zeta)) \subset \mathfrak{u}$. Since $\langle \chi', \zeta \rangle < 0$ and since the limit $\lim_{t \to 0} \zeta(t) \cdot v$ exists for all $v \in \mathfrak{v}$, all elements of $\mathfrak{v}$ are $\chi'$-unstable. Giving both $\mathfrak{v}$ and the $\chi'$-unstable locus $\mathfrak{u}^{\chi'-us}$ the reduced induced subscheme structure, since $\chi$ is adapted we have that the subscheme $\mathfrak{u}^{\chi'-us}$ of $\mathfrak{u}$ is a finite union of linear spaces of the form
    $$ W_m := \{ v \in \mathfrak{u} : p_{m-1}(v) = p_m(v) \}. $$
    Since $\mathfrak{v}$ is also a linear space (in particular, $\mathfrak{v}$ is irreducible), it follows that there exists some $m$ with $\mathfrak{v} \subset W_m$.

    We claim that for this $m$ we have an isomorphism $E_{m-1} \cong E_m$ of principal $H$-bundles. Since $E$ admits a reduction of structure group to $P_H(\eta)$, the principal $H_{m-1}$-bundle 
    $$ F_{m-1} = E/(U_1 \rtimes (\cdots \rtimes ( U_{m-2} \rtimes U_{m-1} ) \cdots)) $$
    admits a reduction of structure group to $P_{H_{m-1}}(\eta)$. Noting that all elements of $H_{m-1}$ commute with $\lambda_1, \dots, \lambda_{m-1}$, the inclusion $\mathfrak{v} \subset W_m$ further implies that all elements of $P_{H_{m-1}}(\eta)$ commute with $\lambda_m$, or in other words that $P_{H_{m-1}}(\eta) \subset L_{H_{m-1}}(\lambda_m) = H_m$. Consequently $F_m$ admits a reduction of structure group to $H_m$. This implies the existence of an isomorphism $F_{m-1} \cong F_m \times^{H_m} H_{m-1}$ of principal $H_{m-1}$-bundles, and hence the existence of an isomorphism $E_{m-1} \cong E_m$ of principal $H$-bundles. This completes the proof of the proposition.
\end{proof}

\subsubsection{Automorphism Groups of Adapted Stable Bundles}

By Proposition \ref{prop:semistability-adapted-character}, all adapted characters define a common open substack of $\mathcal{X} = \mathcal{X}_{\delta}$, which we denote $\mathcal{X}^{\mathrm{adpt}}$, and a common good moduli space $\mathcal{X}^{\mathrm{rig, adpt}} \to M_H^{d,\delta,\mathrm{adpt}}$ of the image $\mathcal{X}^{\mathrm{rig, adpt}}$ of $\mathcal{X}^{\mathrm{adpt}}$ under the gerbe $\rho : \mathcal{X} \to \mathcal{X}^{\mathrm{rig}}$. We will show that this good moduli space is a tame moduli space. To do so, we consider the automorphisms of a principal $H$-bundle which is semistable with respect to an adapted character.

\begin{lem} \label{lem:lifts-of-cocharacters-L-to-H}
    Let $\eta: \mathbb{G}_m \to L$ be a cocharacter. Then any lift of $\eta$ to a cocharacter of $H$ is of the form $t \mapsto u\eta(t)u^{-1}$ for some $u \in U$.
\end{lem}

\begin{proof}
    This result is standard, though we include a proof for lack of a satisfactory reference. Since all unipotent linear algebraic groups over $k$ are split unipotent, by passing to an appropriate composition series we may assume without loss of generality that $U = \mathbb{G}_a$.

    Suppose $\eta' : \mathbb{G}_m \to H$ is a lift of $\eta$. For each $t \in \mathbb{G}_m$, write $\eta'(t) = u(t)\eta(t) \in H = U \rtimes L$. The assignment $t \mapsto u(t)$ then defines a 1-cocycle with respect to $\eta$. But $H^1(\mathbb{G}_m, \mathbb{G}_a) = 0$ (cf. \cite[Chapitre II, Proposition 3.4.2]{Demazure-Gabriel}), so $t \mapsto u(t)$ is cohomologous to a trivial 1-cocycle, from which the lemma follows.
\end{proof}

\begin{prop} \label{prop:finiteness-of-auts-of-Xadpt}
    Assume Assumption \ref{ass:structure-group-bundles-strong} holds for $H$. Let $E$ be a principal $H$-bundle corresponding to a $k$-point of $\mathcal{X}^{\mathrm{adpt}}$, so that for each $m = 1, \dots, n$, $E_{m-1} \not \cong E_m$ as principal $H$-bundles. Let $\overline{Z(H)}$ denote the image of $Z(H)$ under the homomorphism $\mathrm{Aut}(E) = H^0(c_H(E)) \to \mathrm{Aut}_{\mathcal{X}^{\mathrm{rig}}}(\rho(E))$, after first identifying elements of $Z(H)$ with the corresponding right multiplication automorphism of $E$. Then the automorphism group $\mathrm{Aut}_{\mathcal{X}^{\mathrm{rig}}}(\rho(E))$ contains $\overline{Z(H)}$ as a finite index central subgroup.
\end{prop}

\begin{proof}
    It is possible to prove this lemma by using Theorem \ref{thm:gms-etale-local-structure} to reduce the problem to GIT, however we elect to give a more direct proof.

    Fix an adapted character $\chi$ of $L$. Since $\mathcal{I}_{\mathrm{gr}, E}$ can be identified with $H^0(E \times^H U) \subset H^0(c_H(E))$, we may identify
    $$ \mathrm{Aut}_{\mathcal{X}^{\mathrm{rig}}}(\rho(E)) = \frac{H^0(c_H(E))}{H^0(E \times^H U)}. $$
    Since $Z(H)$ is a central subgroup of $\mathrm{Aut}(E)$, $\overline{Z(H)}$ is central in $\mathrm{Aut}_{\mathcal{X}^{\mathrm{rig}}}(\rho(E))$. On the other hand, the map $E \mapsto E_L = E \times^H L$ induces an injection $\mathrm{Aut}_{\mathcal{X}^{\mathrm{rig}}}(\rho(E))^{\circ} \hookrightarrow \mathrm{Aut}(E_L)^{\circ}$. As $E_L$ is a stable principal $L$-bundle, the quotient $\mathrm{Aut}(E_L)/Z(L)$ is finite, and so $\mathrm{Aut}(E_L)^{\circ}$ is a torus. Consequently, in order to show that the quotient $\mathrm{Aut}_{\mathcal{X}^{\mathrm{rig}}}(\rho(E))/\overline{Z(H)}$ is finite, it is enough to show that there exists a positive integer $N \in \Z^{>0}$ such that for any cocharacter $\eta : \mathbb{G}_m \to \mathrm{Aut}_{\mathcal{X}^{\mathrm{rig}}}(\rho(E))$, $\eta^N$ must factor through $\overline{Z(H)}$. We will complete the proof of the proposition by showing that this holds when $N$ is chosen such that for any cocharacter $\theta \in \mathfrak{X}_{\ast}(Z(L)^{\circ})$, $\theta^N$ admits a lift along the isogeny $\mathbb{G}_m^n \times Z(H)' \to Z(L)^{\circ}$ of Assumption \ref{ass:structure-group-bundles-strong}(3)(ii).

    Suppose we are given a cocharacter $\eta : \mathbb{G}_m \to \mathrm{Aut}_{\mathcal{X}^{\mathrm{rig}}}(\rho(E))$. The composition of $\eta$ with $\mathrm{Aut}_{\mathcal{X}^{\mathrm{rig}}}(\rho(E))^{\circ} \hookrightarrow \mathrm{Aut}(E_L)^{\circ}$ necessarily factors through $Z(L)^{\circ} \subset \mathrm{Aut}(E_L)$, so we may view $\eta$ as a cocharacter of $Z(L)^{\circ} \subset \mathrm{Aut}(E_L)$. We claim that this induces a reduction of structure group of $E$ to the subgroup $L_H(\eta)$. Indeed, since the underlying scheme of $H^0(E \times^H U) \cong H^0(E \times^H \mathfrak{u})$ is isomorphic to an affine space over $k$, $H^0(E \times^H U)$ is a split unipotent group, so we may pick a lift $\eta' : \mathbb{G}_m \to \mathrm{Aut}(E)$ of $\eta$ (cf. proof of Lemma \ref{lem:lifting_graded_points_in_center}). The formula
    $$ \eta'(t)(x) = x \cdot \eta_x'(t), \quad x \in E(k) $$
    then defines homomorphisms $\eta_x' : \mathbb{G}_m \to H$, each of which is a lift of $\eta : \mathbb{G}_m \to Z(L)^{\circ}$. By Lemma \ref{lem:lifts-of-cocharacters-L-to-H}, each $\eta_x'$ is conjugate to $\eta$ as cocharacters of $H$. On the other hand, we have $\eta_{xh}' = h^{-1}\eta_x'h$, so the $\eta_x'$ yield an $H$-equivariant morphism $E \to \mathfrak{X}_{\ast}(H) = \mathrm{Hom}_{k-\mathrm{gp}}(\mathbb{G}_m, H)$. It follows from \cite[Exposé XI, Corollaire 5.3 and Exposé XII, Remarques 5.7]{SGA3} that, as a scheme, we have
    $$ \mathfrak{X}_{\ast}(H) = \bigsqcup_{\psi \in \mathfrak{X}_{\ast}(T)/W} H/L_H(\psi) $$
    as a disjoint union of open and closed subschemes, corresponding to the conjugacy classes of cocharacters of $H$. Since $E$ is connected and since each $\eta_x'$ is conjugate to $\eta$, it follows that the morphism $E \to \mathfrak{X}_{\ast}(H)$ factors through $H/L_H(\eta)$. The morphism $E \to H/L_H(\eta)$ then corresponds to a section $C \to E \times^H H/L_H(\eta) = E/L_H(\eta)$, or in other words a reduction of structure group of $E$ to $L_H(\eta)$. This establishes the claim.

    Since $E$ admits a reduction of structure group to $L_H(\eta)$, by Lemma \ref{lem:principal-bundle-semistability-general-character} we have both $\langle \chi, \eta \rangle \geq 0$ and $\langle \chi, \eta^{-1} \rangle \geq 0$, i.e. $\langle \chi, \eta \rangle = 0$. Furthermore, since $\eta^N$ lifts to $\mathbb{G}_m^n \times Z(H)'$, we may write $\eta^N = (\lambda \circ \zeta) + \xi$ for cocharacters $\zeta \in \mathfrak{X}_{\ast}(\mathbb{G}_m^n)$ and $\xi \in \mathfrak{X}_{\ast}(Z(H)')$; we then have $L_H(\eta) = L_H(\lambda \circ \zeta)$ and
    $$ 0 = \langle \chi, \eta \rangle = \langle \chi, \lambda \circ \zeta \rangle = \langle \chi', \zeta \rangle. $$
    We claim that $\zeta$ must be the trivial cocharacter of $\mathbb{G}_m^n$. To see this, consider the subalgebra $\mathfrak{v} := \mathrm{Lie}(U \cap L_H(\lambda \circ \zeta)) \subset \mathfrak{u}$. $\mathfrak{v}$ must contain a vector $v$ which is $\chi'$-stable; if this were not the case then all elements of $\mathfrak{v}$ would be $\chi'$-unstable (since there are no strictly $\chi'$-semistable elements of $\mathfrak{u}$), so as in the proof of Proposition \ref{prop:semistability-adapted-character} there would exist an inclusion $\mathfrak{v} \subset W_m$ for some $m$, and hence an isomorphism $E_{m-1} \cong E_m$, a contradiction. For this vector $v \in \mathfrak{v}$, we have that $\lim_{t \to 0} \zeta(t) \cdot v$ exists in $\mathfrak{u}$, but that $\langle \chi', \zeta \rangle = 0$. This implies that $\zeta$ must be the trivial cocharacter of $\mathbb{G}_m^n$, whence $\eta^N : \mathbb{G}_m \to \mathrm{Aut}_{\mathcal{X}^{\mathrm{rig}}}(\rho(E))$ factors through the inclusion $\overline{Z(H)} \subset \mathrm{Aut}_{\mathcal{X}^{\mathrm{rig}}}(\rho(E))$. This completes the proof of the proposition.
\end{proof}

\begin{cor} \label{cor:adapted-character-tame-ms}
    Assume Assumption \ref{ass:structure-group-bundles-strong} holds for $H$. Then the good moduli space $\mathcal{X}^{\mathrm{rig, adpt}} \to M_H^{d,\delta,\mathrm{adpt}}$ associated to any adapted character of $L$ is a tame moduli space. In particular, the morphism $\pi_{\mathrm{adpt}} : \mathcal{X}^{\mathrm{adpt}} = \mathcal{X}_{\delta}^{\mathrm{adpt}} \to M_H^{d,\delta,\mathrm{adpt}}$ obtained by composing the gerbe $\rho : \mathcal{X}^{\mathrm{adpt}} \to \mathcal{X}^{\mathrm{rig, adpt}}$ with $\mathcal{X}^{\mathrm{rig, adpt}} \to M_H^{d,\delta,\mathrm{adpt}}$ is a coarse moduli space, and the $k$-points of the quasi-projective scheme $M_H^{d,\delta,\mathrm{adpt}}$ are in bijection with isomorphism classes of principal $H$-bundles $E$ over $C$ satisfying the following properties:
    \begin{enumerate}
        \item $E$ has type $d$.
        \item $E_L = E \times^H L$ is a stable principal $L$-bundle.
        \item There are equalities $\dim_k H^0(C, E \times^H \mathfrak{u}) = \dim_k H^0(C, E_L \times^L \mathfrak{u}) = \delta$.
        \item For each $m = 1, \dots, n$, $E_{m-1} \not \cong E_m$ as principal $H$-bundles.
    \end{enumerate}
\end{cor}

\begin{proof}
    We observe that if $E$ is a principal $H$-bundle corresponding to a $k$-point of $\mathcal{X}^{\mathrm{adpt}}$, then the intersection of $Z(H)'$ with $H^0(E \times^H U)$ in $\mathrm{Aut}(E) = H^0(c_H(E))$ is trivial, since the intersection of the torus $Z(H)'$ with the unipotent group $U$ is necessarily trivial. It follows that $\mathcal{I}_{\mathrm{gr}}$ and $Z(H)' \times \mathcal{X}^{\mathrm{adpt}}$ intersect trivially as subgroup stacks of the inertia of $\mathcal{X}^{\mathrm{adpt}}$, so $Z(H)' \times \mathcal{X}^{\mathrm{adpt}}$ descends along the gerbe $\rho$ to give a flat closed subgroup stack $Z(H)' \times \mathcal{X}^{\mathrm{rig,adpt}} \subset \mathcal{I}_{\mathcal{X}^{\mathrm{rig,adpt}}}$, and hence a well-defined rigidification $\mathcal{X}^{\mathrm{rig,adpt}} \!\!\! \fatslash Z(H)'$, which (by a similar argument used to prove Lemma \ref{lem:gms-Zrig}) admits $M_H^{d,\delta,\mathrm{adpt}}$ as a good moduli space. 
    
    On the other hand, by Proposition \ref{prop:finiteness-of-auts-of-Xadpt} the automorphism groups of the points of the rigidification $\mathcal{X}^{\mathrm{rig,adpt}} \!\!\! \fatslash Z(H)'$ are all of dimension $0$. It then follows from \cite[Proposition 7.7 and Lemma 11.10]{Alper-GMS} (cf. proof of Proposition \ref{prop:gms-from-line-bundle}) that the good moduli space $\mathcal{X}^{\mathrm{rig,adpt}} \!\!\! \fatslash Z(H)' \to M_H^{d, \delta, \mathrm{adpt}}$ is in fact a tame moduli space, and in particular a coarse moduli space. Since gerbes are coarse moduli space morphisms, it follows that $\pi_{\mathrm{adpt}} : \mathcal{X}^{\mathrm{adpt}} \to M_H^{d,\delta,\mathrm{adpt}}$ is a coarse moduli space. Finally, since $Z(H)'$ is an algebraic torus then the gerbe $\mathcal{X}^{\mathrm{rig,adpt}} \to \mathcal{X}^{\mathrm{rig,adpt}} \!\!\! \fatslash Z(H)'$ is a good moduli space morphism (cf. proof of Lemma \ref{lem:gms-Zrig}), so the good moduli space $\mathcal{X}^{\mathrm{rig, adpt}} \to M_H^{d,\delta,\mathrm{adpt}}$ is also a tame moduli space.
\end{proof}

\subsubsection{Adapted Stable Bundles for Parabolic Subgroups of $\mathrm{GL}_r$} \label{sec:parabolic-subgroups-gl}

As illustrated in Example \ref{ex:setting-parabolic-subgroup}, Assumption \ref{ass:structure-group-bundles-strong} is applicable in the case where $H$ is a parabolic subgroup of a connected reductive linear algebraic group $G$. Taking $G = \mathrm{GL}_r$, the moduli spaces $M_H^{d, \delta, \mathrm{adpt}}$ are moduli spaces of flags of vector bundles. Here we consider the resulting moduli spaces in more detail, and compare our moduli spaces to the moduli spaces of unstable vector bundles over $C$ of fixed Harder--Narasimhan type constructed by Jackson \cite{Jackson-HNL2} (where $\tau$ is chosen to be of length $\ell = 2$), Brambila Paz--Rios Sierra \cite{BPRS-HNL2-Proceedings} ($\ell = 2$), Qiao \cite{Qiao-HNL2, Qiao-Thesis} ($\ell = 2$, $\ell > 2$ respectively) and Hoskins--Jackson \cite{Hoskins-Jackson-QIS} ($\ell > 2$). In particular, we indicate how our existence result for the quasi-projective moduli spaces $M_H^{d,\delta,\mathrm{adpt}}$ (Theorem \ref{thm:moduli-principal-bundles-intro}/Corollary \ref{cor:adapted-character-tame-ms}) generalizes the existence results for moduli spaces of slope unstable vector bundles of fixed Harder--Narasimhan type arising from \cite{Jackson-HNL2, Qiao-HNL2, Qiao-Thesis, BPRS-HNL2-Proceedings}.

\begin{ex} \label{ex:comparing-adapt-tau-stable}
    In the setting of Example \ref{ex:setting-parabolic-subgroup}, take $G = \mathrm{GL}_r$ (where $r \geq 2$), and let $H$ be the standard parabolic subgroup of $G$ consisting of block upper-triangular matrices with diagonal blocks of sizes $r_1 \times r_1, \dots, r_{\ell} \times r_{\ell}$ (in this order). The data of a principal $H$-bundle $E$ over $C$ is then equivalent to the data of a vector bundle $\mathcal{E}$ over $C$ of rank $r$ together with a filtration
    \begin{equation} \label{eq:vector-bundle-filtration-hj}
        \mathcal{E}^{\bullet} : 0 = \mathcal{E}^{0} \subset \mathcal{E}^{1} \subset \cdots \subset \mathcal{E}^{\ell} = \mathcal{E}
    \end{equation}
    where each $\mathcal{E}^i$ is a subbundle of $\mathcal{E}$, and where each subquotient $\mathcal{E}^i/\mathcal{E}^{i-1}$ is a vector bundle of rank $r_i$. Fixing the type $d$ of $E$ is then equivalent to fixing the length $\ell$ of the filtration and the degrees $d_i = \deg(\mathcal{E}^i/\mathcal{E}^{i-1})$ of the subquotients of the filtration. In this case, we may also identify the space $H^0(E \times^H \mathfrak{u})$ with the space of shifted endomorphisms
    $$ \mathrm{End}_{-1}(\mathcal{E}^{\bullet}) := \{ \phi \in \mathrm{End}(\mathcal{E}) : \phi(\mathcal{E}^i) \subset \mathcal{E}^{i-1} \text{ for all } i = 1, \dots, \ell \}. $$
    
    The principal $L$-bundle $E_L = E \times^H L$ corresponds to the tuple of subquotients $(\mathcal{E}^1, \mathcal{E}^2/\mathcal{E}^1, \dots, \mathcal{E}^{\ell}/\mathcal{E}^{\ell-1})$, and the associated graded principal $H$-bundle $E_L \times^L H$ is given by the direct sum
    $$ \overline{\mathcal{E}} := \bigoplus_{i=1}^{\ell} \mathcal{E}^i/\mathcal{E}^{i-1} $$
    together with the associated split filtration $\overline{\mathcal{E}}^{\bullet}$; $E_L$ is (semi)stable as a principal $L$-bundle if and only if each subquotient $\mathcal{E}^i/\mathcal{E}^{i-1}$ is a slope (semi)stable vector bundle. In particular, if $E_L$ is semistable and if
    $$ \mu(\mathcal{E}^1) > \mu(\mathcal{E}^2/\mathcal{E}^1) > \cdots > \mu(\mathcal{E}^{\ell}/\mathcal{E}^{\ell-1}), $$
    then $\mathcal{E}$ is a slope unstable vector bundle whose Harder--Narasimhan filtration is the filtration \eqref{eq:vector-bundle-filtration-hj}. We remark that in each of \cite{Jackson-HNL2, Qiao-HNL2, Qiao-Thesis, BPRS-HNL2-Proceedings, Hoskins-Jackson-QIS}, the type $d$ is chosen to be a Harder--Narasimhan type $\tau$, and all of the Harder--Narasimhan subquotients of the bundles being parameterized are assumed to be stable (as opposed to being only semistable). Under the above identification, we also have an equality
    $$ H^0(E_L \times^L \mathfrak{u}) = \mathrm{End}_{-1}(\overline{\mathcal{E}}^{\bullet}) = \bigoplus_{1 \leq i < j \leq \ell} \mathrm{Hom}(\mathcal{E}^j/\mathcal{E}^{j-1}, \mathcal{E}^i/\mathcal{E}^{i-1}). $$

    As in Remark \ref{rmk:hj24-4.1}, set $\gamma_i = \varepsilon_i - \varepsilon_{i+1}$ and $\alpha_i = \gamma_{s_i}$, so that $H$ corresponds to the complementary set of simple roots $\Delta \setminus I = \{ \alpha_1, \dots, \alpha_{\ell-1} \}$. If we order the $\alpha_i$ in ascending order, so that we filter $U$ by its rows going top-to-bottom, the principal $H$-bundle $E_m$ corresponds to the vector bundle $\mathcal{E}/\mathcal{E}^m$ together with the filtration
    \begin{multline*}
        \mathcal{E}^1 \subset \mathcal{E}^1 \oplus \mathcal{E}^2/\mathcal{E}^2 \subset \cdots \subset \mathcal{E}^1 \oplus \mathcal{E}^2/\mathcal{E}^1 \oplus \cdots \oplus \mathcal{E}^m/\mathcal{E}^{m-1} \\ \subset (\mathcal{E}^1 \oplus \mathcal{E}^2/\mathcal{E}^1 \oplus \cdots \oplus \mathcal{E}^m/\mathcal{E}^{m-1}) \oplus \mathcal{E}^{m+1}/\mathcal{E}^m \\ \subset (\mathcal{E}^1 \oplus \mathcal{E}^2/\mathcal{E}^1 \oplus \cdots \oplus \mathcal{E}^m/\mathcal{E}^{m-1}) \oplus \mathcal{E}^{m+2}/\mathcal{E}^m \\ \subset \cdots \subset (\mathcal{E}^1 \oplus \mathcal{E}^2/\mathcal{E}^2 \oplus \cdots \oplus \mathcal{E}^m/\mathcal{E}^{m-1}) \oplus \mathcal{E}^{\ell}/\mathcal{E}^m.
    \end{multline*}
    Therefore, the condition that $E_{m-1} \not \cong E_m$ for all $m = 1, \dots, \ell-1$ is equivalent to the condition that for each $m = 2, \dots, \ell$, the inclusion $\mathcal{E}^{m-1}/\mathcal{E}^{m-2} \subset \mathcal{E}^m/\mathcal{E}^{m-2}$ is non-split.

    On the order hand, if we order the $\alpha_i$ in descending order, so that $\lambda_i$ is a multiple of the fundamental coweight of $\alpha_{\ell - i}$ for all $i = 1, \dots, \ell-1$, then the resulting subgroup $U_i \subset U$ consists of the matrices in the unipotent radical $U$ of $H$ whose off-diagonal blocks are all $0$, except for those blocks in column $\ell - i + 1$ (i.e. we are filtering $U$ by its columns going right-to-left). In this case, the condition that $E_{m-1} \not \cong E_m$ for all $m = 1, \dots, \ell-1$ is equivalent to the condition that $\mathcal{E}^m \not \cong \mathcal{E}^{m-1} \oplus \mathcal{E}^m/\mathcal{E}^{m-1}$ for all $m = 2, \dots, \ell$. In light of \cite[Lemma 7.1.4]{Jackson-HNL2} together with the fact that a non-split extension of simple vector bundles (in the sense that all endomorphisms are given by scalar multiplication) is simple, this is further equivalent to the condition that for each $m = 2, \dots, \ell$, the inclusion $\mathcal{E}^{m-1} \subset \mathcal{E}^m$ is non-split. When $\ell > 2$, this is the non-splitting criterion imposed on the \emph{$\widetilde{\tau}$-stable} sheaves considered in \cite{Hoskins-Jackson-QIS} (see Appendix \ref{app:hoskins-jackson-moduli-space}) and the {$(\tau, \delta)$-stable} vector bundles considered in \cite{Qiao-Thesis}; in the case $\ell = 2$, we also recover the non-splitting criterion appearing in each of \cite{Jackson-HNL2, Qiao-HNL2, BPRS-HNL2-Proceedings}.
\end{ex}

\begin{ex} \label{ex:qiao-cohomological-condition}
    Let $H$ be as in the previous example. Fix the type $d$ of the principal $H$-bundles $E$ being parameterized, such that $\mu(\mathcal{E}^i/\mathcal{E}^{i-1}) = \mu_i$ for each $i = 1, \dots, \ell$. Recall from Remark \ref{rmk:unipotent-auts-bundles-non-abelian} that if $\{e\} = U^{(q+1)} \lhd U^{(q)} \lhd \cdots \lhd U^{(1)} = U$ is a central series such that the extended series $\{e\} = U^{(q+1)} \lhd U^{(q)} \lhd \cdots \lhd U^{(1)} = U \lhd H$ is a normal series, then a sufficient condition for $\dim H^0(E \times^H \mathfrak{u}) = \dim H^0(E_L \times^L \mathfrak{u})$ (the pointwise condition on unipotent stabilizers arising from Corollary \ref{cor:univ-prop-X-bund}) to hold for the principal $H$-bundle $E$ is that for any irreducible $L$-representation appearing as a subrepresentation of $\mathrm{Lie}(U^{(m)}/U^{(m+1)})$ for some $m \geq 2$, the corresponding sheaf of sections $\Gamma_C(E_L \times^L V)$ satisfies $$ \mu(\Gamma_C(E_L \times^L V)) > 2p_a(C) - 2. $$

    Taking the central series to be the lower central series of $U$, the corresponding filtration of $\mathfrak{u}$ is given by filtering by off-diagonals, whence
    $$ \mathrm{Lie}(U^{(m)}/U^{(m+1)}) = \bigoplus_{j-i = m} \mathfrak{u}_{i,j}, \quad \mathfrak{u}_{i,j} := \bigoplus_{\stackrel{p = s_{i-1}+1, \dots, s_i}{q = s_{j-1}+1, \dots, s_j}} k \cdot E_{p,q} \cong \mathrm{Hom}(k^{r_j}, k^{r_i}). $$
    Each of the $\mathfrak{u}_{i,j}$ are irreducible as representations of $L$. Moreover, we have
    $$ \deg(\Gamma_C(E_L \times^L \mathfrak{u}_{i,j})) = r_ir_j \left( \mu(\mathcal{E}^i/\mathcal{E}^{i-1}) - \mu(\mathcal{E}^j/\mathcal{E}^{j-1}) \right) = r_ir_j(\mu_i - \mu_j). $$
    As such, $\mu(\Gamma_C(E_L \times^L V)) > 2p_a(C) - 2$ for all irreducible subrepresentations $V \subset \mathrm{Lie}(U^{(m)}/U^{(m+1)})$ for some $m \geq 2$ if and only if
    $$ \mu_i - \mu_j > 2p_a(C) - 2 $$
    for all pairs $(i,j)$ with $j - i \geq 2$. If $d$ is chosen to be a Harder--Narasimhan type, so that $\mu_1 > \cdots > \mu_{\ell}$, this is further equivalent to
    $$ \mu_i - \mu_{i+2} > 2p_a(C) - 2 $$
    for all $i = 1, \dots, \ell-2$, which is \cite[Condition 6.6.2]{Qiao-Thesis}.
\end{ex}

When $\ell = 2$, the condition $\dim H^0(E \times^H \mathfrak{u}) = \dim H^0(E_L \times^L \mathfrak{u}) = \delta$ is equivalent to the condition on unipotent stabilizers imposed on the vector bundles parameterized in \cite{Jackson-HNL2} and \cite{BPRS-HNL2-Proceedings}, and is the pointwise condition on unipotent stabilizers imposed in \cite{Qiao-HNL2}. For general $\ell \geq 2$, the second and third conditions appearing in the statement of Corollary \ref{cor:univ-prop-X-bund} collectively simplifies and extends the final conditions of \cite[Definition 2.3]{Qiao-HNL2} and \cite[Definition 6.6.15]{Qiao-Thesis} imposed on families of $(\tau, \delta)$-stable vector bundles (for Qiao's existence results of moduli spaces of $(\tau, \delta)$-stable vector bundles, \cite[Condition 6.6.2]{Qiao-Thesis} either automatically holds (if $\ell = 2$) or is assumed to hold (if $\ell > 2$), so Remark \ref{rmk:dropping-conditions-L-stabs} is applicable). We conclude that the existence result for the moduli spaces $M_H^{d, \delta, \mathrm{adpt}}$ generalizes the existence results for moduli spaces of slope unstable vector bundles of fixed Harder--Narasimhan type arising from \cite{Jackson-HNL2, Qiao-HNL2, Qiao-Thesis, BPRS-HNL2-Proceedings}.
    
The conditions on pointwise unipotent stabilizers imposed on $\widetilde{\tau}$-stable vector bundles in \cite{Hoskins-Jackson-QIS} do not necessarily imply that the condition $\dim H^0(E \times^H \mathfrak{u}) = \dim H^0(E_L \times^L \mathfrak{u}) = \delta$ holds; Hoskins--Jackson implement a blow-up procedure as part of their non-reductive GIT construction (to handle non-constant unipotent stabilizer dimensions), which allows their resulting moduli space to parameterize unstable vector bundles for which $\dim H^0(E \times^H \mathfrak{u}) < \dim H^0(E_L \times^L \mathfrak{u})$ can occur. However, a \emph{generic} $\widetilde{\tau}$-stable vector bundle satisfies $\dim H^0(E \times^H \mathfrak{u}) = \dim H^0(E_L \times^L \mathfrak{u})$, as shown in Appendix \ref{app:hoskins-jackson-moduli-space}. As such, a generic $\widetilde{\tau}$-stable vector bundle will be classified by some moduli space $M_H^{d, \delta, \mathrm{adpt}}$.

\appendix

\section{Comparison with the Moduli Space of Hoskins--Jackson} \label{app:hoskins-jackson-moduli-space}

In this appendix, we make free use of the conventions, notation and terminology of \cite{Hoskins-Jackson-QIS}. 

Let $k$ be an algebraically closed field of characteristic $0$, and let $(B, \mathcal{O}_B(1))$ be a polarized projective $k$-scheme. Let $\tau = (\tau_1, \dots, \tau_{\ell})$ be a fixed Harder--Narasimhan type for Gieseker unstable coherent sheaves on $B$.

We first recall Hoskins--Jackson's notion of a $\widetilde{\tau}$-stable coherent sheaf on $B$, where $\widetilde{\tau} = (\tau, d)$ is a fixed refinement of the Harder--Narasimhan type $\tau$ (\cite[Definition 6.1]{Hoskins-Jackson-QIS}).

\begin{defi}
    Let $\tau$ be as above, and let $d = (d^{[1]}, \dots, d^{[\ell]}, d^{[1,2]}, \dots d^{[\ell-1, \ell]})$ be a fixed sequence of non-negative integers. Write $\widetilde{\tau} = (\tau, d)$. Let $\mathcal{F}$ be a coherent sheaf on $B$ of Gieseker Harder--Narasimhan type $\tau$, with Gieseker Harder--Narasimhan filtration
    $$ \mathcal{F}^{(\bullet)} : 0 = \mathcal{F}^{(0)} \subset \mathcal{F}^{(1)} \subset \cdots \subset \mathcal{F}^{(\ell)} = \mathcal{F}. $$
    
    $\mathcal{F}$ is said to be \emph{$\widetilde{\tau}$-stable} if the following conditions hold:

    \begin{enumerate}
        \item For each $i = 1, \dots, \ell$, the subquotient $\mathcal{F}_i := \mathcal{F}^{(i)}/\mathcal{F}^{(i-1)}$ is Gieseker stable.
        \item For each $i = 2, \dots, \ell$, the inclusion $\mathcal{F}^{(i-1)} \subset \mathcal{F}^{(i)}$ is non-split.
        \item For each $i = 1, \dots, \ell$, $\dim \mathrm{Hom}(\mathcal{F}/\mathcal{F}^{(i)}, \mathcal{F}^{(i)}) = d^{[i]}$.
        \item For each $i = 2, \dots, \ell$, $\dim \mathrm{Hom}(\mathcal{F}/\mathcal{F}^{(i)}, \mathcal{F}^{(i-1)}) = d^{[i-1,i]}$.
    \end{enumerate}
\end{defi}

Fix the refinement $\widetilde{\tau}$. Let $W = \overline{p^{-1}(Z_{\min}^d)}$ be as in the proof of \cite[Theorem 6.2]{Hoskins-Jackson-QIS} (the scheme $W$ is the subscheme of an appropriately defined Grothendieck Quot scheme). $W$ is acted on by a parabolic subgroup $P_{\tau} \subset \mathrm{SL}_N$, and has the property that the isomorphism classes of $\widetilde{\tau}$-stable coherent sheaves on $B$ are in bijection with the $P_{\tau}$-orbits of the ``quotienting-in-stages stable locus'' $W^{P_{\tau}-qs}$ (cf. \cite[Definition 4.11]{Hoskins-Jackson-QIS}). Hoskins--Jackson's moduli space of $\widetilde{\tau}$-stable sheaves is constructed as a non-reductive GIT quotient of $W^{P_{\tau}-qs}$ by $P_{\tau}$ (the resulting $P_{\tau}$-quotient is a geometric good quotient).

Our objective is to establish the following result.

\begin{prop} \label{app-prop:generic-unipotent-stabilizers}
    If $W^{P_{\tau}-qs}$ is non-empty, then for any coherent sheaf $\mathcal{F}$ corresponding to a generic point of $W^{P_{\tau}-qs}$ with Gieseker Harder--Narasimhan filtration 
    $$ \mathcal{F}^{(\bullet)} : 0 = \mathcal{F}^{(0)} \subset \mathcal{F}^{(1)} \subset \cdots \subset \mathcal{F}^{(\ell)} = \mathcal{F}, $$
    there are equalities 
    $$ \dim \mathrm{End}_{-1}(\mathcal{F}) = \dim \mathrm{End}_{-1}(\mathrm{gr}^{\mathrm{HN}}(\mathcal{F})) = \sum_{i=1}^{\ell-2}(d^{[i]} - d^{[i,i+1]}) + d^{[\ell-1]}, $$ 
    where $\mathrm{End}_{-1}(\mathcal{F}) = \{ \phi \in \mathrm{End}(\mathcal{F}) : \phi(\mathcal{F}^{(i)}) \subset \mathcal{F}^{(i-1)}\}$ and where $\mathrm{gr}^{\mathrm{HN}}(\mathcal{F}) := \bigoplus_{i=1}^{\ell} \mathcal{F}_i$ (endowed with the induced split filtration; here $\mathcal{F}_i := \mathcal{F}^{(i)}/\mathcal{F}^{(i-1)}$).
\end{prop}

Before proving Proposition \ref{app-prop:generic-unipotent-stabilizers}, we first establish the following lemma.

\begin{lem} \label{app-lem:sheaves-constant-dimension}
    Let $\mathcal{F}$ be a coherent sheaf on $B$ with Gieseker Harder--Narasimhan filtration
    $$ \mathcal{F}^{(\bullet)} : 0 = \mathcal{F}^{(0)} \subset \mathcal{F}^{(1)} \subset \cdots \subset \mathcal{F}^{(\ell)} = \mathcal{F}. $$
    Assume that the following conditions hold:
    \begin{enumerate}
        \item For all $i = 1, \dots, \ell-1$, 
        $$ \mathrm{Hom}(\mathcal{F}/\mathcal{F}^{(i)}, \mathcal{F}^{(i)}) = \dim \mathrm{Hom}(\mathcal{F}_{i+1} \oplus \cdots \oplus \mathcal{F}_{\ell}, \mathcal{F}_1 \oplus \cdots \oplus \mathcal{F}_{i}) =: d^{[i]}. $$
        \item For all $i = 2, \dots, \ell-1$, 
        $$ \dim \mathrm{Hom}(\mathcal{F}/\mathcal{F}^{(i)}, \mathcal{F}^{(i-1)}) = \dim \mathrm{Hom}(\mathcal{F}_{i+1} \oplus \cdots \oplus \mathcal{F}_{\ell}, \mathcal{F}_1 \oplus \cdots \oplus \mathcal{F}_{i-1}) =: d^{[i-1,i]}. $$
    \end{enumerate}

    Then there are equalities 
    $$ \dim \mathrm{End}_{-1}(\mathcal{F}) = \dim \mathrm{End}_{-1}(\mathrm{gr}^{\mathrm{HN}}(\mathcal{F})) = \sum_{i=1}^{\ell-2}(d^{[i]} - d^{[i,i+1]}) + d^{[\ell-1]}. $$ 
\end{lem}

\begin{proof}
    We break the proof down into several steps.

    \emph{Step 1.} We show that for all $i = 2, \dots, \ell-1$ that the natural injection\footnote{Here and throughout we identify $\mathrm{Hom}(\mathcal{F}/\mathcal{F}^{(i)}, \mathcal{F}^{(i-1)}) \subset \mathrm{Hom}(\mathcal{F}, \mathcal{F}^{(i-1)})$ as the subspace of homomorphisms which vanish on $\mathcal{F}^{(i)} \subset \mathcal{F}$.}
    $$ \frac{\mathrm{Hom}(\mathcal{F}/\mathcal{F}^{(i-1)}, \mathcal{F}^{(i-1)})}{\mathrm{Hom}(\mathcal{F}/\mathcal{F}^{(i)}, \mathcal{F}^{(i-1)})} \hookrightarrow \mathrm{Hom}(\mathcal{F}_i, \mathcal{F}^{(i-1)}) $$
    induced by restricting along $\mathcal{F}_i \subset \mathcal{F}/\mathcal{F}^{(i-1)}$ is an isomorphism. Indeed, the left hand side has dimension $d^{[i-1]} - d^{[i-1,i]}$. On the other hand, by using the left-exactness of $\mathrm{Hom}(\mathcal{F}_i,-)$ and by considering the short exact sequences
    \[\begin{tikzcd}[ampersand replacement=\&]
	0 \& {\mathcal{F}^{(j-1)}} \& {\mathcal{F}^{(j)}} \& {\mathcal{F}_j} \& 0
	\arrow[from=1-1, to=1-2]
	\arrow[from=1-2, to=1-3]
	\arrow[from=1-3, to=1-4]
	\arrow[from=1-4, to=1-5]
    \end{tikzcd}\]
    for $j = 2, \dots, i-1$, we have $\dim \mathrm{Hom}(\mathcal{F}_i, \mathcal{F}^{(i-1)}) \leq \dim \mathrm{Hom}(\mathcal{F}_i, \mathcal{F}_1 \oplus \cdots \oplus \mathcal{F}_{i-1})$, and this latter space has dimension $d^{[i-1]} - d^{[i-1,i]}$.

    \emph{Step 2.} We show that for all $2 \leq j < i \leq \ell$, the natural (injective) restriction map
    $$ \frac{\mathrm{Hom}(\mathcal{F}^{(i)}/\mathcal{F}^{(j-1)}, \mathcal{F}^{(j-1)})}{\mathrm{Hom}(\mathcal{F}^{(i)}/\mathcal{F}^{(j)}, \mathcal{F}^{(j-1)})} \hookrightarrow \mathrm{Hom}(\mathcal{F}_j, \mathcal{F}^{(j-1)}) $$
    is an isomorphism. Indeed, this follows from \emph{Step 1}, since restricting along $\mathcal{F}_j \subset \mathcal{F}/\mathcal{F}^{(j)}$ is the same as restricting along $\mathcal{F}^{(i)}/\mathcal{F}^{(j-1)} \subset \mathcal{F}/\mathcal{F}^{(j-1)}$ followed by restricting along $\mathcal{F}_j \subset \mathcal{F}^{(i)}/\mathcal{F}^{(j-1)}$.

    \emph{Step 3.} We show that for all $i = 2, \dots, \ell$ that the restriction map
    $$ V^{(i)} := \frac{\mathrm{End}_{-1}(\mathcal{F}^{(i)})}{\sum_{j=2}^{i-1} \mathrm{Hom}(\mathcal{F}^{(i)}/\mathcal{F}^{(j)}, \mathcal{F}^{(j)})} \to \mathrm{Hom}(\mathcal{F}^{(2)}/\mathcal{F}^{(1)}, \mathcal{F}^{(1)}) $$
    is an isomorphism. Observe that each $V^{(i)}$ contains $\mathrm{Hom}(\mathcal{F}^{(i)}/\mathcal{F}^{(1)}, \mathcal{F}^{(1)})/\mathrm{Hom}(\mathcal{F}^{(i)}/\mathcal{F}^{(2)}, \mathcal{F}^{(1)})$ as a subspace. By induction on $i \geq 3$, one sees that the maps $V^{(i)} \to V^{(i-1)}$ induced by restricting along the inclusion $\mathcal{F}^{(i-1)} \subset \mathcal{F}^{(i)}$ are all injective. Moreover, by \emph{Step 2}, the restriction of the map $V^{(i)} \hookrightarrow V^{(2)} = \mathrm{End}_{-1}(\mathcal{F}^{(2)}) = \mathrm{Hom}(\mathcal{F}^{(2)}/\mathcal{F}^{(1)}, \mathcal{F}^{(1)})$ to the subspace $\mathrm{Hom}(\mathcal{F}^{(i)}/\mathcal{F}^{(1)}, \mathcal{F}^{(1)})/\mathrm{Hom}(\mathcal{F}^{(i)}/\mathcal{F}^{(2)}, \mathcal{F}^{(1)})$ is an isomorphism; this establishes the claim.

    \emph{Step 4.} We show that $\dim \mathrm{End}_{-1}(\mathcal{F}) = \dim \mathrm{End}_{-1}(\mathrm{gr}^{\mathrm{HN}}(\mathcal{F}))$. First, it follows from \emph{Step 3} that for each $i = 2, \dots, \ell$ that the natural inclusion
    $$ \sum_{j=1}^{i-1} \mathrm{Hom}(\mathcal{F}^{(i)}/\mathcal{F}^{(j)}, \mathcal{F}^{(j)}) \subset \mathrm{End}_{-1}(\mathcal{F}^{(i)}) $$
    is an isomorphism. It follows that $\mathrm{End}_{-1}(\mathcal{F}) = \mathrm{End}_{-1}(\mathcal{F}^{(\ell)})$ admits a composition series with composition factors
    $$ \frac{\mathrm{Hom}(\mathcal{F}^{(\ell)}/\mathcal{F}^{(i)}, \mathcal{F}^{(i)})}{\mathrm{Hom}(\mathcal{F}^{(\ell)}/\mathcal{F}^{(i+1)}, \mathcal{F}^{(i)})} \ \ (i = 1, \dots, \ell-2), \quad \mathrm{Hom}(\mathcal{F}^{(\ell)}/\mathcal{F}^{(\ell-1)}, \mathcal{F}^{(\ell-1)}). $$
    Consequently, we have
    $$ \dim \mathrm{End}_{-1}(\mathcal{F}) = \sum_{1 \leq j < i \leq \ell} \dim \mathrm{Hom}(\mathcal{F}_i, \mathcal{F}_j) = \dim \mathrm{End}_{-1}(\mathrm{gr}^{\mathrm{HN}}(\mathcal{F})). $$

    To complete the proof of the lemma, we compute the common dimension $\dim \mathrm{End}_{-1}(\mathcal{F}) = \dim \mathrm{End}_{-1}(\mathrm{gr}^{\mathrm{HN}}(\mathcal{F}))$: since by assumption we have $\dim \mathrm{Hom}(\mathcal{F}^{(\ell)}/\mathcal{F}^{(\ell-1)}, \mathcal{F}^{(\ell-1)}) = d^{[\ell-1]}$ and by \emph{Step 1} we have
    $\dim \left( \mathrm{Hom}(\mathcal{F}^{(\ell)}/\mathcal{F}^{(i)}, \mathcal{F}^{(i)}) / \mathrm{Hom}(\mathcal{F}^{(\ell)}/\mathcal{F}^{(i+1)}, \mathcal{F}^{(i)}) \right) = d^{[i]} - d^{[i,i+1]}$,
    from the composition series considered in \emph{Step 4} we obtain that this common dimension is equal to $\sum_{i=1}^{\ell-2}(d^{[i]} - d^{[i,i+1]}) + d^{[\ell-1]}$.
\end{proof}

\begin{proof}[Proof of Proposition \ref{app-prop:generic-unipotent-stabilizers}]
    Assume $W^{P_{\tau}-qs} \neq \emptyset$. Then a generic $k$-point $w \in W^{P_{\tau}-qs}(k)$ is a $k$-point of the dense open subscheme $p^{-1}(Z_{\min}^d) \cap W^{P_{\tau}-qs} \subset W^{P_{\tau}-qs}$. Write 
    $$ \mathcal{F}^{(\bullet)} : 0 = \mathcal{F}^{(0)} \subset \mathcal{F}^{(1)} \subset \cdots \subset \mathcal{F}^{(\ell)} = \mathcal{F} $$
    for the Gieseker Harder--Narasimhan filtration of the coherent sheaf $\mathcal{F}$ corresponding to $w$. The point $p(w)$ corresponds to the associated graded sheaf $\mathrm{gr}^{\mathrm{HN}}(\mathcal{F})$, with the induced split filtration; since $p(w)$ is a $k$-point of $Z_{\min}^d$, this implies the following:

    \begin{enumerate}
        \item for all $i = 1, \dots, \ell-1$, $\dim \mathrm{Hom}(\mathcal{F}_{i+1} \oplus \cdots \oplus \mathcal{F}_{\ell}, \mathcal{F}_1 \oplus \cdots \oplus \mathcal{F}_{i}) = d^{[i]}$; and
        \item for all $i = 2, \dots, \ell-1$, $\dim \mathrm{Hom}(\mathcal{F}_{i+1} \oplus \cdots \oplus \mathcal{F}_{\ell}, \mathcal{F}_1 \oplus \cdots \oplus \mathcal{F}_{i-1}) = d^{[i-1,i]}$.
    \end{enumerate}
    On the other hand, since $w \in W^{P_{\tau}-qs}(k)$ then the following also hold:

    \begin{enumerate}
        \item[(3)] for all $i = 1, \dots, \ell-1$, $\dim \mathrm{Hom}(\mathcal{F}/\mathcal{F}^{(i)}, \mathcal{F}^{(i)}) = d_{\min}^{[i]}$, where $d_{\min}^{[i]} := d_{\min}(W_{\min}^{[i],ss}, U^{[i]})$; and

        \item[(4)] for all $i = 2, \dots, \ell-1$, $\dim \mathrm{Hom}(\mathcal{F}/\mathcal{F}^{(i)}, \mathcal{F}^{(i-1)}) = d_{\min}^{[i-1,i]}$, where $d_{\min}^{[i-1,i]} := d_{\min}(W_{\min}^{[i],ss}, U^{[i-1,i]})$.
    \end{enumerate}

    Since $U^{[i]}$ and $U^{[i-1,i]}$-stabilizer dimensions are upper semi-continuous with respect to taking limits under cocharacters, and since we have $w \in W_{\min}^{[i],ss}(k)$ for all $i = 1, \dots, \ell-1$ by \cite[Lemma 4.34 and Proposition 4.38]{Hoskins-Jackson-QIS}, it follows that $d_{\min}^{[i]} = d^{[i]}$ for all $i = 1, \dots, \ell-1$, and that $d_{\min}^{[i-1,i]} = d^{[i-1,i]}$ for all $i = 2, \dots, \ell-1$. Consequently, Lemma \ref{app-lem:sheaves-constant-dimension} is applicable to $\mathcal{F}$, from which we conclude that
    $$ \dim \mathrm{End}_{-1}(\mathcal{F}) = \dim \mathrm{End}_{-1}(\mathrm{gr}^{\mathrm{HN}}(\mathcal{F})) = \sum_{i=1}^{\ell-2}(d^{[i]} - d^{[i,i+1]}) + d^{[\ell-1]}. $$
    This completes the proof of the result.
\end{proof}

\printbibliography
\end{document}